\documentclass[a4paper,openright,12pt,english]{report}
\usepackage{amsmath,amsfonts}
\usepackage{amssymb}
\usepackage{amsthm,mathrsfs}
\usepackage{epsfig}
\usepackage{verbatim}
\usepackage[english]{babel}
\usepackage{graphicx}
\usepackage{epstopdf}

\usepackage{color}

\usepackage {fancyhdr}

\usepackage{makeidx}

\def\vt{\vartheta}

\def\Cov{{\rm Cov\,}}

\newcommand{\dlines}{\displaylines}

\newcommand{\field}[1]{\mathbb{#1}}

\newcommand{\R}{\field{R}}

\renewcommand{\th}{\theta}

\newcommand{\X}{\mathscr{X}}

\newcommand{\T}{\field{T}}

\newcommand{\N}{\field{N}}

\newcommand{\Lc}{\mathcal{L}}

\newcommand{\Sc}{\mathcal S}

\newcommand{\Nc}{\mathcal{N}}

\newcommand{\Mc}{\mathcal{M}}

\newcommand{\Tb}{\field{T}}

\newcommand{\Z}{\field{Z}}

\newcommand{\bS}{\field{S}}

\newcommand{\Q}{\field{Q}}

\newcommand{\C}{\field{C}}

\newcommand{\Var}{{\rm Var}}

\newcommand{\e}{{\rm e}}

\newcommand{\F}{{\mathscr{F}}}

\newcommand{\B}{{\mathscr B}}

\newcommand{\var}{{\text{Var}}}

\newcommand{\cS}{{\mathbb S}}

\newcommand{\goto}{{\longrightarrow}}

\newcommand{\eps}{\varepsilon}

\newcommand{\for}{\; \mbox{for}\;}

\def\tin#1{\par\noindent\hskip3em\llap{#1\enspace}\ignorespaces}

\def\E{{\mathbb{ E}}}

\def\P{{\mathbb{P}}}

\def\F{{\mathscr{F}}}

\def\mE{\mathcal E}

\def\tr{\text{tr}}

\def\paref#1{(\ref{#1})}

\def\tfrac#1#2{{\textstyle\frac {#1}{#2}}}

\newtheorem{theorem}{Theorem}[section]

\newtheorem{rema}[theorem]{Remark}

\newtheorem{conjecture}[theorem]{Conjecture}

\newtheorem{cor}[theorem]{Corollary}

\newtheorem{definition}[theorem]{Definition}
\newtheorem{defn}[theorem]{Definition}

\newtheorem{example}[theorem]{Example}

\newtheorem{corollary}[theorem]{Corollary}

\newtheorem{lemma}[theorem]{Lemma}

\newtheorem{prop}[theorem]{Proposition}

\newtheorem{proposition}[theorem]{Proposition}

\newtheorem{remark}[theorem]{Remark}

\def\zsmile{\mathop{\smash{\zeta}\vphantom{g}}\limits^{\>\>\scriptscriptstyle\smile}}

\def\fismile{\mathop{\smash{\phi}\vphantom{g}}\limits^{\>\>\scriptscriptstyle\smile}}

\begin{document}
\titlepage
\thispagestyle{empty} \pagenumbering{Roman}

\begin{center}
\text{\sc \LARGE Universit\`a di  Roma Tor Vergata}
\begin{figure} [!htb]
\begin{center}
\includegraphics[width=4cm]{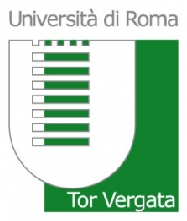}
\end{center}
\label{fig:dessin}
\end{figure}
\bigskip\bigskip\bigskip
\text{\sc Dipartimento di Matematica}

\vspace{1 cm}
\text{\sc \LARGE Tesi di Dottorato}

\vspace{10pt}


\vspace{0.5 cm}

\

\

\LARGE{\bf The geometry of spherical random fields}
\end{center}

\begin{center}
\bigskip
\textbf{\large }
\end{center}
\begin{center}
$\ $

\text{\Large Maurizia Rossi}
\bigskip\bigskip\bigskip\bigskip\bigskip

\bigskip
{\sc \Large Advisor}

{\Large Prof. Paolo Baldi}

\vspace{15pt}
{\sc \Large Co-Advisor}


{\Large Prof. Domenico Marinucci}

\vspace{35pt}
{\Large a.a. 2014/2015}
\end{center}
%
%
%
%



\thispagestyle{empty}
\begin{flushright}
\emph{\LARGE Alla mia famiglia}

\vspace{50pt}

\emph{\Large Milena, Lorenzo, Fabiana, Doriano}


\vspace{50pt}

\emph{\Large e Giada}

\vspace{200pt}

\end{flushright}

\pagenumbering{Roman} \linespread{1.5}
\pagestyle{fancy}
\textwidth 5.5in \textheight 7.5in \topmargin
0.25in \setcounter{tocdepth}{2}
\linespread{1.2}
\renewcommand{\chaptermark}[1]{\markboth{Chap.\ \textbf{\thechapter}\ -\ \textit{#1}}{}}
\renewcommand{\sectionmark}[1]{\markright{Sec.\ \textbf{\thesection}\ -\ \textit{#1}}}
\fancyhf{} \fancyfoot[CE,CO]{\thepage} \fancyhead[CO]{\rightmark}
\fancyhead[CE]{\leftmark}
\renewcommand{\headrulewidth}{0.5pt}
\renewcommand{\footrulewidth}{0.0pt}
\addtolength{\headheight}{0.5pt}
\fancypagestyle{plain}{\fancyhead{}\renewcommand{\headrulewidth}{0pt}}
\fancyhf{} \fancyfoot[CE,CO]{\thepage}
\fancyhead[CO]{\textit{Contents}} \fancyhead[CE]{\textit{Contents}}
\renewcommand{\headrulewidth}{0.5pt}
\renewcommand{\footrulewidth}{0.0pt}
\addtolength{\headheight}{0.5pt}
\fancypagestyle{plain}{\fancyhead{}\renewcommand{\headrulewidth}{0pt}}
\linespread{1.2}
\tableofcontents \clearpage
\fancyhf{} \fancyfoot[CE,CO]{\thepage}
\fancyhead[CO]{\textit{Introduction}}
\fancyhead[CE]{\textit{Introduction}}
\renewcommand{\headrulewidth}{0.5pt}
\renewcommand{\footrulewidth}{0.0pt}
\addtolength{\headheight}{0.5pt}
\fancypagestyle{plain}{\fancyhead{}\renewcommand{\headrulewidth}{0pt}}
\chapter*{Introduction}
\addcontentsline{toc}{chapter}{Introduction}

\begin{quote}
\emph{Siamo liberi di sceglierci ogni volta \\invece che lasciare troppe cose gi\`a decise \\ a scegliere per noi. (Tiromancino)}
\end{quote}


\

\noindent


\noindent This Ph.D.  thesis \emph{The geometry of spherical random fields}
collects research results obtained in these last three years.
The main purpose is  the study
of random fields indexed by the two-dimensional unit sphere $\mathbb S^2\subset \R^{3}$.

Let us first fix some probability space $(\Omega, \F, \P)$.
\begin{definition}\label{defIntro}
A random field $T$ on $\mathbb S^2$ \cite{dogiocam} is a  (possibly complex-valued) measurable map
$$
T: (\Omega \times \mathbb S^2, \F\otimes \B(\mathbb S^2)) \goto (\C, \B(\mathbb C))\ ; \qquad (\omega, x)\mapsto T_x(\omega)\ ,
$$
where $\B(\mathbb S^2)$ (resp. $\B(\C)$) denotes, as usual, the Borel $\sigma$-field on the sphere (resp. the field of complex numbers).
\end{definition}
Often in this work we will write $T(\cdot, x)$ instead of $T_x(\cdot)$.
Loosely speaking, $T$ is a collection
of r.v.'s $(T_x)_{x\in \mathbb S^2}$   indexed by the points of the sphere or, equivalently, it can be seen as  a r.v. $x\mapsto T_x$ taking values in some space of functions on $\mathbb S^2$.

In particular, we are interested in   rotationally  invariant or \emph{isotropic} random fields (e.g. see \cite{dogiocam, balditrapani, mauspin, malyarenkobook}): briefly we mean
that the random field $T$ and the ``rotated'' $T^g:=(T_{gx})_{x\in \mathbb S^2}$, have the same law for every $g\in SO(3)$ (for  details see Definition \ref{invarian}).
$SO(3)$, as usual,  denotes the group of all rotations  of $\R^3$ about the origin, under the operation of composition.

Spherical random fields naturally arise
in a wide class of instances in geophysics, atmospheric sciences, medical imaging
and cosmology.
The application domain we are interested in concerns the latter,  mainly in connection with the
 analysis of Cosmic Microwave Background (CMB) radiation.

We can image that physical experiments for CMB measure, for each point $x\in \mathbb S^2$,
an ellipse on $T_x \mathbb S^2$ - the tangent plane to the sphere at $x$  (\cite{dogiocam, malyarenko}).
The ``width'' of this ellipse  is related to the \emph{temperature} of this radiation whereas the other features (elongation and orientation) are collected in complex \emph{polarization} data.

Indeed, the modern random model for the absolute temperature of CMB is  an isotropic random field on the sphere, according to Definition \ref{defIntro} (see also Part 1).  Instead, to model the polarization of this radiation we need a more complex structure, namely  an invariant random field on the sphere taking values in some space of algebraic curves (the so-called spin random fields - see Part 3).

To test some features of the CMB -- such as whether it is a realization of a Gaussian field, is a question that has attracted a lot of attention in last years: asymptotic theory must hence be developed in the high-frequency sense (see Part 2).

Although our attention has been mostly attracted  by the spherical case,  in this work we decided to treat more general
situations whenever it is possible to extend our results from the sphere to other structures. Actually the interplay between the probabilistic aspects and the geometric ones produces sometimes fascinating insights.
We shall deal for instance with homogeneous spaces of a compact group (Part 1)
as well as vector bundles (Part 3).

This thesis can be split into three strongly correlated parts: namely Part 1: Gaussian fields, Part 2:
High-energy Gaussian eigenfunctions and Part 3: Spin random fields. It is nice  to note that
this work will turn out to have a ``circulant'' structure, in a sense to make clear below
(see Theorem \ref{intro1} and Theorem \ref{introFin}).

\subsection*{Related works}
\addcontentsline{toc}{subsection}{Related works}

Throughout the whole thesis, we refer to the following:
\begin{itemize}
\item  P. Baldi, M. Rossi.  \emph{On L\'evy's Brownian motion indexed by elements of compact groups}, Colloq. Math. 2013 (\cite{mauSO(3)});

\item P. Baldi, M. Rossi. \emph{Representation of Gaussian isotropic spin random fields}, Stoch. Processes Appl. 2014  (\cite{mauspin});

\item  D. Marinucci, M. Rossi. \emph{Stein-Malliavin approximations for nonlinear functionals of random eigenfunctions
on $\mathbb S^d$}, J. Funct. Anal. 2015  (\cite{maudom});

\item  D. Marinucci, G. Peccati, M. Rossi, I. Wigman. (2015+) \emph{Non-Universality of nodal length distribution for
 arithmetic random waves}, Preprint arXiv:1508.00353 (\cite{misto}).

\end{itemize}
However some of the results presented here are works  still in progress,
and should appear  in  forthcoming papers:
\begin{itemize}
\item  M. Rossi. (2015+) \emph{The Defect of random hyperspherical harmonics}, in preparation (\cite{mau});

\item  M. Rossi (2015) \emph{Level curves of spherical Gaussian eigenfunctions}, Preprint.
\end{itemize}
Moreover, we decided not to include some other works: for brevity \cite{simonmaudom} written with S. Campese and D. Marinucci, and to avoid heterogeneity
 \cite{ld, sld}, both
 joint works with P. Baldi and L. Caramellino.


\section*{Part 1: Gaussian fields}
\addcontentsline{toc}{section}{Part 1: Gaussian fields}

\subsection*{Chapters 1 \& 2}

Our investigation starts from a ``typical'' example of random field on the sphere,
i.e. P.~L\'evy's spherical Brownian motion. We mean a centered
Gaussian field $W=(W_x)_{x\in \mathbb S^2}$ whose covariance kernel $K$ is given by
\begin{equation}\label{covMBintro}
K(x,y) := \frac12 \left ( d(x,o) + d(y,o) - d(x,y) \right )\ ,\qquad x,y\in \mathbb S^2\ ,
\end{equation}
where $o$ is some fixed point on the sphere -- say the ``north pole'', and $d$ denotes the usual geodesic distance.
Note that \paref{covMBintro} implies $W_o=0$ (a.s.).

In particular, we recall  P.~L\'evy's idea \cite{levy} for constructing $W$ (see Example \ref{MB}).
 Consider a Gaussian white noise $S$ on the sphere, i.e. an isometry between
the space of square integrable functions on $\mathbb S^2$ and
finite-variance r.v.'s, defined on  $(\Omega, \F, \P)$.
P.~L\'evy defines a spherical Gaussian field $T$ as
\begin{equation}\label{LevyT}
T_x := \sqrt{\pi} S(1_{H_x})\ ,\quad x\in \mathbb S^2\ ,
\end{equation}
where $1_{H_x}$ denotes the indicator function of the half-sphere centered at $x$.
It turns out that $T$ is isotropic and
$$
\E[|T_x - T_y|^2] = d(x,y)\ ,\quad x,y\in \mathbb S^2\ .
$$
From now on, $\E$ denotes the expectation under the probability measure $\P$.

P.~L\'evy's spherical Brownian motion is hence the Gaussian field $W$ defined as
\begin{equation}\label{MB}
W_x := T_x - T_o\ , \quad x\in \mathbb S^2\ .
\end{equation}
 It is worth remarking
that the Brownian motion on the $m$-dimensional  unit sphere $\mathbb S^m\subset \R^{m+1}$ ($m>2$) is analogously defined and P.~L\'evy itself extended the previous construction to the higher dimensional situation.

Our first question is the following.

$\bullet$ Can we extend this technique to construct isotropic Gaussian fields $T$ on $\mathbb S^2$?

We answered this question in the first part of \cite{mauspin}.
We note that \paref{LevyT} can be rewritten as
$$
T_x := \sqrt{\pi} S(L_{g} 1_{H_{o}})\ ,\quad x\in \mathbb S^2\ ,
$$
where $g=g_x$ is any rotation matrix $\in SO(3)$ mapping the north pole $o$ to the point $x$ and $L_{g} 1_{H_{o}}$ is the function
defined as $L_{g} 1_{H_o}(y):= 1_{H_o}(g^{-1} y)$, $y\in \mathbb S^2$.
Actually,
$$\displaylines{
L_{g} 1_{H_o}(y)= 1_{H_o}(g^{-1} y) = 1_{g H_{o}}( y)= 1_{H_{g o}}( y)= 1_{H_{x}}( y)\ ,\quad y\in \mathbb S^2\ .
}$$
 $L$ coincides with the left regular representation (\ref{rappsin}) of $SO(3)$.

Consider now some homogeneous space $\X$ (see Definition \ref{hom}) of a compact group $G$
  (e.g. $\X= \mathbb S^2$ and $G=SO(3)$).
As for the spherical case, we have the following.
\begin{definition}
A random field $T$ on $\X$ is a  (possibly complex-valued) measurable map
$$
T: (\Omega \times \X, \F\otimes \B(\X)) \goto (\C, \B(\mathbb C))\ ; \qquad (\omega, x)\mapsto T_x(\omega)\ ,
$$
where $\B(\X)$ denotes, as usual, the Borel $\sigma$-field on $\X$.
\end{definition}
We develop P.L\'evy's construction to  obtain isotropic Gaussian fields on $\X$.
First we consider a Gaussian white noise $S$ on $\X$,  extended  to the space $L^2(\X)$ of square integrable \emph{complex} functions $f$.   $S$ respects the real character of $f$, i.e. $f$ is real if and only if $S(f)$ is real.
Let us fix once forever some point $x_0\in \X$ and denote $K$ the isotropy group of $x_0$, i.e. the closed subgroup of elements $g\in G$ fixing $x_0$.
Recall that $\X$ is isomorphic to the quotient space $G/K$.

To each $f\in L^2(\X)$ which is moreover left invariant w.r.t. the action of $K$, we
 associate
an isotropic complex-valued Gaussian field $T^f$ on $\X$ as
\begin{equation}\label{defTintro}
T^f_x := S(L_g f)\ ,\quad x\in \X\ ,
\end{equation}
where the function $L_g f$ is defined as
$L_g f(y) := f(g^{-1} y)$, $y\in \X$
 and $g\in G$ is any element that maps the  point $x_0$  to the point $x=g x_0$. $L$ coincides with  the left regular representation of $G$ (see \paref{rappsin}).

The  law of the field $T^f$ is completely characterized
by the associated positive definite function $\phi^f$ which is defined, for $g\in G$, as
\begin{equation}\label{phii}
\phi^f(g) := \Cov \left (T^f_x, T^f_{x_0} \right )=\langle L_g f, f\rangle
\ ,
\end{equation}
where $x$ is such that $gx_0 = x$. As usual, $\langle \cdot, \cdot \rangle$ denotes the inner product in $L^2(\X)$. Moreover we need
the ``relation'' function of $T^f$
\begin{equation}
\zeta^f(g) := \Cov \left (T^f_x, \overline{T^f_{x_0}} \right )=\langle L_g f, \overline{f}\rangle
\ ,
\end{equation}
where $\overline{T^f_{x_0}}$ (and  $\overline{f}$) denotes complex conjugation.

$\bullet$ Now we ask whether every  isotropic, complex-valued Gaussian
random field on $\X$ can be obtained with this construction.

The answer is \emph{no} in general (see Remark \ref{rem-sfera} and \paref{cic} for some counterexample). It is
however positive if we consider  isotropic \emph{real} Gaussian  fields $T$ on $\X$. Our first result is the following (Theorem \ref{real-general}).
\begin{theorem}\label{intro1}
Let $T$ be a real isotropic Gaussian field on $\X$. Then there exists a real   left-$K$-invariant function $f\in L^2(\X)$ such that $T$ and $T^f$ have the same law
$$
T \ \mathop{=}^{\mathcal L} \ T^f\ ,
$$
where $T^f$ is defined as \paref{defTintro}.
\end{theorem}
Actually, we prove that the associated positive definite function $\phi$ on the group $G$ of $T$ is of the form \paref{phii}. Precisely, if $\phi$ is defined  as before as $\phi(g) := \Cov \left (T^f_x, T^f_{x_0} \right )$, where $x= g x_0$, then we show (see Proposition \ref{real-sq}) that there exists a real function $f\in L^2(\X)$ such that
$$
\phi(g) = \langle L_g f, f \rangle\ ,\quad g\in G\ ,
$$
i.e. $T$ and $T^f$  have the same distribution.

\subsection*{Chapter 3}

Assume now that $\X$ is in addition endowed with some metric
$d$. Analogously for the spherical case, P.L\'evy's Brownian motion on the metric space $(\X, d)$ is defined as a real centered Gaussian  field on $\X$ which vanishes at some point $x_0\in\X$ and such that $\E[|X_x-X_y|^2]=d(x,y)$. By polarization, its covariance function  is
\begin{equation}\label{kernel MBintro}
K(x,y)=\frac{1}{2} \,( d(x,x_0) + d(y,x_0) - d(x,y) )\ .
\end{equation}
Note that it is not obvious that the Brownian motion exists on $(\X, d)$,
equivalently that the kernel \paref{kernel MBintro} is positive definite on $\X$.


Positive
definiteness of $K$ for $\X=\R^{m+1}$ and $d$ the Euclidean metric had been proved by Schoenberg  \cite{schoenberg} in 1938 and, as recalled above, P.L\'evy itself
constructed the Brownian motion on $\X=\bS^{m}$, here
 $d$ being the spherical distance.
Later Gangolli \cite{GAN:60} gave an analytical proof of the positive definiteness of the kernel \paref{kernel MBintro} for the same metric space $(\bS^{m},d)$, in a paper that dealt with this question for a large class of homogeneous spaces.

Finally Takenaka in \cite{TKU} proved the positive definiteness of the kernel \paref{kernel MBintro} for the Riemannian metric spaces
of constant sectional curvature equal to $-1,0$ or $1$, therefore adding the hyperbolic disk to the list. To be precise in the case of the hyperbolic space
$\mathcal{H}_m = \lbrace (x_0, x_1, \dots, x_m)\in \R^{m+1} :
x_1^2 + \dots x_m^2 - x_0^2 = 1 \rbrace $, the distance under consideration  is the unique, up to  multiplicative
constants, Riemannian distance that is invariant with respect to the action of $G=L_m$, the Lorentz group.

 $\bullet$ Now we ask the question of the existence of P.L\'evy's Brownian motion on $\X=SO(3)$, endowed with the Riemannian metric induced by the embedding $SO(3)\hookrightarrow \R^9$.

There are deep motivations for this choice, connected to the spin theory, which will be clearer in Part 3.

We answer this question in \cite{mauSO(3)} (Proposition \ref{kernel MB su SO(3)}).
\begin{prop}\label{propIntro}
The kernel $K$ in \paref{kernel MBintro} is not positive definite on $SO(3)$, endowed with the  Riemannian metric induced by the embedding $SO(3)\hookrightarrow \R^9$.
\end{prop}
This is somehow surprising
as, in particular, $SO(3)$ is locally isometric to $SU(2)$, where
positive definiteness of  $K$ is immediate since  isomorphic to
the unit hypersphere $\mathbb S^3$.

Proposition \ref{propIntro} moreover allows to prove the non existence of P. L\'evy's Brownian motion on the group $SO(n)$ of all rotations of $\R^{n+1}$ for $n>3$.  Actually, $SO(n)$ contains a closed subgroup  that is isomorphic
to $SO(3)$.
Indeed the same argument  holds also on e.g.
 the group $SU(n)$ of all $n\times n$ unitary matrices with determinant one, for  $n\ge 3$ (see Corollary \ref{kernel MB su SO(n)}).
Our method could be applied to investigate positive definitess of the Brownian kernel on  other compact Lie groups.

\section*{Part 2: High-energy Gaussian eigenfunctions}
\addcontentsline{toc}{section}{Part 2: High-energy Gaussian eigenfunctions}

\subsection*{Chapters 4, 5 \& 6}

As already briefly stated,  the investigation of spherical random fields has been strongly motivated by cosmological applications (e.g. concerning  CMB): the asymptotic analysis in this setting must be hence developed in the high-energy sense, as follows.

First recall that the eigenvalues of the Laplace-Beltrami operator $\Delta_{\mathbb S^2}$ on $\mathbb S^2$ are integers of the form $-\ell(\ell +1)$, $\ell\in \N$.

  Under Gaussianity, an isotropic random field $T$ on $\mathbb S^2$ can be decomposed in terms of  its random Fourier components $T_\ell$, $\ell\in \N$. The latter  are independent and isotropic centered Gaussian fields, whose covariance kernel is
\begin{equation}\label{ker2intro}
\E[T_\ell(x) T_\ell(y)]=P_\ell(\cos d(x,y))\ ,\quad x,y\in \mathbb S^2\ ,
\end{equation}
where $P_\ell$ is the $\ell$-th Legendre polynomial \cite{szego, dogiocam}
and $d(x,y)$ denotes the spherical distance between $x$ and $y$.

The following spectral representation holds \cite[Propositions 5.13]{dogiocam}
$$
T_x = \sum_{\ell\in \N} c_\ell T_\ell(x)\ ,
$$
where the series converges in $L^2(\Omega\times \mathbb S^2)$ and the nonnegative sequence $(c_\ell)_\ell$ is the power spectrum of the field \cite{dogiocam}.

 $T_\ell$ is known as the $\ell$-th \emph{Gaussian spherical  eigenfunction} or random spherical harmonic (see \paref{Telle} for a precise definition), indeed ``pathwise'' satisfies
$$
\Delta_{\mathbb S^2}T_\ell + \ell(\ell+1) T_\ell = 0\ .
$$
In this second part, we investigate  the high-energy behavior (i.e., as $\ell\to +\infty$) of $T_\ell$.   We are interested in the geometry of the $z$-\emph{excursion set} (\cite{adlertaylor} e.g.), which is defined  for $z\in \R$ as
\begin{equation}\label{excset}
A_\ell(z) := \lbrace x\in \mathbb S^2 : T_\ell(x) > z\rbrace\ .
\end{equation}
For instance, one can investigate the area  of $A_\ell(z)$, the length of the boundary $\partial A_\ell(z)$ -- that is the length of level curves $T^{-1}_\ell(z)$, and the Euler-Poincar\'e characteristic of these domains. For completeness, we recall that these three quantities correspond to the so-called Lipschitz-Killing curvatures on the sphere \cite{adlertaylor}.

Many authors have studied  properties of excursion sets of random fields on the sphere or other manifolds:
for instance, one can investigate the behavior of the excursion probability \cite{yimin},
i.e. as $z\to +\infty$
$$
\P\left ( \sup_{x\in \mathbb S^2} T_x > z\right )\ ,
$$
where $T$ is some random field on the sphere;
 (see also e.g.
\cite{adlertaylor, yimin-cheng, cheng, cheng2,  chenschwar, yimin-cheng,
fluct}).

It is worth remarking that  random spherical harmonics  have attracted great interest  also in other disciplines, such as  Mathematical Physics.
Indeed Berry's Random Wave Model (\cite{berry}) allows to compare - at least for ``generic'' chaotic Riemannian \emph{surfaces} $\mathcal M$ -  a \emph{deterministic} Laplace eigenfunction $f$ on $\mathcal M$ of some large eigenvalue $E$ to a ``typical''
instance of an isotropic, monochromatic \emph{random} wave with wavenumber
$\sqrt{E}$ (see also \cite{wigsurvey}). In view of this conjecture, much effort has been devoted  to $2$-dimensional manifolds such as the  torus $\mathbb T$ (see e.g. \cite{AmP}) and the sphere $\mathbb S^2$ (see e.g.  \cite{vale2}, \cite{vale3},  \cite{Nonlin},
 \cite{Wig}), as stated just above.
%
In this setting, the \emph{nodal} case corresponding to $z=0$ has received  the greatest attention.
Indeed nodal domains (the complement of the set where eigenfunctions are equal to zero) appear
in many problems in engineering, physics and the natural sciences: they
describe the sets that remain \emph{stationary} during vibrations, hence their importance
in such areas as musical instruments industry,  earthquake study and  astrophysics
(for further motivating details see \cite{wigsurvey}).

$\bullet$ In this thesis we want to investigate the geometry of excursion sets \paref{excset} of high-energy Gaussian eigenfunctions $T_\ell$ on the sphere.
The geometric features we are interested in can be written as nonlinear functionals of the random field itself (and its spatial derivatives).

\underline{Excursion area}

The area $S_\ell(z)$ of $z$-excursion sets \paref{excset}  can be written  as
\begin{equation*}
S_\ell(z) = \int_{\mathbb S^2} 1_{(z, +\infty)}(T_\ell(x))\,dx\ ,
\end{equation*}
where $1_{(z, +\infty)}$ is the indicator function of the interval $(z, +\infty)$. The expected value is simply computed to be
$\E[S_\ell(z)] = 4\pi (1 - \Phi(z))$, where $\Phi$ denotes the cumulative distribution function of a standard Gaussian r.v. The variance has been studied in \cite{wigexc, Nonlin, def}:
we have,  as $\ell \to +\infty$,
\begin{equation}\label{varintro}
\Var(S_\ell(z)) = z^2 \phi(z)^2 \cdot \frac{1}{\ell} + O\left (\frac{\log \ell}{\ell^2} \right )\ ,
\end{equation}
where $\phi$ is the standard Gaussian probability density function.
In particular, for $z\ne 0$, \paref{varintro} gives the exact asymptotic form of the
variance.

The nodal case corresponds to the Defect $D_\ell$, which is defined as
$$
D_\ell :=  \int_{\mathbb S^2} 1_{(0, +\infty)}(T_\ell(x))\,dx -  \int_{\mathbb S^2} 1_{(-\infty, 0)}(T_\ell(x))\,dx\ ,
$$
i.e. the difference between the measure of the positive and negative regions.
Note that $D_\ell = 2S_\ell(0) - 4\pi$. We have $\E[D_\ell]=0$ and from \cite{def}
\begin{equation}\label{varD}
\Var(D_\ell) = \frac{C}{\ell^2}( 1 +o(1))\ ,\quad \ell\to +\infty\ ,
\end{equation}
for some $C> \frac{32}{\sqrt{27}}$.

It is worth remarking that the Defect variance is of smaller order than  the non-nodal case.
This situation is similar to the \emph{cancellation phenomenon} observed by Berry in a different setting (\cite{berry}).

In \cite{Nonlin}  Central Limit Theorems  are proved for the excursion area:
\begin{align*}
\frac{S_\ell(z) - \E[S_\ell(z)]}{\sqrt{\Var(S_\ell(z))}}&\mathop{\goto}^{\mathcal L} Z\ , \qquad z\ne 0\ ,\\
\frac{D_\ell}{\sqrt{\Var(D_\ell)}}&\mathop{\goto}^{\mathcal L} Z\ ,
\end{align*}
$Z\sim \mathcal N(0,1)$ being a standard Gaussian r.v. and $\mathop{\goto}^{\mathcal L}$ denoting the convergence in distribution from now on. Often we will write $\mathop{\goto}^{d}$ instead of $\mathop{\goto}^{\mathcal L}$.

A CLT result  is ``only'' an asymptotic result with no information on the \emph{speed of convergence} to the limiting
distribution. More refined results indeed aim at the investigation of the asymptotic behaviour for
various probability metrics, such as Wasserstein, Kolmogorov and total variation  distances, see \paref{prob distance}. In this respect, a major development in the last few years has been provided by the
so-called \emph{fourth-moment} literature, which is summarized in the recent monograph \cite{noupebook}. In short, a
rapidly growing family of results is showing how to establish  bounds on probability
distances between multiple stochastic integrals and the  Gaussian distribution analyzing the fourth-moments/fourth cumulants alone (\cite{taqqu, simon, nou-pe, nou-pe2, simonS} e.g.).

$\bullet$ We establish a quantitative CLT for the excursion area of random spherical harmonics.

In \cite{maudom} we consider a more general situation, i.e.  nonlinear functionals of Gaussian eigenfunctions $(T_{\ell})_{\ell\in \N}$ on the $m$-dimensional unit sphere $\mathbb S^m$, $m\ge 2$.
The eigenvalues of the Laplace-Beltrami operator $\Delta_{\mathbb S^m}$
on $\mathbb S^m$ are integers of the form $-\ell(\ell+m-1)$, $\ell\in \N$.
The $\ell$-th Gaussian eigenfunction $T_\ell$ on $\mathbb S^m$  \paref{Telle}
$$
\Delta_{\mathbb S^m} T_\ell + \ell(\ell +m-1) T_\ell = 0\ , \quad a.s.
$$
is a centered isotropic Gaussian field with covariance function
\begin{equation}\label{kermintro}
\E[T_\ell(x) T_\ell(y)] = G_{\ell;m} (\cos d(x,y))\ ,
\end{equation}
where $G_{\ell;m}$ denotes the normalized Gegenbauer polynomial \cite{szego} and $d$ the usual distance on the $m$-sphere.

 Precisely, we consider
sequences of r.v.'s of the form
$$
S_\ell(M) := \int_{\mathbb S^m} M(T_\ell(x))\,dx\ ,
$$
where $M:\R\to \R$ is some measurable function such that $\E[M(Z)^2]<+\infty$,
$Z\sim \mathcal N(0,1)$. Note that if
we choose $M=1_{(z,+\infty)}$, then $S_\ell(M) = S_\ell(z)$ the excursion volume of random hyperspherical harmonics, i.e. the empirical measure of the set where eigenfunctions lie upon the level $z$.

The main idea for our proof is first to develop $S_\ell(M)$ into Wiener chaoses, i.e. as a series
in $L^2(\P)$ of the type \paref{chaos exp}
$$
S_\ell(M) = \sum_{q=0}^{+\infty} \frac{J_q(M)}{q!} \underbrace{\int_{\mathbb S^m} H_q(T_\ell(x))\,dx}_{:= h_{\ell,q;m}}\ ,
$$
where $H_q$ denotes the $q$-th Hermite polynomial \paref{hermite} (see also \cite{szego, noupebook}) and $J_q(M):= \E[M(Z)H_q(Z)]$. Then, we study the asymptotic behavior of each summand $h_{\ell,q;m}$ of the previous series by means of a careful investigation of asymptotic variances (see Proposition \ref{varianza}) and the Fourth Moment Theorem \paref{th}: we are hence able to prove a quantitative CLT for $h_{\ell,q;m}$ (Proposition \ref{teo1}) in Wasserstein distance \paref{prob distance}. To be more precise, we can restrict ourselves to even integers $\ell$  (see the related discussion in Chapter $5$).

It turns out that, if the projection of $M(Z)$ onto the second order Wiener chaos is not zero ($J_2(M)\ne 0$), then this component dominates the whole series, i.e., as $\ell\to +\infty$
$$
\frac{S_\ell(M) - \E[S_\ell(M)]}{\sqrt{\Var(S_\ell(M))}} = \frac{\frac{J_2(M)}{2}h_{\ell,2;m}}{\sqrt{\Var(S_\ell(M))}} + o_\P(1)\ .
$$
We can therefore  prove the following (Theorem \ref{general}).
\begin{theorem}
If $J_2(M)\ne 0$, then
$$
d_W\left( \frac{S_\ell(M) - \E[S_\ell(M)]}{\sqrt{\Var(S_\ell(M))}}, Z      \right) = O(\ell^{-1/2})\ ,
$$
that gives the rate of convergence, as $\ell\to +\infty$, to the standard Gaussian distribution in Wasserstein distance $d_W$.
\end{theorem}
Moreover  if $M=1_{(z,+\infty)}$, then it easy to compute that $J_2(M)\ne 0 \iff z\ne 0$. The following corollary is therefore immediate (Theorem \ref{mainteo}).
\begin{cor}
If $z\ne 0$, then
$$
d_W\left( \frac{S_\ell(z) - \E[S_\ell(z)]}{\sqrt{\Var(S_\ell(z))}}, Z      \right) = O(\ell^{-1/2})\ ,\quad \ell\to \infty\ .
$$
\end{cor}
We have just obtained a quantitative CLT for the excursione volume of random hyperspherical eigenfunctions in the non-nodal case.

The nodal case which correspond to the Defect $D_\ell$ requires harder work, since it is no longer true that the second chaotic component dominates.
In \cite{mau} we show first the exact rate for the Defect variance (Theorem  \ref{thdefvar}).
\begin{theorem}
 For $m>2$, as $\ell\to +\infty$
$$
\Var(D_\ell) = \frac{C_m}{\ell^m}(1 +o(1))\ ,
$$
where $C_m>0$ is some constant depending on the dimension $m$.
\end{theorem}
Remark that the case $m=2$ has been solved in \cite{def}.
Moreover we prove CLT results for the Defect, in the high-energy limit (Theorem \ref{thDef}).
\begin{theorem}
  For $m\ne 3,4,5$ we have, as $\ell\to +\infty$,
$$
\frac{D_\ell}{\sqrt{\Var(D_\ell)}}\mathop{\to}^{\mathcal L} Z\ ,
$$
where as before $Z\sim \mathcal N(0,1)$.
\end{theorem}
The remaining cases ($m=3,4,5$) require a precise investigation of fourth-order cumulant of r.v.'s $h_{\ell,3;m}$ and are still work in progress in \cite{mau}, where moreover the \emph{quantitative} CLT for the Defect in the Wasserstein distance will be proved:
$$
d_W \left ( \frac{D_\ell}{\sqrt{\Var(D_\ell)}}, Z\right ) =O\left( \ell^{-1/4} \right) \ .
$$

\subsection*{Chapters 7 \& 8}

\underline{Length of level curves}

The length $\mathcal L_\ell(z)$ of level curves $T_\ell^{-1}(z)$ can be formally written as
$$
\mathcal L_\ell(z) = \int_{\mathbb S^2} \delta_z(T_\ell(x))\| \nabla T_\ell(x)\|\, dx\ ,
$$
where $\delta_z$ denotes the Dirac mass in $z$, $\| \cdot \|$ the norm in $\R^2$ and $\nabla$ the gradient.
The expected value is \cite{Wig, wigsurvey}
$$
\E[\mathcal L_\ell(z)] = 4\pi\cdot \frac{\e^{-z^2/2}}{2\sqrt{2}} \sqrt{\ell(\ell+1)}
$$
and for the variance we have \cite{wigsurvey} if $z\ne 0$
$$
\Var\left (\mathcal L_\ell(z)\right ) \sim C\e^{-z^2}z^4\cdot \ell\ ,\quad \ell\to +\infty\ ,
$$
for some $C>0$.  I.Wigman computed (private calculations) moreover the exact constant
$$
C=\frac{\pi^2}{2}\ .$$
For the nodal length $\mathcal L_\ell:= \mathcal L_\ell(0)$ we have   \cite{Wig}
\begin{equation}\label{varN}
\Var(\mathcal L_\ell) \sim \frac{1}{32} \cdot \log \ell\ ,\quad \ell\to +\infty\ .
\end{equation}
Here also we observe the different behavior for asymptotic variances: in the nodal case it is of smaller order (logarithmic)  rather  than  what should be the  ``natural'' scaling $\approx \ell$. This is due to some analytic cancellation in the asymptotic expansion of the length variance which occurs only for $z=0$.
This phenomenon has been called ``obscure'' Berry's cancellation (see \cite{Berry2002, Wig}).

$\bullet$ We investigate the asymptotic distribution of the length of level curves.

We try to answer this question in \cite{mistosfera}: here also we first compute the chaotic expansion of $\mathcal L_\ell(z)$ (Proposition \ref{teoexpS}). Let us denote  $(\partial_1 \widetilde T_\ell, \partial_2 \widetilde T_\ell)$  the normalized gradient, i.e. $ \sqrt{\frac{2}{\ell(\ell+1)}}\nabla T_\ell$.
 \begin{prop}
The chaotic expansion of $\mathcal L_\ell(z)$ is
\begin{eqnarray}\label{chaosexpintro}
\mathcal L_\ell(z) &=& \E [\mathcal L_\ell(z)] + \sqrt{\frac{\ell(\ell+1)}{2}}\sum_{q=2}^{+\infty}\sum_{u=0}^{q}\sum_{k=0}^{u}
\frac{\alpha _{k,u-k}\beta _{q-u}(z)
}{(k)!(u-k)!(q-u)!} \!\!\times\\
&&\hspace{4.5cm} \times \int_{\mathbb S^2}\!\! H_{q-u}(T_\ell (x))
H_{k}(\partial_1 \widetilde T_\ell (x))H_{u-k}(\partial_2
\widetilde T_\ell (x))\,dx,\notag
\end{eqnarray}
where the series converges in $L^2(\P)$ and $(\beta_l(z))_l$, $(\alpha_{n,m})_{n,m}$ are respectively the  chaotic coefficients of the Dirac mass
in $z$ and the norm in $\R^2$.
\end{prop}
By some computations, it is possible to give an exact formula for the second chaotic component (Proposition \ref{2sfera})
$$
\text{proj}(\mathcal L_\ell(z)|C_2) = \sqrt{\frac{\ell(\ell+1)}{2}}
 \frac{1}{4}  \e^{-z^2/2} z^2 \int_{\mathbb S^2} H_2(T_\ell(x))\,dx=
\sqrt{\frac{\ell(\ell+1)}{2}}
 \frac{1}{4}  \e^{-z^2/2} z^2 h_{\ell,2;2}\ .
$$
Note that, as for the excursion area, the second component vanishes if and only if $z=0$.

Computing the exact variance of $\text{proj}(\mathcal L_\ell(z)|C_2) $, it turns out that, for $z\ne 0$
$$
\lim_{\ell\to +\infty} \frac{\Var(\mathcal L_\ell(z))}{\Var(\text{proj}(\mathcal L_\ell(z)|C_2) )}=1\ ,
$$
so that, as $\ell\to \infty$,
$$
\frac{\mathcal L_\ell(z) - \E[\mathcal L_\ell(z)]}{\sqrt{\Var(\mathcal L_\ell(z))}}= \frac{\text{proj}(\mathcal L_\ell(z)|C_2)}{\sqrt{\Var(\mathcal L_\ell(z))}}+o_\P(1)\ .
$$
This implies that the total length has the same asymptotic distribution of the second chaotic projection (Theorem \ref{main th S}).
\begin{theorem}
As $\ell\to +\infty$, if $z\ne 0$, we have
$$
\frac{\mathcal L_\ell(z) - \E[\mathcal L_\ell(z)]}{\sqrt{\Var(\mathcal L_\ell(z))}}\mathop{\to}^{\mathcal L} Z\ ,
$$
where $Z\sim \mathcal N(0,1)$.
\end{theorem}
The nodal case requires harder work, indeed it is not a simple task to derive an explicit expression for the fourth-order chaos and is still a work in progress. We can anticipate that the fourth-order chaotic projection dominates the whole nodal length and limit theorems will come hopefully soon (\cite{mistosfera}).

Furthermore we decided to investigate the nodal case on the standard $2$-torus $\mathbb T$: in \cite{misto} we prove a Non-Central Limit Theorem for
nodal lengths of arithmetic random waves (Theorem \ref{thm:lim dist sep}). The situation is analogous to the sphere: indeed the second chaotic component disappers and the fourth-order chaos dominates. The limit distribution is unexpectly non-Gaussian, indeed  it is a linear combination of
$H_2(Z_1)$ and $H_2(Z_2)$, where $H_2$ is the second Hermite polynomial and $Z_1, Z_2$ are i.i.d. standard Gaussian r.v.'s.

\underline{Euler-Poincar\'e characteristic}

The Euler-Poincar\'e characteristic of $z$-excursion set $\chi(A_\ell(z))$  for random spherical harmonics has been investigated in \cite{fluct}. In the quoted paper, a precise expression for the
asymptotic variance is proven, moreover the Gaussian Kinematic
Formula \cite{adlertaylor} gives immediately the expected value:
$$
\E[\chi(A_\ell(z))]= \sqrt{\frac{2}{\pi}}\e^{-z^2/2} z \frac{\ell(\ell+1)}{2}
+ 2(1 - \Phi(z))\ ,
$$
and
$$
\lim_{\ell\to +\infty} \ell^{-3}\Var(\chi(A_\ell(z)))= \frac{1}{4}(z^3 - z)^2 \phi(z)^2\ .
$$
The same phenomenon happens here: i.e. the nodal variance is of smaller order than the case $z\ne 0$.

For a unified formula for asymptotic variances of excursion area, length of level curves and Euler-Poincar\'e characteristic see (1.11) in \cite{fluct}.

Quantitative CLTs for $\chi(A_\ell(z)$ will be treated in forthcoming papers by the same authors V.Cammarota, D.Marinucci and I.Wigman.

\begin{remark}\rm
A careful investigation of previous results on asymptotic variances, suggests that there is a strict connection between Berry's cancellation phenomenon for Lipschitz-Killing curvatures and chaotic expansions. Indeed the unique  case which shows a different scaling for the variance as well as a zero second order chaotic component is the nodal one. Moreover, in situations analyzed above, there is always a single dominating chaotic projection: the second one at non-zero level and the fourth one in the nodal case.

We conjecture that this qualitative behaviour should be universal somehow: we mean that it should hold for all Lipschitz-Killing curvatures on every ``nice'' compact manifold.
\end{remark}

\section*{Part 3: Spin random fields}
\addcontentsline{toc}{section}{Part 3: Spin random fields}

\subsection*{Chapter 9}

As briefly stated above, in cosmology and astrophysics  spherical random fields are used
to model CMB data \cite{dogiocam}. More precisely, the \emph{temperature} of this radiation is seen as a single
realization of an isotropic random field on $\mathbb S^2$, whereas to model its \emph{polarization}
we need to introduce random fields which do not take ordinary scalar values but have a more complex
geometrical structure, the so-called spin random fields \cite{dogiocam, malyarenko, mauspin, claudiospin}. Roughly speaking, they can be seen
as random structures that at each point of the sphere take as a value a random ``curve''.

This family of random models is indexed by
integers $s\in \Z$: for instance, a spin-$0$ random field is simply a  spherical random field,
whereas a spin-$1$ random field takes as a value at the point $x\in \mathbb S^2$
a random tangent vector to the sphere at $x$.  The polarization of CMB is seen as a realization of a spin-$2$ random field.

From a mathematical point of view, spin random fields are random sections of particular
complex-line bundles on the sphere $\mathbb S^2$, whose construction can be interpreted
in terms of group representation concepts. These
 are special cases of so-called homogeneous vector bundle, which we handle in
 the second part of \cite{mauspin}. Briefly, given a compact topological group $G$ and an irreducible
 representation $\tau$ of its
 closed subgroup $K$, we can construct the $\tau$-homogeneous vector bundle as follows.
 Let $H$ be the (finite-dimensional) Hilbert space of the representation $\tau$.
 Consider the action of $K$ on $G\times K$ defined as
 $$\displaylines{
k(g,h) := (gk, \tau(k^{-1}) h)\ ,
 }$$
and denote by $G\times_\tau H$ the quotient space. The $\tau$-homogeneous vector
bundle is the triple $\xi_\tau:=(G\times_\tau H, \pi_\tau, G/K)$ where  the bundle projection is
$$
\pi_\tau :G\times_\tau H\to G/K;\qquad \theta(g,h)\mapsto gK\ ,
$$
 $\theta(g,h)$ denoting the orbit of the pair $(g,h)$ and $gK$  the lateral class of $g$.
\begin{definition}
A random field $T$ on the $\tau$-homogeneous vector bundle is a random section of $\xi_\tau$, i.e. a measurable map
$$
T: \Omega\times G/K \to G\times_\tau H; \quad (\omega, x)\mapsto T_x(\omega)\ ,
$$
where for each $\omega\in \Omega$, the sample path is a section of $\xi_\tau$, that is
 $\pi_\tau \circ T (\omega) = \text{id}_{G/K}$.
\end{definition}
 Of course this means that for each $x=gK\in G/K$, $T_x$ takes as a value a random element in $\pi^{-1}_\tau(gK)$.

 In the quoted paper, we first introduce a new approach
 to study random fields in $\tau$-homogeneous vector bundles: the ``pullback''
random field. The latter is a (complex-valued) random field $X$ on $G$ whose paths satisfy the following invariance  property
\begin{equation}\label{invS}
X_{gk} = \tau(k^{-1}) X_g\ ,\qquad g\in G, k\in K\ .
\end{equation}
There is one to one correspondence between (random) sections of $\xi_\tau$ and (random) functions on $G$ satysfying \paref{invS} (Proposition \ref{pullback-s-deterministic}) and we prove that $T$ is \emph{equivalent} to its pullback $X$ (see Proposition \ref{prop-pull1}).

Now our attention is devoted to the spherical case. Here $G=SO(3)$, $K\cong SO(2)$: the latter being abelian, its irreducible representations are all one-dimensional and coincide with its linear characters $\chi_s$, $s\in \Z$. Each $k\in SO(2)$ is a counterclockwise rotation of the complex plane $\C$ by an angle $\theta(k)$. The action of $k$ through $\chi_s$ can be seen as a clockwise rotation by the angle $s\cdot \theta(k)$.

The $\chi_s$-homogeneous vector bundle $\xi_s:= \xi_{\chi_s}$ on the sphere is called spin $s$ line bundle and a random field in $\xi_s$ is known as spin $s$ random field.

Our aim is to extend the representation formulas for isotropic Gaussian fields on homogeneous spaces of a compact group in Part 1 to the spin case. The pullback approach allows to deal with isotropic Gaussian  fields $X$ on $SO(3)$ whose sample paths satysfy the invariance property \paref{invS}, that is
$$
X_{gk} = \chi_s(k^{-1}) X_g\ ,\quad g\in SO(3), k\in SO(2)\ .
$$
Indee we prove, with an analogous construction to the one developed in Chapter 2, that for each function $f\in L^2(SO(3)$ bi-$s$-associated, i.e. such that
$$
f(k_1 g k_2) = \chi_s(k_1) f(g) \chi_s(k_2)\ ,\qquad g\in SO(3), k_1, k_2\in SO(2)\ ,
$$
an isotropic complex Gaussian spin $s$ random field $X^f$ is associated (Proposition \ref{propspin=}). Moreover also the converse is true (Theorem \ref{teospin=}).
\begin{theorem}\label{introFin}
Let $T$ be an isotropic complex-Gaussian section of the spin $s$ line bundle and $X$ its pullback random field on $SO(3)$. Then there exists a bi-$s$-associated function $f\in L^2(SO(3))$ such that $X$ and $X^f$ have the same law.
\end{theorem}
Finally,  we prove that our approach is equivalent to the existing ones:
by Malyarenko \cite{malyarenko} (Proposition \ref{noimal}) and Marinucci \& Geller \cite{marinuccigeller} (\S 9.4, especially Lemma  \ref{lemma angolo} ).

\

\noindent The anticipated ``circulant'' structure can be found in  Theorem \ref{intro1} and Theorem \ref{introFin}: this connection is the starting point in the analysis of spin random fields. Indeed open questions concern how to extend results presented in Chapters 1--8 to the spin case. We leave it as a possible topic for future research.

\pagenumbering{arabic}
\renewcommand{\chaptermark}[1]{\markboth{Chap..\ \textbf{\thechapter}\ -\ \textit{#1}}{}}
\renewcommand{\sectionmark}[1]{\markright{Sec.\ \textbf{\thesection}\ -\ \textit{#1}}}
\fancyhf{} \fancyfoot[CE,CO]{\thepage} \fancyhead[CO]{\rightmark}
\fancyhead[CE]{\leftmark}
\renewcommand{\headrulewidth}{0.5pt}
\renewcommand{\footrulewidth}{0.0pt}
\addtolength{\headheight}{0.5pt}
\fancypagestyle{plain}{\fancyhead{}\renewcommand{\headrulewidth}{0pt}}
\linespread{1.2}

\chapter*{Part 1\\ Gaussian fields}
\addcontentsline{toc}{chapter}{Part 1: Gaussian fields}

\chapter{Background:  isotropic random fields }

In this first chapter we recall  basic results concerning both,
 Fourier analysis for a topological compact group $G$,
 and  the structure of isotropic random fields indexed by elements of $G$-homogeneous spaces.

The plan is as follows. In \S 1.1. we give main definitions and fix some notations, whereas in \S 1.2 we investigate Fourier developments for square integrable functions on $G$-homogeneous spaces \cite{faraut}. In \S 1.3 we recollect some useful properties of isotropic random fields from several works (\cite{mauspin, balditrapani, dogiocam} e.g.). Finally, the last section is devoted to the connection between isotropy and positive definite functions on compact groups - highlighting the main features we will deeply need in the sequel.

A great  attention is devoted to the case of the $2$-dimensional unit sphere
$\mathbb S^2$, to whom we particularize the results recalled in each  section of this chapter.

\section{Preliminaries}

Throughout this work $G$ denotes a topological \emph{compact} group (e.g. \cite{faraut}). Let us recall the notion of homogeneous space.
\begin{definition}\label{hom}
A topological space $\X$ is said to be a $G$-homogeneous space if $G$ acts on $\X$
with a continuous and transitive action which we shall denote
$$
G\times \X \goto \X\ ;\qquad (g,x)\mapsto gx\ .
$$
\end{definition}
Remark that $G$ itself is a $G$-homogeneous space, indeed the left multiplication
$(g,h)\mapsto g^{-1}h$
is a continuous and transitive action.

 $\B(\X)$ and
$\B(G)$ stand for the Borel $\sigma$-fields of $\X$ and $G$
respectively
and $dg$  for the Haar measure (see \cite{faraut} e.g.) of $G$.
The latter induces
on $\X$ a $G$-invariant measure which we denote $dx$ abusing notation, given by
$$dx:=\int_G \delta_{gx}\,dg\ ,$$
 where $\delta_{gx}$ as usual stands for the Dirac mass at the singleton $\lbrace gx \rbrace$.
For $G$-invariant we mean that for each integrable function $f$ on $\X$  we have for any $g\in G$
$$\int_\X f(gx)\,dx=\int_\X f(x)\,dx\ .$$
We assume that both these
measures have total mass equal to 1, unless explicitly stated.
For instance, in the case $\X=\mathbb S^d$ the unit $d$-dimensional sphere and $G=SO(d+1)$ the special orthogonal group of order $d+1$,
we have $%
\int_{{\mathbb{S}}^{d}}\,dx=\mu _{d}$ where
\begin{equation}\label{ms}
\mu _{d}:=\frac{2\pi ^{\frac{d+1%
}{2}}}{\Gamma \left( \frac{d+1}{2}\right) }\ .
\end{equation}
We set $L^2(G):=L^2(G,dg)$ and similarly $L^2(\X):=L^2(\X,dx)$;
  the $L^2$-spaces are spaces of \emph{complex-valued} square integrable functions, unless otherwise
stated.

Let us fix a point $x_0\in\X$ once forever and denote by $K$ the
isotropy group of $x_0$
$$K=\lbrace g\in G : gx_0=x_0 \rbrace\ ,$$
 i.e. the (closed) subgroup of the elements $g\in G$
fixing $x_0$. Then it is immediate to check that $\X\cong G/K$ i.e., there exists a $G$-invariant
 isomorphism $\phi: \X \goto G/K$. Actually
the morphism $\widetilde \phi: G\goto \X$ defined as  $\widetilde \phi(g):=gx_0$ is
surjective and $K=\widetilde \phi^{-1}(x_0)$.
 For instance, in the case
$G=SO(3)$ the group of all rotations about the origin of three-dimensional Euclidean space $\R^3$ (under the operation of composition), and $\X=\cS^2$ the two-dimensional unit sphere,
 $x_0$ will be  the north pole  and  the subgroup $K$ of matrices that leave $x_0$ fixed will be isomorphic to $SO(2)$, the special orthogonal group of order two.

The $G$-invariant isomorphism $\X\cong G/K$ suggests that it is possible to identify
functions defined on $\X$ with particular functions on the group $G$.
\begin{definition}
The pullback on $G$ of a function $f:\X\to\C$ is the function $\widetilde f$ defined as
\begin{equation}\label{pul}
\widetilde f(g) := f(gx_0)\ ,\quad g\in G\ .
\end{equation}
\end{definition}
Note that $\widetilde f$  is a right-$K$-invariant function on $G$, i.e.
constant on left closets of $K$. Actually for $g\in G, k\in K$ it is immediate that
$
\tilde f(gk)=f(gkx_0)=f(gx_0)=\tilde f(g)
$.
We have
\begin{equation}\label{int-rule}
\int_\X f(x)\, dx=\int_G \widetilde f(g)\, dg\ ,
\end{equation}
by the
integration rule of image measures,
whenever one of the above integrals has sense.
\begin{remark}\label{id}
In particular, from \paref{pul} and \paref{int-rule}, the map
$f\mapsto \widetilde f$ is an isometry between  $L^2(\X)$ and the (closed) subspace of
right-$K$-invariant
functions in $L^2(G)$.
\end{remark}

\section{Fourier expansions}

In this section we briefly recall Fourier expansions on compact groups
  (for further details see \cite{faraut}).

The \emph{left regular representation} $L$  of $G$ is given, for  $g\in G$ and $f\in
L^2(G)$, by
\begin{equation}\label{rappsin}
L_g f(h)= f(g^{-1}h)\ ,\qquad h\in G\ .
\end{equation}
Let $\widehat G$ be the \emph{dual} of $G$, i.e., the set of
the equivalence classes of irreducible unitary representations of
$G$. The compactness of $G$ implies that $\widehat G$ is at most
countable.

In what follows, we use the same approach as in \cite{balditrapani, mauspin, faraut}).
Let us choose, for every $\sigma\in \widehat G$,  a representative $(D^\sigma, H_\sigma)$
where $D^\sigma$ is a unitary operator  acting
irreducibly on $H_\sigma$ (a complex finite dimensional Hilbert space).

As usual, $\langle \cdot, \cdot \rangle$ denotes the inner
product of $H_\sigma$ and $\dim\sigma:=\dim H_\sigma$ its dimension.
Recall that the \emph{character} of $\sigma$ is the (continuous) function on $G$ defined as
$$
\chi_\sigma(g) := \tr D^\sigma(g)\ ,\qquad g\in G\ ,
$$
where $\tr D^\sigma(g)$ denotes the trace of $D^\sigma(g)$.
Given $f\in
L^2(G)$, for every $\sigma\in \widehat G$ we define the Fourier operator coefficient
\begin{equation}\label{Fourier coefficient}
\widehat f(\sigma) :=\sqrt{\dim \sigma}\int_G f(g)
D^\sigma(g^{-1})\,dg
\end{equation}
which is a linear endomorphism of $H_\sigma$.
Let us denote for $g\in G$
\begin{equation}\label{component}
f^\sigma(g):=\sqrt{\dim \sigma}\, \tr(\widehat f(\sigma) D^\sigma(g))\ ;
\end{equation}
by standard arguments in  Representation Theory \cite{sugiura}, $f^\sigma$ is a continuous function on $G$.
Let us denote $*$ the convolution operator on $G$,
defined  for $f_1, f_2\in L^2(G)$ as
$$
f_1*f_2 (g) := \int_G f_1(h) f_2(h^{-1}g)\,dh\ ,
$$
so that
\begin{equation}\label{conv1}
\widehat{f_1\ast f_2}(\sigma)=\frac 1{\sqrt{\dim \sigma}}\,\widehat
f_2(\sigma) \widehat f_1(\sigma)\ .
\end{equation}
Actually by a Fubini argument and the $G$-invariance property of Haar measure
$$\displaylines{
\widehat{f_1\ast f_2}(\sigma) =  \sqrt{\dim \sigma}\int_G f_1\ast f_2(g)
D^\sigma(g^{-1})\,dg = \cr
= \sqrt{\dim \sigma}\int_G \left( \int_G f_1(h) f_2(h^{-1}g)\,dh \right)
D^\sigma(g^{-1})\,dg = \cr
= \sqrt{\dim \sigma}\int_G f_1(h)\left( \int_G  f_2(g)D^\sigma(g^{-1})\,dg \right) D^\sigma(h^{-1})
\,dh= \frac 1{\sqrt{\dim \sigma}}\,\widehat
f_2(\sigma) \widehat f_1(\sigma)\ .
}$$
Moreover note that
$$
f^\sigma = \text{dim}\,\sigma\cdot  f\ast \chi_\sigma \ .
$$
The
 Peter-Weyl Theorem and Schur's orthogonality relations
 (see \cite{dogiocam} or \cite{sugiura} e.g.) imply the following, known as Plancherel's Theorem.
\begin{theorem}\label{PW}
Let $f\in L^2(G)$. Then
\begin{equation}\label{PW for compact groups}
f(g) = \sum_{\sigma\in \widehat G} f^\sigma(g)\ ,
\end{equation}
the convergence of the series taking place in $L^2(G)$;
$$
\| f\|_{L^2(G)}^2 = \sum_{\sigma \in \widehat G} \| \widehat f(\sigma) \|_{\text{H.S.}}^2\ ,
$$
where $\| \cdot \|_{\text{H.S.}}$ denotes the Hilbert-Schmidt norm \cite{faraut}. If $f_1, f_2\in L^2(G)$, then
$$
\langle f_1, f_2 \rangle_{L^2(G)} = \sum_{\sigma\in \widehat G} \tr \widehat f_1(\sigma) \widehat
f_2(\sigma)^* \ .
$$
\end{theorem}
Recall that a function $f$ is said to be a  \emph{central} (or class) function if for every $g\in G$
$$
f(hgh^{-1}) = f(g), \quad h\in G\ .
$$
\begin{prop}\label{classExp}
The set of characters $\lbrace \chi_\sigma : \sigma\in \widehat G \rbrace$
is an orthonormal basis of the space of square integrable
central functions on $G$.
\end{prop}
Now fix any orthonormal basis $v_1, v_2, \dots, v_{\dim\sigma }$ of
$H_\sigma$ and for $i,j=1,\dots , \text{dim $\sigma$}$ denote
$D^\sigma_{ij}(g):= \langle D^\sigma(g) v_j, v_i\rangle$
the $(i,j)$-th coefficient of
the matrix representation for $D^\sigma(g)$ with respect to this
basis. The matrix representation for $\widehat f(\sigma)$ has entries
$$
\dlines{\widehat f(\sigma)_{i,j}=\langle \widehat f(\sigma) v_j,
v_i\rangle = \sqrt{\dim \sigma} \int_G f(g) \langle D^\sigma(g^{-1})
v_j, v_i\rangle \,dg = \cr = \sqrt{\dim \sigma} \int_G f(g)
D^\sigma_{i,j}(g^{-1})\,dg = \sqrt{\dim \sigma} \int_G f(g)
\overline{D^\sigma_{j,i}(g)}\,dg\ ,}
$$
and Theorem \ref{PW} becomes
\begin{equation}\label{PW1}
f(g)=\sum_{\sigma\in \widehat G} \sqrt{\dim \sigma} \sum_{i,j=1}^{\dim\sigma
} \widehat f(\sigma)_{j,i} D^\sigma_{i,j}(g)\ ,
\end{equation}
the above series still converging in $L^2(G)$. The Peter-Weyl Theorem also states that  the set of functions
$\lbrace \sqrt{\text{dim}\,\sigma}D^\sigma_{i,j}\ , \sigma \in \hat G, i,j=1,\dots, \text{dim}\,\sigma \rbrace$ is a complete  orthonormal basis for
$L^2(G)$. Therefore \paref{PW1} is just the corresponding Fourier development and  $\widehat f(\sigma)_{j,i}$  is
the coefficient corresponding to the element $\sqrt{\text{dim}\,\sigma} D^\sigma_{i,j}(g)$ of this basis.

Let $L^2_\sigma(G)\subset L^2(G)$ be the $\sigma$-isotypical subspace,
 i.e. the subspace generated by the functions $D^\sigma_{i,j}, i,j=1,\dots , \text{dim}\,\sigma$; it is a $G$-module that can
 be decomposed into the
orthogonal direct sum of $\dim\sigma $ irreducible and
equivalent $G$-modules
$(L^2_{\sigma,j}(G))_{j=1,\dots,\dim\sigma }$ where each
$L^2_{\sigma,j}(G)$ is spanned by the functions
$D^\sigma_{i,j}$ for $i=1,\dots, \dim\sigma $, loosely
speaking by the $j$-th column of the matrix $D^\sigma$.
Note that $f^\sigma$ is the component (i.e. the orthogonal projection) of $f$ in $L^2_\sigma(G)$.

Equivalently the Peter-Weyl Theorem can be stated as
\begin{equation}\label{PW2}
L^2(G)= \bigoplus_{\sigma \in \widehat G} \,\, \mathop{\oplus}\limits_{j=1}^{\dim\sigma }
L^2_{\sigma,j}(G)\ ,
\end{equation}
the direct sums being orthogonal.

Let us now deduce the Fourier expansion of functions $f\in L^2(\X)$. It can be easily obtained
 from Theorem \ref{PW} and Remark \ref{id}, indeed their pullbacks $\widetilde f$
belong to $L^2(G)$ and form a $G$-invariant closed subspace of $L^2(G)$.
We can therefore associate to $f\in L^2(\X)$ the family of operators $\bigl(
\widehat{\widetilde f}(\sigma) \bigr)_{\sigma\in \widehat G}$.
Let $H_{\sigma,0}$ denote the subspace of $H_\sigma$ (possibly
reduced to $\{0\}$) formed by the vectors that remain fixed under the action of
$K$, i.e. for every $k\in K, v\in
H_{\sigma,0}$, $D^\sigma(k)v=v$.
Right-$K$-invariance implies that the image of $\widehat
{\widetilde f}(\sigma)$ is contained in $H_{\sigma,0}$:
\begin{equation}\label{proiezione triviale}
\begin{array}{c}
\displaystyle\widehat {\widetilde f}(\sigma) =\sqrt{\dim \sigma}
\int_G \widetilde f(g) D^\sigma(g^{-1})\,dg=\\
=\sqrt{\dim \sigma}
\int_G \widetilde f(gk) D^\sigma(g^{-1})\,dg
\displaystyle= \sqrt{\dim \sigma} \int_G\widetilde f(h)
D^\sigma(kh^{-1} )\,dh =\\
= D^\sigma(k)\sqrt{\dim \sigma}
\int_G\widetilde f(h) D^\sigma(h^{-1} )\,dh=D^\sigma(k)\widehat
{\widetilde f}(\sigma)\ .
\end{array}
\end{equation}
Let us denote by $P_{\sigma,0}$ the projection of $H_\sigma$ onto
$H_{\sigma,0}$, so that $\widehat {\widetilde
f}(\sigma)=P_{\sigma,0}\widehat {\widetilde f}(\sigma)$,
and $\widehat G_0$ the set of irreducible unitary
representations of $G$ whose restriction to $K$ contains the trivial
representation.
If $\sigma\in \widehat G_0$ let us consider a basis of $H_\sigma$
such that the elements $\lbrace v_{p+1}, \dots , v_{\dim\sigma} \rbrace$,
for some  integer $p=p(\sigma)\ge 0$, span $H_{\sigma,0}$. Then the first $p$ rows
of the representative matrix of $\widehat {\widetilde f}(\sigma)$ in this
basis contain only zeros. Actually, by
(\ref{proiezione triviale}) and $P_{\sigma,0}$ being self-adjoint,
for $i\le p$
$$
\widehat {\widetilde f}_{i,j}(\sigma) = \langle \widehat {\widetilde
f}(\sigma)  v_j, v_i \rangle = \langle P_{\sigma,0} \widehat
{\widetilde f}(\sigma)   v_j, v_i \rangle=\langle \widehat
{\widetilde f}(\sigma)   v_j, P_{\sigma,0}v_i \rangle= 0\ .
$$
Identifying $L^2(\X)$ as the closed subspace of right-$K$-invariant functions in $L^2(G)$, the Peter-Weyl Theorem entails that
$$
L^2(\X) = \bigoplus_{\sigma \in \widehat G_0} \oplus_{j=p+1}^{\text{dim}\,\sigma} L^2_{\sigma, j}(G)\ ,
$$
the direct sums being orthogonal.

Now we consider an important  class of functions we shall need in the sequel.
\begin{definition}
A function $f:G \goto \C$ is said to be \emph{bi-$K$-invariant} if for every $g\in G, k_1, k_2\in K$
\begin{equation}\label{bicappa}
f(k_1gk_2)=f(g)\ .
\end{equation}
\end{definition}
If moreover $f\in L^2(G)$, the equality in \paref{bicappa} entails that, for every
$k_1,k_2\in K$, $\sigma\in\widehat G$,
$$
\widehat f(\sigma)=D^\sigma(k_1)\widehat f(\sigma)D^\sigma(k_2)
$$
and therefore  a function $f\in L^2(G)$ is {bi-$K$-invariant} if and only if
for every $\sigma\in \hat G$
\begin{equation}\label{fourier-bicappa}
\widehat f(\sigma) = P_{\sigma,0} \widehat f(\sigma) P_{\sigma,0}\ .
\end{equation}
Note that we can identify of course bi-$K$-invariant functions in $L^2(G)$ with
left-$K$-invariant functions in $L^2(\X)$.

\subsection{Spherical harmonics}

Now we focus on the case of $\X=\mathbb S^2$ under the action of $G=SO(3)$, first
specializing previous results and then
recalling basic facts we will need in the rest of this work (see \cite{faraut}, \cite{dogiocam} e.g.
for further details).
The isotropy
group $K\cong SO(2)$ of the north pole  is abelian, therefore its unitary irreducible
representations are unitarily equivalent to its linear characters which we shall denote
 $\chi_s, s\in \Z$, throughout the whole work.

A complete set of unitary irreducible matrix representations of $SO(3)$  is given by
the so-called Wigner's $D$ matrices
$\lbrace D^\ell, \ell \ge 0 \rbrace$, where each $D^\ell(g)$ has dimension
$(2\ell+1)\times (2\ell+1)$ and acts on a representative space that we shall denote $H_\ell$.
The restriction to $K$ of each $D^\ell$ being unitarily equivalent to the direct sum of the
representations $\chi_m$, $m=-\ell, \dots, \ell$, we can suppose
$v_{-\ell}, v_{-\ell +1}, \dots, v_\ell$ to be an orthonormal basis for $H_\ell$ such that
for every $m : |m| \le \ell$
\begin{equation}\label{restrizione}
D^\ell(k)v_m = \chi_m(k)v_m\ , \qquad k\in K\ .
\end{equation}
Let $D^\ell_{m,n}=\langle D^\ell v_n, v_m \rangle$ be the $(m,n)$-th entry of $D^\ell$
with respect to the basis fixed above.
It follows from (\ref{restrizione})  that for every $g\in SO(3), k_1, k_2\in K$,
\begin{equation}\label{prop fnz di Wigner}
D^\ell_{m,n}(k_1gk_2) = \chi_m(k_1)D^\ell_{m,n}(g)\chi_n(k_2)\ .
\end{equation}
The functions  $D^\ell_{m,n}:SO(3) \goto \C$, $ \ell\ge 0, m,n=-\ell,\dots,\ell$
are usually called Wigner's $D$ functions.

Given $f\in L^2(SO(3))$,  its $\ell$-th Fourier coefficient  (\ref{Fourier coefficient}) is
\begin{equation}\label{coefficiente ellesimo}
\widehat f(\ell) := \sqrt{2\ell+1}\,\int_{SO(3)} f(g)
D^\ell(g^{-1})\,dg
\end{equation}
and its Fourier development (\ref{PW1})  becomes
\begin{equation}\label{PW SO(3)}
f(g)=\sum_{\ell \ge 0} \sqrt{2\ell + 1} \sum_{m.n=-\ell}^{\ell
} \widehat f(\ell)_{n,m} D^\ell_{m,n}(g)\ .
\end{equation}
If $\widetilde f$ is the pullback  of  $f\in L^2(\mathbb S^2)$, (\ref{restrizione}) entails that
for every $\ell\ge 0$
$$
\widehat{\widetilde f}(\ell)_{n,m} \ne 0 \iff n= 0\ .
$$
Moreover if $f$ is left-$K$-invariant, then
$$
\widehat{\widetilde f}(\ell)_{n,m} \ne 0 \iff n,m = 0\ .
$$
In words, an orthogonal basis for the space of the square integrable right-$K$-invariant
functions on $SO(3)$ is given by the central columns  of the matrices $D^\ell$, $\ell \ge 0$. Furthermore
the subspace of the bi-$K$-invariant functions is spanned by the  central functions
$D^\ell_{0,0}(\cdot),\, \ell \ge 0$, which are \emph{real-valued}.

The important role of the other columns of Wigner's $D$ matrices will appear further in this work.
\begin{definition}
For every $\ell \ge 0, m=-\ell \dots, \ell$,  let us define the spherical harmonic $Y_{\ell,m}$ as
\begin{equation}\label{armoniche sferiche1}
Y_{\ell,m}(x) := \sqrt{\frac{2\ell+1}{4\pi}}\, \overline{D^\ell_{m,0}(g_x)}\ , \qquad x\in \cS^2\ ,
\end{equation}
where $g_x$ is any rotation mapping the north pole of the sphere to $x$.
\end{definition}
Remark that this is a good definition thanks to the
invariance of each $D^\ell_{m,0}(\cdot)$ under the right action of $K$.
The functions in (\ref{armoniche sferiche1}) form an orthonormal basis of the space $L^2(\mathbb S^2)$
considering the sphere with total mass equal to $4\pi$.

Often, e.g. in the second part ``High-energy eigenfunctions'',
we  work with \emph{real} spherical harmonics, i.e. the orthonormal set
of functions given by
\begin{equation}\label{realSH}
\frac{1}{\sqrt 2}\left ( Y_{\ell,m} + \overline{Y_{\ell,m}}    \right )\ , \qquad \frac{1}{i\sqrt 2}\left ( Y_{\ell,m} - \overline{Y_{\ell,m}}    \right )\
\end{equation}
which abusing notation we will again denote by $Y_{\ell,m}$ for $\ell\ge 0$, $m=1, \dots, 2\ell +1$.

Every $f\in L^2(\mathbb S^2)$ admits the Fourier development of the form
\begin{equation}
f(x) = \sum_{\ell=0}^{+\infty} \sum_{m=-\ell}^{\ell} a_{\ell,m} Y_{\ell,m}(x)\ ,
\end{equation}
where the above series converges in $L^2(\mathbb S^2)$ and
$$
a_{\ell,m} = \int_{\mathbb S^2} f(x) \overline{Y_{\ell,m}(x)}\,dx\ .
$$
Moreover the Fourier expansion of a left-$K$-invariant function $f\in L^2(\cS^2)$ is
\begin{equation}\label{PW invariant sphere}
f=\sum_{\ell=0}^{+\infty} \beta_\ell Y_{\ell,0}\ ,
\end{equation}
where $\beta_\ell := \int_{\mathbb S^2} f(x) Y_{\ell,0}(x)\,dx$.
The functions $Y_{\ell,0},\,\ell\ge 0$ are called  \emph{central spherical harmonics}.

\

We stress that there exists an alternative characterization of spherical harmonics,
as \emph{eigenfunctions} of the spherical Laplacian $\Delta_{\mathbb S^2}$ (see \cite{dogiocam}
 e.g.). We shall deeply use this formulation in Part 2: High-energy Gaussian eigenfunctions.

Recall that the spherical Laplacian is the Laplace-Beltrami operator on $\mathbb S^2$ with its
canonical metric of constant sectional curvature $1$, moreover
its (totally discrete) spectrum
is given by the set of eigenvalues $\lbrace -\ell(\ell+1)=: - E_\ell$, $\ell\in \mathbb N \rbrace$.

It can be proved that for $\ell \ge 0, m=-\ell, \dots, \ell $
$$
\Delta_{\mathbb S^2} Y_{\ell,m} + E_{\ell} Y_{\ell,m} = 0\ ,
$$
and the subset of spherical harmonics
$\lbrace Y_{\ell,m}, m=-\ell, \dots, \ell \rbrace$ is an orthonormal basis for the eigenspace $\mathcal H_\ell$  corresponding to the $\ell$-th
eigenvalue.
The Spectral Theorem for self-adjoint compact operators then entails that  $\mathcal H_\ell$ and $\mathcal H_{\ell'}$ are orthogonal whenever $\ell \ne \ell'$ and  moreover
$$
L^2(\mathbb S^2) = \bigoplus_{\ell\ge 0} \mathcal H_\ell\ ,
$$
which coincides with the Peter-Weyl decomposition for the sphere.

\section{Isotropic random fields}

Let us recall main definitions and facts about isotropic random fields on homogeneous spaces
(see \cite{dogiocam, mauspin, balditrapani} e.g.). First fix some probability space $(\Omega, \F, \P)$  and
denote $L^2(\P):=L^2(\Omega, \P)$ the space of finite-variance random variables.
\begin{definition}\label{campo aleatorio su uno spazio omogeneo}
A (complex-valued) random field $T=(T_x)_{x \in \X}$ on the $G$-ho\-mo\-ge\-neous space $\X$ is
a collection of (complex-valued) random variables indexed by elements of $\X$
such that the map
\begin{align}
\nonumber
T:\Omega \times \X \goto \C\ ;\qquad
(\omega, x) \mapsto T_x(\omega)
\end{align}
is $\F \otimes \B(\X)$-measurable.
\end{definition}
Note that often we shall write $T(x)$ instead of $T_x$.
\begin{definition}
We say that the random field $T$ on the $G$-homogeneous space $\X$
is second order if $T_x\in L^2(\P)$ for every $x\in \X$.
\end{definition}
\begin{definition}\label{contRF}
We say that the random field $T$ on the $G$-homogeneous space $\X$
is a.s. continuous if the functions
$\X\ni x\mapsto T_x$ are a.s. continuous.
\end{definition}
In this work the minimal regularity assumption for the paths of a random field $T$ is the a.s. square integrability. From now on, $\E$ shall stand for the expectation under the probability measure $\P$.
\begin{definition}
We say that the random field $T$ on the $G$-homogeneous space $\X$  is
\begin{enumerate}
\item a.s. square integrable if
\begin{equation}
\int_\X | T_x |^2\,dx < +\infty\;\; a.s.\label{integrabile q.c.}
\end{equation}
\item mean square integrable if
\begin{equation}
\E \left[\int_\X | T_x |^2\,dx \right] < +\infty\ .\label{integrabile in media quadratica}
\end{equation}
\end{enumerate}
\end{definition}
\noindent Note that the mean square integrability implies the $a.s.$ square integrability.
\begin{remark}\rm\label{variabile aleatoria che prende valori in L2X}
If $T$ is a.s. square integrable,
it can be regarded to as a random variable taking a.s. its values in $L^2(\X)$
i.e. $T(\cdot)=(x\goto T_x(\cdot))$.
\end{remark}
We define now the notion of isotropy.
Let $T$ be a.s. square integrable. For every $f\in L^2(\X)$, we can consider the integral
$$
T(f):=\int_\X T_x \overline{f(x)}\,dx
$$
which defines a r.v. on $(\Omega, \F, \P)$.
For every $g\in G$, let  $T^g$ be the \emph{rotated field} defined as
$$
T^g_x:=T_{gx},\qquad x\in \X\ .
$$
Losely speaking, $T$ is isotropic if its law and the law of the rotated field $T^g$
coincide for every $g\in G$.
We give the following formal definition of isotropy (see \cite{maupec, mauspin,  balditrapani})
\begin{definition}\label{invarian}
An a.s. square integrable random field $T$ on the homogeneous space
$\X$ is said to be (strict sense) $G$-invariant or
isotropic if
the joint laws of
\begin{equation}\label{l2-invar}
(T(f_1),\dots,T(f_m))\quad\mbox{and}\quad (T(L_gf_1),\dots,T(L_gf_m))=
(T^g(f_1),\dots,T^g(f_m))
\end{equation}
coincide for every $g\in G$ and $f_1,f_2,\dots,f_m\in L^2(\X)$.
\end{definition}
This definition is somehow different from the one usually considered in the literature, where
the requirement is the equality of the finite dimensional distributions, i.e. that the random vectors
\begin{equation}\label{invar-continui1}
(T_{x_1},\dots,T_{x_m})\qquad\mbox{and}\qquad(T_{gx_1},\dots,T_{gx_m})
\end{equation}
have the same law for every choice of $g\in G$ and
$x_1,\dots,x_m\in \X$. Remark that (\ref{l2-invar})
implies (\ref{invar-continui1}) (see \cite{maupec}) and
that, conversely, by standard approximation arguments
(\ref{invar-continui1}) implies (\ref{l2-invar}) if $T$ is
continuous.

We see now how the Peter-Weyl decomposition  naturally applies to
random fields.
It is worth remarking that every a.s. square integrable random field $T$ on $\X$
uniquely defines an a.s. square integrable random field on $G$ (whose paths
are the pullback functions of the paths $x\mapsto T_x$).
Therefore w.l.o. we can  investigate the case $\X=G$.

To every a.s. square integrable random field $T$ on  $G$ we
can associate the set of operator-valued r.v.'s $(\widehat
{T}(\sigma))_{\sigma\in \widehat G}$ defined ``pathwise'' as
\begin{equation}
\widehat { T}(\sigma) = \sqrt{\dim\sigma}\int_G T_g D^\sigma(g^{-1})\,dg\ .
\end{equation}
From  (\ref{PW for compact groups}) therefore
\begin{equation}\label{l2as-dev}
T_g = \sum_{\sigma\in \widehat G} T^\sigma_g
\end{equation}
where the convergence takes place in $L^2(G)$ a.s.
Remark that
\begin{equation}
T^\sigma_g:=\sqrt{\dim\sigma} \tr(\widehat {
T}(\sigma) D^\sigma(g))= \sqrt{\dim \sigma} \sum_{i,j=1}^{\dim\sigma
} \widehat T(\sigma)_{j,i} D^\sigma_{i,j}(g)\ ,
\end{equation}
is the projection of $T$ on $H_\sigma$ and $T^\sigma$ is continuous.

For the proof of the following see \cite{balditrapani}.
\begin{prop}
Let T be an a.s. square integrable random field on G. Then T is isotropic if and only
if, for every $g\in G$, the two families of r.v's
$$
(\widehat T(\sigma))_{\sigma\in \widehat G} \quad \text{and} \quad
(\widehat T(\sigma)D^\sigma(g))_{\sigma\in \widehat G}
$$
are equi-distributed.
\end{prop}
If the random field $T$ is second order and isotropic (so that
 \paref{integrabile in media quadratica} holds by a standard Fubini argument),
it is possible to say more about the convergence of the series
in (\ref{l2as-dev}).

Indeed we have the following result, that we reproduce here from  \cite{dogiocam}.
\begin{theorem}\label{Stochastic-PW}
Let $T$ be a second order and
isotropic random field on the compact group $G$.
Then
\begin{equation}
T=\sum_{\sigma \in \hat G} T^\sigma\ .
\end{equation}
The convergence of the infinite series is both in the sense of
$L^2(\Omega\times G, \P \otimes dg)$ and $L^2(\P)$ for every fixed $g$, that is, for any enumeration
$\lbrace \sigma_k: k\ge 1 \rbrace$ of $\hat G$, we have both
\begin{align}
&\lim_{N\to +\infty} \E \left [\int_G \left | T_g - \sum_{k=1}^{N} T^{\sigma_k}_g \right |^2\,dg \right ]=0\ ,\label{align1}\\
&\lim_{N\to +\infty} \E  \left [\left | T_g - \sum_{k=1}^{N} T^{\sigma_k}_g \right |^2\right ] =0\ .\label{2}
\end{align}
\end{theorem}
\noindent The previous theorem has the following interesting consequence (for a proof see \cite{maupec}).
\begin{prop}\label{Mean square continuity of invariant}
Every second order and  isotropic random field $T$ on the homogeneous space  $\X$ of a compact group is
 mean square continuous, i.e.
 \begin{equation}
\lim_{y\to x} \E[|T_y - T_x|^2 ]=0\ .
 \end{equation}
\end{prop}
It is worth remarking some features of Fourier coefficients of a second order and  isotropic random field $T$ (see  \cite{balditrapani}).
\begin{theorem}
If $\sigma\in \widehat G$ is not the trivial representation, then
$$
\E[\widehat T(\sigma)]=0\ ,
$$
moreover for $\sigma_1, \sigma_2\in \widehat G$, we have
\begin{enumerate}\label{th coeff}
\item if $\sigma_1, \sigma_2$ are not equivalent, the r.v.'s $\widehat T(\sigma_1)_{i,j}$ and $\widehat T(\sigma_2)_{k,l}$ are orthogonal for
$i,j=1, \dots, \text{dim}\,\sigma_1$ and $k,l=1, \dots , \text{dim}\,\sigma_2$;
\item if $\sigma_1=\sigma_2=\sigma$, and $\Gamma(\sigma)=\E[\widehat T(\sigma) \widehat T(\sigma)^*]$, then $\Cov(\widehat T(\sigma)_{i,j}, \widehat T(\sigma)_{k,l})=\delta_j^l \Gamma(\sigma)_{i,k}$. In particular
coefficients belonging to different columns are orthogonal and the covariance between entries in
different rows of a same column does not depend on the column.
\end{enumerate}
\end{theorem}
Theorem \ref{th coeff} states that the entries of $\widehat T(\sigma)$ might not be pairwise orthogonal and this happens when
the matrix $\Gamma$ is not diagonal. This phenomenon is actually already been remarked by other authors
(see \cite{malyarenkobook} e.g.).
Of course there are situations in
which orthogonality is still guaranteed: when the dimension of $H_{\sigma,0}$ is one at most (i.e. in every
irreducible $G$-module the dimension of the space $H_{\sigma,0}$ of the $K$-invariant vectors in one at most) as
is the case for $G = SO(m+1)$, the special orthogonal group of dimension $m+1$, $K = SO(m)$ and  $G/K \cong \mathbb S^m$ the $m$-dimensional unit sphere.  In this case actually the matrix $\widehat T(\sigma)$ has
just one row that does not vanish and $\Gamma(\sigma)$ is all zeros, but one entry in the diagonal.

\

Let us now focus on Gaussian  fields, which will receive the greatest attention in this work.
First it is useful to  recall the following.
\begin{definition}\label{Gaussian}
Let $Z=Z_1+iZ_2$ be a complex random variable (we mean that $Z_1,Z_2$ are real random variables).
We say that
\begin{itemize}
\item $Z$ is a \emph{complex-valued} Gaussian random variable if $(Z_1,Z_2)$ are jointly Gaussian;
\item $Z$ is a \emph{complex} Gaussian random variable if $Z_1,Z_2$ are independent Gaussian random variables
with the same variance.
\end{itemize}
Furthermore we say that the random vector
$(Y_1,Y_2,\dots,Y_m)$ is a complex (resp. com\-plex-valued)  Gaussian vector if
$$\sum_i a_i Y_i$$
is a complex (resp. complex-valued) Gaussian random variable for every choice of $a_1,a_2,\dots,a_m\in \mathbb C$.
\end{definition}
\noindent From this definition it follows  that if $T$ is complex-valued Gaussian,
meaning that the r.v.
$T(f)$
is complex-valued Gaussian for every $f\in L^2(\X)$,  then its Fourier coefficients
are complex-valued Gaussian r.v.'s. Furthermore, if each representation of $G$ occurs  at most once
in the Peter-Weyl decomposition of $L^2(\X)$ and $T$ is Gaussian and isotropic, we have that these Fourier
coefficients are pairwise independent from Theorem \ref{th coeff}. This is the case for instance for $G=SO(m+1)$ and $\X=\mathbb S^m$.

In \cite{BMV} a characterization
of isotropic real Gaussian fields on homogeneous spaces of compact groups is given: under some mild
additional assumption also the converse is true, namely that if a random field is isotropic
and its Fourier coefficients are independent, then it is necessarily Gaussian.
For more discussions on this topic see also \cite{balditrapani}.

\subsection*{Isotropic spherical random fields}

Let us consider a random field $T=(T_x)_{x\in \cS^2}$ on $\cS^2$
according to Definition (\ref{campo aleatorio su uno spazio omogeneo}).
We assume that $T$ is $a.s.$ square integrable.
From previous sections,
$T$ admits the following stochastic Fourier expansion
\begin{equation}\label{espansione stocastica sulla sfera}
T_x = \sum_{\ell \ge 0} \sum_{m=-\ell}^{\ell} a_{\ell,m} Y_{\ell,m}(x)
\end{equation}
where  $a_{\ell,m}=\int_{\cS^2} T_x \overline{Y_{\ell,m}(x)}\,dx$ are the Fourier coefficients w.r.t. the basis of spherical harmonics and the convergence is
in the sense of $L^2( \cS^2)$ $a.s.$  \\
If the random field $T$ is in addition second order and isotropic,  Theorem (\ref{Stochastic-PW}) states that the convergence
of the series in (\ref{espansione stocastica sulla sfera}) holds both
in the sense of $L^2(\Omega\times \cS^2, \P\otimes dx)$ and $L^2(\P)$ for every fixed $x$, and
furthermore, Corollary \ref{Mean square continuity of invariant} states that $T$ is mean square continuous.

\noindent Moreover from  Theorem \ref{th coeff}
we obtain easily
\begin{equation}\label{media zero}
\E(a_{\ell,m}) = 0\; \mbox{for every}\;  m=-\ell,\dots , \ell\; \mbox{and}\; \ell > 0
\end{equation}
so that $\E (T_x)=\E (a_{0,0})/ \sqrt{4\pi}$, as $Y_{0,0}= 1/\sqrt{4\pi}$, according to the
fact that the mean of an isotropic random field is constant.
If $c$ is any additive constant, the random field $T^c:=T+c$ has the same
Fourier expansion as $T$, except for the term $a_{0,0}^c Y_{0,0}= c + a_{0,0}Y_{0,0}$ because for every $\ell > 1$ the spherical harmonics $Y_{\ell,m}$ are orthogonal to the constants.
In what follows we often consider \emph{centered} isotropic random fields,
this is generally done by ensuring that also the trivial coefficient $a_{0,0}$ is a centered
random variable. However we will often require that $a_{0,0}=0$, i.e.,
the average of the random field vanishes on $\cS^2$:
\begin{equation}
\int_{\cS^2} T_x\,dx = 0\ .
\end{equation}
As in the Peter-Weyl decomposition of $L^2(\cS^2)$ two irreducible representations with $\ell\not=\ell'$ are not equivalent,
the random coefficients $a_{\ell,m}, m=-\ell, \dots, \ell$ are pairwise orthogonal and moreover
the variance of $a_{\ell,m}$ does not depend on $m$. We denote
$$c_\ell:=\E[|a_{\ell,m}|^2]$$
the variance of $a_{\ell,m}$.  The (nonnegative) sequence
$(c_\ell)_\ell$ is known as the \emph{angular power spectrum} of the field.

It turns out that $T$ is Gaussian and isotropic if and only $a_{\ell,m}$ are Gaussian independent random variables.

In this case, from \paref{espansione stocastica sulla sfera}, setting
$$
T_\ell(x) := \sum_{m=-\ell}^{\ell} \frac{a_{\ell,m}}{\sqrt{c_\ell}} Y_{\ell,m}(x)
$$
we can write
$$
T_x = \sum_\ell \sqrt{c_\ell}\, T_\ell(x) \ ,
$$
where $T_\ell$ is known as the $\ell$-th Gaussian eigenfunctions on $\mathbb S^2$ or random spherical harmonics (see \paref{Telle} for further details).
\section{Positive definite functions}

To every second order random field $T$ one can associate
the \emph{covariance kernel} $R:\X\times \X \goto \C$ defined as
$$
R(x,y)=\Cov(T_x,T_y)\ .
$$
This kernel is positive definite, as, for every choice of $x_1,\dots,x_m\in\X$ and of $\xi\in\C^m$ we have
$$
\displaylines{
\sum_{i,j=1}^mR(x_i,x_j)\xi_i\overline{\xi_j}=
\sum_{i,j=1}^m\Cov(T_{x_i},T_{x_j})\xi_i\overline{\xi_j}
=\Var\left(\sum_{i}T_{x_i}\xi_i\right)\ge0\ .
}
$$
If in addition $T$ is isotropic we have, for every $g\in G$,
$$
R(gx,gy)=R(x,y)
$$
and, in this case, $R$ turns out to be continuous, thanks to proposition (\ref{Mean square continuity of invariant}).
Moreover to this kernel one can associate the function on $G$
\begin{equation}\label{funzione fi}
\phi(g):=R(gx_0,x_0)\ .
\end{equation}
This function $\phi$ is
\begin{itemize}
\item  continuous, as a consequence of the continuity of $R$.
\item \emph{bi-$K $-invariant} i.e. for every $k_1,k_2\in K $ and $g\in G$ we have
$$\phi(k_1gk_2)=R(k_1gk_2x_0,x_0)\underbrace{=R(k_1gx_0,x_0)}_{k_2x_0=x_0} \underbrace{=R(gx_0,k_1^{-1}x_0)}_{R\,is\,G-invariant}=R(gx_0,x_0)=\phi(g)$$
\item  \emph{positive definite}, actually as $R$ is a positive definite kernel, for every $g_1,\dots,g_m\in G$ and
$\xi_1,\dots,\xi_m\in \C$ we have
\begin{align}\label{ab}
\sum_{i,j} \phi(g_i^{-1}g_j)\overline{\xi_i} \xi_j = \sum_{i,j} R(g_i^{-1}g_jx_0,x_0)\overline{\xi_i} \xi_j=\sum_{i,j} R(g_jx_0,g_ix_0)\overline{\xi_i} \xi_j\ge 0\ .
\end{align}
\noindent By standard approximation arguments \paref{ab} imply that for every continuous function $f$ we have
\begin{equation}\label{dis1}
\int_G \int_G \phi(h^{-1}g) f(h) \overline{f(g)}\,dg\,dh \ge 0\ .
\end{equation}
\end{itemize}
Moreover $\phi$ determines the covariance kernel $R$ by observing that if $g_x x_0=x, g_yx_0=y$, then
$$R(x,y)=R(g_xx_0,g_yx_0)=R(g_y^{-1}g_xx_0,x_0)=\phi(g_y^{-1}g_x)\ .$$
\noindent Now it is useful to introduce the following functions and their properties.
\begin{definition}\label{f chech}
Let $\zeta$ be a function defined on $G$. We denote by $\zsmile$
the function
$$\zsmile(g):=\overline{\zeta(g^{-1})}$$
\end{definition}
\begin{remark}\rm\label{algebra gruppo}
 We have just defined a map
\begin{align}
\zeta \goto \zsmile
\end{align}
that is an \emph{involution} of the convolution algebra $L^2(G)$ that becomes an $H^*$-algebra. $L^2(G)$ is known as the \emph{group algebra} of $G$. \qed
\end{remark}
\begin{remark}\rm\label{aggiunto}
If $\zeta\in L^2(G)$, then for every $\sigma \in \hat G$ we have
$$\hat \zsmile (\sigma)= \hat \zeta(\sigma)^*\ .$$
Actually,
$$
\hat \zsmile (\sigma) = \int_G \zsmile(g) D^\sigma(g^{-1})\,dg=
\int_G \overline{\zeta(g^{-1})} D^\sigma(g^{-1})\,dg\ .
$$
Thus, for every $v\in H_\sigma$,
$$\dlines{
\langle \hat \zsmile (\sigma) v,v \rangle = \int_G \overline{\zeta(g^{-1})} \langle D^\sigma(g^{-1})v,v \rangle \,dg=\cr
=  \int_G \overline{\zeta(g^{-1})} \langle v,D^\sigma(g)v \rangle \,dg= \int_G \overline{\zeta(g^{-1})} \overline{\langle D^\sigma(g)v,v \rangle} \,dg=\cr
=\overline{ \langle \hat \zeta(\sigma)v, v \rangle}=\langle v, \hat \zeta(\sigma)v \rangle\ .
}$$
\end{remark}
Remark that every positive definite function $\phi$ on $G$ (see \cite{sugiura} p.123) satisfies
$$
\fismile = \phi\ .
$$
The following proposition states some (not really unexpected) properties of continuous positive definite functions that we shall need later.
\begin{prop}\label{structure of positive definite}
Let $\phi$ a continuous positive function and $\sigma\in \hat G$.

a) Let, $\phi(\sigma):H_\sigma\to H_\sigma$ the operator coefficient $\phi(\sigma)=\int_G \phi(g)D^\sigma(g^{-1})\, dg$. Then $\phi(\sigma)$ Hermitian positive definite.

b) Let $ \phi^\sigma:G\to \C$ the component of $\phi$ corresponding to $\sigma$. Then $\phi^\sigma$ is also a positive definite function.
\end{prop}
\begin{proof}
a) Let us fix a basis $v_1,\dots,v_{d_\sigma}$ of $H_\sigma$, we have
\begin{equation}\label{dis3}
\langle \hat \phi(\sigma) v,v \rangle = \int_G \phi(g)
\langle D^\sigma(g^{-1})v,v\rangle\,dg\ .
\end{equation}
By the invariance of the Haar measure
$$
\dlines{ \int_G \phi(g)
\langle D^\sigma(g^{-1})v,v\rangle\,dg = \int_G \int_G \phi(h^{-1}g) \langle D^\sigma(g^{-1}h)v,v\rangle\,dg\,dh=\cr =
\int_G \int_G \phi(h^{-1}g) \langle D^\sigma(h)v,D^\sigma(g)v\rangle\,dg\,dh=\cr
=\int_G \int_G \phi(h^{-1}g) \sum_k ({D^\sigma(h)}{v})_k\overline{({D^\sigma(g)} {v})}_k\,dg\,dh =\cr
\sum_k \int_G \int_G \phi(h^{-1}g) f_k(h)\overline{f_k(g)}\,dg\,dh \ge 0\ ,
}
$$
where we have set, for every $k$, $f_k(g)=({D^\sigma(g)}{v})_k$ and  (\ref{dis1}) allows to conclude.
Let $\phi^\sigma$ be the projection of $\phi$ onto the $\sigma$-isotypical subspace $L^2_\sigma(G)\subset L^2(G)$.

b) The Peter-Weyl theorem states that
\begin{equation}
\phi= \sum_{\sigma\in \hat G} \phi^\sigma\ ,
\end{equation}
the convergence of the series taking place in $L^2(G)$.

%
\noindent Let $f\in L^2_\sigma(G)$ in \ref{dis1} and replace $\phi$ with its Fourier series.
Recall that $f$ is a continuous function.
We have
$$
\dlines{
0\le \int_G \int_G \sum_{\sigma'} \phi^{\sigma'}(h^{-1}g) f(h) \overline{f(g)}\,dg\,dh=
\int_G \sum_{\sigma'} \underbrace{\int_G  \phi^{\sigma'}(h^{-1}g) f(h)\,dh}_{= f\ast \phi^{\sigma'}} \overline{f(g)}\,dg\ .}
$$
Now recalling that the subspaces $L^2_{\sigma'}(G)$ are pairwise orthogonal
under the product of convolution, we obtain
$$
f\ast \phi^{\sigma'} \ne 0 \iff \sigma'=\sigma\ .
$$
Therefore for every $\sigma\in \hat G$
\begin{equation}
\int_G \int_G \phi^{\sigma}(h^{-1}g) f(h) \overline{f(g)}\,dg\,dh=\int_G \int_G \phi(h^{-1}g) f(h) \overline{f(g)}\,dg\,dh \ge 0
\end{equation}
for every $f\in L^2_\sigma(G)$. Let now $f\in L^2(G)$ and let $f=\sum_{\sigma'} f^{\sigma'}$ its Fourier series. The same argument as above gives
$$
\int_G \int_G \phi^{\sigma}(h^{-1}g) f(h) \overline{f(g)}\,dg\,dh=\int_G \int_G \phi^{\sigma}(h^{-1}g) f^\sigma(h) \overline{f^\sigma(g)}\,dg\,dh\ge 0\ ,
$$
so that $\phi^\sigma$ is a positive definite function.
\end{proof}

\noindent Another important property enjoyed by positive definite and continuous functions on $G$ is shown
in the following classical theorem (see  \cite{GAN:60}, Theorem 3.20, p.151).
\begin{theorem}\label{gangolli-true}
Let $\zeta$ be a continuous positive definite function on $G$. Let $\zeta^\sigma$ be
the component of $\zeta$ on the $\sigma$-isotypical subspace $L^2_\sigma(G)$, then
$$
\zeta=\sum_{\sigma\in \hat G} \sqrt{\text{dim}\,\sigma} \tr\, \widehat \zeta(\sigma) < +\infty\ ,$$
and the Fourier series
$$
\zeta = \sum_{\sigma\in \hat G} \zeta^\sigma
$$
converges uniformly on $G$.
\end{theorem}
\begin{remark}\rm
This theorem is an extension of a classical result for trigonometric series:
\emph{every continuous function on the unit circle with all nonnegative Fourier coefficients has
its Fourier series converging uniformly on the unit circle}. \qed
\end{remark}

\chapter{Representation of isotropic Gaussian fields}

In this chapter we recollect the first part of \cite{mauspin}: as stated in the Introduction, starting from P. L\'evy's construction of his spherical Brownian motion, we prove a representation formula for isotropic Gaussian fields on homogeneous spaces $\X$ of a compact group $G$ (\S 2.1 and \S 2.2).

In particular, we show that to every square integrable bi-$K$-invariant
function $f$ on $G$ a Gaussian isotropic random field on $\X$ can be associated
 and also that every \emph{real} Gaussian
isotropic  random field on $\X$ can be obtained in this way.

This kind of result is extended to the case of  random fields in the
spin-line bundles of the sphere in the second part of \cite{mauspin} and will be presented in the last chapter of this thesis.

\section{Construction of isotropic Gaussian fields}\label{sec4}

In this section we point out the method for Gaussian isotropic random
fields on the homogeneous space $\X$ of a compact group $G$. We start with the construction
of a white noise on $\X$.

Let $(X_n)_n$ be a sequence of i.i.d. standard Gaussian r.v.'s on some probability
space $(\Omega, \F, \P)$ and  denote by $\mathscr{H}\subset L^2(\P)$ the real Hilbert
space generated by $(X_n)_n$.
Let $(e_n)_n$ be an orthonormal basis of $L^2_{\mathbb{R}}(\mathscr{X})$, the space of real square integrable
functions on $\X$.
We define an isometry $S:L^2_{\mathbb{R}}(\mathscr{X})\to\mathscr{H}$ by
$$
L^2_{\mathbb{R}} (\mathscr{X})\ni \sum_k \alpha_k
e_k\enspace\leftrightarrow\enspace \sum_k \alpha_k X_k \in
\mathscr{H}\ .
$$
It is easy to extend $S$ to an isometry on  $L^2(\mathscr{X})$,
indeed if $f\in L^2(\mathscr{X})$, then $f=f_1+if_2$, with $f_1, f_2
\in L^2_{\mathbb{R}}(\mathscr{X})$, hence just set
$S(f)=S(f_1)+iS(f_2)$. Such an isometry respects the real character
of the function $f\in L^2(\X)$ (i.e. if $f$ is real then $S(f)$ is a
real r.v.).

Let $f$ be a left $K $-invariant function in $L^2(\mathscr{X})$.
We then define a random field $(T^f_x)_{x\in \mathscr{X}}$
associated to $f$ as follows: set $T^f_{x_0}=S(f)$ and, for every $x\in \X$,
\begin{equation}\label{def campo}
T^f_x =S(L_gf)\ ,
\end{equation}
where $g\in G$ is such that $gx_0=x$ ($L$ still denotes the left
regular action of $G$).
This is a good definition: in fact if also $\widetilde{g}\in G$ is such that
$\widetilde{g}x_0=x$, then $\widetilde{g}=gk$ for some $k\in K$ and therefore $L_{\widetilde g}f(x)=f(k^{-1}g^{-1}x)=
f(g^{-1}x)=L_gf(x)$ so that
$$
S(L_{\widetilde{g}}f)=S(L_gf)\ .
$$
The random field $T^f$ is mean square integrable, indeed
$$
\dlines{ \E \Bigl[\int_\X |T^f_x|^2\,dx \Bigr]
< +\infty\ .}
$$
Actually,
if $g_x$ is any element of $G$ such that $g_xx_0=x$ (chosen in some
measurable way), then, as $\E[|T^f_x|^2]=\E [|S(L_{g_x} f)|]^2=\|
L_{g_x}f \|^2_{L^2(\X)}=\| f \|^2_{L^2(\X)}$, we have $\E \int_\X |T^f_x|^2\,dx= \| f \|^2_{L^2(\X)}$.
$T^f$ is a centered and \emph{complex-valued Gaussian} random field.
Let us now check that $T^f$ is isotropic. Recall that
the law of a complex-valued Gaussian random vector $Z=(Z_1, Z_2)$ is
completely characterized by its mean value $\E[Z]$, its covariance
matrix $\E[ \left(Z - \E[Z]\right) \left( Z- \E[Z] \right)^*]$ and
the \emph{pseudocovariance} or \emph{relation matrix} $\E[ \left(Z -
\E[Z]\right) \left( Z- \E[Z] \right)']$. We have

(i) as $S$ is an isometry
$$
\displaylines{
\mathbb{E}[T^f_{gx}\overline{T^f_{gy}}]=\mathbb{E}[S(L_{gg_x}f)\overline{S(L_{gg_y}f)}]=
\langle L_{gg_x }f, L_{gg_y}f\rangle_{L^2(\X)}
=\cr
=\langle L_{g_x }f, L_{g_y}f \rangle_{L^2(\X)}=
\mathbb{E}[T^f_x\overline{T^f_y}]\ .}
$$

(ii) Moreover, as complex conjugation commutes both with $S$ and the
left regular representation of $G$,
$$
\displaylines{
\mathbb{E}[T^f_{gx}T^f_{gy}]=\mathbb{E}[S(L_{gg_x}f)\overline{S(L_{gg_y}\overline{f})}]=
\langle L_{gg_x }f, L_{gg_y}\overline{f}\rangle_{L^2(\X)}=\cr
=
\langle L_{g_x }f, L_{g_y}\overline{f}\rangle_{L^2(\X)}=\mathbb{E}[T^f_xT^f_y]\ .}
$$
Therefore $T^f$ is isotropic because it has the same covariance and
relation kernels as the rotated field $(T^f)^g$ for every $g\in G$.

If $R^f(x,y)=\E[T^f_x \overline{T^f_y}]$ denotes its covariance
kernel, then the associated positive definite function
$\phi^f(g):=R(gx_0,x_0)$ satisfies
\begin{equation}\label{convolution for phi}
\begin{array}{c}
\displaystyle\phi^f(g)=\E[S(L_g f)\overline{S(f)}]=
\langle L_gf, f\rangle
=\\
\noalign{\vskip3pt}
\displaystyle= \int_G \widetilde f(g^{-1}h) \overline{\widetilde f(h)}\,dh= \int_G \widetilde f(g^{-1}h) \breve {\widetilde f} (h^{-1})\,dh=\widetilde f \ast \breve {\widetilde f} (g^{-1})\ , \\
\end{array}
\end{equation}
where $\widetilde f$ is the pullback on $G$ of $f$ and the convolution $\ast$ is in $G$. Moreover the relation function of $T^f$, that is
$
\zeta^f(g) := \E[T^f_{gx_0} T^f_{x_0}]
$
satisfies
\begin{equation}\label{convolution for zeta}
\zeta^f(g)=\E[S(L_gf)S(f)]=\langle L_gf, \overline{f}\rangle\ .
\end{equation}
One may ask whether every a.s. square integrable, isotropic, complex-valued Gaussian centered random field on $\X$ can be obtained with this construction: the answer
is \emph{no} in general. It is however positive if we consider
{\it real} isotropic Gaussian random fields (see Theorem \ref{real-general} below). Before considering the case of a general homogeneous space $\X$, let us look first at the case of the sphere, where things are particularly simple.

\begin{remark}\label{rem-sfera} \rm (Representation of
real Gaussian isotropic random fields on $\cS^2$) If $\X=\cS^2$
under the action of $SO(3)$, every isotropic, \emph{real} Gaussian
and centered random field is of the form \paref{def campo} for some
left-$K$-invariant function $f:\cS^2\to \R$. Indeed let us consider on $L^2(\cS^2)$ the Fourier
basis $Y_{\ell,m}$, $\ell=1,2,\dots$, $m=-\ell,\dots,\ell$,
given by the spherical harmonics (\ref{armoniche sferiche1}).
Every continuous positive definite left-$K$-invariant function $\phi$ on $\cS^2$ has a Fourier expansion of the form (\ref{PW invariant sphere})
\begin{equation}\label{fi per sfera}
\phi = \sum_{\ell \ge 0} \alpha_\ell Y_{\ell,0}\ ,
\end{equation}
where (Proposition
\ref{structure of positive definite}) $\alpha_\ell \ge 0$ and
$$
\sum_{\ell \ge 0} \sqrt{2\ell+1}\,\alpha_\ell<+\infty
$$
(Theorem \ref{gangolli-true}).
The $Y_{\ell,0}$'s being real, the function $\phi$ in (\ref{fi per sfera})
is\emph{ real}, so that, 
$\phi(g)=\phi(g^{-1})$
(in this remark and in the next example we identify functions on $\cS^2$ with
their pullbacks on $SO(3)$ for simplicity of notations).

If $\phi$ is the positive definite left-$K$-invariant function associated to
$T$, then, keeping in mind that $Y_{\ell,0}*Y_{\ell',0}=(2\ell+1)^{-1/2}
Y_{\ell,0}\,\delta_{\ell,\ell'}$, a  ``square root'' $f$ of $\phi$ is given by
\begin{equation}
f = \sum_{\ell \ge 0} \beta_\ell \, Y_{\ell,0}\ ,
\end{equation}
where $\beta_\ell$ is a complex number such that
$$
\frac{|\beta_\ell |^2}{\sqrt{2\ell+1}}= \sqrt{\alpha_\ell}\ .
$$
Therefore there exist infinitely many real
functions $f$ such that $\phi(g)=\phi(g^{-1})=
f \ast \breve{f}(g)$, corresponding to the choices $\beta_\ell=\pm
( (2\ell+1)\alpha_\ell )^{1/4}$. For each of these, the random field $T^f$ has
the same distribution as $T$, being real and having the same associated positive
definite function. \qed
\end{remark}
As stated in the Introduction, this method generalizes P. L\'evy's construction of his spherical Brownian motion. In the following example, we show
the connection between this construction and our method. Moreover, it is easy to extend the following to the case of the hyperspherical Brownian motion.
\begin{example}\rm{(P.L\'evy's spherical Brownian field)}\label{MB}.
Let us choose as a particular instance of the previous construction
$f=c1_{H}$, where $H$ is the half-sphere centered at the north pole
$x_0$ of $\cS^2$ and $c$ is some constant to be chosen later.

Still denoting by $S$ a white noise on $\cS^2$, from (\ref{def campo})
we have
\begin{equation}
T^f_x = c S(1_{H_x})\ ,
\end{equation}
where $1_{H_x}$ is the half-sphere centered at $x\in \cS^2$. Now,
let $x, y\in \bS^2$ and denote by $d(x,y) = \theta$ their distance, then,
$S$ being an isometry,
\begin{equation}
\Var(T^f_x - T^f_y) = c^2 \| 1_{H_x\vartriangle H_y}\|^2\ .
\end{equation}
The symmetric difference $H_x\vartriangle H_y$ is formed by the union of
two wedges whose total surface is equal to $\frac{\theta}{\pi}$
(recall that we
consider the surface of $\cS^2$ normalized with total mass $=1$).
Therefore, choosing $c= \sqrt{\pi}$, we have
\begin{equation}
\Var(T^f_x - T^f_y) = d(x,y)
\end{equation}
and furthermore $\Var(T^f_x) = \tfrac\pi2$.
Thus
\begin{equation}\label{aa}
\Cov(T^f_x, T^f_y) = \tfrac12 \bigl( \Var(T^f_x) + \Var(T^f_y) -
\Var(T^f_x - T^f_y)\bigr) = \tfrac\pi2 - \tfrac12 d(x,y)\ .
\end{equation}
Note that the positive definiteness of (\ref{aa}) implies that the distance $d$
is a Schoenberg restricted negative definite kernel on $\cS^2$ (see \paref{neg def}). The random
field $W$
\begin{equation}
W_x := T^f_x - T^f_{o}\ ,
\end{equation}
where $o$ denotes the north pole of the sphere
is \emph{P.L\'evy's spherical Brownian field}, as $W_o=0$ and its
covariance kernel is
\begin{equation}\label{kernel del mb}
\Cov(W_x, W_y) = \tfrac12 \left( d(x,o) + d(y,o) - d(x,y) \right)\ .
\end{equation}
In particular the kernel at the r.h.s. of (\ref{kernel del mb})
is positive definite (see also \cite{GAN:60}).
Let us compute the expansion into spherical harmonics of the
positive definite function $\phi$ associated to the random field
$T^f$ and to $f$.  We have $\phi(x)=\frac\pi2-\frac
12\,d(x,o)$, i.e.  $\phi(x)=\frac\pi2-\frac12\, \vt$ in spherical coordinates, $\vt$ being the colatitude of $x$, whereas
$Y_{\ell,0}(x)=\sqrt{2\ell+1}\,P_\ell(\cos\vt)$ where $P_\ell$ is the $\ell$-th Legendre polynomial. This formula for the central spherical harmonics
differs slightly from the usual one, as we consider the total
measure of $\cS^2$ to be $=1$. Then, recalling the normalized
measure of the sphere is $\frac 1{4\pi}\,\sin\vt\, d\vt\, d\phi$ and
that $Y_{\ell,0}$ is orthogonal to the constants
$$
\dlines{ \int_{\cS^2}\phi(x)Y_{\ell,0}(x)\, dx=-\frac
14\,\sqrt{2\ell+1}\int_0^\pi\vt P_\ell(\cos\th)\sin \vt\, d\vt=\cr
=-\frac 14\,\sqrt{2\ell+1}\int_{-1}^{1}\arccos t\, P_\ell(t)\, dt
=\frac 14\,\sqrt{2\ell+1}\int_{-1}^{1}\arcsin t\, P_\ell(t)\,
dt=\frac14\,\sqrt{2\ell+1}\,c_\ell\ , \cr }
$$
where
$$
c_\ell=\pi\Bigl\{\frac{3\cdot 5\cdots(\ell-2)}{2\cdot 4\cdots
(\ell+1))}\Bigr\}^2\,\quad \ell=1,3,\dots
$$
and $c_\ell=0$ for $\ell$ even (see \cite{WW}, p.~325). As for
the function $f=\sqrt{\pi}\,1_{H}$, we have
$$
\int_{\cS^2}f(x)Y_{\ell,0}(x)\, dx=\frac {\sqrt{\pi}}2
\,\sqrt{2\ell+1}\int_0^{\pi/2}P_\ell(\cos\vt)\sin\vt\, d\vt= \frac
{\sqrt{\pi}}2 \,\sqrt{2\ell+1}\int_0^1P_\ell(t)\, dt\ .
$$
The r.h.s. can be computed using Rodriguez formula for the
Legendre polynomials (see again \cite{WW}, p.~297) giving
that it vanishes for $\ell$ even and equal to
\begin{equation}\label{2m}
(-1)^{m+1}\,\frac {\sqrt{\pi}}2 \,\sqrt{2\ell+1}\,\frac
{(2m)!{2m+1\choose m}}{2^{2m+1}(2m+1)!}
\end{equation}
for $\ell=2m+1$. Details of this computation are given in Remark
\ref{rod}. Simplifying the factorials the previous
expression becomes
$$
\dlines{ (-1)^m\,\frac {\sqrt{\pi}}2 \,\sqrt{2\ell+1}\,\frac
{(2m)!}{2^{2m+1}m!(m+1)!}=(-1)^m\,\frac {\sqrt{\pi}}2
\,\sqrt{2\ell+1}\,\frac{3\cdots (2m-1)}{2\cdots (2m+2)}=\cr
=(-1)^m\,\frac 12 \,\sqrt{2\ell+1}\,\sqrt{c_\ell}\ .\cr }
$$
Therefore the choice $f=\sqrt{\pi}\, 1_H$ corresponds to taking
alternating signs when taking the square roots. Note that the
choice $f'=\sum_\ell \beta_\ell Y_{\ell,0}$ with $\beta_\ell=\frac 12
\,\sqrt{2\ell+1}\,\sqrt{c_\ell}$ would have given a function
diverging at the north pole $o$. Actually it is elementary to
check that the series $\sum_\ell ( {2\ell+1})\,\sqrt{c_\ell}$
diverges so that $f'$ cannot be continuous by Theorem \ref{gangolli-true}.\qed
\end{example}
\begin{remark}\label{rod}\rm Rodriguez formula for the Legendre polynomials states that
$$
P_\ell(x)=\frac 1{2^\ell\ell!}\,\frac
{d^\ell\hfil}{dx^\ell}\,(x^2-1)^\ell\ .
$$
Therefore
\begin{equation}\label{integral}
\int_0^1P_\ell(x)\, dx=\frac 1{2^\ell\ell!}\,\frac
{d^{\ell-1}\hfil}{dx^{\ell-1}}\,(x^2-1)^\ell\Big|^1_0\ .
\end{equation}
The primitive vanishes at $1$, as the polynomial $(x^2-1)^\ell$ has
a zero of order $\ell$ at $x=1$ and all its derivatives up to the
order $\ell-1$ vanish at $x=1$. In order to compute the primitive at
$0$ we make the binomial expansion of $(x^2-1)^\ell$ and take the
result of the $(\ell-1)$-th derivative of the term of order $\ell-1$
of the expansion. This is actually the term of order $0$ of the
primitive. If $\ell$ is even then $\ell-1$ is odd so that this term
of order $\ell-1$ does not exist (in the expansion only even
powers of $x$ can appear). If $\ell=2m+1$, then the term of order
$\ell-1=2m$ in the expansion is
$$
(-1)^m{2m+1\choose m} z^{2m}
$$
and the result of the integral in \paref{integral} is actually, as
given in \paref{2m},
$$
(-1)^{m+1}\,\frac {(2m)!}{2^{2m+1}(2m+1)!}\,{2m+1\choose m}\ \cdotp
$$\qed
\end{remark}

\section{Representation formula}

The result of Remark \ref{rem-sfera} concerning $\mathbb S^2$ can be extended to the case of
a general homogeneous space $\X$.
We shall need the following ``square root'' theorem in the proof of the representation formula
of Gaussian isotropic random fields on $\X$.
\begin{theorem}\label{square-root} Let $\phi$ be a bi-$K$-invariant positive
definite continuous function on $G$. Then there exists a  bi-$K$-invariant function $f\in L^2(G)$
such that $\phi=f*\breve
f$. Moreover, if $\phi$ is real valued then $f$ also can be chosen
to be real valued.
\end{theorem}
\begin{proof}
For every $\sigma\in\widehat G$, $\widehat \phi(\sigma)$ is
Hermitian positive definite. Therefore there exist matrices
$\Lambda(\sigma)$ such that
$\Lambda(\sigma)\Lambda(\sigma)^*=\widehat\phi(\sigma)$.
Let
$$
f= \sum_{\sigma\in\widehat G} \underbrace{\sqrt{\dim\sigma}\,
\tr\bigl(\Lambda(\sigma) D^\sigma\bigr)}_{=f^\sigma}\ .
$$
This actually defines a function $f\in L^2(G)$ as it is easy to see that
$$
\Vert f^\sigma\Vert^2_2=\sum_{i,j=1}^{\dim\sigma}(\Lambda(\sigma)_{ij})^2=
\tr(\Lambda(\sigma)\Lambda(\sigma)^*)=\tr(\widehat\phi(\sigma))
$$
so that
$$
\Vert f\Vert^2_2=\sum_{\sigma\in\widehat G}\Vert f^\sigma\Vert^2_2=
\sum_{\sigma\in\widehat G}\tr(\widehat\phi(\sigma))<+\infty
$$
thanks to \paref{gangolli-true}.
By Remark \ref{aggiunto} and (\ref{conv1}), we have
$$
\phi= f\ast \breve f\ .
$$
Finally the matrix $\Lambda(\sigma)$ can be chosen to be Hermitian
and with this choice $f$ is bi-$K $-invariant
 as the relation \paref{fourier-bicappa}
$\widehat f (\sigma)=P_{\sigma,0}\widehat f(\sigma)P_{\sigma,0}$
still holds. The last statement follows from next proposition.
\end{proof}
\begin{prop}\label{real-sq}
Let $\phi$ be a real positive definite function on a compact group $G$,
then there exists a real function $f$ such that $\phi=f*\breve f$.
\end{prop}
\begin{proof}
Let
$$
\phi(g)=\sum_{\sigma\in\widehat G}
\phi^\sigma(g)=\sum_{\sigma\in\widehat
G}\sqrt{\dim
\sigma}\,\tr(\widehat\phi(\sigma)D^\sigma(g))
$$
be the Peter-Weyl decomposition of $\phi$ into isotypical components. We know
that the Hermitian matrices $\widehat\phi(\sigma)$ are positive
definite, so that there exist square roots
$\widehat\phi(\sigma)^{1/2}$ i.e. matrices such that
$\widehat\phi(\sigma)^{1/2}{\widehat\phi(\sigma)^{1/2}}^*=\widehat\phi(\sigma)$
and the functions
$$
f(g)=\sum_{\sigma\in\widehat G}\sqrt{\dim
\sigma}\,\tr(\widehat\phi(\sigma)^{1/2}D^\sigma(g))
$$
are such that $\phi=f*\breve f$. We need to prove that these square
roots can be chosen in such a way that $f$ is also real. Recall that
a representation of a compact group $G$ can be classified as being
of real, complex or quaternionic type (see \cite{B-D}, p. 93 e.g. for details).

\tin{a)} If $\sigma$ is of real type then there exists a conjugation $J$
of $H_\sigma\subset L^2(G)$ such that $J^2=1$. A conjugation is a
$G$-equivariant antilinear endomorphism. It is well known that in
this case one can choose a basis $v_1,\dots, v_{d_\sigma}$ of
$H_\sigma$ formed of ``real'' vectors, i.e. such that $Jv_i=v_i$. It
is then immediate that the representative matrix $D^\sigma$ of the
action of $G$ on $H_\sigma$ is real. Actually, as $J$ is equivariant
and $Jv_i=v_i$,
$$
D^\sigma_{ij}(g)=\langle gv_j,v_i\rangle=\overline{\langle
Jgv_j,Jv_i\rangle}=\overline{\langle
gv_j,v_i\rangle}=\overline{D^\sigma_{ij}(g)}\ .
$$
With this choice of the basis, the matrix $\widehat\phi(\sigma)$ is real
and also $\widehat\phi(\sigma)^{1/2}$ can be chosen to be real and
$g\mapsto\sqrt{\dim \sigma}\,
\tr(\widehat\phi(\sigma)^{1/2}D^\sigma(g))$ turns out to be real
itself.
\tin{b)} If $\sigma$ is of complex type, then
it is not isomorphic to its dual representation $\sigma^*$.
As $D^{\sigma^*}(g):=D^\sigma (g^{-1})^t=
\overline{D^\sigma (g)}$ and $\phi$ is real-valued,  we have
\begin{equation*}\label{easy}
\widehat \phi(\sigma^*) = \overline{\widehat \phi(\sigma)}\ ,
\end{equation*}
so that we can choose $\widehat \phi(\sigma^*)^{1/2} = \overline{\widehat \phi(\sigma)^{1/2}}$ and, as $\sigma$ and $\sigma^*$ have the same dimension, the function
$$
g\mapsto \sqrt{\dim \sigma} \tr(\widehat \phi(\sigma)^{1/2} D^\sigma(g))+ \sqrt{\dim \sigma^*} \tr(\widehat \phi(\sigma^*)^{1/2} D^{\sigma^*}(g))
$$
turns out to be real.
\tin{c)} If $\sigma$ is quaternionic, let $J$ be the corresponding
conjugation. It is immediate that the vectors $v$ and $Jv$ are
orthogonal and from this it follows that
$\dim \sigma=2k$ and that there exists an orthogonal basis for $H_\sigma$
of the form
\begin{equation}\label{anti-basis}
v_1,  \dots, v_k, w_1=J(v_1), \dots, w_k=J(v_k)\ .
\end{equation}
In such a basis the representation matrix of any linear transformation $U:H_\sigma\to H_\sigma$ which commutes with $J$ has the form
\begin{equation}\label{eq blocchi-gen}
\begin{pmatrix}
A  &  B\\
\noalign{\vskip4pt}
-\overline{B} & \overline{A} \\
\end{pmatrix}
\end{equation}
and in particular $D^\sigma(g)$ takes the form
\begin{equation}\label{eq blocchi}
D^\sigma(g)=\begin{pmatrix}
A(g)  &  B(g)\\
\noalign{\vskip4pt}
-\overline{B(g)} & \overline{A(g)} \\
\end{pmatrix}\ .
\end{equation}
By \paref{eq blocchi} we have also, $\phi$ being real valued,
\begin{equation}\label{matr per fi}
\widehat \phi(\sigma) = \begin{pmatrix}
\int_G \phi(g)A(g^{-1})\,dg  &  \int_G \phi(g)B(g^{-1})\,dg\\
\noalign{\vskip6pt}
 -\int_G \phi(g)\overline{B(g^{-1})}\,dg & \int_G
\phi(g)\overline{A(g^{-1})}\,dg \\
\end{pmatrix}:=\begin{pmatrix}
\phi_A   &  \phi_B \\
\noalign{\vskip4pt}
-\overline{ \phi_B} & \overline{\phi_A } \\
\end{pmatrix}\ .
\end{equation}

More interestingly, if $\phi$ is any function such that, with respect to the basis above,
$\widehat\phi(\sigma)$ is of the form \paref{matr per fi}, then the corresponding component
$\phi^\sigma$ is necessarily a real valued function: actually
$$
\dlines{
\phi^\sigma(g)=\tr(\widehat\phi(\sigma)D^\sigma(g))
=\tr\bigl(\phi_A A(g)-\phi_B\overline {B(g)}-\overline{\phi_B}B(g)+\overline {\phi_A}\,\overline {A(g)}\bigr)=\cr
=\tr\bigl(\phi_A A(g)+\overline{\phi_A A(g)}\bigr)-
\tr\bigl(\phi_B\overline {B(g)}+\overline{\phi_B\overline {B(g)}}\bigr)\ .\cr
}
$$
We now prove that the Hermitian square root, $U$ say, of $\widehat\phi(\sigma)$ is of the form
\paref{matr per fi}. Actually note that $\widehat\phi(\sigma)$ is self-adjoint, so that it
can be diagonalized and all its eigenvalues are real (and positive by Proposition
\ref{structure of positive definite} a)). Let $\lambda$ be an eigenvalue and $v$ a
corresponding eigenvector. Then, as
$$
\widehat\phi(\sigma)Jv=J\widehat\phi(\sigma)=J\lambda v=\lambda Jv\ ,
$$
$Jv$ is also an eigenvector associated to $\lambda$. Therefore there exists a basis as in
\paref{anti-basis} that is formed of eigenvectors, i.e. of the form
$v_1,\dots,v_k, w_1,\dots,w_k$ with $Jv_j=w_j$ and
$v_j$ and $w_j$ associated to the same positive eigenvalue $\lambda_j$. In this basis
$\widehat\phi(\sigma)$ is of course diagonal with the (positive) eigenvalues on the diagonal.
Its Hermitian square root $U$ is also diagonal, with the square roots of the eigenvalues
on the diagonal. Therefore $U$ is also the form \paref{matr per fi} and the corresponding
function $\psi(g)=\tr(UD(g))$ is real valued and such that $\psi*\breve\psi=\phi^\sigma$.
\end{proof}
Note that the decomposition of Theorem \ref{square-root} is not
unique, as the Hermitian square root of the positive definite
operator $\widehat\phi(\sigma)$ is not unique itself.
%
%
Now we prove the main result of this chapter.
\begin{theorem}\label{real-general} Let $\X$ be the homogeneous space of a compact group $G$ and let $T$ be an a.s. square
 integrable isotropic Gaussian \emph{real} random field on $\X$.
 Then there exists a left-$K$-invariant function $f\in L^2(\X)$ such that $T^f$ has the same distribution as $T$.
\end{theorem}
\begin{proof} Let $\phi$ be the invariant positive definite function associated to $T$.
Thanks to \paref{convolution for phi} it is sufficient to prove that
there exists a \emph{real} $K$-invariant
function $f\in L^2(\X)$ such that $\phi(g)=\widetilde f \ast \breve{\widetilde
f}(g^{-1})$. Keeping in mind that $\phi(g)=\phi(g^{-1})$, as $\phi$
is real, 
this follows from
Theorem \ref{square-root}.
\end{proof}
As remarked above $f$ is not unique.

Recall that a complex valued Gaussian r.v. $Z=X+iY$ is said to be
\emph{complex Gaussian} if the
r.v.'s $X,Y$ are jointly Gaussian, are independent and have the same variance. A $\C^m$-valued r.v.
$Z=(Z_1,\dots, Z_m)$ is said to be complex-Gaussian if the r.v.
$\alpha_1Z_1+\dots+\alpha_mZ_m$ is complex-Gaussian for every choice of $\alpha_1,\dots,\alpha_m\in \C$.
\begin{remark}An a.s. square integrable random field $T$ on $\X$ is  \emph{complex Gaussian}
if and only if the complex valued r.v.'s
$$
\int_\X T_xf(x)\, dx
$$
are complex Gaussian for every choice of $f\in L^2(\X)$.
\end{remark}
Complex Gaussian random fields will play an important role in the last chapter of this work. By
now let us remark
that, in general, it is not possible to obtain a complex Gaussian random field by the procedure
\paref{def campo}.

\begin{prop}\label{zeta-prop} Let $\zeta(x,y)=\E[T_xT_y]$ be the relation kernel of a centered complex
Gaussian random field $T$. Then $\zeta\equiv 0$.
\end{prop}
\begin{proof} It easy to check that a centered complex valued r.v. $Z$ is complex Gaussian if and only
if $\E[Z^2]=0$. As for every $f\in L^2(\X)$
$$
\int_\X\int_\X \zeta(x,y)f(x)f(y)\, dx dy=\E\Bigl[\Bigl(\int_\X T_xf(x)\, dx\Bigr)^2\Bigr]=0\ ,
$$
it is easy to derive
that $\zeta\equiv0$.
\end{proof}
Going back to the situation of Remark \ref{rem-sfera}, the relation function $\zeta$ of the
random field $T^f$ is easily found to be
\begin{equation}\label{cic}
\zeta^f=\sum_{\ell \ge 0} \beta_\ell^2\, Y_{\ell,0}\ ,
\end{equation}
and cannot vanish unless $f\equiv 0$ and $T^f$ vanishes itself.
Therefore no isotropic complex Gaussian random field on the sphere can be obtained by
the construction \paref{def campo}.

\chapter{On L\'evy's Brownian fields}

In 1959 P.L\'evy  \cite{levy} asked the question of the existence of a random field $X$
indexed by the points of
a metric space $(\X,d)$ and generalizing the Brownian motion, i.e. of a real Gaussian process
which would be centered, vanishing at some point $x_0\in\X$ and such that $\E(|X_x-X_y|^2)=d(x,y)$.
By polarization, the covariance function of such a process would be
\begin{equation}\label{kernel MB}
K(x,y)=\frac{1}{2} \,( d(x,x_0) + d(y,x_0) - d(x,y) )
\end{equation}
so that this question is equivalent to the fact that the kernel $K$ is positive definite. As anticipated in the Introduction, $X$ is called P.L\'evy's Brownian field on $(\X,d)$. Positive
definiteness of $K$ for $\X=\R^{m+1}$ and $d$ the Euclidean metric had been proved by
\cite{schoenberg} in 1938 and P.L\'evy itself
constructed the Brownian field on $\X=\bS^{m}$, the euclidean sphere of $\R^{m+1}$,
 $d$ being the distance along the geodesics (Example \ref{MB}).
Later Gangolli \cite{GAN:60} gave an analytical proof of the positive definiteness of the kernel
 \paref{kernel MB} for the same metric space $(\bS^{m},d)$, in a paper that dealt with this
 question for a large class of homogeneous spaces.

Finally Takenaka in \cite{TKU} proved the positive definiteness of the kernel \paref{kernel MB}
for the Riemannian metric spaces
of constant sectional curvature equal to $-1,0$ or $1$, therefore adding the hyperbolic disk to the
list. To be precise in the case of the hyperbolic space
$\mathcal{H}_m = \lbrace (x_0, x_1, \dots, x_m)\in \R^{m+1} :
x_1^2 + \dots x_m^2 - x_0^2 = 1 \rbrace $, the distance under consideration  is the unique, up to
multiplicative
constants, Riemannian distance that is invariant with respect to the action of $G=L_m$, the Lorentz group.

In \cite{mauSO(3)}
we investigate this question for the cases
$\X=SO(n)$. The answer is that the kernel \paref{kernel MB} is not
positive definite on $SO(n)$ for $n>2$. This is somehow surprising
as, in particular, $SO(3)$ is locally isometric to $SU(2)$, where
positive definiteness of the kernel $K$ is immediate as shown below.

It is immediate that this imply the non existence on $SU(n), n>2$.

The plan of this chapter is as follows.
In \S\ref{elem} we recall some elementary facts about invariant distances and positive definite kernels. In
\S\ref{sud} we treat the case $G=SU(2)$, recalling well known facts about the invariant distance and Haar measure of this group.
Positive definiteness of $K$ for $SU(2)$ is just a simple remark, but these facts are needed in
\S\ref{sot} where we treat the case $SO(3)$ and deduce from the case $SO(n)$, $n\ge 3$.

\section{Some elementary facts}\label{elem}

In this section we recall some well known facts about Lie groups (see mainly
 \cite{faraut} and also \cite{MR2088027, sugiura}).

\subsection{Invariant distance of a compact Lie group}

In this chapter $G$ denotes a compact \emph{Lie group}. It is well known that  $G$ admits {at least} a bi-invariant Riemannian metric
(see \cite{MR2088027} p.66 e.g.), that
we shall denote $\lbrace \langle \cdot, \cdot \rangle_g \rbrace_{g\in G}$ where of course
$\langle \cdot, \cdot \rangle_g$ is the inner product defined on the tangent space $T_g G$ to the manifold $G$ at $g$ and the family $\lbrace \langle \cdot, \cdot \rangle_g \rbrace_{g\in G}$ smoothly depends on $g$. By the bi-invariance property, for $g\in G$ the diffeomorphisms
$L_g$ and $R_g$ (resp. the left multiplication and the right multiplication of the group) are isometries.
Since the tangent space $T_g G$ at any point $g$ can be translated to the tangent space $T_e G$ at the identity element $e$ of the group, the metric $\lbrace \langle \cdot, \cdot \rangle_g \rbrace_{g\in G}$ is completely characterized by $\langle \cdot, \cdot \rangle_e$.
Moreover, $T_e G$ being the Lie algebra $\mathfrak g$ of $G$, the bi-invariant metric corresponds to an inner product $\langle \cdot, \cdot \rangle$ on
$\mathfrak g$ which is invariant under the adjoint representation $Ad$ of $G$. Indeed there is a one-to-one correspondence between bi-invariant Riemannian metrics on $G$ and $Ad$-invariant inner products on
$\mathfrak g$. If in addition $\mathfrak g$ is {semisimple}, then the negative Killing form
of $G$ is an $Ad$-invariant inner product on $\mathfrak g$ itself.

If there exists a unique
(up to a multiplicative factor) bi-invariant metric on $G$ (for a sufficient condition see
\cite{MR2088027}, Th. $2.43$) and $\mathfrak g$ is semisimple, then
this metric is necessarily proportional to the negative Killing form of $\mathfrak g$. It is well known that this is the case for $SO(n), (n\ne 4)$ and $SU(n)$; furthermore, the (natural) Riemannian metric on $SO(n)$ induced by the embedding
$SO(n) \hookrightarrow \R^{n^2}$ corresponds to the negative Killing form of ${so}(n)$.

Endowed with this bi-invariant Riemannian metric, $G$ becomes a {metric space}, with a distance $d$ which is
bi-invariant. Therefore the function $g\in G \to d(g,e)$ is a class function as
\begin{equation}\label{cfunction}
d(g,e)=d(hg,h)=d(hgh^{-1},hh^{-1})=d(hgh^{-1},e), \qquad g,h\in G\ .
\end{equation}
It is well known that {geodesics} on $G$ through the identity $e$
are exactly the one parameter subgroups of $G$ (see \cite{MR0163331} p.113 e.g.), thus a geodesic from $e$
is the curve on $G$
\begin{equation*}
\gamma_X(t) : t\in [0,1] \to \exp(tX)
\end{equation*}
for some $X\in \mathfrak g$. The length of  this geodesic is
\begin{equation*}
L(\gamma_X) = \| X \| = \sqrt{\langle X, X \rangle}\ .
\end{equation*}
Therefore
\begin{equation*}\label{distanza}
d(g,e) = \inf_{X\in \mathfrak g: \exp X=g} \| X \|\ .
\end{equation*}

\subsection{Brownian kernels on a metric space}

Let $(\X, d)$ be a metric space.
\begin{lemma}\label{fondamentale}
The kernel $K$ in (\ref{kernel MB}) is positive definite on $\X$ if
and only if  $d$ is a restricted negative definite kernel, i.e., for
every choice of elements $x_1,  \dots, x_n\in \X$ and of complex
numbers $\xi_1, \dots, \xi_n$ with $\sum_{i=1}^n \xi_i =0$
\begin{equation}\label{neg def}
\sum_{i,j=1}^n d(x_i,x_j) \xi_i \overline{\xi_j} \le 0\ .
\end{equation}
\end{lemma}

\begin{proof}
For every $x_1, \dots, x_n\in \X$ and complex numbers $\xi_1,\dots,
\xi_n$
\begin{equation}\label{eq1}
\sum_{i,j} K(x_i,x_j) \xi_i \overline{\xi_j} = \frac{1}{2} \Bigl(
\overline{a} \sum_i d(x_i, x_0)\xi_i  + a \sum_j d(x_j,
x_0)\overline{\xi_j}- \sum_{i,j}d(x_i, x_j)\xi_i \overline{\xi_j}
\Bigr)
\end{equation}
where $a:= \sum_i \xi_i$.  If  $a=0$ then it is immediate that in (\ref{eq1}) the l.h.s. is $\ge 0$ if and only if the r.h.s. is $\le 0$. Otherwise set
$\xi_{0}:=-a$ so that $\sum_{i=0}^n \xi_i =0$. The following equality
\begin{equation}
\sum_{i,j=0}^{n} K(x_i,x_j) \xi_i \overline{\xi_j} =
\sum_{i,j=1}^n K(x_i,x_j) \xi_i \overline{\xi_j}
\end{equation}
is then easy to check, keeping in mind that $K(x_i,x_0)=K(x_0,
x_j)=0$, which finishes the proof.
\end{proof}
For a more general proof see \cite{GAN:60} p. $127$ in the proof of Lemma 2.5.

If $\X$ is the homogeneous space of some topological group $G$, and
$d$ is a $G$-invariant distance, then (\ref{neg def}) is satisfied
if and only if for every choice of elements $g_1,\dots,g_n\in G$ and
of complex numbers $\xi_1,\dots, \xi_n$ with $\sum_{i=1}^n \xi_i =0$
\begin{equation}\label{neg def2}
\sum_{i,j=1}^n d(g_ig_j^{-1}x_0,x_0) \xi_i \overline{\xi_j} \le 0
\end{equation}
where $x_0\in \X$ is a fixed point. We shall say that
the function $g\in G \to d(gx_0, x_0)$ is restricted negative definite on $G$ if it satisfies (\ref{neg def2}).

In our case of interest $\X=G$  a compact  (Lie) group  and
$d$ is a bi-invariant distance as in $\S 3.1$.

The Peter-Weyl development  for the class function $d(\cdot,e)$ on $G$ (see Theorem \ref{classExp}) is
\begin{equation}\label{PW dev}
d(g,e)= \sum_{\ell \in \widehat G} \alpha_\ell \chi_\ell(g)\ ,
\end{equation}
where $\widehat G$ denotes the family of equivalence classes of irreducible representations of $G$ and $\chi_\ell$  the character of the
$\ell$-th irreducible representation of $G$.

\begin{remark}\label{coeff neg}\rm
A function $\phi$ with a development as in (\ref{PW dev}) is
restricted negative definite if and only if $\alpha_\ell \le 0$ but
for the trivial representation.

Actually note first that, by standard approximation arguments,
$\phi$ is restricted negative definite if and only if for every
continuous function $f:G\to \C$ with $0$-mean (i.e. orthogonal to
the constants)
\begin{equation}\label{neg def measure}
\int_G\int_G \phi(gh^{-1}) f(g)\overline{f(h)}\,dg\, dh \le 0
\end{equation}
$dg$ denoting the Haar measure of $G$.
Choosing $f=\chi_\ell$ in the l.h.s. of
(\ref{neg def measure}) and denoting $d_\ell$ the dimension of the corresponding representation, a straightforward computation gives
\begin{equation}\label{semplice}
\int_G\int_G \phi(gh^{-1}) \chi_\ell(g)\overline{\chi_\ell(h)}\,dg\,
dh = \frac{\alpha_\ell}{d_\ell}
\end{equation}
so that if $\phi$ restricted negative definite, $\alpha_\ell\le 0$
necessarily.

Conversely, if $\alpha_\ell \le 0$ 
but for the
trivial representation, then $\phi$ is restricted negative definite,
as the characters $\chi_\ell$'s are positive definite and orthogonal
to the constants.
\end{remark}

\section{$SU(2)$}\label{sud}

The special unitary group $SU(2)$ consists of the complex unitary $2\times 2$-matrices $g$ such that
$\det(g)=1$.
Every $g\in SU(2)$ has the form
\begin{equation}\label{matrice}
g= \begin{pmatrix} a  & b \\
 -\overline{b} & \overline{a}
\end{pmatrix}, \qquad a,b\in \C,\, |a|^2 + |b|^2 = 1\ .
\end{equation}
If $a=a_1 + ia_2$ and $b=b_1 + ib_2$, then the map
\begin{align}\label{omeomorfismo}
\Phi(g)=
(a_1, a_2, b_1, b_2)
\end{align}
is an {homeomorphism} (see \cite{faraut}, \cite{sugiura} e.g.) between $SU(2)$ and the unit sphere $\cS^3$
of $\R^4$. Moreover the right translation
\begin{equation*}
R_g : h\to hg, \qquad h,g\in SU(2)
\end{equation*}
of $SU(2)$ is a rotation (an element of $SO(4)$) of $\cS^3$ (identified with $SU(2)$).
The homeomorphism (\ref{omeomorfismo}) preserves the invariant measure, i.e., if $dg$ is the normalized Haar measure
on $SU(2)$, then $\Phi(dg)$ is the normalized Lebesgue measure on $\cS^3$.
As the $3$-dimensional polar coordinates on $\cS^3$ are
\begin{equation}\label{polar coord}
\begin{array}{l}
a_1=\cos \theta,\cr
a_2= \sin \theta\,\cos \varphi,\cr
b_1= \sin \theta\,\sin \varphi\,\cos \psi,\cr
b_2= \sin \theta\,\sin \varphi\,\sin \psi\ ,
\end{array}
\end{equation}
$(\theta, \varphi, \psi) \in [0,\pi] \times [0,\pi]\times [0,2\pi]$, the normalized Haar integral of $SU(2)$ for an integrable function $f$ is
\begin{equation}\label{int}
\int_{SU(2)} f(g)\,dg = \frac{1}{2\pi^2}\int_0^\pi\sin \varphi\,  d\varphi\,
\int_0^\pi \sin^2 \theta\,d\theta\, \int_0^{2\pi}
f(\theta, \varphi, \psi)\, d\psi
\end{equation}
The bi-invariant Riemannian metric on $SU(2)$ is necessarily proportional to the negative
Killing form  of its Lie algebra ${su(2)}$  (the real vector space of
$2\times 2$ anti-hermitian complex matrices).
We consider the bi-invariant metric corresponding to the $Ad$-invariant inner product on ${su(2)}$
\begin{equation*}
\langle X, Y \rangle= -\frac12\,{\tr(XY)},\qquad X,Y \in {su(2)}\ .
\end{equation*}
Therefore as an orthonormal basis of ${su(2)}$ we can consider the matrices
$$
\dlines{
X_1= \begin{pmatrix} 0  & 1 \\
 -1 & 0
\end{pmatrix}, \quad
X_2= \begin{pmatrix} 0  & i \\
 i & 0
\end{pmatrix}, \quad
X_3 = \begin{pmatrix} i  & 0\\
 0 & -i
\end{pmatrix}
}$$
The homeomorphism (\ref{omeomorfismo}) is actually an isometry between $SU(2)$ endowed with this
distance and $\cS^3$. Hence the restricted negative definiteness of the kernel $d$ on $SU(2)$ is an immediate consequence of this property on $\cS^3$ which is known
to be true as mentioned in the introduction (\cite{GAN:60}, \cite{levy}, \cite{TKU}).
In order to develop a comparison with $SO(3)$, we shall give a different proof of this fact in \S\ref{s-final}.
\section{$SO(n)$}\label{sot}

We first investigate the case $n=3$. The group
$SO(3)$ can also be realized as a quotient of $SU(2)$. Actually
the adjoint representation $Ad$ of $SU(2)$
is a surjective morphism from $SU(2)$ onto $SO(3)$ with kernel $\lbrace \pm e \rbrace$ (see \cite{faraut} e.g.).
Hence the well known result
\begin{equation}\label{iso}
SO(3) \cong {SU(2)}/{\lbrace \pm e \rbrace}\ .
\end{equation}
Let us explicitly recall this morphism: 
if $a=a_1 +ia_2, b=b_1 + ib_2$ with $|a|^2 + |b|^2=1$ and
$$
\widetilde g=\begin{pmatrix} a  & b \\
 -\overline{b} & \overline{a}
\end{pmatrix}$$
then the orthogonal matrix $Ad(\widetilde g)$  is given by
\begin{equation}\label{matr}
g=\begin{pmatrix}
a_1^2-a_2^2-(b_1^2-b_2^2)&-2a_1a_2-2b_1b_2&-2(a_1b_1-a_2b_2)\cr
2a_1a_2-2b_1b_2&(a_1^2-a_2^2)+(b_1^2-b_2^2)&-2(a_1b_2+a_2b_1)\cr
2(a_1b_1+a_2b_2)&-2(-a_1b_2+a_2b_1)&|a|^2-|b|^2
\end{pmatrix}
\end{equation}
The isomorphism in (\ref{iso}) might suggest that the
positive definiteness of the Brownian kernel on $SU(2)$ implies
a similar result for $SO(3)$.
This is not true and actually it turns out that the distance $(g,h) \to d(g,h)$ on $SO(3)$
induced by its bi-invariant Riemannian metric \emph{is not
a restricted negative definite kernel} (see Lemma \ref{fondamentale}).

As for $SU(2)$, the bi-invariant Riemannian metric on $SO(3)$ is proportional to the negative Killing form  of its Lie algebra
${ so(3)}$ (the real $3\times 3$ antisymmetric real matrices).
We shall consider the $Ad$-invariant inner product on ${so(3)}$ defined as
\begin{equation*}
\langle A,B \rangle = -\frac12\,{\tr(AB)}\ , \qquad A,B\in {so(3)}\ .
\end{equation*}
An  orthonormal basis for ${so(3)}$ is therefore given by
the matrices
$$\dlines{
A_1= \begin{pmatrix} 0  & 0 & 0\\
 0 & 0 & -1\\
 0 & 1 & 0
\end{pmatrix}, \quad
A_2= \begin{pmatrix} 0  & 0 & 1 \\
 0 & 0 & 0\\
 -1 & 0 & 0
\end{pmatrix}, \quad
A_3 = \begin{pmatrix} 0 & -1  & 0\\
 1 & 0 & 0\\
 0 & 0 & 0
\end{pmatrix}
}$$
Similarly to the case of $SU(2)$, it is easy to compute the distance
from  $g\in SO(3)$ to the identity. Actually $g$ is conjugated to the matrix of the form
\begin{equation*}
\Delta(t)= \begin{pmatrix} \cos t  & \sin t & 0 \\
 -\sin t & \cos t & 0\\
 0 & 0 & 1
\end{pmatrix} = \exp(tA_1)
\end{equation*}
where $t\in [0,\pi]$ is the {rotation angle} of $g$.
Therefore if $d$ still denotes the distance induced by the bi-invariant metric,
\begin{equation*}
d(g,e) = d( \Delta(t), e ) = t
\end{equation*}
i.e. the distance from $g$ to $e$ is the rotation angle of $g$.

Let us denote $\lbrace \chi_\ell \rbrace_{\ell \ge 0}$ the set of characters for $SO(3)$.
It is easy to compute the Peter-Weyl development in (\ref{PW dev}) for $d(\cdot, e)$  as  the
characters $\chi_\ell$ are also simple functions of the rotation angle.
More precisely, if $t$ is the rotation angle of $g$ (see \cite{dogiocam} e.g.),
\begin{equation*}
\chi_\ell(g)= \frac{\sin\frac{(2\ell + 1)t}{2}}{\sin\frac{t}{2}}=1 + 2\sum_{m=1}^\ell \cos(mt)\ .
\end{equation*}
We shall prove that the coefficient
$$
\alpha_\ell=\int_{SO(3)} d(g,e)\chi_\ell(g)\, dg
$$
is positive for some $\ell\ge 1$. As both $d(\cdot,e)$ and $\chi_\ell$ are functions of the rotation angle $t$,
we have
$$
\alpha_\ell=\int_0^\pi t\Bigl( 1 + 2\sum_{j=1}^\ell \cos(jt)\Bigr)\, p_T(t)\, dt
$$
where $p_T$ is the density of $t=t(g)$, considered as a r.v. on the probability space $(SO(3),dg)$. The next statements are devoted to the computation of the density $p_T$. This is certainly well
known but we were unable to find a reference in the literature. We first compute the density of the trace of $g$.
\begin{prop}
The distribution of the trace of a matrix in $SO(3)$ with respect to the normalized Haar measure is given by the density
\begin{equation}\label{trace3}
f(y)=\frac 1{2\pi}\,(3-y)^{1/2}(y+1)^{-1/2}1_{[-1,3]}(y)\ .
\end{equation}
\end{prop}
\begin{proof}
The trace of the matrix \paref{matr} is equal to
$$
\tr(g)=3a_1^2-a_2^2-b_1^2-b_2^2\ .
$$
Under the normalized Haar measure of $SU(2)$ the vector $(a_1,a_2,b_1,b_2)$ is uniformly distributed on the sphere $\cS^3$.
Recall the normalized Haar integral (\ref{int})
so that, taking the corresponding marginal, $\th$ has density
\begin{equation}\label{denth}
f_1(\th)=\frac 2\pi\,\sin^2(\th)\, d\th\ .
\end{equation}
Now
$$
\dlines{
3a_1^2-a_2^2-b_1^2-b_2^2=4\cos^2\th-1\ .\cr
}
$$
Let us first compute the density of $Y=\cos^2 X$, where
$X$ is distributed according to the density \paref{denth}. This is elementary as
$$
\dlines{
F_Y(t)=\P(\cos^2 X\le t)=
\P(\arccos(\sqrt{t})\le X\le \arccos(-\sqrt{t}))=\cr
=\frac 2\pi\!\int\limits_{\arccos(\sqrt{t})}^{\arccos(-\sqrt{t})}
\sin^2(\th)\, d\th\ .\cr
}
$$
Taking the derivative it is easily found that the density of $Y$ is, for $0<t<1$,
$$
\dlines{
F'_Y(t)=\frac 2\pi\,(1-t)^{1/2}t^{-1/2}\ .\cr
}
$$
By an elementary change of
variable the distribution of the trace $4Y-1$ is therefore given by \paref{trace3}.
\end{proof}
\begin{cor}
The distribution of the rotation angle of a matrix in $SO(3)$ is
\begin{equation*}
p_T(t)=\frac1\pi\,(1 - \cos t)\,1_{[0,\pi]}(t)\ .
\end{equation*}
\end{cor}
\begin{proof}
It suffices to remark that if $t$ is the rotation angle of $g$, then its trace is equal to $2\cos t + 1$. $p_T$ is therefore the distribution of $W=\arccos ( \frac{Y-1}{2})$, $Y$ being distributed as \paref{trace3}. The elementary details are left to the reader.
\end{proof}
Now it is easy to compute the Fourier development of the function $d(\cdot, e)$.
\begin{prop}\label{kernel MB su SO(3)}
The kernel $d$ on $SO(3)$ is not restricted negative definite.
\end{prop}
\begin{proof}
It is enough to show that in the Fourier development
$$
d(g,e)=\sum_{\ell \ge 0} \alpha_\ell \chi_\ell(g)
$$
$\alpha_\ell > 0$ for some $\ell \ge 1$ (see Remark \ref{coeff neg}).
We have
$$
\dlines{ \alpha_\ell =\int_{SO(3)} d(g,e) \chi_\ell(g) dg = \frac 1\pi
\int_0^\pi t \Bigl( 1 + 2\sum_{m=1}^\ell \cos(mt) \Bigr) (1-\cos t)\, dt =\cr
=\frac{1}{\pi}\underbrace{\int_0^\pi t (1-\cos t)\, dt}_{:= I_1} + \frac{2}{\pi}\sum_{m=1}^{\ell} \underbrace{\int_0^\pi t \cos(mt)\, dt}_{:= I_2}
- \frac{2}{\pi}\sum_{m=1}^{\ell} \underbrace{\int_0^\pi t \cos(mt)\cos t\, dt}_{:= I_3}\ .\cr
}$$
Now integration by parts gives
$$
I_1 = \frac{\pi^2}{2} + 2,\quad
I_2=
\frac{ (-1)^m -1}{m^2}\ \raise2pt\hbox{,}
$$
whereas, if $m\ne 1$, we have
$$
\dlines{I_3=\int_0^\pi t \cos(mt)\cos t\, dt
= \frac{ m^2 +1}{ (m^2 -1)^2} ( (-1)^m +1)}
$$
and for $m=1$,
$$
I_3=\int_0^\pi t\cos^2 t\, dt=\frac {\pi^2}4\ .
$$
Putting things together we find
$$
\alpha_\ell =\frac{2}{\pi} \Bigl( 1 + \sum_{m=1}^{\ell} \frac{ (-1)^m -1}{m^2} + \sum_{m=2}^{\ell} \frac{ m^2 +1}{ (m^2 -1)^2} ( (-1)^m +1) \Bigr)\ .
$$
If $\ell=2$, for instance, we find $\alpha_2=\frac{2}{9\pi}>0$, but it is easy to see
that $\alpha_\ell>0$ for every $\ell$ even.
\end{proof}

Consider now the case $n>3$. $SO(n)$ (resp. $SU(n)$) contains a closed subgroup $H$ that is isomorphic to $SO(3)$ and
the restriction to $H$ of any bi-invariant distance $d$ on $SO(n)$ (resp. $SU(n)$) is a bi-invariant distance $\widetilde d$ on $SO(3)$.
By Proposition \ref{kernel MB su SO(3)},  $\widetilde d$ is not
restricted negative definite, therefore there exist $g_1, g_2, \dots, g_m\in H$,
$\xi_1, \xi_2, \dots, \xi_m \in \R$ with $\sum_{i=1}^m \xi_i =0$ such that
\begin{equation}
\sum_{i,j} d(g_i, g_j) \xi_i \xi_j =\sum_{i,j} \widetilde d(g_i, g_j) \xi_i \xi_j > 0\ .
\end{equation}
We have therefore
\begin{cor}\label{kernel MB su SO(n)}
Any bi-invariant distance $d$ on $SO(n)$ and $SU(n)$, $n\ge 3$ is not
a restricted negative definite kernel.
\end{cor}
Remark that the same argument applies to other compact groups.
Moreover  the bi-invariant Riemannian metric on $SO(4)$ is not unique, meaning that it is not necessarily
proportional to the negative Killing form of $so(4)$.
In this case Corollary \ref{kernel MB su SO(n)} states that every such
bi-invariant distance cannot be restricted negative definite.

\section{Final remarks}\label{s-final}

We were intrigued by the different behavior of the invariant distance of $SU(2)$ and $SO(3)$ despite these groups
are locally isometric and decided to compute also for $SU(2)$ the development
\begin{equation}\label{sviluppo}
d(g,e) = \sum_{\ell} \alpha_\ell \chi_\ell(g)\ .
\end{equation}
This is not difficult as, denoting by $t$ the distance of $g$ from $e$, the characters of $SU(2)$ are
$$
\chi_\ell(g)=\frac{\sin((\ell+1)t)}{\sin t},\quad t\not=k\pi
$$
and $\chi_\ell(e)=\ell+1$ if $t=0$, $\chi_\ell(-)=(-1)^\ell(\ell+1)$ if $t=\pi$. Then it is elementary to compute, for $\ell>0$,
$$
\alpha_\ell =\frac 1\pi\,\int_0^\pi t\sin((\ell+1)t)\sin t\, dt=\begin{cases}
-\frac 8\pi\,\frac{m+1}{m^2(m+2)^2}&\ell\mbox{ odd}\cr
0&\ell\mbox{ even}
\end{cases}
$$
thus confirming the restricted negative definiteness of $d$ (see Remark \ref{coeff neg}).
Remark also that the coefficients corresponding to the even numbered
representations, that are also representations of $SO(3)$, here vanish.

\chapter*{Part 2 \\High-energy Gaussian eigenfunctions}
\addcontentsline{toc}{chapter}{Part 2: High-energy Gaussian eigenfunctions}

\chapter{Background: Fourth-Moment phenomenon and Gaussian eigenfunctions}\label{background}

As made clear by the title, this chapter is first devoted to the so-called Fourth Moment phenomenon. Main results in this area are
 summarized in the recent monograph \cite{noupebook}:
 a
beautiful connection has been established between Malliavin calculus and  Stein's method for normal approximations
to prove Berry-Esseen bounds and quantitative Central Limit
Theorems
for functionals of a Gaussian random field.

Finally we recall definitions and fix some notation for Gaussian eigenfunctions on the $d$-dimensional unit sphere $\mathbb S^d$ ($d\ge 2$) whose properties we will deeply investigate in the sequel of this work.

\section{Fourth-moment theorems}

\subsection{Isonormal Gaussian fields}

Let $H$ be a (real) separable Hilbert space with inner product $\langle \cdot, \cdot \rangle_H$ and $(\Omega, \F, \P)$ some probability space.
\begin{definition}
The isonormal Gaussian field $T$ on $H$ is a centered Gaussian random field $(T(h))_{h\in H}$ whose covariance kernel is given by
\begin{equation}
\Cov(T(h), T(h')) = \langle h, h' \rangle_H\ , \qquad h,h'\in H\ .
\end{equation}
\end{definition}
Consider, from now on, the case $H=L^{2}(X,\mathcal{X},\mu )$ the space of
square integrable functions on the measure space
 $(X,\mathcal{X},\mu )$, where $X$ is a
Polish space, $\mathcal{X}$ is the $\sigma $-field on $X$ and $\mu $ is a
positive, $\sigma $-finite and non-atomic measure on $(X,\mathcal{X})$.
As usual the inner product is given by $%
\langle f,g\rangle _{H}=\int_{X}f(x)g(x)\,d\mu (x)$.

Let us recall the
construction of an isonormal Gaussian field on $H$. Consider a
(real) Gaussian measure  over $(X, \mathcal X)$, i.e. a centered Gaussian family $W$
\begin{equation*}
W=\{W(A):A\in \mathcal{X},\mu (A)<+\infty \}
\end{equation*}%
such that for $A,B\in \mathcal{X}$ of $\mu$-finite measure, we have
\begin{equation*}
{\mathbb{\ E}}[W(A)W(B)]=\mu(A\cap B)\ .
\end{equation*}%
We define a random field $T=(T(f))_{f\in H}$ on $H$ as follows. For each $f\in H$,
let
\begin{equation}
T(f)=\int_{X}f(x)\,dW(x)\   \label{isonormal}
\end{equation}%
be the Wiener-It\^o integral of $f$ with respect to $W$. The random field $T$ is the
isonormal Gaussian field on $H$: indeed it is centered Gaussian and by construction
\begin{equation*}
\mathrm{Cov\,}(T(f),T(g))=\langle f,g\rangle _{H}\ .
\end{equation*}%

\subsection{Wiener chaos and contractions}

Let us recall now the notion of Wiener chaos. Define the space of
constants $C_{0}:=\mathbb{R}\subseteq L^2(\P)$, and for $q\geq 1$,
let $C_{q}$ be the closure in $L^{2}(\P):=L^2(\Omega, \F, \P)$  of the linear subspace
generated by random variables of the form
\begin{equation*}
H_{q}(T(f))\ ,\qquad f\in H,\ \Vert f\Vert _{H}=1\ ,
\end{equation*}%
where $H_{q}$ denotes the $q$-th Hermite polynomial, i.e.
\begin{equation}\label{hermite}
H_q (t) := (-1)^q \phi^{-1}(t) \frac{d^q}{d t^q} \phi(t)\ ,\quad t\in \R\ ,
\end{equation}
$\phi$ being the density function of a standard Gaussian r.v. $Z\sim \mathcal N(0,1)$. $C_{q}$ is
called the $q$-th Wiener chaos.

The following, well-known property is very important: let $Z_{1},Z_{2}\sim \mathcal{N}(0,1)$ be jointly
Gaussian; then, for all $q_{1},q_{2}\geq 0$
\begin{equation}  \label{hermite orto}
{\mathbb{\ E}}[H_{q_{1}}(Z_{1})H_{q_{2}}(Z_{2})]=q_{1}!\,{\mathbb{\ E}}%
[Z_{1}Z_{2}]^{q_{1}}\,\delta _{q_{2}}^{q_{1}}\ .
\end{equation}
\begin{theorem}
The Wiener-It\^o chaos expansion holds
\begin{equation*}
L^2(\P)=\bigoplus_{q=0}^{+\infty }C_{q}\ ,
\end{equation*}%
the above sum being orthogonal from (\ref{hermite orto}).
Equivalently, each random variable $F\in L^2(\P)$ admits a unique
decomposition in the $L^2(\P)$-sense of the form
\begin{equation}
F=\sum_{q=0}^{\infty }J_{q}(F)\ ,  \label{chaos exp}
\end{equation}%
where $J_{q}:L^2(\P){\goto} C_{q}$ is the orthogonal
projection operator onto the $q$-th Wiener chaos. Remark that $J_{0}(F)={\mathbb{\ E}}[F]$.
\end{theorem}
Often we will use the symbols $\text{proj}(F| C_q)$ or $F_q$ instead of $J_q(F)$.

We denote by $H^{\otimes q}$ and $H^{\odot q}$ the $q$-th tensor product and
the $q$-th symmetric tensor product of $H$ respectively. Therefore $%
H^{\otimes q} = L^2(X^{q}, \mathcal{X}^{q},\mu ^{q})$ and $H^{\odot
q}=L^2_s(X^{q}, \mathcal{X}^{q},\mu ^{q})$, where by $L^2_s$ we mean the symmetric and
square integrable functions w.r.t. $\mu^q$. Note that for $(x_1,x_2,\dots,
x_q)\in X^q$ and $f\in H$, we have
\begin{equation*}
f^{\otimes q}(x_1,x_2,\dots,x_q)=f(x_1)f(x_2)\dots f(x_q)\ .
\end{equation*}
Now for $q\ge 1$, let us define the map $I_{q}$ as
\begin{equation}
I_{q}(f^{\otimes q}):=\,H_{q}(T(f))\ ,\qquad f\in H\ ,  \label{isometria}
\end{equation}%
which can be extended to a linear isometry between $H^{\odot q}$ equipped
with the modified norm $\sqrt{q!}\,\Vert \cdot \Vert _{H^{\odot q}}$ and the $%
q $-th Wiener chaos $C_{q}$. Moreover for $q=0$, set $I_{0}(c)=c\in \mathbb{R%
}$. Hence \eqref{chaos exp} becomes
\begin{equation}
F=\sum_{q=0}^{\infty }I_{q}(f_{q})\ ,  \label{chaos exp2}
\end{equation}%
where the kernels $f_{q}, q\ge 0$  are uniquely determined, $f_{0}={\mathbb{\ E}}[F]$ and for $q\ge 1$ $f\in
H^{\odot q}$.

In our setting, it is well known that for $h\in H^{\odot q}$, $I_q(h)$
coincides with the multiple Wiener-Ito integral of order $q$ of $h$ with respect to the
Gaussian measure $W$, i.e.
\begin{equation}  \label{int multiplo}
I_q(h)= \int_{X^q} h(x_1,x_2,\dots x_q)\,dW(x_1) dW(x_2)\dots dW(x_q)
\end{equation}
and, loosely speaking, $F$ in (\ref{chaos exp2}) can be seen as a series of
(multiple) stochastic integrals.

For every $p,q\ge 1$, $f\in H^{\otimes p}, g\in H^{\otimes q}$ and $%
r=1,2,\dots, p\wedge q$, the so-called \emph{contraction} of $f$ and $g$ of
order $r$ is the element $f\otimes _{r}g\in H^{\otimes p+q-2r}$ given by
\begin{equation}\label{contrazione}
\begin{split}
f\otimes _{r}g\, &(x_{1},\dots,x_{p+q-2r})=\\
=\int_{X^{r}}f(x_{1},\dots,x_{p-r},y_{1},\dots,&y_{r})
g(x_{p-r+1},\dots,x_{p+q-2r},y_{1},\dots,y_{r})\,d\mu(\underline{y})\text{ ,}
\end{split}
\end{equation}
where we set $d\mu(\underline{y}):=d\mu(y_{1})\dots d\mu
(y_{r})$.

For $p=q=r$, we have $f\otimes _{r}g = \langle f,g \rangle_{H^{\otimes_r}}$
and for $r=0$, $f\otimes _{0}g = f\otimes g$. Note that
$f\otimes_r g$ is not necessarily symmetric, let us  denote by $f\widetilde \otimes
_{r}g$ its canonical symmetrization.

The following
multiplication formula is well-known: for $p,q=1,2,\dots$, $f\in H^{\odot
p}$, $g\in H^{\odot q}$, we have
\begin{equation*}
I_p(f)I_q(g)=\sum_{r=0}^{p\wedge q} r! {\binom{p }{r}} {\binom{q }{r}}%
I_{p+q-2r}(f\widetilde \otimes _{r}g)\ .
\end{equation*}

\subsection{Some language of Malliavin calculus}

Let $\mathcal S$ be the set of all cylindrical r.v.'s of the type
$F=f(T(h_1), \dots , T(h_m))$, where $m\ge 1$, $f:\R^m \to \R$ is an infinitely differentiable
function with compact support and $h_i\in H$.
The Malliavin derivative $DF$ (or $D^1 F$) of $F$ w.r.t. $T$ is the element $\in L^2(\Omega, H)$ defined as
$$
DF = \sum_i \frac{\partial f}{\partial x_i }(T(h_1), \dots , T(h_m)) h_i\ .
$$
We can define, by iteration, the $r$-th derivative $D^r F$ which is an element of $L^2(\Omega, H^{\odot r})$ for every $r\ge 2$. Recall that
$\mathcal S$ is dense in $L^q(\P)$ for each $q\ge 1$.

For $r\ge 1$ and $q\ge 1$,
let us denote by $\mathbb{D}^{r,q}$ the closure of $\mathcal S$ w.r.t.
the norm $\| \cdot \|_{\mathbb{D}^{r,q}}$ defined by the relation
\begin{equation*}
\Vert F\Vert _{\mathbb{D}^{r,q}}:=\left( {\mathbb{\ E}}[|F|^{q}]+\dots +{%
\mathbb{\ E}}[\Vert D^{r}F\Vert _{H^{\odot r}}^{q}]\right) ^{\frac{1}{q}%
}\ .
\end{equation*}
For $q,r\ge 1$, the $r$-th Malliavin derivative of the random
variable $F=I_{q}(f)\in C_{q}$ where $f\in H^{\odot q}$, is  given by
\begin{equation}
D^{r}F=\frac{q!}{(q-r)!}I_{q-r}(f)\text{ ,}  \label{marra2}
\end{equation}%
for $r\leq q$, and $D^{r}F=0$ for $r>q$.

It is possible to show that if we consider the chaotic representation (\ref{chaos exp2}), then $F\in \mathbb{D}^{r,2}$ if and only if
$$
\sum_{q=r}^{+\infty} q^r q!\, \| f_q\|^2_{H^{\odot q}} < +\infty
$$
and in this case
$$
D^r F = \sum_{q=r}^{+\infty}\frac{q!}{(q-r)!}I_{q-r}(f_q)\ .
$$
%
We need to introduce also the generator of the Ornstein-Uhlenbeck semigroup,
defined as
\begin{equation*}
L:=-\sum_{q=1}^{\infty }q\,J_{q}\ ,
\end{equation*}%
where $J_{q}$ is the orthogonal projection operator on $C_{q}$, as in (\ref%
{chaos exp}). The domain of $L$ consists of $F\in L^2(\P)$ such that
\begin{equation*}
\sum_{q=1}^{+\infty }q^{2}\Vert J_{q}(F)\Vert _{L^2(\P)}^{2}<+\infty
\ .
\end{equation*}%
The pseudo-inverse operator of $L$ is defined as
\begin{equation*}
L^{-1}=-\sum_{q=1}^{\infty }\frac{1}{q}J_{q}
\end{equation*}%
and satisfies for each $F\in L^2(\P)$
\begin{equation*}
LL^{-1}F=F-{\mathbb{\ E}}[F]\text{ ,}
\end{equation*}
equality that justifies its name.

\subsection{Main theorems}

We will need
the following definition throughout the rest of this thesis.

\begin{defn} {\rm Denote by $\mathscr{P}$ the collection of all probability measures on $\R$, and let $d : \mathscr{P}\times \mathscr{P}\to \R$ be a distance on $\mathscr{P}$. We say that the $d$ {\it metrizes weak convergence on} $\mathscr{P}$ if the following double implication holds for every collection $\{\P, \P_n : n\geq 1\}\subset \mathscr{P} $, as $n\to \infty$:
$$
\P_n \mbox{\ converges weakly to } \P \quad \mbox{if and only if} \quad d(\P_n , \P)\to 0.
$$
Given two random variables $X_1,X_2$ and a distance $d$ on $\mathscr{P}$, by an abuse of notation we shall
write $d(X_1,X_2)$ to indicate the quantity $d(\mathbf{D}(X_1) ,\mathbf{D}(X_2))$, where $\mathbf{D}(X_i)$
indicates the distribution of $X_i$, $i=1,2$. Recall that, given random variables $\{X, X_n : n\geq 1\}$, one
has that $\mathbf{D}(X_n)$ converges weakly to $\mathbf{D}(X)$ if and only if $X_n$ converges in distribution
to $X$. In this case, we write
$$
X_n \stackrel{\rm d}{\longrightarrow} X\quad \text{or}\quad X_n \stackrel{\mathcal L}{\longrightarrow} X\ ,
$$
whereas $X \stackrel{\rm d}{=} Y$ or $X \stackrel{\mathcal L}{=} Y$ indicates that $\mathbf{D}(X) = \mathbf{D}(Y)$.}
\end{defn}

Outstanding examples of distances metrizing weak convergence are the {\it Prokhorov distance} (usually denoted by $\rho$) and the {\it Fortet-Mourier distance} (or {\it bounded Wasserstein} distance, usually denoted by $\beta$). These are given by
$$
\rho\, (\P, \Q) = \inf\left\{ \epsilon >0 : \P(A)\leq \epsilon +\Q(A^\epsilon),\,\, \, \mbox{for every Borel set}\, A\subset \R\right\},
$$
where $A^{\epsilon} := \{ x : |x-y| <\epsilon, \mbox{ for some } y \in A\}$, and
$$
\beta\, (\P, \Q) = \sup\left\{  \left| \int_\R f \, d(\P - \Q) \right| : \|f\|_{BL}\leq 1  \right\}.
$$
where $\|\cdot\|_{BL} = \|\cdot \|_L + \|\cdot \|_\infty$, and $\|\cdot \|_L$ is the usual Lipschitz seminorm (see e.g. \cite[Section 11.3]{D} for further details on these notions).


Let us recall moreover the
usual Kolmogorov $d_{K}$, total variation $d_{TV}$ and Wasserstein $d_{W}$
distances between r.v.'s $X,Y$: for $\mathcal D \in \lbrace K, TV, W \rbrace$
\begin{equation}
d_{\mathcal D}(X,Y) :=\sup_{h\in  H_{\mathcal D}}\left\vert {\mathbb{E}}[h(X)]-{%
\mathbb{E}}[h(Y)]\right\vert \text{ ,}  \label{prob distance}
\end{equation}
where $H_{K} = \lbrace  1(\cdot \le z), z\in \mathbb R \rbrace$,
 $H_{TV} = \lbrace  1_A(\cdot), A\in \B(\mathbb R) \rbrace$  and $H_{W}$ is
 the set of Lipschitz functions with Lipschitz
constant one.

It is a standard fact (see e.g. \cite[Proposition C.3.1]{noupebook}) that $d_K$ {\it does not} metrize, in general, weak convergence on $\mathscr{P}$.

The connection between stochastic calculus and probability metrics
is summarized in the following result (see e.g. \cite{noupebook},
Theorem 5.1.3), which will provide the basis for most of our results to
follow.

From now on, $\mathcal N(\mu, \sigma^2)$ shall denote the Gaussian law with mean $\mu$ and variance $\sigma^2$.
\begin{prop}
\label{BIGnourdinpeccati} 
Let $F\in
\mathbb{D}^{1,2}$ such that $\mathbb{E}[F]=0,$ $\mathbb{E}[F^{2}]=\sigma
^{2}<+\infty .$ Then we have for $Z\sim \mathcal{N}(0,\sigma^2)$
\begin{equation*}
d_{W}(F,Z)\leq \sqrt{\frac{2}{\sigma ^{2}\,\pi }}\mathbb{E}%
[\left\vert \sigma ^{2}-\langle DF,-DL^{-1}F\rangle _{H}\right\vert ]\text{ .%
}
\end{equation*}%
Also, assuming in addition that $F$ has a density%
\begin{eqnarray*}
d_{TV}(F,Z) &\leq &\frac{2}{\sigma ^{2}}\mathbb{E}[\left\vert
\sigma ^{2}-\langle DF,-DL^{-1}F\rangle _{H}\right\vert ]\text{ ,} \\
d_{K}(F,Z) &\leq &\frac{1}{\sigma ^{2}}\mathbb{E}[\left\vert
\sigma ^{2}-\langle DF,-DL^{-1}F\rangle _{H}\right\vert ]\text{ .}
\end{eqnarray*}%
Moreover if $F\in \mathbb{D}^{1,4}$, we have also%
\begin{equation*}
\mathbb{E}[\left\vert \sigma ^{2}-\langle DF,-DL^{-1}F\rangle
_{H}\right\vert ]\leq \sqrt{\mathrm{Var}[\langle DF,-DL^{-1}F\rangle _{H}]}%
\text{ .}
\end{equation*}
\end{prop}
Furthermore, in the special case where $F=I_{q}(f)$ for $f\in H^{\odot q}$, $q\ge 2$
then from \cite{noupebook}, Theorem 5.2.6
\begin{equation}
\mathbb{E}[\left\vert \sigma ^{2}-\langle DF,-DL^{-1}F\rangle
_{H}\right\vert ]\leq \sqrt{\Var \left ( \frac{1}{q} \|DF\|_H^2    \right )}\ ,
\end{equation}
and  Lemma 5.2.4 gives
\begin{equation}\label{casoparticolare}
\Var \left ( \frac{1}{q} \|DF\|_H^2    \right )=\frac{1}{q^{2}}\sum_{r=1}^{q-1}r^{2}r!^{2}{%
\binom{q}{r}}^{4}(2q-2r)!\Vert f\widetilde{\otimes }_{r}f\Vert _{H^{\otimes
2q-2r}}^{2}\ .
\end{equation}
In addition it is possible to show the powerful chain of inequalities: for $q\ge 2$
$$
\Var \left ( \frac{1}{q} \|DF\|_H^2    \right )\le \frac{q-1}{3q}k_4(F) \le (q-1) \Var \left ( \frac{1}{q} \|DF\|_H^2    \right )\ ,
$$
where
$$
k_4(F) := \E[F^4] - 3(\sigma^2)^2
$$
is the \emph{fourth} cumulant of $F$.

\begin{remark}\rm
Note that in (\ref{casoparticolare}) we can replace $\Vert f\widetilde{%
\otimes }_{r}f\Vert _{H^{\otimes 2q-2r}}^{2}$ with the norm of the unsymmetrized
contraction $\Vert f\otimes _{r}f\Vert _{H^{\otimes 2q-2r}}^{2}$ for the upper
bound, since
$$\Vert f\widetilde{\otimes }_{r}f\Vert _{H^{\otimes
2q-2r}}^{2}\leq \Vert f\otimes _{r}f\Vert _{H^{\otimes 2q-2r}}^{2}$$
 by the
triangular inequality.
\end{remark}
We shall make an extensive use of the following.
\begin{cor}
Let $F_n$, $n\ge 1$, be a sequence of random variables belonging to the $q$-th Wiener chaos, for some fixed integer $q\ge 2$.  Then we have the following bound: for
$\mathcal D \in \lbrace K, TV, W\rbrace$
\begin{equation}\label{th}
d_{\mathcal D}\left (\frac{F_n}{\sqrt{\Var(F_n)}},Z \right )\le C_{\mathcal D}(q) \sqrt{\frac{k_4 \left ( F_n  \right )}{\Var(F_n)^2}}\ ,
\end{equation}
where $Z\sim \mathcal N(0,1)$, for some constant $C_{\mathcal D}(q)>0$.  In particular, if the right hand side in \paref{th} vanishes for $n\to +\infty$, then the following convergence in distribution holds
$$
\frac{F_n}{\sqrt{\Var(F_n)}}\mathop{\goto}^{\mathcal L} Z\ .
$$
\end{cor}

\newpage
\section{Gaussian eigenfunctions on the $d$-sphere}

\subsection{Some more notation}

 For any two positive sequence $a_{n},b_{n}$,
we shall write $a_{n}\sim b_{n}$ if $\lim_{n\rightarrow \infty }\frac{a_{n}}{%
b_{n}}=1$ and $a_{n}\ll b_{n}$ or $a_{n}=O(b_{n})$ if the sequence $\frac{%
a_{n}}{b_{n}}$ is bounded; moreover   $a_n = o(b_n)$ if $\lim_{n\rightarrow \infty }\frac{a_{n}}{%
b_{n}}=0$. Also, we write as usual $dx$ for the Lebesgue
measure on the unit $d$-dimensional sphere ${\mathbb{S}}^{d}\subset \R^{d+1}$, so that $%
\int_{{\mathbb{S}}^{d}}\,dx=\mu _{d}$ where $\mu _{d}:=\frac{2\pi ^{\frac{d+1%
}{2}}}{\Gamma \left( \frac{d+1}{2}\right) }$, as already stated in \paref{ms}.
Recall that
the triple $(\Omega, \F, \P)$  denotes a probability space and $\E$
stands for the expectation w.r.t $\P$; convergence (resp. equality) in law
is denoted by $\mathop{\rightarrow}^{\mathcal{L}}$ or equivalently $\mathop{\rightarrow}^{d}$ (resp. $\mathop{=}^{\mathcal{L}}$ or $\mathop{=}^{d}$)
and finally, as usual, $\mathcal N(\mu, \sigma^2)$ stands for a Gaussian
random variable with mean $\mu$ and variance $\sigma^2$.

Let $\Delta _{{\mathbb{S}}^{d}}$ $(d\geq 2)$ denote  the Laplace-Beltrami operator on  $\mathbb{S}%
^{d}$ and  $\left( Y_{\ell ,m;d}\right) _{\ell ,m}$ the
orthonormal system of (real-valued) spherical harmonics, i.e. for $\ell \in
\mathbb{N}$ the set of eigenfunctions
\begin{equation*}
\Delta _{{\mathbb{S}}^{d}}Y_{\ell ,m;d}+\ell (\ell +d-1)Y_{\ell ,m;d}=0\
,\quad m=1,2,\dots ,n_{\ell ;d}\ .
\end{equation*}%
For $d=2$ compare with \paref{realSH}.
As well-known, the spherical harmonics $\left( Y_{\ell ,m;d}\right)
_{m=1}^{n_{\ell ;d}}$ represent a family of linearly independent homogeneous
polynomials of degree $\ell $ in $d+1$ variables restricted to ${\mathbb{S}}%
^{d}$ of size
\begin{equation*}
n_{\ell ;d}:=\frac{2\ell +d-1}{\ell }{\binom{\ell +d-2}{\ell -1}}\ \sim \
\frac{2}{(d-1)!}\ell ^{d-1}\text{ ,}\quad \text{as}\ \ell \rightarrow
+\infty \ ,
\end{equation*}%
see e.g. \cite{andrews} for further details.
\subsection{Definition and properties}
\begin{definition}
For $\ell \in \mathbb{N}$, the Gaussian eigenfunction $T_{\ell }$
on ${\mathbb{S}}^{d}$ is defined as
\begin{equation}\label{Telle}
T_{\ell }(x):=\sum_{m=1}^{n_{\ell ;d}}a_{\ell ,m}Y_{\ell ,m;d}(x)\ ,\quad
x\in {\mathbb{S}}^{d}\ ,
\end{equation}%
with the random coefficients $\left( a_{\ell ,m}\right) _{m=1}^{n_{\ell ;d}}$
Gaussian i.i.d. random variables, satisfying the relation
\begin{equation*}
{\mathbb{E}}[a_{\ell ,m}a_{\ell ,m^{\prime }}]=\frac{\mu _{d}}{n_{\ell ;d}}%
\delta _{m}^{m^{\prime }}\ ,
\end{equation*}%
where
$\delta _{a}^{b}$ denotes the Kronecker delta function and $\mu_d=\frac{2\pi ^{\frac{d+1%
}{2}}}{\Gamma \left( \frac{d+1}{2}\right) }$ the hypersurface volume of  $\mathbb S^d$, as in \paref{ms}.
\end{definition}

It is then readily checked that $\left( T_{\ell }\right) _{\ell \in \mathbb{N%
}}$ represents a sequence of isotropic, zero-mean Gaussian random fields on $%
{\mathbb{S}}^{d}$, according to Definition  \paref{campo aleatorio su uno spazio omogeneo} and moreover
$$
\E[T_\ell(x)^2] = 1\ ,\qquad x\in \mathbb S^d\ .
$$
 $T_\ell$ is a continuous random field (Definition \ref{contRF}) and  its isotropy  simply means that the probability laws of
the two random fields $T_{\ell }(\cdot )$ and $T_{\ell }^{g}(\cdot
):=T_{\ell }(g\,\cdot )$ are equal (in the sense of finite dimensional distributions)
for every $g\in SO(d+1)$ (see \paref{invar-continui1}).

As briefly states in the Introduction of this thesis, it is also well-known that every Gaussian and isotropic random field $T$ on $%
\mathbb{S}^{d}$
satisfies in the  $L^{2}(\Omega \times {\mathbb{S}}^{d})$-sense the spectral representation
(see \cite{malyarenko, adlertaylor, dogiocam} e.g.)
\begin{equation*}
T(x)=\sum_{\ell =1}^{\infty }c_{\ell }T_{\ell }(x)\ ,\quad x\in {\mathbb{S}}%
^{d}\text{ ,}
\end{equation*}%
where  for every $x\in \mathbb S^d$, $\mathbb{E}\left[ T(x)^{2}\right] =\sum_{\ell =1}^{\infty }c_{\ell
}^{2}<\infty $; hence the spherical Gaussian eigenfunctions $\left( T_{\ell
}\right) _{\ell \in \mathbb{N}}$ can be viewed as the Fourier components (Chapter 1) of
the field $T$ (note that w.l.o.g. we are implicitly assuming that $T$ is
centered). Equivalently these random  eigenfunctions \paref{Telle}
could be defined
by their covariance function, which equals
\begin{equation}
{\mathbb{E}}[T_{\ell }(x)T_{\ell }(y)]=G_{\ell ;d}(\cos d(x,y))\ ,\quad
x,y\in {\mathbb{S}}^{d}\ .  \label{defT}
\end{equation}%
Here and in the sequel, $d(x,y)$ is the spherical distance between $x,y\in
\mathbb{S}^{d}$, and $G_{\ell ;d}:[-1,1]{\longrightarrow }\mathbb{R}$ is the
$\ell $-th normalized Gegenbauer polynomial, i.e.
$$G_{\ell ;d}\equiv \frac{P_{\ell }^{(\frac{d%
}{2}-1,\frac{d}{2}-1)}}{\left(
\begin{array}{c}
\ell +\frac{d}{2}-1 \\
\ell%
\end{array}%
\right)}\ ,$$
 where $P_{\ell }^{(\alpha ,\beta )}$ are the Jacobi
polynomials; throughout the whole thesis therefore
 $G_{\ell ;d}(1)=1$. As a special case, for $d=2$, it equals $G_{\ell ;2}\equiv
P_{\ell },$ the degree-$\ell $ Legendre polynomial.
Remark that the Jacobi polynomials $%
P_{\ell }^{(\alpha ,\beta )}$ are orthogonal on the interval $[-1,1]$ with
respect to the weight function $w(t)=(1-t)^{\alpha }(1+t)^{\beta }$ and
satisfy
$
P_{\ell }^{(\alpha ,\beta )}(1)=\left(
\begin{array}{c}
\ell +\alpha \\
\ell%
\end{array}%
\right)
$,
see e.g. \cite{szego} for more details.

\subsection{Isonormal representation}

Let us give the isonormal representation (\ref{isonormal}) on $L^{2}({%
\mathbb{S}}^{d})
$
for the Gaussian random eigenfunctions $T_{\ell }$, $\ell\ge 1$.
We shall show that the following identity in law holds:
\begin{equation*}
T_{\ell }(x)\overset{\mathcal L}{=}\int_{{\mathbb{S}}^{d}}\sqrt{\frac{n_{\ell ;d}}{%
\mu _{d}}}G_{\ell ;d}(\cos d(x,y))\,dW(y)\text{ ,}\qquad x\in \cS^d\ ,
\end{equation*}%
where $W$ is a Gaussian white noise on ${\mathbb{S}}^{d}$. To compare with (%
\ref{isonormal}), $$T_{\ell }(x)=T(f_{x})\ ,$$
 where $T$ is the isonormal Gaussian field
on $L^2(\cS^d)$ and $$f_{x}(\cdot ):=\sqrt{\frac{%
n_{\ell ;d}}{\mu _{d}}}G_{\ell ;d}(\cos d(x,\cdot ))\ .$$
 Moreover we have
immediately that
\begin{equation*}
{\mathbb{E}}\left[\int_{{\mathbb{S}}^{d}}\sqrt{\frac{n_{\ell ;d}}{\mu _{d}}}%
G_{\ell ;d}(\cos d(x,y))\,dW(y)\right]=0\ ,
\end{equation*}%
and by the reproducing formula for Gegenbauer polynomials (\cite{szego})
$$\displaylines{
{\mathbb{E}}\left[\int_{{\mathbb{S}}^{d}}\sqrt{\frac{n_{\ell ;d}}{\mu _{d}}}%
G_{\ell ;d}(\cos d(x_{1},y_{1}))\,dW(y_{1})\int_{{\mathbb{S}}^{d}}\sqrt{%
\frac{n_{\ell ;d}}{\mu _{d}}}G_{\ell ;d}(\cos d(x_{2},y_{2}))\,dW(y_{2})%
\right] = \cr
=\frac{n_{\ell ;d}}{\mu _{d}}\int_{{\mathbb{S}}^{d}}G_{\ell ;d}(\cos
d(x_{1},y))G_{\ell ;d}(\cos d(x_{2},y))dy=G_{\ell ;d}(\cos d(x_{1},x_{2}))%
\text{ .}
}$$

\chapter{Empirical measure of excursion sets}

The asymptotic behavior (i.e. for growing eigenvalues) of Gaussian eigenfunctions on a compact Riemannian manifold  is a topic which has
recently drawn considerable attention, see e.g. \cite{Sodin-Tsirelson,  AmP, bogomolnyschmit}.

In particular,  in view of Berry's Random Wave model \cite{wigsurvey} much effort has been devoted to the case of the  sphere
$\mathbb S^2$ (see \cite{nazarovsodin, Wig, def, Nonlin}).

As anticipated in the Introduction, the aim of this chapter is the investigation of the asymptotic distribution of the empirical measure $S_\ell(z)$ of $z$-excursion sets
of random spherical harmonics.  A Central Limit Theorem has already been proved \cite{Nonlin}, but it can provide little guidance to the
actual distribution of random functionals, as it is only an asymptotic
result with no information on the \emph{speed of convergence} to the limiting
distribution.

In \cite{maudom} therefore we  exploit the results about fourth-moments phenomenon (see \cite{noupebook} and Chapter 4) to
establish \emph{quantitative} Central Limit Theorems
for the excursion volume of Gaussian eigenfunctions on the $d$-dimensional unit sphere $\mathbb{S}^d$, $d\ge 2$ (see also \cite{mauproc}).

We note that
there are already results in the literature giving rates of convergence in CLTs for value distributions of
eigenfunctions of the spherical Laplacian, see in particular \cite{meckes}, which investigates however
the complementary situation to the one considered here,
i.e. the limit for eigenfunctions of fixed degree $\ell$ and increasing dimension $d$.

To achieve our goal, we will provide a number of
intermediate results of independent interest, namely the asymptotic analysis for
the variance of moments of Gaussian eigenfunctions, the rates of convergence for various probability
metrics for so-called Hermite subordinated processes, the analysis of arbitrary polynomials
of finite order and square integrable nonlinear transforms. All these results could be useful
to attack other problems, for instance quantitative Central Limit Theorems for Lipschitz-Killing
curvatures of arbitrary order. A more precise statement of
our results and a short outline of the proof is given in \S \ref{main results}.

\section{Main results and outline of the proofs}
\label{main results}

The excursion volume of $T_\ell$ \paref{Telle},
 for any fixed $z\in \mathbb{R}$ can be written as
\begin{equation}\label{voldef}
S_{\ell }(z)=\int_{{\mathbb{S}}^{d}}1(T_{\ell }(x)> z)\,dx\text{ ,}
\end{equation}%
where $1(\cdot > z)$ denotes the indicator function of the interval $%
(z,\infty)$; note that \\${\mathbb{E}}[S_{\ell }(z)]=\mu _{d}(1-\Phi (z))$, where
$\Phi (z)$ is the standard Gaussian cdf and $\mu_d$ as in \paref{ms}. The variance of this excursion volume will be
shown below to have the following asymptotic behavior (as $\ell \to +\infty$)
\begin{equation}
\Var(S_{\ell}(z))=\frac{z^2\phi(z)^2}{2}\frac{\mu_d^2}{n_{\ell,d}}+o(\ell^{-d})\ ,
\end{equation}
where $\phi$ denotes the standard Gaussian density and $n_{\ell,d} \sim \frac{2}{(d-1)!}\ell^{d-1}$ is the dimension of the eigenspace related to the eigenvalue $-\ell(\ell+d-1)$, as in Chapter \ref{background}.

Note that the variance is of order smaller than $\ell^{-(d-1)}$ if and only if $z=0$.

The main result of this chapter is then as follows.

\begin{theorem}\label{mainteo}
The excursion volume $S_\ell(z)$ in \paref{voldef}  of Gaussian eigenfunctions $T_\ell$
on $\mathbb S^d$, $d\ge 2$,
satisfies a quantitative CLT as $\ell \to +\infty$, with rate of convergence in the Wasserstein
distance \paref{prob distance} given by, for $z\ne 0$ and $Z\sim \mathcal{N}(0,1)$
\begin{equation*}
d_{W}\left( \frac{S_{\ell }(z)-\mu_d (1-\Phi(z))}{\sqrt{\mathrm{Var}[S_{\ell
}(z)]}}, Z\right) =O\left(\ell^{-1/2}\right)\text{
.}
\end{equation*}
\end{theorem}
An outline of the main steps and auxiliary results to prove this theorem
is given in the following subsection.

\subsection{Steps of the proofs}\label{outline}

The first tool to investigate quantitative CLTs for the excursion volume
of Gaussian eigenfunctions on ${\mathbb{S}}^{d}$ (compare for $d=2$ with \cite{Nonlin}) is
to study the asymptotic behavior, as $\ell \rightarrow \infty $, of the random variables $%
h_{\ell ;q,d}$ defined for $\ell =1,2,\dots $ and $q=0,1,\dots $ as
\begin{equation}
h_{\ell ;q,d}=\int_{{\mathbb{S}}^{d}}H_{q}(T_{\ell }(x))\,dx\ ,  \label{hq}
\end{equation}%
where $H_{q}$ represent the family of Hermite polynomials \paref{hermite} (see also \cite{noupebook}).
The rationale to investigate these sequences is the fact that
the excursion volume, and more generally any square integrable transform of $T_\ell$,
admits the Wiener-Ito chaos decomposition \paref{chaos exp}
(for more details e.g. \cite{noupebook},  \S 2.2),
i.e. a series expansion in the $L^2(\P)$-sense
of the form
\begin{equation}\label{volseries}
S_{\ell }(z) = \sum_{q=0}^{+\infty} \frac{J_q(z)}{q!}\, h_{\ell ;q,d}\ ,
\end{equation}
where  $J_0(z) =1-\Phi(z)$ and for $q\ge 1$, $J_q(z)=\E[1(Z>z)H_q(Z)]$,
$\Phi$ and $\phi$
denoting again respectively the cdf and the density of $Z\sim \mathcal N(0,1)$.

The main idea in our argument will then be to establish first a CLT for each of the summands
in the series, and then to deduce from this a CLT for  the excursion volume.
The starting point will then
be the analysis of the asymptotic variances for $h_{\ell ;q,d}$, as $\ell\to +\infty$.

To this aim, note first that, for all $d$
\begin{equation*}
h_{\ell ;0,d}=\mu _{d}\ ,\qquad h_{\ell ;1,d}=0
\end{equation*}%
a.s., and therefore it is enough to restrict our discussion to $q\geq 2$.
Moreover ${\mathbb{E}}[h_{\ell; q, d}]=0$ and
\begin{equation}
\mathrm{Var}[h_{\ell ;q,d}]=q!\mu _{d}\mu _{d-1}\int_{0}^{\pi }G_{\ell
;d}(\cos \vartheta )^{q}(\sin \vartheta )^{d-1}\,d\vartheta  \label{int var}
\end{equation}%
(see \S \ref{tech} for more details). Gegenbauer polynomials satisfy the symmetry
relationships \cite{szego}
\begin{equation*}
G_{\ell ;d}(t)=(-1)^{\ell }G_{\ell ;d}(-t)\ ,
\end{equation*}%
whence the r.h.s. integral in (\ref{int var}) vanishes identically when both
$\ell $ and $q$ are odd; therefore in these cases $h_{\ell ;q,d}= 0$ a.s. For
the remaining cases we have
\begin{equation}\label{int var doppio}
\mathrm{Var}[h_{\ell ;q,d}]=2q!\mu _{d}\mu _{d-1}\int_{0}^{\frac{\pi }{2}%
}G_{\ell ;d}(\cos \vartheta )^{q}(\sin \vartheta )^{d-1}\,d\vartheta \ .
\end{equation}%
We have hence the following asymptotic result for these variances,
whose proof is given in \S \ref{subvar}.
\begin{prop}
\label{varianza} As $\ell \rightarrow \infty ,$ for $q,d\ge 3$,
\begin{equation}  \label{int1}
\int_{0}^{\frac{\pi }{2}}G_{\ell ;d}(\cos \vartheta )^{q}(\sin \vartheta
)^{d-1}\,d\vartheta =\frac{c_{q;d}}{\ell ^{d}}(1+o_{q;d}(1))\ .
\end{equation}
The constants $c_{q;d}$ are given by the formula
\begin{equation}\label{ecq}
c_{q;d}=\left(2^{\frac{d}{2} - 1}\left (\frac{d}2-1 \right)!\right)^q\int_{0}^{+\infty
}J_{\frac{d}{2}-1}(\psi )^{q}\psi ^{-q\left( {\textstyle\frac{d}{2}}%
-1\right) +d-1}d\psi\ ,
\end{equation}%
where $J_{\frac{d}{2}-1}$ is the Bessel function
of order $\frac{d}{2}-1$. The r.h.s. integral in (\ref{ecq}) is absolutely convergent for any
pair $(d,q)\neq (3,3)$ and conditionally convergent for $d=q=3$.
\end{prop}
It is well known that for $d\geq 2$, the second moment of the Gegenbauer
polynomials is given by
\begin{equation}
\int_{0}^{\pi }G_{\ell ;d}(\cos \vartheta )^{2}(\sin \vartheta
)^{d-1}\,d\vartheta =\frac{\mu _{d}}{\mu _{d-1}\,n_{\ell ;d}}\ ,
\label{momento 2}
\end{equation}%
whence
\begin{equation}
\mathrm{Var}(h_{\ell ;2,d})=2\frac{\mu _{d}^{2}}{n_{\ell ;d}}\ \sim\  4\mu_d \mu_{d-1}\frac{c_{2;d}}{\ell ^{d-1}}\ ,\qquad \text{as}\ \ell \rightarrow
+\infty \ , \label{q=2}
\end{equation}%
where $c_{2;d} :=
\frac{(d-1)!\mu _{d}}{4 \mu_{d-1}}$.
For $d=2$ and every $q$, the asymptotic behavior of these integrals was resolved in \cite%
{def}. In particular, it was shown that for $q=3$ or $q\geq 5$
\begin{equation}
\mathrm{Var}(h_{\ell ;q,2})=(4\pi )^{2}q!\int_{0}^{\frac{\pi}{2}}P_{\ell }(\cos
\vartheta )^{q}\sin \vartheta \,d\vartheta =(4\pi )^{2}q!\frac{c_{q;2}}{\ell
^{2}}(1+o_{q}(1))\ ,  \label{int2}
\end{equation}%
where
\begin{equation}
c_{q;2}=\int_{0}^{+\infty }J_{0}(\psi )^{q}\psi \,d\psi \ ,  \label{cq2}
\end{equation}%
$J_{0}$ being the Bessel function of order $0$ and the above integral being
absolutely convergent for $q\ge 5$ and conditionally convergent for $%
q=3$. On the other hand, for $q=4$, as $\ell \rightarrow \infty $,
\begin{equation}
\mathrm{Var}[h_{\ell;4,2}]\ \sim \ 24^{2}\frac{\text{log}\ell }{\ell ^{2}}= (4\pi)^2 4!\, c_{4;2}\frac{\log \ell}{\ell^2}\
,  \label{q=4d=2}
\end{equation}%
where we set $c_{4;2}:=\frac{3}{2\pi^2}$.
Clearly for any $d,q\geq 2$, the constants $c_{q;d}$ are
nonnegative and it is obvious that $c_{q;d}>0$ for all even $q$.
Moreover, as we will recall in the next chapter, there exists an explicit formula
in the case $q=3$.

We conjecture that this strict inequality holds for every $(d,q)$,
but leave this issue as an open question for future research;
also, in view of the previous discussion on the symmetry
properties of Gegenbauer polynomials, to simplify the discussion
in the sequel we restrict ourselves to even multipoles $\ell $.

As argued earlier, the following step is to establish quantitative CLTs for
$h_{\ell ;q,d}$ (see \S \ref{clthl}) in various probability metrics
 \paref{prob distance}. Here the crucial point to stress
is that the Gaussian eigenfunctions $(T_{\ell})_\ell$
can be always expressed as stochastic integrals with respect to a
 Gaussian white noise measure on $\mathbb{S}^d$, as seen in
Chapter \ref{background}.  As a consequence, the random
sequences $h_{\ell ;q,d}$ can themselves be represented as multiple Wiener-Ito integrals,
and therefore fall inside the domain of
quantitative CLTs by means of the Nourdin-Peccati approach (Chapter \ref{background}).
It is thus sufficient to investigate the so-called circular components of their normalized
fourth-order cumulants (Proposition \ref{BIGnourdinpeccati}) to establish the following Proposition
\ref{teo1}.
%
%
\begin{prop}
\label{teo1}
For all $d,q\ge 2$ and
${\mathcal{D}}\in \lbrace K, TV, W \rbrace$ we have, as $\ell\to +\infty$,
\begin{equation}\label{bounds}
d_{\mathcal{D}}\left( \frac{h_{\ell ;q,d}}{\sqrt{\mathrm{Var}[h_{\ell ;q,d}]}},%
Z\right) =O \left (\ell^{-\delta(q;d)} (\log \ell)^{-\eta(q;d)}\right )\ ,
\end{equation}%
where for $d=2$
\begin{align}\label{rate2}
&\delta(2;2)=\delta(3;2)=1/2\ , \quad \delta(4;2)=0\ ,\quad  \delta(q;2)=1/4\ \text{  for }q\ge 5\ ;\\
\nonumber
&\eta(4;2)=1\ , \quad \eta(5;2)=\eta(6;2)=-1\ ,\quad  \delta(q;2)=0\ \text{  for }q=2,3\text{ and for }q\ge 7\ ;
\end{align}
whereas for $d\ge 3$ we have
$$\eta(q;d) = 0\quad \text{      for } q\ge 2\ ;
$$
$$\delta(2;d)=(d-1)/2\ , \quad \delta(3;d)=(d-5)/2\ ,\quad  \delta(4;d)=(d-3)/2\ $$
$$ \delta(q;d) = (d-1)/4\ \text{  for }q\ge 5\ .$$
\end{prop}
Let us set $R(\ell;q,d) := \ell^{-\delta(q;d)} (\log \ell)^{-\eta(q;d)}$.
The following corollary is  immediate.

\begin{cor}
\label{cor1} For all $q$ such that $(d,q)\neq (3,3), (3,4),(4,3),(5,3)$ and $%
c_{q;d}>0$, $d\ge 2$,
\begin{equation}
\frac{h_{2\ell ;q,d}}{\sqrt{Var[h_{2\ell ;q,d}]}}\mathop{\goto}^{\mathcal{L}}%
Z\ ,\qquad \text{as}\ \ell \rightarrow +\infty \ ,
\label{convergenza hq}
\end{equation}
where $Z\sim \mathcal{N}(0,1)$.
\end{cor}

\begin{remark}\rm
\textrm{For $d=2$, the CLT in (\ref{convergenza hq}) was already provided by
\cite{Nonlin}; nevertheless Theorem \ref{teo1} improves the existing bounds
on the speed of convergence to the asymptotic Gaussian distribution. More
precisely, for $d=2,q=2,3,4$ the same rate of convergence as in (\ref{rate2}%
) was given in their Proposition 3.4; however for arbitrary $q$ the total
variation rate was only shown to satisfy (up to logarithmic terms) $%
d_{TV}=O(\ell ^{-\delta _{q}}),$ where $\delta _{4}=\frac{1}{10},$ $\delta
_{5}=\frac{1}{7},$ and $\delta _{q}=\frac{q-6}{4q-6}<\frac{1}{4}$ for $q\geq
7$.}
\end{remark}

\begin{remark}\rm
\textrm{The cases not included in Corollary \ref{cor1} correspond to the
pairs where $q=4$ and $d=3$, or $q=3$ and $d=3,4,5$; in these circumstances
the bounds we establish on fourth-order cumulants in Proposition \ref{cumd}
are not sufficient to
ensure that the CLT holds. We leave these computations as a topic for future research.}
\end{remark}
As briefly anticipated earlier in this subsection,
the random variables $h_{\ell;q,d}$ defined in (\ref{hq}) are the basic
building blocks for the analysis of any square integrable nonlinear transforms
of Gaussian  eigenfunctions on ${\mathbb{S}}^{d}$. Indeed, let us
first consider generic polynomial functionals of the form
\begin{equation}
Z_{\ell }=\sum_{q=0}^{Q}b_{q}\int_{{\mathbb{S}}^{d}}
T_{\ell
}(x)^{q}\,dx\ ,\qquad Q\in \mathbb{N},\text{ }b_{q}\in \mathbb{R},
\label{Z}
\end{equation}%
which include, for instance, the so-called polyspectra (see e.g. \cite{dogiocam}, p.148)
of isotropic random
fields defined on ${\mathbb{S}}^{d}$. Note
\begin{equation}
Z_{\ell }=\sum_{q=0}^{Q}\beta _{q}h_{2\ell ;q,d}  \label{Z2}
\end{equation}%
for some $\beta _{q}\in \mathbb{R}$. It is easy to establish CLTs
 for generic polynomials (\ref{Z2}) from convergence results on
$ h_{2\ell ;q,d}$, see e.g. \cite{peccatitudor}. It is more
difficult to investigate the speed of convergence in the CLT in terms of the
probability metrics we introduced earlier; indeed, in \S \ref{genpoly} we establish the
following.

\begin{prop}
\label{corollario1} As $\ell \rightarrow \infty$, for $Z\sim \mathcal N(0,1)$
\begin{equation*}
d_{\mathcal{D}}\left( \frac{Z_{\ell }-{\mathbb{E}}[Z_{\ell }]}{\sqrt{\mathrm{%
Var}[Z_{\ell }]}},Z\right) =O(R(Z_{\ell };d))\text{ ,}
\end{equation*}%
where $d_{\mathcal{D}}=d_{TV},d_{W},d_{K}$ and for $d\ge 2$
\begin{equation*}
R(Z_{\ell };d)=%
\begin{cases}
\ell ^{-\left( \frac{d-1}{2}\right) }\qquad & \text{if}\quad \beta _{2}\neq
0\ , \\
\max_{q=3,\dots ,Q\,:\,\beta _{q},c_{q;d}\neq 0}\,R(\ell ;q,d)\qquad & \text{if%
}\quad \beta _{2}=0\ .%
\end{cases}%
\end{equation*}
\end{prop}

The previous results can be summarized as follows: for polynomials of
\emph{Hermite rank $2$}, i.e. such that 
$\beta
_{2}\neq 0$ (more details later on the notion of Hermite rank) the asymptotic behavior of $Z_{\ell }$ is
dominated by the single term $h_{\ell ;2,d},$ whose variance is of order $\ell
^{-(d-1)}$ rather than $\ell ^{-d}$ as for the other terms. On the other
hand, when $\beta _{2}=0,$ the convergence rate to the asymptotic Gaussian
distribution for a generic polynomial is the slowest among the rates for the
Hermite components into which $Z_{\ell }$ can be decomposed, i.e.
 the terms $\beta _{q}h_{2\ell ;q,d}$ in \paref{Z2}.

The fact that the bound for generic polynomials is of the same order as for
the Hermite case (and not slower) is indeed rather unexpected; it can be
shown to be due to the cancelation of some cross-product terms, which are
dominant in the general Nourdin-Peccati framework, while they vanish for
 spherical eigenfunctions of arbitrary dimension (see (\ref%
{eq=0}) and Remark \ref{rem0}). An inspection of our proof will reveal that
this result is a by-product of the orthogonality of eigenfunctions
corresponding to different eigenvalues; it is plausible that similar ideas
may be exploited in many related circumstances, for instance random
eigenfunction on generic compact Riemannian manifolds.


Proposition \ref{corollario1} shows that the asymptotic behavior of arbitrary
polynomials of Hermite rank $2$ is of particularly simple nature. Our result
below will show that this feature holds in much greater generality, at least
as far as the Wasserstein distance $d_W$ is concerned. Indeed, we shall consider
the case of functionals of the form
\begin{equation}
S_{\ell }(M)=\int_{{\mathbb{S}}^{d}}M(T_{\ell }(x))\,dx\ ,  \label{S}
\end{equation}%
where $M:\mathbb{R}\rightarrow \mathbb{R}$ is some
measurable  function such that $\E[M(Z)^2]<+\infty$, where
$Z\sim \mathcal N(0,1)$. As in Chapter \ref{background},
for such transforms the
following chaos expansion holds in the $L^2(\P)$-sense \paref{chaos exp}
\begin{equation}
M(T_{\ell })=\sum_{q=0}^{\infty }\frac{J_{q}(M)}{q!}H_{q}(T_{\ell })\ ,\qquad J_{q}(M):={\mathbb{E}}[M(T_{\ell
})H_{q}(T_{\ell })]\ .  \label{exp}
\end{equation}%
Therefore the asymptotic analysis, as $\ell \rightarrow \infty $, of $S_{\ell }(M)$
in (\ref{S}) directly follows from the Gaussian approximation
for $h_{\ell ;q,d}$ and their polynomial transforms $Z_{\ell }$. More
precisely, in \S \ref{genvol} we prove the following result.

\begin{prop}
\label{general} Let $Z\sim \mathcal N(0,1)$. For functions $M$ in (\ref{S}) such that $$\mathbb{E}\left[
M(Z)H_{2}(Z)\right] =J_{2}(M)\neq 0\ ,$$ we have
\begin{equation}
d_{W}\left( \frac{S_{2\ell }(M)-{\mathbb{E}}[S_{2\ell }(M)]}{\sqrt{%
Var[S_{2\ell }(M)]}},Z\right) =O(\ell ^{-1/2}) \
,\qquad \text{as}\ \ell \rightarrow \infty\ ,  \label{sun2}
\end{equation}
in particular
\begin{equation}
\frac{S_{2\ell }(M)-{\mathbb{E}}[S_{2\ell }(M)]}{\sqrt{\mathrm{Var}[S_{2\ell
}(M)]}}\mathop{\goto}^{\mathcal{L}} Z\ .  \label{sun1}
\end{equation}
\end{prop}

Proposition \ref{general} provides a Breuer-Major like result on nonlinear
functionals, in the high-frequency limit (compare for instance \cite%
{nourdinpodolskij}). While the CLT in \eqref{sun1} is somewhat expected, the
square-root speed of convergence \eqref{sun2} to the limiting distribution
 may be considered quite remarkable; it is mainly due to some specific
features in the chaos expansion of Gaussian eigenfunctions, which is dominated
by a single term at $q=2.$ Note that the function $M$ need not be smooth in
any meaningful sense; indeed, as explained above, our
main motivating rationale here is the analysis of the asymptotic behavior of the excursion
volume in \paref{voldef} $%
S_{\ell }(z)=S_{\ell }(M)$, where $M(\cdot )=M_{z}(\cdot )=1(\cdot > z)$ is again the
indicator function of the interval $(z, +\infty)$. An application
of Proposition \ref{general} (compare \paref{volseries} to \paref{exp}) provides a quantitative CLT for $%
S_{\ell }(z)$, $z\neq 0$, thus completing the proof of our main result.

The plan of the rest of this chapter is as follows: in \S \ref{2.2} we specialize results in Chapter \ref{background} to the hypersphere, in \S \ref{clthl} we establish the
quantitative CLT for the sequences $h_{\ell;q,d,}$, while \S \ref{genpoly} extends these
results to generic finite-order polynomials. The results for general nonlinear transforms and excursion
volumes are given in \S \ref{genvol}; most technical proofs and (hard) estimates, including in particular
evaluations of asymptotic variances, are collected in \S \ref{tech}.

\section{Polynomial transforms in Wiener chaoses}\label{2.2}

As mentioned earlier in \S \ref{main results}, we shall be concerned first
with random variables $h_{\ell ;q,d}$, $\ell \geq 1$, $q,d\geq 2$
\begin{equation*}
h_{\ell ;q,d}=\int_{{\mathbb{S}}^{d}}H_{q}(T_{\ell }(x))\,dx\ ,
\end{equation*}%
and their (finite) linear combinations
\begin{equation}
Z_{\ell }=\sum_{q=2}^{Q}\beta _{q}h_{\ell ;q,d}\text{ ,}\qquad \beta _{q}\in
\mathbb{R},Q\in \mathbb{N}\ .  \label{genovese}
\end{equation}
Our first objective is to represent \paref{genovese} as a (finite) sum of
(multiple) stochastic integrals as in (\ref{chaos exp2}), in order to apply
the results recalled in Chapter \ref{background}.
Note that by (\ref{isometria}), we have
\begin{equation*}
\displaylines{ H_{q}(T_{\ell }(x))=I_{q}(f_{x}^{\otimes q})=\cr
=\int_{({\mathbb{S}}^{d})^{q}}\left( \frac{n_{\ell;d}}{\mu _{d}}\right)
^{q/2}G_{\ell ;d}(\cos d(x,y_{1}))\dots G_{\ell ;d}(\cos
d(x,y_{q}))\,dW(y_{1})...dW(y_{q})\text{ ,} }
\end{equation*}%
so that
\begin{equation*}
h_{\ell ;q,d}\overset{\mathcal L}{=}\int_{({\mathbb{S}}^{d})^{q}}g_{\ell
;q}(y_{1},...,y_{q})\,dW(y_{1})...dW(y_{q})\ ,
\end{equation*}%
where%
\begin{equation}
g_{\ell ;q}(y_{1},...,y_{q}):=\int_{{\mathbb{S}}^{d}}\left( \frac{n_{\ell ;d}%
}{\mu _{d}}\right) ^{q/2}G_{\ell ;d}(\cos d(x,y_{1}))\dots G_{\ell ;d}(\cos
d(x,y_{q}))\,dx\text{ .}  \label{moltobello}
\end{equation}%
Thus we just established that $h_{\ell ;q,d}\overset{\mathcal L}{=}I_{q}(g_{\ell ;q})$
and therefore
\begin{equation}
Z_{\ell }\overset{\mathcal L}{=}\sum_{q=2}^{Q}I_{q}(\beta _{q}\,g_{\ell ;q})\ ,
\end{equation}%
as required. It should be noted that for such random variables $Z_{\ell }$,
the conditions of the Proposition \ref{BIGnourdinpeccati} are trivially
satisfied.

\section{The quantitative CLT for Hermite transforms}
\label{clthl}

In this section we prove Proposition \ref{teo1} with the help of Proposition \ref%
{BIGnourdinpeccati} and \eqref{casoparticolare} in particular. The
identifications of \S \ref{2.2} lead to some very explicit expressions for the
contractions \eqref{contrazione}, as detailed in the following result.

For $\ell\ge 1, q\ge 2$, let $g_{\ell ;q}$ be defined as in \eqref{moltobello}.

\begin{lemma}
\label{contractions}
For all $q_{1},q_{2}\ge 2$, $r=1,...,q_{1}\wedge q_{2}-1$,
we have the identities
$$\displaylines{ \left\Vert g_{\ell ;q_{1}}\otimes _{r}g_{\ell
;q_{2}}\right\Vert _{H^{\otimes
n}}^{2}
= \int_{(\cS^{d})^4}\!\!\!\!G_{\ell
;d}^{r}\cos d( x_{1},x_{2})  G_{\ell
;d}^{q_{1}\wedge q_{2}-r}\!(\cos d( x_{2},x_{3}) )   G_{\ell
;d}^{r}(\cos d(x_{3},x_{4}) ) \times \cr
\times  G_{\ell
;d}^{q_{1}\wedge q_{2}-r}(\cos d( x_{1},x_{4}) )\,d\underline{x}\ ,}
$$
where we set $d\underline{x} := dx_{1}dx_{2}dx_{3}dx_{4}$ and
$n:= q_{1}+q_{2}-2r$.
\end{lemma}

\begin{proof}
Assume w.l.o.g. $q_1\le q_2$ and set for simplicity of notation $d\underline{t}:=dt_{1}\dots dt_{r}$.
The contraction \paref{contrazione} here takes the form%
$$\displaylines{
(g_{\ell;q_{1} }\otimes _{r}g_{\ell;q_{2}})(y_{1},...,y_{n})=\cr
=\int_{(\cS^d)^r} g_{\ell;q_{1} }(y_{1},\dots,y_{q_{1}-r},t_{1},\dots,t_{r})g_{\ell;q_{2}
}(y_{q_{1}-r+1},\dots,y_{n},t_{1},\dots,t_{r})\,d\underline{t}=\cr
=\int_{(\cS^d)^r}  \int_{\cS^{d}}\left(\frac{n_{\ell; d}}{\mu_d }
\right) ^{q_1/2} G_{\ell ;d}(\cos d(
x_{1},y_{1}) )\dots 
G_{\ell
;d}(\cos d( x_{1},t_{r}) )\,dx_{1}\times \cr
\times \int_{\cS^{d}}\left(\frac{n_{\ell; d}}{\mu_d }
\right) ^{q_2/2}G_{\ell ;d}(\cos d(
x_{2},y_{q_{1}-r+1}) )\dots 
G_{\ell ;d}(\cos d(
x_{2},t_{r}) )\,dx_{2}\, d\underline{t}=\cr
=\int_{(\cS^{d})^2}\left(\frac{n_{\ell; d}}{\mu_d }
\right) ^{n/2}G_{\ell ;d}(\cos d(
x_{1},y_{1})) \dots  G_{\ell ;d}(\cos d(
x_{1},y_{q_1-r}) )\times \cr
\times G_{\ell ;d}(\cos d(
x_{2},y_{q_{1}-r+1}) )\dots G_{\ell ;d}(\cos d(
x_{2},y_{n}) )G_{\ell
;d}^{r}(\cos d( x_{1},x_{2}) )\,dx_{1}dx_{2}\ ,
}$$
where in the last equality
we have repeatedly used the reproducing property of Gegenbauer polynomials (\cite{szego}).
Now set $d\underline{y}:=dy_{1}\dots dy_{n}$. It follows at once that
$$\displaylines{
\left\Vert g_{\ell; q_{1} }\otimes _{r}g_{\ell; q_{2} }\right\Vert _{%
{H}^{\otimes n}}^{2}
=\int_{(\cS^d)^{n}} (g_{\ell; q_{1} }\otimes _{r}g_{\ell; q_{2} })^{2}
(y_{1},\dots,y_{n})\,d\underline{y}=\cr
=\int_{(\cS^d)^{n}} \int_{(\cS^{d})^2}\Bigl(\frac{n_{\ell; d}}{\mu_d
} \Bigr) ^{n}\!G_{\ell ;d}(\cos d(
x_{1},y_{1})) \dots G_{\ell ;d}(\cos d(
x_{2},y_{n})) G_{\ell
;d}^{r}\cos d(x_{1},x_{2}) dx_1 dx_2\times \cr
\times
 \int_{(\cS^{d})^2}G_{\ell ;d}(\cos d(
x_{4},y_{1})) \dots G_{\ell ;d}(\cos d(
x_{3},y_{n})) G_{\ell
;d}^{r}(\cos d( x_{3},x_{4}))
\,dx_{3}dx_{4}\,d\underline{y}=\cr
=\int_{(\cS^{d})^4}\!\!\!G_{\ell
;d}^{r}(\cos d( x_{1},x_{2})) G_{\ell
;d}^{q_{1}-r}(\cos d( x_{2},x_{3}))   G_{\ell
;d}^{r}(\cos d(x_{3},x_{4}))      G_{\ell
;d}^{q_{1}-r}(\cos d( x_{1},x_{4}))
\,d\underline{x}\ ,
}$$
as claimed.
\end{proof}
We need now to introduce some further notation, i.e. for $q\ge 2$ and $r= 1,\dots, q-1$
\begin{equation*}
\displaylines{ \mathcal{K}_{\ell }(q;r):=\int_{(\cS^{d})^4}G_{\ell
;d}^{r}(\cos d( x_{1},x_{2}) ) G_{\ell ;d}^{q-r}(\cos d( x_{2},x_{3})
)\times \cr \times G_{\ell ;d}^{r}(\cos d(x_{3},x_{4}) ) G_{\ell
;d}^{q-r}(\cos d( x_{1},x_{4}) ) \,dx_{1}dx_{2}dx_{3}dx_{4}, }
\end{equation*}%
Lemma \ref{contractions} asserts that
\begin{equation}\label{K}
\mathcal{K}_{\ell }(q;r)=\left\Vert g_{\ell ;q}\otimes _{r}g_{\ell
;q}\right\Vert _{H^{\otimes 2q-2r}}^{2}\text{ ;}
\end{equation}%
it is immediate to check that
\begin{equation}\label{simm}
\mathcal{K}_{\ell }(q;r)= \mathcal{K}_{\ell }(q;q-r)\ .
\end{equation}
In the following two propositions we bound each term of the form $\mathcal{K}(q;r)$  (from (\ref{simm})
it is enough to consider $r=1,\dots, \left[\frac{q}2\right]$).
As noted in
\S \ref{outline}, these bounds improve the existing literature even for the case $d=2,$
from which we start our analysis.

For $d=2,$ as previously recalled, Gegenbauer polynomials become standard
Legendre polynomials $P_{\ell },$ for which it is well-known that (see (\ref%
{momento 2}))
\begin{equation}
\int_{{\mathbb{S}}^{2}}P_{\ell }(\cos d( x_{1},x_{2}) )^{2}\,dx_{1}=O\left(
\frac{1}{\ell }\right) \text{ ;}  \label{tri1}
\end{equation}%
also, from \cite{Nonlin}, Lemma 3.2 we have that%
\begin{equation}
\int_{{\mathbb{S}}^{2}}P_{\ell }(\cos d( x_{1},x_{2}) )^{4}\,dx_{1}=O\left(
\frac{\log \ell }{\ell ^{2}}\right) \text{ .}  \label{tri2}
\end{equation}%
Finally, it is trivial to show that%
\begin{equation}
\int_{{\mathbb{S}}^{2}}\left\vert P_{\ell }(\cos d( x_{1},x_{2})
)\right\vert\, dx_{1}\leq \sqrt{\int_{{\mathbb{S}}^{2}}P_{\ell
}(\cos d( x_{1},x_{2}) )^{2}\,dx_{1}}=O\left( \frac{1}{\sqrt{%
\ell }}\right)   \label{tri3}
\end{equation}%
and by Cauchy-Schwartz inequality
\begin{equation}
\int_{{\mathbb{S}}^{2}}\left\vert P_{\ell }(\cos d( x_{2},x_{3})
)\right\vert ^{3}\,dx_{2} 
=O\left( \sqrt{\frac{\log \ell }{\ell
^{3}}}\right) \text{ .}  \label{tri4}
\end{equation}%
These results will be the main tools to establish the upper bounds which are
collected in the following Proposition, whose proof is deferred to the last section.
\begin{prop}
\label{cum2} For all $r=1,2,\dots ,q-1,$ we have%
\begin{eqnarray}
 \mathcal{K}_{\ell }(q;r) &=& O\left(\frac{1}{\ell ^{5}}%
\right)\text{ for }q=3\text{ ,}  \label{hotel1} \\
\mathcal{K}_{\ell }(q;r) &=& O\left(\frac{1}{\ell ^{4}}%
\right)\text{ for }q=4\text{ ,}  \label{hotel2}\\
\mathcal{K}_{\ell }(q;r) &=& O\left(\frac{\log \ell }{%
\ell ^{9/2}}\right)\text{  for }q=5,6\text{ }  \label{hotel3}
\end{eqnarray}%
and%
\begin{equation}
\mathcal{K}_{\ell }(q;1) =\mathcal{K}_{\ell
}(q;q-1) =O\left(\frac{1}{\ell ^{9/2}}\right)\text{ , }
\mathcal{K}_{\ell }(q;r) =O\left(\frac{1}{\ell ^{5}}\right)\text{
, }r=2,...,q-2,\text{ for }q\geq 7\text{ .}  \label{hotel4}
\end{equation}
\end{prop}
We can now move to the higher-dimensional case, as follows. Let us start
with the bounds for all order moments of Gegenbauer polynomials.
From (\ref{momento 2})
\begin{equation}
\int_{{\mathbb{S}}^{d}}G_{\ell ;d}(\cos d( x_{1},x_{2}) )^{2}dx_{1}=O\left(
\frac{1}{\ell ^{d-1}}\right) \text{ ;}  \label{trid2}
\end{equation}%
also, from Proposition \ref{varianza}, we have that if $q=2p,$ $p=2,3,4...$,
\begin{equation}
\int_{{\mathbb{S}}^{d}}G_{\ell ;d}(\cos d( x_{1},x_{2}) )^{q}dx_{1}=O\left(
\frac{1}{\ell ^{d}}\right) \text{ .}  \label{tridq}
\end{equation}%
Finally, it is trivial to show that%
$$
\int_{{\mathbb{S}}^{d}}\left\vert G_{\ell ;d}(\cos d( x_{2},x_{3})
)\right\vert dx_{2}\le
$$
\begin{equation}
\leq \sqrt{\int_{{\mathbb{S}}^{d}} G_{\ell
;d}(\cos d( x_{2},x_{3}) )^{2}dx_{2}}=O\left( \frac{1}{\sqrt{%
\ell ^{d-1}}}\right) \text{ ,}  \label{trid1}
\end{equation}%
$$
\int_{{\mathbb{S}}^{d}}\left\vert G_{\ell ;d}(\cos d( x_{2},x_{3})
)\right\vert ^{3}dx_{2}\le
$$
\begin{equation}
\leq \sqrt{\int_{{\mathbb{S}}^{d}}G_{\ell
;d}(\cos d( x_{2},x_{3}) )^{2}dx_{2}}\sqrt{\int_{{\mathbb{S}}^{d}}G_{\ell
;d}(\cos d( x_{1},x_{2}) )^{4}dx_{1}}=O\left( \frac{1}{\ell ^{d-{\textstyle%
\frac{1}{2}}}}\right) \text{ }  \label{trid3}
\end{equation}%
and for $q\geq 5$ odd,
$$
\int_{{\mathbb{S}}^{d}}\left\vert G_{\ell ;d}(\cos d( x_{2},x_{3})
)\right\vert ^{q}dx_{2}\le
$$
\begin{equation}
\leq \sqrt{\int_{{\mathbb{S}}^{d}}G_{\ell
;d}(\cos d( x_{2},x_{3}) )^{4}dx_{2}}\sqrt{\int_{{\mathbb{S}}^{d}}G_{\ell
;d}(\cos d( x_{1},x_{2}) )^{2(q-2)}dx_{1}}=O\left( \frac{1}{\ell ^{d}}%
\right) \text{ .}  \label{tridgen}
\end{equation}%
Analogously to the $2$-dimensional case, we can exploit the previous results
to obtain the following bounds, whose proof is again collected in \S \ref{tech}.
\begin{prop}
\label{cumd} For all $r=1,2,\dots, q-1$
\begin{eqnarray}
 \mathcal{K}_{\ell }(q;r) &=&O\left( \frac{1}{\ell ^{2d+%
{\frac{d-5}{2}}}}\right) \text{ for }q=3\text{ ,}  \label{hoteld1} \\
\mathcal{K}_{\ell }(q;r) &=&O\left( \frac{1}{\ell ^{2d+%
{\frac{d-3}{2}}}}\right) \text{ for }q=4\text{ ,}  \label{hoteld2}
\end{eqnarray}%
and for $r=2,\dots ,q-2$
\begin{equation}
\mathcal{K}_{\ell }(q;1) = \mathcal{K}_{\ell
}(q;q-1) =O\left( \frac{1}{\ell ^{2d+{\frac{d-1}{2}}}}\right)
\text{ , }\mathcal{K}_{\ell }(q;r)=O\left( \frac{1}{%
\ell ^{3d-1}}\right) \text{  }\text{ for }q\geq 5\text{ .}
\label{hoteld4}
\end{equation}
\end{prop}
Exploiting the results in this section and the variance evaluation in Proposition \ref{varianza}
in \S \ref{tech},
we have the proof of our first quantitative CLT.
\begin{proof}[Proof Proposition \ref{teo1}]
By Parseval's identity, the case $q=2$ can be treated as a sum of independent random variables and the proof
follows from standard Berry-Esseen arguments, as in Lemma 8.3 of \cite{dogiocam} for the case $d=2$.
For $q\ge 3$, from Proposition \ref{BIGnourdinpeccati} and \paref{casoparticolare}, for $d_{\mathcal{D}}=
d_K, d_{TV}, d_W$
\begin{equation}
d_{\mathcal{D}}\left(\frac{h_{\ell ;q}}{\sqrt{\Var[h_{\ell ;q,d}]}},Z\right)
 = O\left(\sup_{r}\sqrt{\frac{\mathcal{K}_{\ell }(q;r) }{%
 \Var[h_{\ell ;q,d}]^{2}}}\right)\text{ .}
\end{equation}%
The proof is thus an immediate consequence of
the previous equality and the results in Proposition \ref{varianza},
 Proposition  \ref{cum2} and Proposition \ref{cumd}.
\end{proof}%
%

\section{General polynomials}\label{genpoly}

In this section, we show how the previous results can be extended to establish quantitative
CLTs, for the case of general, non Hermite polynomials.
To this aim, we need to introduce some more notation, namely
(for $Z_{\ell }$ defined as in (\ref{genovese}))%
\begin{equation*}
\mathcal{K}(Z_{\ell }; d):=\max_{q:\beta _{q}\neq 0}\max_{r=1,...,q-1}\mathcal{K%
}_{\ell }(q;r)\text{ ,}
\end{equation*}
and as in Proposition \ref{corollario1}
\begin{equation*}
R(Z_{\ell}; d )=\begin{cases}
\ell ^{-\frac{d-1}2}\ ,\quad&\text{for }\beta _{2}\neq 0\ , \\
\max_{q=3,\dots,Q:\beta _{q}, c_{q;d}\neq 0}\,R(\ell ;q,d)\ ,\quad&\text{for }\beta _{2}=0\ .%
\end{cases}
\end{equation*}
In words, $\mathcal{K}(Z_{\ell }; d)$ is the largest contraction term among
those emerging from the analysis of the different Hermite components, and $%
R(Z_{\ell };d)$ is the slowest convergence rate of the same components. The
next result is stating that these are the only quantities to look at when
considering the general case.

%

\begin{proof}[Proof Theorem \ref{corollario1}]
We apply Proposition \ref{BIGnourdinpeccati}.
In our case $H=L^2(\mathbb S^d)$ and %
$$\displaylines{
\Var[\langle  DZ_{\ell },-DL^{-1}Z_{\ell }\rangle _{H%
}]=\Var\left[ \langle \sum_{q_{1}=2}^{Q}\beta _{q_{1}}Dh_{\ell;q_{1},d},
-\sum_{q_{21}=2}^{Q}\beta _{q_{2}}DL^{-1}h_{\ell;q_{2},d}\rangle_{H}\right]=\cr
=\Var\left[ \sum_{q_{1}=2}^{Q}\sum_{q_{1}=2}^{Q}\beta _{q_{1}}\beta
_{q_{2}}\langle Dh_{\ell;q_{1},d},-DL^{-1}h_{\ell;q_{2},d}\rangle_{%
H}\right] \text{ .}
}$$
From Chapter \ref{background} recall that for $q_{1}\neq q_{2}$
\begin{equation*}
E[\langle Dh_{\ell;q_{1},d},-DL^{-1}h_{\ell ;q_{2},d}\rangle_{%
H}]=0\text{ ,}
\end{equation*}%
whence we write
$$\displaylines{
\Var\left[ \sum_{q_{1}=2}^{Q}\sum_{q_{2}=2}^{Q}\beta _{q_{1}}\beta
_{q_{2}}\langle Dh_{\ell;q_{1},d},-DL^{-1}h_{\ell;q_{2},d}\rangle_{%
H}\right]=\cr
=\sum_{q_{1}=2}^{Q}\sum_{q_{2}=2}^{Q}\beta _{q_{1}}^{2}\beta
_{q_{2}}^{2}\Cov\left( \langle Dh_{\ell;q_{1},d},-DL^{-1}h_{\ell;q_{1},d}
\rangle_{H},\langle Dh_{\ell;q_{2},d},-DL^{-1}h_{\ell ;q_{2},d}\rangle_{H}\right)+\cr
+\sum_{q_{1}, q_3=2}^{Q}\sum_{q_{2}\neq
q_{1}}^{Q}\sum_{q_{4}\neq q_{3}}^{Q}\beta _{q_{1}}\beta
_{q_{2}}\beta _{q_{3}}\beta _{q_{4}}\times \cr
\times
\Cov\left( \langle Dh_{\ell;q_{1},d},-DL^{-1}h_{\ell;q_{2},d}\rangle_{H},\langle
Dh_{\ell ;q_{3},d},-DL^{-1}h_{\ell ;q_{4},d}\rangle_{H}\right).
}$$
Now of course we have
$$\displaylines{
\Cov\left( \langle Dh_{\ell ;q_{1},d},-DL^{-1}h_{\ell
;q_{1},d}\rangle_{H},\langle Dh_{\ell
;q_{2},d},-DL^{-1}h_{\ell ;q_{2},d}) _{H}\right) \le \cr
\leq \left(\Var\left[\langle Dh_{\ell ;q_{1},d},-DL^{-1}h_{\ell
;q_{1},d}\rangle_{H}\right] \Var\left[ \langle Dh_{\ell
;q_{2},d},-DL^{-1}h_{\ell ;q_{2},d}\rangle_{H}\right] \right)^{1/2},
}$$
$$\displaylines{
\Cov\left( \langle Dh_{\ell ;q_{1},d},-DL^{-1}h_{\ell
;q_{2},d}\rangle_{H},\langle Dh_{\ell
;q_{3},d},-DL^{-1}h_{\ell ;q_{4},d}\rangle_{H}\right) \le \cr
\leq \left( \Var\left[ \langle Dh_{\ell ;q_{1},d},-DL^{-1}h_{\ell
;q_{2},d}\rangle_{H}\right] \Var\left[ \langle Dh_{\ell
;q_{3},d},-DL^{-1}h_{\ell ;q_{4},d}\rangle_{H}\right]
\right)^{1/2}.
}$$
Applying \cite{noupebook}, Lemma 6.2.1 it is immediate to show that%
$$\displaylines{
\Var\left[ \langle Dh_{\ell ;q_{1},d},-DL^{-1}h_{\ell ;q_{1},d}\rangle
_{H}\right]\le \cr
\leq q_{1}^{2}\sum_{r=1}^{q_{1}-1}((r-1)!)^{2}{
q_{1}-1 \choose
r-1%
}^{4}(2q_{1}-2r)!\left\Vert g_{\ell; q_{1}}\otimes
_{r}g_{\ell;q_{1}}\right\Vert _{H^{\otimes 2q_1-2r}}^{2} =\cr
=q_{1}^{2}\sum_{r=1}^{q_{1}-1}((r-1)!)^{2}{
q_{1}-1 \choose
r-1%
} ^{4}(2q_{1}-2r)! \mathcal{K}_{\ell }(q_1;r) \text{
.}
}$$
Also, for $q_{1}<q_{2}$%
$$\displaylines{
\Var\left[ \langle Dh_{\ell ;q_{1},d},-DL^{-1}h_{\ell ;q_{2},d}\rangle
_{H}\right]=\cr
=q_{1}^{2}\sum_{r=1}^{q_{1}}((r-1)!)^{2}{
q_{1}-1 \choose
r-1%
} ^{2}{
q_{2}-1 \choose
r-1%
} ^{2}(q_{1}+q_{2}-2r)!\left\Vert g_{\ell;q_{1}}\widetilde{\otimes }%
_{r}g_{\ell ; q_{2}}\right\Vert _{H^{\otimes
(q_{1}+q_{2}-2r)}}^{2}=\cr
=q_{1}^{2}((q_{1}-1)!)^{2}{
q_{2}-1 \choose
q_{1}-1%
}^{2}(2q_{1}-2r)!\left\Vert g_{\ell;q_{1}}\widetilde{\otimes }%
_{q_{1}}g_{\ell ; q_{2}}\right\Vert _{H^{\otimes (q_{2}-q_1)}}^{2} + \cr
+q_{1}^{2}\sum_{r=1}^{q_{1}-1}((r-1)!)^{2}{
q_{1}-1 \choose
r-1%
}^{2}{
q_{2}-1 \choose
r-1%
} ^{2}(q_{1}+q_{2}-2r)!\left\Vert g_{\ell;q_{1} }\widetilde{\otimes }%
_{r} g_{\ell ;q_{2}}\right\Vert _{H^{\otimes
(q_{1}+q_{2}-2r)}}^{2}
=:A+B\text{ .}
}$$%
Let us focus on the first summand $A$, which includes terms that, from Lemma \ref{contractions},
take the form%
$$\displaylines{
\left\Vert g_{\ell;q_{1}}\widetilde{\otimes }_{q_{1}}g_{\ell;q_{2}
}\right\Vert_{H^{\otimes (q_{2}-q_1)}}^{2}\le
\left\Vert g_{\ell;q_{1}}\otimes_{q_{1}}g_{\ell;q_{2}
}\right\Vert_{H^{\otimes (q_{2}-q_1)}}^{2}=\cr
=\int_{(\cS^d)^{q_2-q_1}}\int_{(\cS^{d})^2}\left( \frac{n_{\ell; d}%
}{\mu_d }\right)^{q_{2}-q_{1}}
 G_{\ell;d }(\cos d( x_{2},y_{1}) )\cdots \times \cr
\times \cdots G_{\ell;d }(\cos d( x_{2},y_{q_{2}-q_1}) )
G_{\ell;d }(\cos d(
x_{1},x_2) )^{q_1}\,dx_1 dx_2\times\cr
\times \int_{(\cS^d)^2}G_{\ell;d}(\cos d( x_{3},y_{1}) )
...G_{\ell;d}(\cos d(
x_{3},y_{q_{2}-q_1}) )G_{\ell;d }(\cos d(
x_{3},x_4) )^{q_1}\,dx_3 dx_4\,d\underline{y} =: I\ ,
}$$
where for the sake of simplicity we have set $d\underline{y}:=dy_{1}...dy_{q_2-q_1}$.
Applying $q_2-q_1$ times the reproducing formula for Gegenbauer polynomials (\cite{szego})  we get
\begin{equation}\label{anvedi}
I = \int_{(\cS^{d})^4}G_{\ell;d}(\cos d(
x_{1},x_{2}) )^{q_{1}}G_{\ell;d}(\cos d(
x_{2},x_{3}) )^{q_{2}-q_{1}}G_{\ell;d}(\cos d(
x_{3},x_{4}) )^{q_{1}}\,d\underline{x}\ .
\end{equation}
In graphical terms, these contractions correspond to the diagrams such that
all $q_{1}$ edges corresponding to vertex $1$ are linked to vertex 2, vertex
$2$ and $3$ are connected by $q_{2}-q_{1}$ edges, vertex $3$ and $4$ by $%
q_{1}$ edges, and no edges exist between $1$ and $4,$ i.e. the diagram has
no proper loop.
Now immediately we write
$$\displaylines{
\paref{anvedi} =\int_{\cS^{d}}G_{\ell;d}(\cos d( x_{1},x_{2})
)^{q_{1}}\,dx_{1}\int_{\cS^{d}}G_{\ell;d}(\cos d( x_{3},x_{4})
)^{q_{1}}\,dx_{4}\times \cr
\times \int_{(\cS^{d})^2}G_{\ell;d}(\cos d(
x_{2},x_{3}) )^{q_{2}-q_{1}}\,dx_{2}dx_{3}=\cr
=\frac{1}{(q_{1}!)^{2}} \Var[ h_{\ell;q_1,d}]^{2}
\int_{(\cS^{d})^2}G_{\ell;d}
(\cos d( x_{2},x_{3}) )^{q_{2}-q_{1}}\,dx_{2}dx_{3}\text{ .}
}$$
Moreover we have
\begin{equation}\label{eq=0}
\int_{(\cS^{d})^2}G_{\ell;d}(\cos d(
x_{2},x_{3}) )^{q_{2}-q_{1}}\,dx_{2}dx_{3}=0\ ,\quad \text{if}\ q_{2}-q_{1}=1\
\end{equation}%
and from \paref{momento 2} if $q_{2}-q_{1}\geq 2$
\begin{equation*}
\int_{(\cS^{d})^2}G_{\ell;d}(\cos d(
x_{2},x_{3}) )^{q_{2}-q_{1}}\,dx_{2}dx_{3}\leq \mu_d \int_{\cS^{d}}G_{\ell;d}(\cos d(x,y))^{2}\,dx=
O\left(\frac{1}{\ell^{d-1} }\right)\ .
\end{equation*}%
It follows that
\begin{equation}
\left\Vert g_{\ell;q_{1} }\otimes_{q_{1}}g_{\ell;q_{2}
}\right\Vert_{H^{\otimes (q_2-q_1)}}^{2}=  O\left(\Var[ h_{\ell;q_1,d}]^{2}\frac{1}{\ell^{d-1} }\right)
\label{efficientbound}
\end{equation}%
always. For the second term, still from \cite{noupebook}, Lemma $6.2.1$ we have%
\begin{eqnarray*}
B \leq \frac{q_{1}^{2}}{2}\sum_{r=1}^{q_{1}-1}((r-1)!)^{2}\left(
\begin{array}{c}
q_{1}-1 \\
r-1%
\end{array}%
\right) ^{2}\left(
\begin{array}{c}
q_{2}-1 \\
r-1%
\end{array}%
\right) ^{2}(q_{1}+q_{2}-2r)!\times \\
\times \left( \left\Vert g_{\ell;q_{1} }\otimes _{q_{1}-r}g_{\ell;q_{1}
}\right\Vert _{H^{\otimes 2r}}^{2}+\left\Vert g_{\ell;q_{2}
}\otimes _{q_{2}-r}g_{\ell;q_{2} }\right\Vert _{H^{\otimes
2r}}^{2}\right)=
\end{eqnarray*}%
\begin{equation}
=\frac{q_{1}^{2}}{2}\sum_{r=1}^{q_{1}-1}((r-1)!)^{2}\left(
\begin{array}{c}
q_{1}-1 \\
r-1%
\end{array}%
\right) ^{2}\left(
\begin{array}{c}
q_{2}-1 \\
r-1%
\end{array}%
\right) ^{2}(q_{1}+q_{2}-2r)!\left( \mathcal{K}_{\ell }(q_{1};r)+\mathcal{K}%
_{\ell }(q_{2};r)\right)\ , \label{juve2}
\end{equation}%
where the last step follows from Lemma \ref{contractions}.

Let us first investigate the case $d=2$. From  \paref{q=2}, \paref{int2} and \paref{q=4d=2}
it is immediate that
\begin{equation}
\Var[Z_{\ell }]=\sum_{q=2}^{Q}\beta _{q}^{2}\Var[h_{\ell ;q,2}]=
\begin{cases}
O(\ell ^{-1})\ ,\quad &\text{for }\beta _{2}\neq 0 \\
O(\ell ^{-2}\log \ell)\ ,\quad &\text{for }\beta _{2}=0\text{ , }\beta _{4}\neq 0 \\
O(\ell ^{-2})\ ,\quad &\text{otherwise.}%
\end{cases}%
\end{equation}%
Hence we have that for $\beta _{2}\neq 0$ and $Z\sim \mathcal{N}(0,1)$
\begin{eqnarray*}
d_{TV}\left(\frac{Z_{\ell }-EZ_{\ell }}{\sqrt{\Var[Z_{\ell }]}},Z\right)
=O\left(\frac{\sqrt{\mathcal{K}_{\ell }(2;r)}}{\Var[Z_{\ell }]}\right)
=O\left(\ell ^{-1/2}\right)\text{ ;}
\end{eqnarray*}%
for $\beta _{2}=0$ , $\beta _{4}\neq 0$ ,%
\begin{eqnarray*}
d_{TV}\left(\frac{Z_{\ell }-EZ_{\ell }}{\sqrt{\Var[Z_{\ell }]}},Z\right)
=O\left(\frac{\sqrt{\mathcal{K}_{\ell }(4;r)}}{\Var[Z_{\ell }]}\right)
=O\left(\frac{1}{\log \ell }\right)
\end{eqnarray*}%
and for $\beta _{2}=\beta _{4}=0$, $\beta _{5}\neq 0$ and $c_5 >0$
\begin{eqnarray*}
d_{TV}\left(\frac{Z_{\ell }-EZ_{\ell }}{\sqrt{\Var[Z_{\ell }]}},Z\right)
=O\left(\frac{\sqrt{\mathcal{K}_{\ell }(5;r)}}{\Var[Z_{\ell }]}\right)
=O\left(\frac{\log \ell }{\ell ^{1/4}}\right)\text{ ,}
\end{eqnarray*}%
and analogously we deal with the remaining cases, so that we obtain the claimed result for $d=2$.

For $d\ge 3$ from \paref{momento 2} and Proposition \ref{varianza}, it holds
\begin{equation*}
\Var[Z_{\ell }]=\sum_{q=2}^{Q}\beta _{q}^{2}\Var[h_{\ell ;q,d}]=
\begin{cases}
O(\ell ^{-(d-1)})\ ,\quad &\text{for }\beta _{2}\neq 0\ , \\
O(\ell ^{-d})\ ,\quad &\text{otherwise}\ .
\end{cases}%
\end{equation*}
Hence  we have for $\beta _{2}\neq 0$
\begin{eqnarray*}
d_{TV}\left(\frac{Z_{\ell }-EZ_{\ell }}{\sqrt{\Var[Z_{\ell }]}},Z\right)
=O\left(\frac{\sqrt{\mathcal{K}_{\ell }(2;r)}}{\Var[Z_{\ell }]}\right) =
O\left(\frac{1}{\ell ^{\frac{d-1}{2}}}\right)\text{ .}
\end{eqnarray*}%
Likewise for $\beta _{2}=0$ , $\beta _{3},c_{3;d}\neq 0$,
\begin{eqnarray*}
d_{TV}\left(\frac{Z_{\ell }-EZ_{\ell }}{\sqrt{\Var[Z_{\ell }]}},Z\right)
=O\left(\frac{\sqrt{\mathcal{K}_{\ell }(3;r)}}{\Var[Z_{\ell }]}\right)
=O\left(\frac{1}{\ell ^\frac{d-5}{4}}\right)\text{ %
}
\end{eqnarray*}%
and for $\beta _{2}=\beta _{3}=0$, $\beta _{4}\neq 0$
\begin{eqnarray*}
d_{TV}\left(\frac{Z_{\ell }-EZ_{\ell }}{\sqrt{\Var[Z_{\ell }]}},Z\right)
=O\left(\frac{\sqrt{\mathcal{K}_{\ell }(4;r)}}{\Var[Z_{\ell }]}\right)
=O\left(\frac{1}{\ell ^{\frac{d-3}{2}}}\right)\text{ .%
}
\end{eqnarray*}%
Finally if $\beta_2=\beta_3=\beta_4=0$, $\beta_q, c_{q;d} \ne 0$ for some $q$, then
\begin{eqnarray*}
d_{TV}\left(\frac{Z_{\ell }-EZ_{\ell }}{\sqrt{\Var[Z_{\ell }]}},Z\right)
=O\left(\frac{\sqrt{\mathcal{K}_{\ell }(q;r)}}{\Var[Z_{\ell }]}\right)
=O\left(\sqrt{\frac{\ell^{2d} }{\ell ^{2d +\frac{d}{2}-\frac{1}{2}}}}\right)=
O\left(\frac{1}{\ell ^{\frac{d-1}{4}}}\right)\text{ .}
\end{eqnarray*}
\end{proof}
\begin{remark}\rm
\label{rem0}\textrm{To compare our result in these specific circumstances
with the general bound obtained by Nourdin and Peccati, we note that for \paref{anvedi},
these
authors are exploiting the inequality%
\begin{equation*}
\displaylines{ \left\Vert g_{\ell;q_{1}}\otimes_{q_{1}} g_{\ell;q_{2}}
\right\Vert_{H^{\otimes (q_2-q_1)}} ^{2}\le \left\Vert g_{\ell;q_{1}}
\right\Vert _{H^{\otimes q_{1}}}^{2}\left\Vert g_{\ell;q_{2}}\otimes_{q_{2}-q_{1}}g_{\ell;q_{2}}
\right\Vert_{H^{\otimes 2q_{1}}}\ , }
\end{equation*}%
see \cite{noupebook}, Lemma $6.2.1$. In the special framework we
consider here (i.e., orthogonal eigenfunctions), this provides, however,  a less efficient bound than
(\ref{efficientbound}): indeed from \paref{anvedi}, repeating the same argument as in Lemma
\ref{contractions},
one obtains
$$\displaylines{
\left\Vert g_{\ell ;q_{1}}\otimes _{q_{1}}g_{\ell ;q_{2}}
\right\Vert_{H^{\otimes (q_2-q_1)}}^{2}=
\int_{({\mathbb{S}}^{d})^4}G_{\ell ;d}(\cos
d(x_{1},x_{2}))^{q_1}G_{\ell ;d}(\cos d(x_{2},x_{3}))^{q_2-q_1}\times \cr
\times G_{\ell ;d}(\cos
d(x_{3},x_{4}))^{q_1}\,d\underline{x}
 \le \int_{({\mathbb{S}}^{d})^2} G_{\ell
;d}(\cos d(x_{1},x_{2}))^{q_1}\,dx_1 dx_2\times \cr
\times \Big (\int_{({\mathbb{S}}^{d})^4}\!\!\!\!G_{\ell ;d}(\cos d(x_{1},x_{2}))^{q_1}
G_{\ell ;d}(\cos
d(x_{2},x_{3}))^{q_2-q_1}\times \cr
\times G_{\ell ;d}(\cos
d(x_{3},x_{4}))^{q_1}G_{\ell ;d}(\cos d(x_{1},x_{4}))^{q_2-q_1}\,d\underline{x} \Big )^{1/2}=\cr
=O\left(  \Var[h_{\ell ;q_{1},d}] \sqrt{\mathcal{K}_{\ell
}(q_{2},q_{1})}\right) \text{ ,}
}$$
yielding a bound of order
\begin{equation}
O\left( \sqrt{\frac{ \Var[h_{\ell;q_{1},d}] \sqrt{\mathcal{K}%
_{\ell }(q_{2},q_{1})}}{\Var[h_{\ell ;q_{1},d}]^{2}}}\right)
=O\left( \frac{\sqrt[4]{\mathcal{K}_{\ell }(q_{2},q_{1})}}{\sqrt{\Var[h_{\ell
;q_{1},d}]}}\right)  \label{cdip}
\end{equation}%
rather than%
\begin{equation}
O\left( \sqrt{\frac{\mathcal{K}_{\ell }(q_{2},q_{1})}{ \Var[h_{\ell
;q_{1},d}]^{2}}}\right) \text{ ;}  \label{cdip2}
\end{equation}%
for instance, for $d=2$ note that (\ref{cdip}) is typically $=O(\ell
\times \ell ^{-9/8})=O(\ell ^{-1/8}),$ while we have established for (\ref{cdip2}%
) bounds of order $O(\ell ^{-1/4})$. }
\end{remark}

\begin{remark}\rm
\textrm{Clearly the fact that $\left\Vert g_{\ell;q_{1}
}\otimes_{q_{1}}g_{\ell;q_{2} }\right\Vert_{H^{\otimes (q_2-q_1)}}^{2}=0$ for $q_{2}=q_{1}+1$
entails that the contraction $g_{\ell;q_{1} }\otimes_{q_{1}}g_{\ell;q_{2} }$ is identically
null. Indeed repeating the same argument as in Lemma \ref{contractions}
$$\displaylines{
g_{\ell;q_{1} }\otimes_{q_{1}}g_{\ell; q_{1}+1}=\cr
=
\int_{(\cS^{d})^{2}}G_{\ell ;d}(\cos d( x_{1},y) )G_{\ell ;d}(\cos d(
x_{1},x_2) )^{q_1}\,dx_{1}dx_{2}=\cr
=\int_{{\mathbb{S}}^{d}}G_{\ell ;d}(\cos d( x_{1},y) )\left(  \int_{{%
\mathbb{S}}^{d}}G_{\ell ;d}(\cos d( x_{1},x_2) )^{q_1}\,dx_{2} \right)\,dx_1=0\text{ ,}
}$$
as expected, because the inner integral in the last equation does not depend on $x_1$
by rotational invariance. }
\end{remark}

\section{Nonlinear functionals and excursion volumes}
\label{genvol}

The techniques and results developed previously are restricted to
finite-order polynomials. In the special case of the Wasserstein distance,
we shall show below how they can indeed be extended to general nonlinear
functionals of the form \paref{S}
\begin{equation*}
S_{\ell }(M)=\int_{{\mathbb{S}}^{d}}M(T_{\ell }(x))dx\text{ ;}
\end{equation*}%
here $M:\mathbb{R}\rightarrow \mathbb{R}$ is a measurable function such that
$\mathbb{E}[M(Z)^{2}]<\infty $, $Z\sim \mathcal N(0,1)$ as in \S \ref{outline}, and $J_{2}(M)\neq 0,$ where we recall
that $J_{q}(M):=\mathbb{E}[M(Z)H_{q}(Z)]$ .

\begin{remark}\rm
\textrm{Without loss of generality, the first two coefficients $%
J_{0}(M),J_{1}(M)$ can always be taken to be zero in the present framework.
Indeed, $J_{0}(M):=\mathbb{E}[M(Z)]=0,$ assuming we work with
centered variables and moreover as we noted earlier $h_{\ell ;1,d}=\int_{{%
\mathbb{S}}^{d}}T_{\ell }(x)\,dx=0$. }
\end{remark}
\begin{proof}[Proof Proposition \ref{general}]
As in \cite{wigexc}, from \paref{exp} we write the expansion%
\begin{equation*}
S_{\ell }(M) =\int_{\cS^{d}}\sum_{q=2}^{\infty }\frac{J_{q}(M)H_{q}(T_{\ell
}(x))}{q!}dx\ .
\end{equation*}
Precisely, we write for $d=2$
\begin{equation}\label{sum2}
S_{\ell }(M) =\frac{J_{2}(M)}{2%
}h_{\ell;2,2}+ \frac{J_{3}(M)}{3!}h_{\ell;3,2} + \frac{J_{4}(M)}{4!}h_{\ell;4,2} +\int_{\cS^{2}}\sum_{q=5}^{\infty }\frac{J_{q}(M)H_{q}(T_{\ell }(x))}{q!%
}dx\text{ ,}
\end{equation}
%
whereas for $d\ge 3$
\begin{eqnarray}\label{sum2d}
S_{\ell }(M) =\frac{J_{2}(M)}{2}h_{\ell;2,d}
+\int_{\cS^{d}}\sum_{q=3}^{\infty }\frac{J_{q}(M)H_{q}(T_{\ell }(x))}{q!%
}dx\text{ .}
\end{eqnarray}%
Let us first investigate the case $d=2$.
Set for the sake of simplicity%
\[
S_{\ell }(M;1):=\frac{J_{2}(M)}{2%
}h_{\ell;2,2}+ \frac{J_{3}(M)}{3!}h_{\ell;3,2} + \frac{J_{4}(M)}{4!}h_{\ell;4,2}\text{ ,}
\]%
\[
S_{\ell }(M;2):=\int_{\cS^{2}}\sum_{q=5}^{\infty }\frac{J_{q}(M)H_{q}(T_{\ell
}(x))}{q!}dx\text{ .}
\]%
Consider $Z\sim \mathcal N(0,1)$ and $Z_\ell \sim \mathcal{N}\left(0,\frac{\Var[S_{\ell }(M;1)]}{%
\Var[S_{\ell }(M)]}\right)$. Hence from \paref{sum2} and the triangular inequality
$$\displaylines{
d_{W}\left( \frac{S_{\ell }(M)}{\sqrt{\Var[S_{\ell }(M)]}},Z\right)\le\cr
\le d_{W}\left( \frac{S_{\ell }(M)}{\sqrt{\Var[S_{\ell }(M)]}},\frac{%
S_{\ell }(M;1)}{\sqrt{\Var[S_{\ell }(M)]}}\right) +d_{W}\left( \frac{S_{\ell
}(M;1)}{\sqrt{\Var[S_{\ell }(M)]}},Z_\ell\right)
+d_{W}\left( Z_\ell,%
Z\right)\le}$$
$$\displaylines{
\le \frac{1}{\sqrt{\Var[S_{\ell }(M)]}}\mathbb{E}\left[\left(
\int_{\cS^{2}}\sum_{q=5}^{\infty }\frac{J_{q}(M)H_{q}(T_{\ell }(x))}{q!}dx%
\right)^2\right]^{\tfrac12}+\cr
+d_{W}\left( \frac{S_{\ell }(M;1)}{\sqrt{\Var[S_{\ell }(M)]}},Z_\ell\right) +d_{W}\left( Z_\ell,Z\right)\ .
}$$
Let us bound the first term of the previous summation. Of course
$$\displaylines{
\Var[S_{\ell }(M)] = \Var[S_{\ell }(M;1)] +
\Var[S_{\ell }(M;2)]\ ;
}$$
now we have (see \cite{wigexc})
$$
\Var[S_{\ell }(M;1)] =\frac{J_{2}^{2}(M)}{2^2}\Var[h_{\ell ;2,2}]+\frac{%
J_{3}^{2}(M)}{6^2}\Var[h_{\ell ;3,2}]
+\frac{J_{4}^{2}(M)}{(4!)^2}\Var[h_{\ell ;4,2}]
$$
and moreover
$$\displaylines{
\Var[S_{\ell }(M;2)]=
\mathbb{E}\left[ \left(\int_{\cS^{2}}\sum_{q=5}^{\infty }\frac{J_{q}(M)H_{q}(T_{\ell
}(x))}{q!}dx \right)^2\right ]=\sum_{q=5}^{\infty }\frac{J_{q}^{2}(M)}{(q!)^2}%
\Var[h_{\ell ;q,2}]\ll \cr
\ll \frac{1}{\ell ^{2}}\sum_{q=5}^{\infty }\frac{J_{q}^{2}(M)}{q!}\ll
\frac{1}{\ell ^{2}}\text{ ,}
}$$
where the last bounds follows from \paref{int2} and \paref{cq2}. Remark that $$
\Var(S_\ell(M))=\sum_{q=0}^{\infty }\frac{J_{q}^{2}(M)}{q!} <+\infty\ .
$$
Therefore recalling also \paref{q=2} and \paref{q=4d=2}
\[
\frac{1}{\Var[S_{\ell }(M)]}\mathbb{E}\left[ \left(\int_{\cS^{2}}\sum_{q=5}^{\infty }%
\frac{J_{q}(M)H_{q}(T_{\ell }(x))}{q!}dx\right)^2\right]\ll \frac{1}{\ell }%
\text{ .}
\]%
On the other hand, from Proposition \ref{corollario1}%
\[
d_{W}\left( \frac{S_{\ell }(M;1)}{\sqrt{\Var[S_{\ell }(M)]}},Z_\ell\right) =O\left(\frac{1}{\sqrt{\ell
}}\right)
\]%
and finally, using Proposition 3.6.1 in \cite{noupebook},
\begin{eqnarray*}
d_{W}\left( Z_\ell,%
Z\right) &\leq &\sqrt{\frac{2}{\pi }}\left|\frac{%
\Var[S_{\ell }(M;1)]}{\Var[S_{\ell }(M)]}-1\right|=O\left(\frac{1}{\ell }\right)\text{ ,}
\end{eqnarray*}%
so that the proof for $d=2$ is completed.

The proof in the general case $d\ge 3$ is indeed
analogous, just setting
\[
S_{\ell }(M;1):=\frac{J_{2}(M)}{2%
}h_{\ell;2,d}\text{ ,}
\]%
\[
S_{\ell }(M;2):=\int_{\cS^{2}}\sum_{q=3}^{\infty }\frac{J_{q}(M)H_{q}(T_{\ell
}(x))}{q!}dx
\]%
and recalling from \paref{momento 2} that $\Var[h_{\ell;2,d}]=O( \frac{1}{\ell^{d-1}})$ whereas for
$q\ge 3$, $\Var[h_{\ell;q,d}]=O( \frac{1}{\ell^{d}})$ from Proposition \ref{varianza}.
\end{proof}

We are now in the position to establish our main result, concerning the volume
of the excursion sets, which we recall for any fixed $z\in \mathbb{R}$ is given by%
\begin{equation*}
S_{\ell }(z):=S_{\ell }(\mathbb{I}(\cdot > z))=\int_{{\mathbb{S}}^{d}}\mathbb{I%
}(T_{\ell }(x) > z)dx\text{ .}
\end{equation*}%
Again, ${\mathbb{E}}[S_{\ell }(z)]=\mu _{d}(1-\Phi (z))$, where
$\Phi (z)$ is the cdf of the standard Gaussian law, and in this case we have
$M=M_z:=\mathbb{I}(\cdot > z)$, $J_{2}(M_z)=z\phi (z)$, $\phi$ denoting
 the standard Gaussian density. The proof of Theorem
 \ref{mainteo} is then just an immediate consequence of Proposition \ref{general}.

\begin{remark}\rm
\textrm{It should be noted that the rate obtained here is much sharper than
the one provided by \cite{pham} for the Euclidean case with $d=2$. The
asymptotic setting we consider is rather different from his, in that we
consider the case of spherical eigenfunction with diverging eigenvalues,
whereas he focusses on functionals evaluated on increasing domains $%
[0,T]^{d} $ for $T\rightarrow \infty .$ However the contrast in the
converging rates is not due to these different settings, indeed \cite{vale3}
establish rates of convergence analogous to those by \cite{pham} for
spherical random fields with more rapidly decaying covariance structure than
the one we are considering here. The main point to notice is that the slow
decay of Gegenbauer polynomials entails some form of long range dependent
behaviour on random spherical harmonics; in this sense, hence, our results
may be closer in spirit to the work by \cite{dehlingtaqqu} on empirical
processes for long range dependent stationary processes on $\mathbb{R}$. }
\end{remark}

\section{Technical proofs}\label{tech}

\subsection{On the variance of $h_{\ell ;q,d}$}\label{subvar}

In this section we study the variance  of $h_{\ell ;q,d}$ defined in \eqref{hq}. By (\ref%
{hermite orto}) and the definition of Gaussian random eigenfunctions
\eqref{defT}, it follows that \eqref{int var} hold at once:
\begin{equation*}
\displaylines{ \Var[h_{\ell;q,d}]= \E \left[ \left( \int_{\cS^d}
H_q(T_\ell(x))\,dx \right)^2 \right] = \int_{(\cS^d)^2} \E[ H_q(T_\ell(x_1))
H_q(T_\ell(x_2))]\,dx_1 dx_2 = \cr = q! \int_{(\cS^d)^2} \E[T_\ell(x_1)
T_\ell(x_2)]^q\,dx_1 dx_2 = q! \int_{(\cS^d)^2} G_{\ell;d}(\cos
d(x_1,x_2))^q\,dx_1 dx_2=\cr = q! \mu_d \mu_{d-1} \int_0^{\pi}
G_{\ell;d}(\cos \vartheta)^q (\sin \vartheta)^{d-1}\, d\vartheta. }
\end{equation*}%
Now we prove Proposition \ref{varianza}, inspired by the proof of \cite{def}, Lemma $5.2$.

\begin{proof}[Proof of Proposition \protect\ref{varianza}]

By the Hilb's asymptotic formula for Jacobi polynomials (see \cite{szego}, Theorem $8.21.12$),
we have uniformly for $\ell\ge 1$, $\vartheta\in [0, \tfrac{\pi}2]$
$$\displaylines{
(\sin \vartheta)^{\frac{d}{2} - 1}G_{\ell ;d} (\cos \vartheta) =
\frac{2^{\frac{d}{2} - 1}}{{\ell + \frac{d}{2}-1\choose \ell}}
\left( a_{\ell, d} \left(\frac{\vartheta}{\sin
\vartheta}\right)^{\tfrac12}J_{\frac{d}{2} - 1}(L\vartheta) +
\delta(\vartheta) \right)\ , }$$ where $L=\ell + \frac{d-1}{2}$,
\begin{equation}\label{al}
a_{\ell, d} = \frac{\Gamma(\ell + \frac{d}{2})}{(\ell + \frac{d-1}{2})^{\tfrac{d}{2}-1} \ell !}\
\sim\  1\quad \text{as}\ \ell \to \infty,
\end{equation}
and the remainder is
\begin{equation*}
\delta(\vartheta) \ll \begin{cases} \sqrt{\vartheta}\,
\ell^{-\tfrac32}\ & \qquad \ell^{-1}< \vartheta
< \tfrac{\pi}2\ ,             \\
\vartheta^{\left(\tfrac{d}2-1\right) + 2}\, \ell^{\tfrac{d}2-1}\ & \qquad 0 < \vartheta
< \ell^{-1}\ .
\end{cases}
\end{equation*}
Therefore, in the light of \eqref{al} and $\vartheta \to \frac{\vartheta}{\sin \vartheta}$ being bounded,
\begin{equation}\label{eq:G moment main error}
\begin{array}{c}
\displaystyle\int_{0}^{\frac{\pi}{2}} G_{\ell ;d} (\cos \vartheta)^q (\sin
\vartheta)^{d-1}d\vartheta =\\
\displaystyle=\left(\frac{2^{\frac{d}{2} - 1}}{{\ell + \frac{d}{2}-1\choose
\ell}}\right)^q a^q_{\ell,d} \int_0^{\frac{\pi}{2}} ( \sin
\vartheta)^{- q(\frac{d}{2} -1)}
 \Big( \frac{\vartheta}{\sin \vartheta} \Big)^{\frac{q}{2}}
J^q_{\frac{d}{2}-1}(L\vartheta) (\sin \vartheta)^{d-1} d\vartheta\ +\qquad\cr
\displaystyle\qquad+O\left(\frac{1}{\ell^{q(\frac{d}{2}-1)}} \int_0^{\frac{\pi}{2}}
( \sin \vartheta)^{- q(\frac{d}{2} -1)}
|J_{\frac{d}{2}-1}(L\vartheta)|^{q-1}
\delta(\vartheta)(\sin \vartheta)^{d-1}d\vartheta\right),\cr
\end{array}
\end{equation}
where we used $${\ell + \frac{d}{2}-1\choose \ell} \ll \frac{1}{\ell^{\frac{d}{2}-1}}$$
(note that we readily neglected the smaller terms, corresponding to higher powers of $\delta(\vartheta)$).
We rewrite \eqref{eq:G moment main error} as
\begin{equation}
\label{eq:G moment=M+E}
\begin{split}
\int_{0}^{\frac{\pi}{2}} G_{\ell ;d} (\cos \vartheta)^q (\sin
\vartheta)^{d-1} d\vartheta
=N+E\ ,
\end{split}
\end{equation}
where
\begin{equation}
\label{eq:Mdql def}
N=N(d,q;\ell) := \left(\frac{2^{\frac{d}{2} - 1}}{{\ell + \frac{d}{2}-1\choose
\ell}}\right)^q a^q_{\ell,d} \int_0^{\frac{\pi}{2}} ( \sin
\vartheta)^{- q(\frac{d}{2} -1)}
 \Big( \frac{\vartheta}{\sin \vartheta} \Big)^{\frac{q}{2}}
J_{\frac{d}{2}-1}(L\vartheta)^q (\sin \vartheta)^{d-1} d\vartheta\
\end{equation}
and
\begin{equation}
\label{eq:Edql def}
E=E(d,q;\ell) \ll \frac{1}{\ell^{q(\frac{d}{2}-1)}} \int_0^{\frac{\pi}{2}}
( \sin \vartheta)^{- q(\frac{d}{2} -1)}
|J_{\frac{d}{2}-1}(L\vartheta)|^{q-1}
\delta(\vartheta)(\sin \vartheta)^{d-1}d\vartheta\ .
\end{equation}
To bound the error term $E$ we split the range of the integration in \eqref{eq:Edql def}
and write
\begin{equation}
\label{eq:int error split}
\begin{split}
E \ll & \frac{1}{\ell^{q(\frac{d}{2}-1)}}
\int\limits_{0}^{\frac{1}{\ell}}( \sin \vartheta)^{- q(\frac{d}{2} -1)}
|J_{\frac{d}{2}-1}(L\vartheta)|^{q-1}
\vartheta^{\left(\tfrac{d}2-1\right) + 2}\, \ell^{\tfrac{d}2-1}(\sin \vartheta)^{d-1}\,d\vartheta +
\\&+\frac{1}{\ell^{q(\frac{d}{2}-1)}}
\int_{\frac{1}{\ell}}^{\frac{\pi}{2}}
( \sin \vartheta)^{- q(\frac{d}{2} -1)}
|J_{\frac{d}{2}-1}(L\vartheta)|^{q-1}
\sqrt{\vartheta}\,
\ell^{-\tfrac32}(\sin \vartheta)^{d-1}\,d\vartheta\ .
\end{split}
\end{equation}
For the first integral in \eqref{eq:int error split}
recall that $J_{\frac{d}{2}-1}(z) \sim z^{\frac{d}{2}-1}$ as
$z\to 0$, so that as $\ell\to \infty$,
$$\displaylines{
\frac{1 }{\ell^{(q-1)(\frac{d}{2}-1)}} \int_0^{\frac{1}{\ell}}
\left(\frac{\vartheta}{ \sin \vartheta}\right)^{ q(\frac{d}{2} -1)-d+1}
|J_{\frac{d}{2}-1}(L\vartheta)|^{q-1}
\vartheta^{-(q-1)\left(\tfrac{d}2-1\right) + d+1}\, \,d\vartheta
\ll\cr
}$$
\begin{equation}\label{err1}
\ll \int_0^{\frac{1}{\ell}}\vartheta^{d+1}\,d\vartheta = \frac{1}{\ell^{d+2}}\ ,
\end{equation}
which is enough for our purposes. Furthermore,
since for $z$ big $|J_{\frac{d}{2}-1}(z)|=O(z^{-\tfrac12})$ (and keeping in mind
that $L$ is of the same order of magnitude as $\ell$), we may bound the second integral
in \eqref{eq:int error split} as
$$\displaylines{\ll
\frac{1}{\ell^{q(\frac{d}{2}-1)+\frac32}}\int_{\frac{1}{\ell}}^{\frac{\pi}{2}}
\left( \frac{\vartheta}{\sin \vartheta}\right)^{ q(\frac{d}{2} -1)-d+1}
|J_{\frac{d}{2}-1}(L\vartheta)|^{q-1}
\vartheta^{-q(\frac{d}{2}-1)+d-\frac12}\,d\vartheta \ll \cr
\ll
\frac{1}{\ell^{q(\frac{d}{2}-1)+\frac32}}
\int_{\frac{1}{\ell}}^{\frac{\pi}{2}}
(\ell\vartheta)^{-\frac{q-1}{2}}
\vartheta^{-q(\frac{d}{2}-1)+d-\frac12}\,d\vartheta
=\frac{1}{\ell^{q(\frac{d}{2}-\frac12)+2}}
\int_{\frac{1}{\ell}}^{\frac{\pi}{2}}
\vartheta^{-q(\frac{d}{2}-\frac12)+d}\,d\vartheta \ll \cr
}$$
\begin{equation}\label{err2}
\ll \frac{1}{\ell^{(d+2)\wedge \left(q\left(\tfrac{d}2 -\tfrac{1}2\right) +1\right)}} = o(\ell^{-d})\ ,
\end{equation}
where the last equality in \paref{err2} holds for $q\ge 3$.
From \paref{err1} (bounding the first integral in \eqref{eq:int error split}) and
\paref{err2} (bounding the second integral in \eqref{eq:int error split}) we finally
find that the error term in
\eqref{eq:G moment=M+E} is
\begin{equation}
\label{resto}
E =o(\ell^{-d})
\end{equation}
for $q\ge 3$, admissible for our purposes.

Therefore, substituting \paref{resto} into \eqref{eq:G moment=M+E} we have
$$\displaylines{
\int_{0}^{\frac{\pi}{2}} G_{\ell ;d} (\cos \vartheta)^q (\sin
\vartheta)^{d-1}\, d\vartheta =\cr
=\left(\frac{2^{\frac{d}{2} -
1}}{{\ell + \frac{d}{2}-1\choose \ell}}\right)^q a^q_{\ell,d}
\int_0^{\frac{\pi}{2}} ( \sin \vartheta)^{- q(\frac{d}{2} -1)}
 \Big( \frac{\vartheta}{\sin \vartheta} \Big)^{\frac{q}{2}}
J_{\frac{d}{2}-1}(L\vartheta)^q (\sin \vartheta)^{d-1}d\vartheta + o(\ell^{-d}) =
}$$
\begin{align}
\label{bene}
=\left(\frac{2^{\frac{d}{2} -
1}}{{\ell + \frac{d}{2}-1\choose \ell}}\right)^q a^q_{\ell,d} \frac{1}{L}
\int_0^{L\frac{\pi}{2}} ( \sin  \frac{\psi}{L})^{- q(\frac{d}{2} -1)}
 \Big( \frac{ \frac{\psi}{L}}{\sin \frac{\psi}{L}} \Big)^{\frac{q}{2}}\times \\
\nonumber
\times
J_{\frac{d}{2}-1}(\psi)^q (\sin  \frac{\psi}{L})^{d-1}\, d\psi + o(\ell^{-d})\ ,
\end{align}
where in the last equality we transformed $\psi/L=\vartheta$; it then remains to evaluate
the first term in \paref{bene}, which we denote by
\begin{equation*}
N_L := \left(\frac{2^{\frac{d}{2} -
1}}{{\ell + \frac{d}{2}-1\choose \ell}}\right)^q a^q_{\ell,d} \frac{1}{L}
\int_0^{L\frac{\pi}{2}} ( \sin \psi/L)^{- q(\frac{d}{2} -1)}
 \Big( \frac{\psi/L}{\sin \psi/L} \Big)^{\frac{q}{2}}
J_{\frac{d}{2}-1}(\psi)^q (\sin \psi/L)^{d-1}\, d\psi\ .
\end{equation*}
Now recall that as $\ell\to \infty$
\begin{equation*}
{\ell + \frac{d}{2}-1\choose \ell}\ \sim \ \frac{\ell^{\frac{d}{2}-1}}{(\frac{d}2-1)!}\ ;
\end{equation*}
moreover \paref{al} holds,   therefore we find of course that as $L\to \infty$
\begin{equation}\label{mado}
N_L\  \sim\  \frac{(2^{\frac{d}{2} -
1}(\frac{d}2-1)!)^q}{L^{q(\frac{d}{2}-1)+1}}
\int_0^{L\frac{\pi}{2}} ( \sin \psi/L)^{- q(\frac{d}{2} -1)}
 \Big( \frac{\psi/L}{\sin \psi/L} \Big)^{\frac{q}{2}}
J_{\frac{d}{2}-1}(\psi)^q (\sin \psi/L)^{d-1}\, d\psi\ .
\end{equation}
In order to finish the proof of Proposition \ref{varianza},
it is enough to check that, as $L\to \infty$
$$\displaylines{
L^d \, \frac{(2^{\frac{d}{2} -
1}(\frac{d}2-1)!)^q}{L^{q(\frac{d}{2}-1)+1}}
\int_0^{L\frac{\pi}{2}} ( \sin  \frac{\psi}{L})^{- q(\frac{d}{2} -1)}
 \Big( \frac{ \frac{\psi}{L}}{\sin  \frac{\psi}{L}} \Big)^{\frac{q}{2}}
J_{\frac{d}{2}-1}(\psi)^q \left (\sin \frac{\psi}{L} \right )^{d-1}\, d\psi\, \to\, c_{q;d}\ ,
}$$
actually from \paref{bene} and \paref{mado}, we have
$$\displaylines{
\lim_{\ell\to +\infty} \ell^d \int_{0}^{\frac{\pi}{2}} G_{\ell ;d} (\cos \vartheta)^q (\sin
\vartheta)^{d-1}\, d\vartheta = \cr
= \lim_{L\to +\infty} L^d \, \frac{(2^{\frac{d}{2} -
1}(\frac{d}2-1)!)^q}{L^{q(\frac{d}{2}-1)+1}}
\int_0^{L\frac{\pi}{2}} ( \sin  \frac{\psi}{L})^{- q(\frac{d}{2} -1)}
 \Big( \frac{ \frac{\psi}{L}}{\sin  \frac{\psi}{L}} \Big)^{\frac{q}{2}}
J_{\frac{d}{2}-1}(\psi)^q \left (\sin \frac{\psi}{L} \right )^{d-1}\, d\psi \ .
}$$
Now we write
$$\frac{\psi/L}{\sin \psi/L} = 1 + O\left( \psi^2/L^2  \right),$$
so that
$$\displaylines{
L^d\, \frac{\left(2^{\frac{d}{2} - 1}(\frac{d}2-1)!\right)^q}{L^{q(\frac{d}{2}-1)+1}}
\int_0^{L\frac{\pi}{2}} ( \sin \psi/L)^{- q(\frac{d}{2} -1)}
 \Big( \frac{\psi/L}{\sin \psi/L} \Big)^{\frac{q}{2}}
J_{\frac{d}{2}-1}(\psi)^q (\sin \psi/L)^{d-1}\, d\psi=\cr
=\left(2^{\frac{d}{2} - 1}(\frac{d}2-1)!\right)^q
\int_0^{L\frac{\pi}{2}}
 \Big( \frac{\psi/L}{\sin \psi/L} \Big)^{q(\tfrac{d}{2} -\tfrac12)-d+1}
J_{\frac{d}{2}-1}(\psi)^q \psi^{- q(\frac{d}{2} -1)+d-1}\, d\psi=\cr
=\left(2^{\frac{d}{2} - 1}(\frac{d}2-1)!\right)^q
\int_0^{L\frac{\pi}{2}}
 \Big( 1 + O\left( \psi^2/L^2  \right) \Big)^{q(\tfrac{d}{2} -\tfrac12)-d+1}
J_{\frac{d}{2}-1}(\psi)^q \psi^{- q(\frac{d}{2} -1)+d-1}\, d\psi=\cr
=\left(2^{\frac{d}{2} - 1}(\frac{d}2-1)!\right)^q
\int_0^{L\frac{\pi}{2}}
J_{\frac{d}{2}-1}(\psi)^q \psi^{- q(\frac{d}{2} -1)+d-1}\, d\psi +\cr
+
O\left( \frac{1}{L^2}\int_0^{L\frac{\pi}{2}}
J_{\frac{d}{2}-1}(\psi)^q \psi^{- q(\frac{d}{2} -1)+d+1}\,d\psi  \right).
}$$
Note that as  $L\to +\infty$, the first term of the previous summation
 converges to $c_{q;d}$ defined in \paref{ecq}, i.e.
\begin{equation}\label{to cq}
\left(2^{\frac{d}{2} - 1}(\frac{d}2-1)!\right)^q
\int_0^{L\frac{\pi}{2}}
J_{\frac{d}{2}-1}(\psi)^q \psi^{- q(\frac{d}{2} -1)+d-1}\, d\psi \to c_{q;d}\ .
\end{equation}
It remains to bound the remainder
$$
 \frac{1}{L^2}\int_0^{L\frac{\pi}{2}}
|J_{\frac{d}{2}-1}(\psi)|^q \psi^{- q(\frac{d}{2} -1)+d+1}\,d\psi =
O(1) +  \frac{1}{L^2}\int_1^{L\frac{\pi}{2}}
|J_{\frac{d}{2}-1}(\psi)|^q \psi^{- q(\frac{d}{2} -1)+d+1}\,d\psi\ .
$$
Now for the second term on the r.h.s.
$$\displaylines{
\int_1^{L\frac{\pi}{2}}
|J^q_{\frac{d}{2}-1}(\psi)| \psi^{- q(\frac{d}{2} -1)+d+1}\,d\psi \ll
\int_1^{L\frac{\pi}{2}} \psi^{- q(\frac{d}{2} -\frac{1}{2})+d+1}\,d\psi=\cr
= O(1 + L^{- q(\frac{d}{2} -\frac{1}{2})+d+2})\ .
}$$
Therefore we obtain
$$\displaylines{
\Bigl(2^{\frac{d}{2} - 1}(\frac{d}2-1)!\Bigr)^q \int_0^{L\frac{\pi}{2}}
J_{\frac{d}{2}-1}(\psi)^q \psi^{- q(\frac{d}{2} -1)+d-1}\, d\psi +\cr
+
O\left( \frac{1}{L^2}\int_0^{L\frac{\pi}{2}}
J_{\frac{d}{2}-1}(\psi)^q \psi^{- q(\frac{d}{2} -1)+d+1}\,d\psi
\right)
= \cr
=\left(2^{\frac{d}{2} - 1}(\frac{d}2-1)!\right)^q
\int_0^{L\frac{\pi}{2}} J_{\frac{d}{2}-1}(\psi)^q \psi^{-
q(\frac{d}{2} -1)+d-1}\, d\psi + O(L^{-2} + L^{- q(\frac{d}{2}
-\frac{1}{2})+d})\ , }$$ so that we have just checked the
statement of the present proposition for $q > \frac{2d}{d-1}$. This is indeed enough for each $q\ge
3$  when
$d\ge 4$ .

It remains to investigate separately just the case $d=q=3$.
Recall that for $d=3$ we have an explicit formula for the Bessel
function of order $\frac{d}{2}-1$ (\cite{szego}), that is
\begin{equation*}
J_{\frac{1}{2}}(z) = \sqrt{\frac{2}{\pi z}} \sin (z)\ ,
\end{equation*}
and hence the integral in \paref{ecq} is indeed convergent for $q=d=3$ by integrations by parts.

We have hence to study the convergence of the following integral
\begin{equation*}
\frac{8}{\pi^\frac{3}{2} }
\int_0^{L\frac{\pi}{2}}
\left( \frac{\psi/L}{\sin \psi/L} \right)
\frac{\sin^3 \psi}{\psi} \, d\psi\ .
\end{equation*}
To this aim, let us consider a large parameter $K\gg 1$ and divide the integration range
into $[0,K]$ and $[K, \frac{\pi}{2}]$;
the main contribution comes from the first term, whence we have to prove that the latter vanishes.
Note that
\begin{equation}
\label{eq:int K-> bound}
\int_K^{L\frac{\pi}{2}} \left( \frac{\psi/L}{\sin \psi/L} \right)
\frac{\sin^3 \psi}{\psi} \, d\psi \ll \frac{1}{K}\ ,
\end{equation}
where we use integration by part with the bounded function $I(T) = \int_0^T \sin^3 z\, dz$.
On $[0,K]$, we write
$$\displaylines{
\frac{8}{\pi^\frac{3}{2} }\int_0^{K} \left( \frac{\psi/L}{\sin \psi/L} \right)
\frac{\sin^3 \psi}{\psi} \, d\psi
= \frac{8}{\pi^\frac{3}{2} }\int_0^{K}
\frac{\sin^3 \psi}{\psi} \, d\psi + O\left( \frac{1}{L^2}\int_0^{K}
\psi \sin^3 \psi\, d\psi   \right) = \cr
=\frac{8}{\pi^\frac{3}{2} }\int_0^{K}
\frac{\sin^3 \psi}{\psi} \, d\psi +O\left( \frac{K^2}{L^2} \right).
}$$
Consolidating the latter with \eqref{eq:int K-> bound} we find that
$$\displaylines{
\frac{8}{\pi^\frac{3}{2} }
\int_0^{L\frac{\pi}{2}}
\left( \frac{\psi/L}{\sin \psi/L} \right)
\frac{\sin^3 \psi}{\psi} \, d\psi = \frac{8}{\pi^\frac{3}{2} }\int_0^{K}
\frac{\sin^3 \psi}{\psi} \, d\psi + O\left(  \frac{1}{K} +  \frac{K^2}{L^2} \right).
}$$
Now as $K\to +\infty$,
$$
\frac{8}{\pi^\frac{3}{2} }\int_0^{K}
\frac{\sin^3 \psi}{\psi} \, d\psi \to c_{3;3}\ ;
$$
to conclude the proof, it is then enough to choose $K=K(L)\rightarrow\infty$
sufficiently slowly, i.e. $K=\sqrt{L}$.
\end{proof}

\subsection{Proofs of Propositions \ref{cum2} and \ref{cumd}}

\begin{proof}[Proof \paref{cum2}]
The bounds (\ref{hotel1}), (\ref{hotel2}) are known and indeed the
corresponding integrals can be evaluated explicitly in terms of Wigner's 3j
and 6j coefficients, see \cite{dogiocam} e.g. The
bounds in (\ref{hotel3}),(\ref{hotel4}) derives from a simple improvement in
the proof of Proposition 2.2 in \cite{Nonlin}, which can be obtained when
focussing only on a subset of the terms (the circulant ones) considered in
that reference. In the proof to follow, we exploit repeatedly
 (\ref{tri1}), (\ref{tri2}), (\ref{tri3}) and \paref{tri4}.

Let us start investigating the case $q=5$:%
$$\displaylines{
\mathcal{K}_{\ell }(5;1)=\int_{(\cS^{2})^4}\left\vert
P_{\ell }(\cos d( x_{1},x_{2}) )\right\vert ^{4}\left\vert
P_{\ell }(\cos d( x_{2},x_{3}) )\right\vert \times \cr
\times \left\vert
P_{\ell }(\cos d( x_{3},x_{4}) )\right\vert ^{4}\left\vert
P_{\ell }(\cos d( x_{4},x_{1}) )\right\vert
dx_{1}dx_{2}dx_{3}dx_{4} \le \cr
\mathcal{\leq }\int_{(\cS^{2})^4}\left\vert P_{\ell
}(\cos d( x_{1},x_{2}) )\right\vert ^{4}\left\vert P_{\ell
}(\cos d( x_{2},x_{3}) )\right\vert \left\vert P_{\ell
}(\cos d( x_{3},x_{4}) )\right\vert
^{4}dx_{1}dx_{2}dx_{3}dx_{4} \le \cr
\mathcal{\leq }\int_{(\cS^{2})^3}\left\vert P_{\ell
}(\cos d( x_{1},x_{2}) )\right\vert ^{4}\left\vert P_{\ell
}(\cos d( x_{2},x_{3}) )\right\vert
\int_{\cS^{2}}\left\vert P_{\ell }(\cos d( x_{3},x_{4})
)\right\vert ^{4}dx_{4}\,\, dx_{1}dx_{2}dx_{3} \le\cr
\mathcal{\leq } O\left( \frac{\log \ell }{\ell ^{2}}\right) \times
\int_{\cS^{2}\times \cS^{2}}\left\vert P_{\ell }(\cos d(
x_{1},x_{2}) )\right\vert ^{4}\left\{ \int_{\cS^{2}}\left\vert
P_{\ell }(\cos d( x_{2},x_{3}) )\right\vert dx_{3}\right\}
dx_{1}dx_{2} \le \cr
\leq O\left( \frac{\log \ell }{\ell ^{2}}\right) \times O\left( \frac{1}{%
\sqrt{\ell }}\right) \times \int_{\cS^{2}\times \cS^{2}}\left\vert P_{\ell
}(\cos d( x_{1},x_{2}) )\right\vert ^{4}dx_{1}dx_{2} \le \cr
\leq O\left( \frac{\log \ell }{\ell ^{2}}\right) \times O\left( \frac{1}{%
\sqrt{\ell }}\right) \times O\left( \frac{\log \ell }{\ell ^{2}}\right)
=O\left( \frac{\log ^{2}\ell }{\ell ^{9/2}}\right) \text{ ;}
}$$
$$\displaylines{
\mathcal{K}_{\ell }(5;2)=\int_{(\cS^{2})^4}\left\vert
P_{\ell }(\cos d( x_{1},x_{2}) )\right\vert ^{3}\left\vert
P_{\ell }(\cos d( x_{2},x_{3}) )\right\vert^2 \times \cr
\times
 \left\vert
P_{\ell }(\cos d( x_{3},x_{4}) )\right\vert ^{3}\left\vert
P_{\ell }(\cos d( x_{4},x_{1}) )\right\vert^2
dx_{1}dx_{2}dx_{3}dx_{4}\le \cr
\mathcal{\leq }\int_{(\cS^{2})^4}\left\vert P_{\ell
}(\cos d( x_{1},x_{2}) )\right\vert ^{3}\left\vert P_{\ell
}(\cos d( x_{2},x_{3}) )\right\vert^2 \left\vert P_{\ell
}(\cos d( x_{3},x_{4}) )\right\vert
^{3}dx_{1}dx_{2}dx_{3}dx_{4} \le \cr
\mathcal{\leq }\int_{(\cS^{2})^3}\left\vert P_{\ell
}(\cos d( x_{1},x_{2}) )\right\vert ^{3}\left\vert P_{\ell
}(\cos d( x_{2},x_{3}) )\right\vert^2
\int_{\cS^{2}}\left\vert P_{\ell }(\cos d( x_{3},x_{4})
)\right\vert ^{3}dx_{4}\,\, dx_{1}dx_{2}dx_{3} \le \cr
\mathcal{\leq }O\left( \sqrt{\frac{\log \ell }{\ell ^{3}}}\right) \times
\int_{\cS^{2}\times \cS^{2}}\left\vert P_{\ell }(\cos d(
x_{1},x_{2}) )\right\vert ^{3}\left\{ \int_{\cS^{2}}\left\vert
P_{\ell }(\cos d( x_{2},x_{3}) )\right\vert^2 dx_{3}\right\}
dx_{1}dx_{2} \le \cr
\leq O\left( \sqrt{\frac{\log \ell }{\ell ^{3}}}\right) \times O\left( \frac{1}{%
\ell }\right) \times \int_{\cS^{2}\times \cS^{2}}\left\vert P_{\ell
}(\cos d( x_{1},x_{2}) )\right\vert ^{3}dx_{1}dx_{2} \le \cr
\leq O\left( \sqrt{\frac{\log \ell }{\ell ^{3}}}\right)\times O\left( \frac{1}{%
\ell}\right) \times O\left( \sqrt{\frac{\log \ell }{\ell ^{3}}}\right)
=O\left( \frac{\log \ell }{\ell ^{4}}\right) \text{ .}
}$$
For $q=6$ and $r=1$ we simply note that $\mathcal{K}_{\ell }(6;1)\le \mathcal{K}_{\ell }(5;1)$, actually
$$\displaylines{
\mathcal{K}_{\ell }(6;1)=\int_{(\cS^{2})^4}\left\vert
P_{\ell }(\cos d( x_{1},x_{2}) )\right\vert ^{5}\left\vert
P_{\ell }(\cos d( x_{2},x_{3}) )\right\vert \times \cr
\times\left\vert
P_{\ell }(\cos d( x_{3},x_{4}) )\right\vert ^{5}\left\vert
P_{\ell }(\cos d( x_{4},x_{1}) )\right\vert
dx_{1}dx_{2}dx_{3}dx_{4}\le \cr
\leq \int_{(\cS^{2})^4}\left\vert P_{\ell }(\cos d(
x_{1},x_{2}) )\right\vert ^{4}\left\vert P_{\ell }(\cos d(
x_{2},x_{3}) )\right\vert \times \cr
\times \left\vert P_{\ell }(\cos d(
x_{3},x_{4}) )\right\vert ^{4}\left\vert P_{\ell }(\cos d(
x_{4},x_{1}) )\right\vert dx_{1}dx_{2}dx_{3}dx_{4}=
\mathcal{K}_{\ell }(5;1)=O\left( \frac{\log ^{2}\ell }{\ell ^{9/2}}\right)\ .
}$$
Then we find with analogous computations as for $q=5$ that

$$\displaylines{
\mathcal{K}_{\ell }(6;2)=\int_{(\cS^{2})^4}\left\vert
P_{\ell }(\cos d( x_{1},x_{2}) )\right\vert ^{4}\left\vert
P_{\ell }(\cos d( x_{2},x_{3}) )\right\vert ^{2}\times \cr
\times \left\vert
P_{\ell }(\cos d( x_{3},x_{4}) )\right\vert ^{4}\left\vert
P_{\ell }(\cos d( x_{4},x_{1}) )\right\vert
^{2}dx_{1}dx_{2}dx_{3}dx_{4}\le \cr
\le
\int_{(\cS^{2})^4}\left\vert P_{\ell
}(\cos d( x_{1},x_{2}) )\right\vert ^{4}\left\vert P_{\ell
}(\cos d( x_{2},x_{3}) )\right\vert ^{2}\times \cr
\times \left\vert P_{\ell
}(\cos d( x_{3},x_{4}) )\right\vert ^{4}\left\vert P_{\ell
}(\cos d( x_{4},x_{1}) )\right\vert
^{2}dx_{1}dx_{2}dx_{3}dx_{4}\le \cr
\mathcal{\leq }\int_{(\cS^{2})^2}\left\vert P_{\ell }(\cos d(
x_{1},x_{2}) )\right\vert ^{4}dx_{1}
\int_{\cS^{2}}\left\vert P_{\ell }(\cos d( x_{2},x_{3})
)\right\vert ^{2}dx_{2} \times \cr
\times \int_{\cS^{2}}\left\vert P_{\ell
}(\cos d( x_{3},x_{4}) )\right\vert ^{4}dx_{4}
dx_{3}=\cr
=O\left( \frac{\log \ell }{\ell ^{2}}\right) \times O\left( \frac{1}{\ell }%
\right) \times O\left( \frac{\log \ell }{\ell ^{2}}\right) =O\left( \frac{%
\log ^{2}\ell }{\ell ^{5}}\right)
}$$
and likewise
$$\displaylines{
\mathcal{K}_{\ell }(6;3)=\int_{(\cS^{2})^4}\left\vert
P_{\ell }(\cos d( x_{1},x_{2}) )\right\vert ^{3}\left\vert
P_{\ell }(\cos d( x_{2},x_{3}) )\right\vert ^{3}\times \cr
\times \left\vert
P_{\ell }(\cos d( x_{3},x_{4}) )\right\vert ^{3}\left\vert
P_{\ell }(\cos d( x_{4},x_{1}) )\right\vert
^{3}dx_{1}dx_{2}dx_{3}dx_{4}\le \cr
\mathcal{\leq }\int_{(\cS^{2})^4}\left\vert P_{\ell
}(\cos d( x_{1},x_{2}) )\right\vert ^{3}\left\vert P_{\ell
}(\cos d( x_{2},x_{3}) )\right\vert ^{3}\left\vert P_{\ell
}(\cos d( x_{3},x_{4}) )\right\vert
^{3}dx_{1}dx_{2}dx_{3}dx_{4}=\cr
=O\left( \frac{\sqrt{\log \ell }}{\ell ^{3/2}}\right) \times O\left( \frac{%
\sqrt{\log \ell }}{\ell ^{3/2}}\right) \times O\left( \frac{\sqrt{\log \ell }%
}{\ell ^{3/2}}\right) =O\left( \frac{\log ^{3/2}\ell }{\ell ^{9/2}}\right)
\text{ .}
}$$
Finally for $q=7$
$$\displaylines{
\mathcal{K}_{\ell }(7;1) =\int_{\cS^{2}\times ...\times S^{2}}\left\vert
P_{\ell }(\cos d( x_{1},x_{2}) )\right\vert ^{6}\left\vert
P_{\ell }(\cos d( x_{2},x_{3}) )\right\vert \times \cr
\times \left\vert
P_{\ell }(\cos d( x_{3},x_{4}) )\right\vert ^{6}\left\vert
P_{\ell }(\cos d( x_{4},x_{1}) )\right\vert
dx_{1}dx_{2}dx_{3}dx_{4} \le \cr
\leq \int_{\cS^{2}\times S^{2}}\left\vert P_{\ell }(\cos d(
x_{1},x_{2}) )\right\vert ^{6}dx_{1}
\int_{\cS^{2}}\left\vert P_{\ell }(\cos d( x_{2},x_{3})
)\right\vert dx_{3} \times  \cr
\times \int_{\cS^{2}}\left\vert P_{\ell
}(\cos d( x_{3},x_{4}) )\right\vert ^{6}dx_{4}\,
dx_{2}
=O\left( \frac{1}{\ell ^{2}}\right) \times O\left( \frac{1}{\ell ^{1/2}}%
\right) \times O\left( \frac{1}{\ell ^{2}}\right) =O\left( \frac{1}{\ell
^{9/2}}\right)
}$$
and repeating the same argument we obtain
$$
\mathcal{K}_{\ell }(7;2)=O\left( \frac{1}{\ell ^{5}}%
\right)\qquad \text{and}\qquad \mathcal{K}_{\ell }(7;3) =
O\left( \frac{\log^{9/2} \ell}{\ell ^{11/2}}\right)\ .
$$
From
 \paref{simm}, we have indeed computed the bounds for $\mathcal{K}_{\ell
}(q;r)$, $q=1,\dots,7$ and $r=1,\dots, q-1$.

To conclude the proof  we note that, for $q>7$
\begin{equation*}
\max_{r=1,...,q-1}\mathcal{K}_{\ell }(q;r)=\max_{r=1,...,\left[ \frac{q}{2}%
\right] }\mathcal{K}_{\ell }(q;r)\leq \max_{r=1,...,3}\mathcal{K}_{\ell
}(6;r)=O\left( \frac{1}{\ell ^{9/2}}\right) \text{ .}
\end{equation*}%
Moreover in particular
\begin{equation*}
\max_{r=2,...,\left[ \frac{q}{2}\right] }\mathcal{K}_{\ell }(q;r)\leq
\mathcal{K}_{\ell }(7;2)\vee \mathcal{K}_{\ell }(7;3)=O\left( \frac{1}{\ell
^{5}}\right) \text{ ,}
\end{equation*}%
so that the dominant terms are of the form $\mathcal{K}_{\ell }(q;1).$
\end{proof}

\begin{proof}[Proof \paref{cumd}]
 The proof relies on the same argument of the proof of Proposition \ref{cum2},
therefore we shall omit some calculations.
In what follows we exploit repeatedly the
inequalities \paref{tridq}, \paref{trid1}, \paref{trid3} and \paref{tridgen}.

For $q=3$ we immediately have
$$\displaylines{
\mathcal{K}_{\ell }(3;1)=\int_{(\cS^{d})^4}\left\vert G_{\ell ;d}(\cos d( x_{1},x_{2})
)\right\vert ^{2}\left\vert G_{\ell ;d}(\cos d(
x_{2},x_{3}) )\right\vert \times \cr
\times  \left\vert G_{\ell
;d}(\cos d( x_{3},x_{4}) )\right\vert
^{2}\left\vert G_{\ell ;d}(\cos d( x_{4},x_{1})
)\right\vert dx_{1}dx_{2}dx_{3}dx_{4}\le \cr
\mathcal{\leq }\int_{(\cS^{d})^4}\left\vert
G_{\ell ;d}(\cos d( x_{1},x_{2}) )\right\vert
^{2}\left\vert G_{\ell ;d}(\cos d( x_{2},x_{3})
)\right\vert \left\vert G_{\ell ;d}(\cos d(
x_{3},x_{4}) )\right\vert
^{2}dx_{1}dx_{2}dx_{3}dx_{4} =\cr
= O\left( \frac{1 }{\ell ^{d-1}}\right) \times O\left( \frac{1}{%
\sqrt{\ell^{d-1} }}\right) \times O\left( \frac{1 }{\ell ^{d-1}}\right)
=O\left( \frac{1 }{\ell ^{2d +\tfrac{d}{2} - \tfrac{5}{2}}}\right) \text{ .}
}$$
Likewise for $q=4$%
$$\displaylines{
\mathcal{K}_{\ell }(4;1)=\int_{(\cS^{d})^4}\left\vert G_{\ell ;d}(\cos d( x_{1},x_{2})
)\right\vert ^{3}\left\vert G_{\ell ;d}(\cos d(
x_{2},x_{3}) )\right\vert \times \cr
\times \left\vert G_{\ell
;d}(\cos d( x_{3},x_{4}) )\right\vert
^{3}\left\vert G_{\ell ;d}(\cos d( x_{4},x_{1})
)\right\vert dx_{1}dx_{2}dx_{3}dx_{4}\le \cr
\mathcal{\leq }\int_{(\cS^{d})^4}\left\vert
G_{\ell ;d}(\cos d( x_{1},x_{2}) )\right\vert
^{3}\left\vert G_{\ell ;d}(\cos d( x_{2},x_{3})
)\right\vert \left\vert G_{\ell
;d}(\cos d( x_{3},x_{4}) )\right\vert^3\,dx_{1}dx_{2}dx_{3}dx_{4} = \cr
= O\left( \frac{1 }{\ell ^{d-\tfrac12}}\right) \times
O\left(\frac{1 }{\ell ^{\tfrac{d}2-\tfrac12}}\right) \times
O\left( \frac{1 }{\ell ^{d-\tfrac12}}\right) =O\left( \frac{1
}{\ell ^{2d + \tfrac{d}2-\tfrac32}}\right)
}$$
and moreover
$$\displaylines{
\mathcal{K}_{\ell }(4;2)=\int_{(\cS^{d})^4}\left\vert G_{\ell ;d}(\cos d( x_{1},x_{2})
)\right\vert ^{2} \times \cr
\times \left\vert G_{\ell ;d}(\cos d(
x_{2},x_{3}) )\right\vert^2 \left\vert G_{\ell
;d}(\cos d( x_{3},x_{4}) )\right\vert
^{2}\left\vert G_{\ell ;d}(\cos d( x_{4},x_{1})
)\right\vert^2 dx_{1}dx_{2}dx_{3}dx_{4}\le \cr
\mathcal{\leq }\int_{(\cS^{d})^4}\left\vert
G_{\ell ;d}(\cos d( x_{1},x_{2}) )\right\vert
^{2}\left\vert G_{\ell ;d}(\cos d( x_{2},x_{3})
)\right\vert^2 \left\vert G_{\ell
;d}(\cos d( x_{3},x_{4}) )\right\vert^2\,d\underline{x} =\cr
=O\left( \frac{1 }{\ell ^{d-1}}\right) \times O\left(
\frac{1 }{\ell ^{d-1}}\right) \times O\left( \frac{1 }{\ell
^{d-1}}\right) =O\left( \frac{1 }{\ell ^{3d - 3}}\right) \text{ ,}
}$$
where we set $d\underline{x}:=dx_{1}dx_{2}dx_{3}dx_{4}$.
Similarly,
for $q=5$ we get the bounds
$$\displaylines{
\mathcal{K}_{\ell }(5;1)=\int_{\cS^{d}\times ...\times
\cS^{d}}\left\vert G_{\ell ;d}(\cos d( x_{1},x_{2})
)\right\vert ^{4} \times \cr
\times \left\vert G_{\ell ;d}(\cos d(
x_{2},x_{3}) )\right\vert \left\vert G_{\ell
;d}(\cos d( x_{3},x_{4}) )\right\vert
^{4}\left\vert G_{\ell ;d}(\cos d( x_{4},x_{1})
)\right\vert dx_{1}dx_{2}dx_{3}dx_{4}\le \cr
\mathcal{\leq }\int_{\cS^{d}\times ...\times \cS^{d}}\left\vert
G_{\ell ;d}(\cos d( x_{1},x_{2}) )\right\vert
^{4}\left\vert G_{\ell ;d}(\cos d( x_{2},x_{3})
)\right\vert \left\vert G_{\ell ;d}(\cos d(
x_{3},x_{4}) )\right\vert
^{4}d\underline{x} =\cr
= O\left( \frac{1}{\ell ^{d}}\right) \times O\left( \frac{1}{%
\ell^{\tfrac{d}2-\tfrac12}}\right) \times O\left( \frac{1 }{\ell ^{d}}\right)
=O\left( \frac{ 1}{\ell ^{2d +\tfrac{d}2-\tfrac12 }}\right)
}$$
and
$$
\mathcal{K}_{\ell }(5;2)=O\left( \frac{1 }{\ell ^{3d -2}}\right)\ .
$$
It is immediate to check that
$$
\mathcal{K}_{\ell }(6;1)=\mathcal{K}_{\ell }(7;1)=O\left( \frac{1 }{\ell^{2d +\tfrac{d}2-\tfrac12 }}\right)\ ,\quad
\mathcal{K}_{\ell }(6;2)=\mathcal{K}_{\ell }(7;2)=O\left( \frac{1 }{\ell ^{2d + d -1}}\right)\ ,
$$
whereas
$$
\mathcal{K}_{\ell }(6;3)=O\left( \frac{1 }{\ell ^{2d + d - \tfrac32}}\right)\quad \text{and} \quad
\mathcal{K}_{\ell }(7;3)=O\left( \frac{1}{\ell ^{2d + d -\tfrac12}}\right)\ .
$$
The remaining terms are indeed bounded thanks to \paref{simm}.

In order to finish the proof, it is enough to note, as for  that for $q>7$%
\begin{equation}
\max_{r=1,...,q-1}\mathcal{K}_{\ell }(q;r)=\max_{r=1,...,\left[ \frac{q}{2}%
\right] }\mathcal{K}_{\ell }(q;r)\leq \max_{r=1,...,3}\mathcal{K}_{\ell
}(6;r)=O\left( \frac{1}{\ell^{2d +\tfrac{d}2-\tfrac12 }}\right) \text{ .}
\end{equation}%
In particular we have
\begin{equation}
\max_{r=2,...,\left[ \frac{q}{2}\right] }\mathcal{K}_{\ell }(q;r)\leq
\mathcal{K}_{\ell }(7;2)\vee \mathcal{K}_{\ell }(7;3)=O\left( \frac{1}{\ell
^{3d-1}}\right) \text{ ,}
\end{equation}%
so that the dominant terms are again of the form $\mathcal{K}_{\ell }(q;1).$
\end{proof}

\chapter{On the Defect distribution}

In this chapter we refer to \cite{mau}, where  the high-energy limit distribution of the Defect of random hyperspherical harmonics is investigated. Indeed in the previous chapter quantitative Central Limit Theorems for the empirical measure of $z$-excursion sets has been shown but for the case $z= 0$.

We  find the exact asymptotic rate for the Defect variance and a  CLT for the case of the $d$-sphere, $d>5$. The CLT in the $2$-dimensional case has been already proved in \cite{Nonlin} whereas the variance has been investigated in \cite{def}.

The remaining cases ($d=3,4,5$) will be investigated in \cite{mau}, where moreover quantitative CLTs will be proved (work still in progress).

\section{Preliminaries}

Consider the sequence of random eigenfunctions $T_\ell$, $\ell\in \N$ \paref{Telle} on $\mathbb S^d$, $d\ge 2$. As in the previous chapter, the empirical measure  of excursion
sets \paref{excset} can be written, for $z\in \mathbb R$, as
\begin{equation}
S_\ell(z) := \int_{\mathbb S^d} 1(T_\ell(x) > z)\,dx\ ,
\end{equation}
where $1_{(z,+\infty)}$ is the indicator function of the interval $(z,+\infty)$.
The case $z\ne 0$ has been treated in \cite{Nonlin, maudom} (see also Chapter 5).

Now consider  the Defect, i.e. the difference between ``cold'' and ``warm'' regions
\begin{equation}\label{defect}
D_\ell :=\int_{\mathbb S^d} 1(T_\ell(x) > 0)\,dx -\int_{\mathbb S^d} 1(T_\ell(x) < 0)\,dx\ ;
\end{equation}
note that
$$
D_\ell =2S_\ell(0) - \mu_d\ ,
$$
$\mu_d$ denoting the hyperspherical volume \paref{ms}.

Recall that the Heaviside function is defined for $t\in \R$ as
$$
\mathcal H(t) := \begin{cases} 1\ &t>0\\
0\ &t=0\\
-1\ &t<0\ ,
\end{cases}
$$
thus \paref{defect} can be rewritten simply as
$$
D_\ell = \int_{\mathbb S^d} \mathcal H(T_\ell(x))\,dx\ .
$$

Note that exchanging expectation and integration on $\mathbb S^d$ we have
$$
\E[D_\ell]= \int_{\mathbb S^d}\E\left [ \mathcal H(T_\ell(x))  \right ]\,dx = 0\ ,
$$
since $\E\left [ \mathcal H(T_\ell(x))  \right ]=0$ for every $x$, by the symmetry of the Gaussian distribution.
%
%

\section{The Defect variance}

The proofs in this section are inspired by \cite{def}, where the case $d=2$ has been investigated.
\begin{lemma}\label{lemI}
For $\ell$ even we have
\begin{equation}
\Var(D_\ell)= \frac{4}{\pi}\mu_d \mu_{d-1} \int_{0}^{\frac{\pi}{2}}  \arcsin (G_{\ell;d}(\cos \theta)) (\sin \theta)^{d-1}\,d\theta\ ,
\end{equation}
where $\mu_d$ is the hyperspherical volume \paref{ms} and $G_{\ell;d}$ the $\ell$-th normalized Gegenbauer polynomial (Chapter \ref{background} or \cite{szego}).
\end{lemma}
\begin{proof}
The proof is indeed analogous to the proof of Lemma 4.1 in \cite{def}. First note that
$$\displaylines{
\Var(D_\ell) =  \int_{\cS^d} \int_{\cS^d}\E[\mathcal H(T_\ell(x))\mathcal H_\ell(T_\ell(y))]\,dx dy=\cr
=\mu_d \int_{\cS^d} \E[\mathcal H(T_\ell(x))\mathcal H(T_\ell(N))]\,dx\ ,
}$$ by the isotropic property of the random field $T_\ell$, where $N$ denotes some fixed point in $\mathbb S^d$. As explained in the proof of Lemma 4.1, we have
$$
 \E[\mathcal H(T_\ell(x))\mathcal H_\ell(T_\ell(N))]= \frac{2}{\pi} \arcsin (G_{\ell;d}(\cos \vartheta))\ ,
$$
where $\vartheta$ is the geodesic distance between $x$ and $N$.
Moreover evaluating the last integral in hyperspherical coordinates we get
\begin{equation}\label{vardef1}
\Var(D_\ell) = \mu_d \mu_{d-1} \int_0^\pi \frac{2}{\pi} \arcsin (G_{\ell;d}(\cos \vartheta))(\sin \vartheta)^{d-1} d\vartheta\ .
\end{equation}
For  $\ell$ even, we can hence write
\begin{equation}
\Var(D_\ell)=\frac{4}{\pi}\mu_d \mu_{d-1} \int_0^{\pi/2} \arcsin(G_{\ell;d}(\cos \vartheta))(\sin \vartheta)^{d-1} d\vartheta\ .
\end{equation}
\end{proof}
Note that, by \paref{vardef1}, if $\ell$ is odd, then $D_\ell=0$, therefore we can restrict ourselves to even $\ell$ only.

The main result of this section is the following.
\begin{theorem}\label{thdefvar}
The defect variance  is asymptotic to, as $\ell\to +\infty$ along even integers
$$
\Var(D_\ell) = \frac{C_d}{\ell^d}(1 + o(1))\ ,
$$
where $C_d$ is a strictly positive constant depending on $d$, that can be expressed by the formula
$$\displaylines{
C_d=\frac{4}{\pi}\mu_d \mu_{d-1}\int_{0}^{+\infty} \psi^{d-1}
\Big ( \arcsin \left (2^{\frac{d}{2} - 1}\left (\frac{d}2-1 \right)!\,
J_{\frac{d}{2}-1}(\psi )\psi ^{-\left( {\textstyle\frac{d}{2}}
-1\right)}\right)              +\cr
-  2^{\frac{d}{2} - 1}\left (\frac{d}2-1 \right)!\,
J_{\frac{d}{2}-1}(\psi )\psi ^{-\left( {\textstyle\frac{d}{2}}
-1\right)}          \Big )\,d\psi\ .
}$$
\end{theorem}
\begin{proof}
Here we are inspired by \cite[Proposition 4.2, Theorem 1.2]{def}.
Since \cite{szego}
$$
\int_0^{\pi/2} G_{\ell;d}(\cos \vartheta)(\sin \vartheta)^{d-1} d\vartheta=0\ ,
$$
from Lemma \ref{lemI} we can write
$$\displaylines{
\Var(D_\ell)=
\frac{4}{\pi}\mu_d \mu_{d-1} \int_{0}^{\frac{\pi}{2}} \left (\arcsin (G_{\ell;d}(\cos \theta))-
G_{\ell;d}(\cos \theta)\right )(\sin \theta)^{d-1}\,d\theta\ .
}$$
Let now
$$
\arcsin(t) - t = \sum_{k=1}^{+\infty} a_k t^{2k+1}
$$
be the Taylor expansion of the arcsine,
where
$$
a_k = \frac{(2k)!}{4^k (k!)^2 (2k+1)}\sim \frac{1}{2\sqrt{\pi}k^{3/2}}\ ,\quad k\to +\infty\ .
$$ Since the Taylor series is uniformly absolutely convergent,
we may write
$$
\Var(D_\ell)
=\frac{4}{\pi}\mu_d \mu_{d-1} \sum_{k=1}^{+\infty} a_k \int_{0}^{\frac{\pi}{2}}
G_{\ell;d}(\cos \theta)^{2k+1}(\sin \theta)^{d-1}\,d\theta\ .
$$
Now from Proposition \ref{varianza} we have
\begin{equation}\label{limG}
\lim_{\ell\to +\infty} \ell^d \int_{0}^{\frac{\pi}{2}}
G_{\ell;d}(\cos \theta)^{2k+1}(\sin \theta)^{d-1}\,d\theta= c_{2k+1;d}\ ,
\end{equation}
where
\begin{equation*}
c_{2k+1;d}=\left(2^{\frac{d}{2} - 1}\left (\frac{d}2-1 \right)!\right)^{2k+1}\int_{0}^{+\infty
}J_{\frac{d}{2}-1}(\psi )^{2k+1}\psi ^{-(2k+1)\left( {\textstyle\frac{d}{2}}%
-1\right) +d-1}d\psi\ .  \label{cq}
\end{equation*}%
Therefore we would expect that
\begin{equation}\label{guess}
\Var(D_\ell) \sim \frac{C_d}{\ell^d}\ ,
\end{equation}
where
\begin{equation}\label{guess1}
C_d = \frac{4}{\pi}\mu_d \mu_{d-1} \sum_{k=1}^{+\infty} a_kc_{2k+1;d}\ .
\end{equation}
Before proving \paref{guess} that is the statement of this theorem,
let us check that $C_d>0$, assuming \paref{guess1} true.
This is easy since  the r.h.s. of \paref{guess1} is a series of  nonnegative terms and from \cite[p. 217]{andrews} we have
\begin{equation}
c_{3;d} = \left(2^{\frac{d}{2} - 1}\left (\frac{d}2-1 \right)!\right)^3
\frac{3^{\frac{d}{2} -\frac32}}{2^{3\left (\frac{d}2-1 \right)-1}\sqrt \pi\,
\Gamma \left ( \frac{d}{2} -\frac12   \right )}>0\ .
\end{equation}
Moreover, assuming \paref{guess1} true, we get
$$\displaylines{
C_d = \cr
= \frac{4}{\pi}\mu_d \mu_{d-1} \sum_{k=1}^{+\infty} a_k
\left(2^{\frac{d}{2} - 1}\left (\frac{d}2-1 \right)!\right)^{2k+1}\int_{0}^{+\infty
}J_{\frac{d}{2}-1}(\psi )^{2k+1}\psi ^{-(2k+1)\left( {\textstyle\frac{d}{2}}
-1\right) +d-1}d\psi=\cr
=\frac{4}{\pi}\mu_d \mu_{d-1}\int_{0}^{+\infty} \sum_{k=1}^{+\infty} a_k
\left(2^{\frac{d}{2} - 1}\left (\frac{d}2-1 \right)!\,
J_{\frac{d}{2}-1}(\psi )\psi ^{-\left( {\textstyle\frac{d}{2}}
-1\right)}\right)^{2k+1}\psi^{d-1}\, d\psi=\cr
=\frac{4}{\pi}\mu_d \mu_{d-1}\int_{0}^{+\infty}
\Big( \arcsin \left (2^{\frac{d}{2} - 1}\left (\frac{d}2-1 \right)!\,
J_{\frac{d}{2}-1}(\psi )\psi ^{-\left( {\textstyle\frac{d}{2}}
-1\right)}\right)              +\cr
-  2^{\frac{d}{2} - 1}\left (\frac{d}2-1 \right)!\,
J_{\frac{d}{2}-1}(\psi )\psi^{-\left( {\textstyle\frac{d}{2}}
-1\right)}           \Big)\psi^{d-1}\,d\psi\ ,
}$$
which is the second statement of this theorem.
To justify the exchange
of the  integration and summation order, we consider the finite summation
$$
\sum_{k=1}^{m} a_k \int_0^{+\infty}
\left(2^{\frac{d}{2} - 1}\left (\frac{d}2-1 \right)!\,
J_{\frac{d}{2}-1}(\psi )\psi ^{-\left( {\textstyle\frac{d}{2}}
-1\right)}\right)^{2k+1}\psi^{d-1}\, d\psi
$$
using $a_k\sim \frac{c}{k^{3/2}}$ ($c>0$) and the asymptotic behavior of Bessel functions for large argument \cite{szego} to bound the contributions of tails, and take the limit $m\to +\infty$.

Let us now formally prove the asymptotic result for the variance \paref{guess}.
Note that
$$\displaylines{
 \sum_{k=m+1}^{+\infty} a_k \int_{0}^{\frac{\pi}{2}}
\left | G_{\ell;d}(\cos \theta) \right |^{2k+1}(\sin \theta)^{d-1}\,d\theta\le\cr
\le   \sum_{k=m+1}^{+\infty} a_k \int_{0}^{\frac{\pi}{2}}
\left | G_{\ell;d}(\cos \theta) \right |^{5}(\sin \theta)^{d-1}\,d\theta
\ll \frac{1}{\ell^d} \sum_{k=m+1}^{+\infty} \frac{1}{k^{3/2}}\ll  \frac{1}{\sqrt{m}\, \ell^d}\ .
}$$
Therefore we have for $m=m(\ell)$ to be chosen
$$
\Var(D_\ell)=\frac{4}{\pi}\mu_d \mu_{d-1} \sum_{k=1}^{m} a_k \int_{0}^{\frac{\pi}{2}}
G_{\ell;d}(\cos \theta)^{2k+1}(\sin \theta)^{d-1}\,d\theta + O\left (   \frac{1}{\sqrt{m}\, \ell^d} \right)\ .
$$
From \paref{limG}, we can write
$$
\Var(D_\ell) =C_{d,m} \cdot \frac{1}{\ell^d} + o(\ell^{-d}) + \frac{1}{\sqrt{m}\, \ell^d}\ ,
$$
where
$$
C_{d,m}:=\frac{4}{\pi}\mu_d \mu_{d-1} \sum_{k=1}^{m} a_k c_{2k+1;d}\ .
$$
Now since $C_{d,m}\to C_d$ as $m\to +\infty$, we can conclude.
\end{proof}
\section{The CLT}

In this section we prove a CLT for the Defect of random eigenfunctions on the $d$-sphere, $d\ne 3,4,5$, whose proof is inspired by the proof of Corollary 4.2 in \cite{Nonlin}.
\begin{theorem}\label{thDef}
As $\ell\to +\infty$ along even integers, we have
$$
\frac{D_\ell}{\sqrt{\Var(D_\ell)}}\mathop{\to}^{\mathcal L} Z\ ,
$$
where $Z\sim \mathcal N(0,1)$.
\end{theorem}

Since our aim for the next future is to find a quantitative CLT, we will first compute its chaotic expansion.

\subsection{Chaotic expansions}

Let us write the chaotic expansion \paref{chaos exp} for the Defect in the form
$$
D_\ell = \sum_{q=0}^{+\infty} \frac{J_q(D_\ell)}{q!}\int_{\mathbb S^d} H_q(T_\ell(x))\,dx\ .
$$
Recalling that $D_\ell=2S_\ell(0) -\mu_d$, let us find the chaotic expansion for $S_\ell(0)$. Note that
$
\E[S_\ell(0)]= \frac12 \mu_d$.
For $q\ge 1$
$$\displaylines{
J_q(S_\ell(0)) = \int_\R 1(z>0) (-1)^q \phi^{-1}(z) \frac{d^q}{dz^q} \phi(z) \phi(z)\,dz=\cr
= (-1)^q \int_0^{+\infty}  \frac{d^q}{dz^q} \phi(z) \,dz =\cr
=-(-1)^{(q-1)} \phi(z) \phi^{-1}(z) \frac{d^{(q-1)}}{dz^{(q-1)}} \phi(z) \Big |_0^{+\infty}=\cr
=-\phi(z) H_{q-1}(z)|_0^{+\infty}=\phi(0) H_{q-1}(0)=\begin{cases} 0\ & q\  \text{even}\\
\frac{(-1)^{\frac{q-1}{2}}}{\sqrt{2\pi} 2^{\frac{q-1}{2}} \left ( \frac{q-1}{2}\right )!}=
\frac{(-1)^{\frac{q-1}{2}}}{\sqrt{2\pi} (q-1)!!} & q\  \text{odd}\ .
\end{cases}
}$$
Therefore the Wiener-It\^o chaos decomposition for the Defect is
$$
D_\ell =  2 S_\ell(0)-\mu_d  =  \sum_{k=1}^{+\infty}
\sqrt{\frac{2}{\pi}}\frac{(-1)^{k}}{ (2k+1)! (2k)!!}\int_{\mathbb S^d} H_{2k+1}(T_\ell(x))\,dx\ ,
$$
with
$$
J_{2k}(D_\ell)=0\ , \qquad J_{2k+1}(D_\ell) =\sqrt{\frac{2}{\pi}}\frac{(-1)^{k}}{(2k)!!}\ .
$$

\subsection{Proof of Theorem \ref{thDef}}

Let $m\in \N, m\ge 2$ to be chosen later and set
$$
D_{\ell,m} :=\sum_{k=1}^{m-1}
\sqrt{\frac{2}{\pi}}\frac{(-1)^{k}}{ (2k+1)! (2k)!!}\int_{\mathbb S^d} H_{2k+1}(T_\ell(x))\,dx\ .
$$
Simple estimates give
$$\displaylines{
\E\left [  \left (  \frac{D_\ell}{\sqrt{\Var(D_\ell)}} -   \frac{D_{\ell,m}}{\sqrt{\Var(D_{\ell,m})}}    \right )^2 \right ]\le \cr
\le 2\E\left [  \left (     \frac{D_\ell}{\sqrt{\Var(D_\ell)}} -   \frac{D_{\ell,m}}{\sqrt{\Var(D_{\ell})}}    \right )^2 \right ] + 2\E\left [  \left (    \frac{D_{\ell,m}}{\sqrt{\Var(D_\ell)}} -   \frac{D_{\ell,m}}{\sqrt{\Var(D_{\ell,m})}}    \right )^2 \right ]\le \cr
\le  2\E\left [  \left (    \frac{D_\ell}{\sqrt{\Var(D_\ell)}} -   \frac{D_{\ell,m}}{\sqrt{\Var(D_{\ell})}}    \right )^2 \right ] + 2 \left (  \frac{\Var(D_{\ell,m})}{\Var(D_\ell)} +1 -  2\sqrt{ \frac{\Var(D_{\ell,m})}{\Var(D_{\ell})}}   \right )\ .
}$$
The first term to control is therefore
$$
\frac{D_\ell}{\sqrt{\Var[D_\ell]}} - \frac{D_{\ell,m}}{\sqrt{\Var[D_\ell]}}\ .
$$
We have, repeating the same argument as in \cite{Nonlin}
$$\displaylines{
\E\left [\left (\frac{D_\ell}{\sqrt{\Var[D_\ell]}} -
\frac{D_{\ell,m}}{\sqrt{\Var[D_\ell]}}  \right )^2 \right ]
=\frac{1}{\Var[D_\ell]}\E\left [\left (D_\ell -
D_{\ell,m}  \right )^2 \right ] =\cr
=\frac{1}{\Var[D_\ell]}\E\left [\left (\sum_{k=m}^{+\infty}
\sqrt{\frac{2}{\pi}}\frac{(-1)^{k}}{ (2k+1)! (2k)!!}\int_{\mathbb S^d} H_{2k+1}(T_\ell(x))\,dx \right )^2 \right ] =\cr
=\frac{1}{\Var[D_\ell]}\sum_{k=m}^{+\infty}\frac{2}{\pi}\frac{1}{ ((2k+1)! (2k)!!)^2}
\Var\left (
\int_{\mathbb S^d} H_{2k+1}(T_\ell(x))\,dx\right ) =\cr
=\frac{1}{\Var[D_\ell]}\left (\frac{1}{\ell^d}\sum_{q=m}^{+\infty}a_q c_{2q+1;d} + o(\ell^{-d})  \right )\le \cr
\le \frac{1}{\Var[D_\ell]}\left (\frac{1}{2\sqrt{\pi}}
\frac{1}{\ell^d}\sum_{q=m}^{+\infty}\frac{c_{5;d}}{q^{\frac32}} + o(\ell^{-d})  \right ) = \frac{1}{\Var[D_\ell]}\times O \left ( \frac{1}{\ell^d \sqrt{m}  }   \right )=O \left (
\frac{1}{\sqrt m}\right)\ ,
}$$
where we used Theorem \ref{thdefvar}.

Moreover for the second term we have, from \paref{limG} and Theorem \ref{thdefvar}
$$\displaylines{
\frac{\Var(D_{\ell,m})}{\Var(D_\ell)} +1 -  2\sqrt{ \frac{\Var(D_{\ell,m})}{\Var(D_{\ell})}} = 2 +O\left(  \frac{1}{\sqrt{m}} \right ) -2 \sqrt{1 +O\left(  \frac{1}{\sqrt{m}} \right )  } =\cr
= O\left(  \frac{1}{\sqrt{m}} \right )\ .
}$$
Putting thigns together we immediately get
$$
\E\left [  \left (  \frac{D_\ell}{\sqrt{\Var(D_\ell)}} -   \frac{D_{\ell,m}}{\sqrt{\Var(D_{\ell,m})}}    \right )^2 \right ]= O\left(  \frac{1}{\sqrt{m}} \right )\ .
$$
For every fixed $m$, Corollary \ref{cor1} and \S \ref{genpoly} gives,  if $d\ne 3,4,5$
$$
\frac{D_{\ell,m}}{\sqrt{\Var(D_{\ell,m})}}\to Z\ ,
$$
where $Z\sim \mathcal N(0,1)$, so that  since $m$ can be chosen arbitrarily large we must have
$$
\frac{D_{\ell}}{\sqrt{\Var(D_{\ell})}}\to Z\ .
$$

\section{Final remarks}

For the remaining cases $d=3,4,5$, we need to prove a CLT for the random variables $h_{\ell,3;d}$, as $\ell\to +\infty$. Indeed bounds on fourth order cumulants obtained in Theorem \ref{teo1} are not enough to garantee the convergence to the standard Gaussian distribution. In \cite{mau} we will investigate the exact rate for fourth order cumulants of $h_{\ell,3;d}$, which will allow to extend Theorem \ref{thDef} to dimensions $d=3,4,5$.

Moreover, we will prove quantitative CLTs for the Defect in the high-energy limit which should be of the form
$$
d_W\left (  \frac{D_{\ell}}{\sqrt{\Var(D_{\ell})}}     ,  Z    \right ) = O\left( \ell^{-1/4}  \right)\ ,
$$
where $d_W$ denotes Wasserstein distance \paref{prob distance}.

\chapter{Random length of level curves}

In this chapter, our aim is to investigate the high-energy behavior  for the length of level curves of Gaussian spherical eigenfunctions $T_\ell$, $\ell\in \N$ \paref{Telle} on $\mathbb S^2$.

\section{Preliminaries}

\subsection{Length{{} : mean and variance}}

Consider the total  length of level curves of random
eigenfunctions, i.e. the sequence the random variables $\{\Lc_{\ell}(z)\}_{\ell \in \N}$ given by, for $z\in \R$,
\begin{equation}\label{e:lengthS}
\Lc_\ell(z) := \text{length}(T_\ell^{-1}(z)).
\end{equation}
As already anticipated in the Introduction of this thesis, the expected value of $\Lc_{\ell}(z)$ was computed (see e.g. \cite{wigsurvey}) to be
\begin{equation}
\E[\mathcal L_\ell(z)]= 4\pi \frac{\e^{-z^2/2}}{2\sqrt{2}}\sqrt{\ell(\ell+1)},
\end{equation}
consistent to Yau's conjecture \cite{Yau, Yau2}.
The asymptotic behaviour of the variance $\Var(\Lc_{\ell}(z))$ of $\Lc_{\ell}(z)$ was
 resolved in \cite{wigsurvey, Wig} as follows.

For $z\ne 0$, we have
\begin{equation}
\Var(\mathcal L_\ell(z))\sim \ell\cdot C z^4\e^{-z^2} \ ,\quad \ell\to +\infty\ ,
\end{equation}
for some $C>0$. Moreover, I.Wigman computed the exact constant (private computations)
$$
C= \frac{\pi^2}{2}\ .
$$
For the nodal case ($z=0$) we have
\begin{equation*}
\Var(\mathcal L_\ell(0))\sim \frac{1}{32} \cdot \log \ell\ ,\quad \ell\to +\infty\ .
\end{equation*}
The order of magnitude of $\Var(\Lc_{\ell}(0))$ is
 smaller  than what would be a natural guess
(i.e., $ \ell $ as for the non-nodal case); this situation is due to some cancelation in the asymptotic expansion of the nodal variance (``obscure Berry's cancellation'' --  see \cite{wigsurvey, Wig}) and is similar to the {\em cancellation phenomenon} observed by Berry
in a different setting \cite{Berry2002}.

\subsection{Main result}

{{}

Our principal aim  is to study the asymptotic behaviour, as $\ell\to \infty$, of the distribution of the sequence of normalized random variables
\begin{equation}\label{e:culpS}
\widetilde{\Lc}_{\ell}(z) :=  \frac{\mathcal{L}_{\ell}(z) - \E[\mathcal{ L}_{\ell}(z)]}{\sqrt{\Var(\mathcal{L}_{\ell}(z) )}}, \quad \ell\geq 1.
\end{equation}

The following statement is the main result of this chapter.

\begin{theorem}
\label{main th S}  For $z\ne 0$ the sequence $\big\{\widetilde{\mathcal L}_\ell(z) : \ell\geq 1\big\}$ converges  in distribution to a standard Gaussian r.v. $Z$. In particular
\begin{equation}
\lim_{\ell\to +\infty} d(\widetilde{\mathcal L}_\ell(z), Z) = 0\ ,
\end{equation}
where $d$ denotes either the Kolmogorov distance \paref{prob distance}, or an arbitrary distance
metrizing the weak convergence on $\mathscr P$ the space of all probability measures on $\R$ (see Chapter \ref{background}).
\end{theorem}
}

{{}

\subsection{ Wiener chaos and Berry's cancellation}\label{ss:berryintroS}

The proof of our result rely on a pervasive use of  Wiener-It\^o chaotic expansions (see Chapter \ref{background} e.g.) and the reader is referred
to the two monographs \cite{noupebook, P-T} for an exhaustive discussion.


According to \eqref{Telle}, the Gaussian spherical eigenfunctions considered in this work are built starting from a family of i.i.d. Gaussian r.v.'s $\{a_{\ell,m} : \ell\ge 1, m=1,\dots, 2\ell+1\}$, defined on some probability space $(\Omega, \mathscr{F}, \mathbb{P})$ and verifying the following properties:
$$\E[a_{\ell,m}]=0\ ,\quad \E[a_{\ell,m} a_{\ell',m'}]=\frac{4\pi}{2\ell+1}\delta_{m}^{m'} \delta_{\ell}^{\ell'}\ .$$
 We define ${\bf A}$ to be the closure in $L^2(\mathbb{P})$ of all real finite linear combinations of $\{a_{\ell,m} : \ell\ge 1, m=1,\dots, 2\ell+1\}$.  ${\bf A}$ is a real centered Gaussian space (that is, a linear space of jointly Gaussian centered real-valued random variables, that is stable under convergence in $L^2(\mathbb{P})$) (compare to Chapter \ref{background}).

\begin{definition}\label{d:chaosS}{\rm For every $q=0,1,2,...$ the $q$th {\it Wiener chaos} associated with ${\bf A}$ (compare to Chapter \ref{background}), written $C_q$, is the closure in $L^2(\mathbb{P})$ of all real finite linear combinations of random variables with the form
$$
H_{p_1}(\xi_1)H_{p_2}(\xi_2)\cdots H_{p_k}(\xi_k),
$$
where the integers $p_1,...,p_k \geq 0$ verify $p_1+\cdot+p_k = q$, and $(\xi_1,...,\xi_k)$ is a real centered Gaussian vector with identity covariance matrix extracted from ${\bf A}$ (note that, in particular, $C_0 = \mathbb{R}$).}
\end{definition}
 $C_q \,\bot\, C_m$ (where the orthogonality holds in the sense of $L^2(\mathbb{P})$) for every $q\neq m$, and moreover
\begin{equation}\label{e:chaosS}
L^2(\Omega, \sigma({\bf A}), \mathbb{P}) = \bigoplus_{q=0}^\infty C_q,
\end{equation}
that is: each real-valued functional $F$ of ${\bf A}$ can be (uniquely) represented in the form
\begin{equation}\label{e:chaos2S}
F = \sum_{q=0}^\infty {\rm proj}(F \, | \, C_q),
\end{equation}
where ${\rm proj}(\bullet \, | \, C_q)$ stands for the projection operator onto $C_q$, and the series converges in $L^2(\mathbb{P})$. Plainly, ${\rm proj}(F \, | \, C_0) = \E[ F]$.

\

Now recall the definition of $T_\ell$ given in \eqref{Telle}: the following elementary statement  shows that the Gaussian field
$$
\left\{ T_\ell(\theta),\,  \frac{\partial}{\partial \theta_1} T_\ell (\theta),\,  \frac{\partial}{\partial \theta_2} T_\ell (\theta) : \theta =(\theta_1,\theta_2)\in \mathbb{S}^2\right\}
$$
is a subset of ${\bf A}$, for every $\ell\in \N$.

\begin{prop} \label{p:fieldS}Fix $\ell\in \N$, let the above notation and conventions prevail. Then, for every $j=1,2$ one has that
\begin{equation}\label{e:partialS}
\partial_j T_\ell(\theta) := \frac{\partial}{\partial \theta_j} T_\ell(\theta) = \sum_{m} a_{\ell,m} \frac{\partial}{\partial \theta_j} Y_{\ell,m}(\theta),
\end{equation}
and therefore $T_\ell (\theta), \, \partial_1 T_\ell (\theta), \, \partial_2 T_\ell (\theta) \in {\bf A}$, for every $\theta\in \mathbb{S}^2$. Moreover, for every fixed $\theta\in \mathbb{S}^2$, one has that $T_\ell (\theta), \, \partial_1 T_\ell (\theta), \, \partial_2 T_\ell (\theta)$ are stochastically independent (see e.g. \cite{Wig}).
\end{prop}
We shall often use the fact that
$$\displaylines{
\Var[\partial_j T_\ell (\theta)] = \frac{\ell(\ell+1)}{2}\ ,
}$$
and, accordingly, for $\theta=(\theta_1, \theta_2)\in \mathbb{S}^2$ and $j=1,2$, we will denote by $\partial_j \widetilde T_\ell (\theta)$ the normalized derivative
\begin{equation}\label{e:normaS}
\partial_j \widetilde T_\ell (\theta) :=  \sqrt{\frac{2}{\ell(\ell+1)}} \frac{\partial}{\partial \theta_j}  T_\ell (\theta) \ .
\end{equation}
The next statement gathers together some of the main technical achievements of the present chapter. It shows in particular that the already evoked `arithmetic Berry cancellation phenomenon' (see \cite{Berry2002}) -- according to which the variance of the nodal length $\mathcal{L}_\ell:=\mathcal L_\ell(0)$  (as defined in \eqref{e:lengthS}) has asymptotically the same order as $\log \ell$ (rather than the expected order $\ell$) -- \emph{should} be a consequence of the following:
\begin{itemize}

\item[\bf (i)] The projection of $\mathcal{L}_\ell(z)$ on the second Wiener chaos $C_2$ is {\it exactly equal to zero} for every $\ell\in \N$ if and only if $z=0$ (and so holds for the projection of $\mathcal{L}_\ell(0)$ onto any chaos of odd order $q\geq 3$).
\item[\bf (ii)] For $z\ne 0$, the variance of $ {\rm proj}(\mathcal{L}_{\ell}(z) \, | C_2)$ has the order $ \ell$, as $\ell\to +\infty$, and one has moreover that
$$
\Var(\Lc_{\ell}(z)) \sim \Var\left( {\rm proj}(\mathcal{L}_{\ell}(z) \, | C_2)\right)\ .
$$

\end{itemize}

\subsection{Plan}

The rest of the chapter is organized as follows:  \S 7.2 contains a study of the chaotic
representation of nodal lengths, \S 7.3 focuses on the projection of nodal lengths on the second
Wiener chaos, whereas \S 7.4 contains a proof of our main result.
}

\section{Chaotic expansions}\label{expanS}

The aim of this section is to derive an explicit expression for each projection of the type ${\rm proj}(\Lc_\ell(z) \, | \, C_q)$, $q\geq 1$. In order to accomplish this task, we first focus on a family of auxiliary random variables $\{\Lc_\ell^\varepsilon (z) : \varepsilon > 0 \}$ that approximate $\Lc_\ell(z)$ in the sense of the $L^2(\P)$-norm.

\subsection{Preliminary results}

$\bullet$ For each $z\in \R$,
 $\mathcal L_n(z)$
is
 the $\omega$-a.s. limit, for $\varepsilon \to 0$, of the $\varepsilon$-approximating r.v.
\begin{equation}\label{napproxS}
\mathcal L_\ell^\varepsilon(z,\omega) := \frac{1}{2\varepsilon}
\int_{\mathbb S^2} 1_{[z-\varepsilon, z+\varepsilon]}(T_\ell(\theta,\omega))\|\nabla T_\ell (\theta,\omega) \|\,d\theta\ ,
\end{equation}
where
$$
\nabla T_\ell := (\partial_1  T_\ell,\partial_2
T_\ell),
$$
$\| \cdot \|$ is the norm in $\mathbb R^2$, and we have used the notation \eqref{e:partialS}.

We know that

$\bullet$ $\mathcal L_\ell (z)\in L^2(\P)$, for every $z\in \R$ (\cite{wigsurvey, Wig})\ .

\noindent Now we want to prove that $\mathcal L_\ell (z)$ is the $L^2(\Omega)$-limit, for
$\varepsilon \to 0$
of $\mathcal L_\ell^\varepsilon(z)$.
Remark first that analogous arguments as in \cite{Wig}, prove that
the function $z\mapsto \E[\mathcal L_\ell (z)^2]$ is continuous (further details will appear in \cite{mistosfera}).

\begin{lemma}\label{approxS}
It holds that
$$
\lim_{\varepsilon\to 0} \E[ (\mathcal L_\ell^\varepsilon (z) - \mathcal L_\ell(z) )^2  ] = 0\ .
$$

\end{lemma}
\begin{proof}
Since $\mathcal L^\varepsilon_\ell (z)\to_{\varepsilon} \mathcal L_\ell (z)$ a.s., it is enough to show that
$$
\E[\mathcal L^\varepsilon_\ell (z)^2]\to E[\mathcal L_\ell (z)^2]\ ,
$$
and then use the well-known fact that convergence a.s. plus convergence of the norms
implies convergence in mean square \cite[Proposition 3.39]{cannarsa}.

By Fatou's Lemma (for the first inequality) we have
$$\displaylines{
\E[\mathcal L_\ell (z)^2]\le \liminf_\varepsilon \E[\mathcal L^\varepsilon_\ell (z)^2]\le
\limsup_\varepsilon \E[\mathcal L^\varepsilon_\ell (z)^2]\mathop{=}^{*}\cr
\mathop{=}^{*}\limsup_\varepsilon \E\left [\left (
\int_{\mathbb \R} \mathcal L_\ell (u)\frac{1}{2\varepsilon}1_{[-\varepsilon, \varepsilon]}
(u-z)\,du\right) ^2 \right ]\ ,
}$$
where to establish the equality $\mathop{=}^{*}$
we have used the co-area formula \cite[(7.14.13)]{adlertaylor}, which in our case gives
$$\displaylines{
\mathcal L_\ell^\varepsilon(z) =\frac{1}{2\varepsilon}
\int_{\mathbb S^2} 1_{[z-\varepsilon, z+\varepsilon]}(T_\ell (\theta))\|\nabla T_\ell (\theta) \|\,d\theta=\cr
=\int_\R du \int_{T_\ell ^{-1}(u)}\frac{1}{2\varepsilon}
 1_{[-\varepsilon, \varepsilon]}(u-z)\,d\theta
 =\frac{1}{2\varepsilon}\int_\R \mathcal L_\ell (u)1_{[-\varepsilon, \varepsilon]}(u-z)\,du\ .
}$$
Now by Jensen inequality we find that
$$\displaylines{
\limsup_\varepsilon \E\left [\left (
\int_{\mathbb \R} \mathcal L_\ell (u)\frac{1}{2\varepsilon}1_{[-\varepsilon, \varepsilon]}
(u-z)\,du\right) ^2 \right ]\le \cr
\le \limsup_\varepsilon
\int_{\mathbb \R} \E\left [\mathcal L_\ell (u)^2\right ]\frac{1}{2\varepsilon}1_{[-\varepsilon, \varepsilon]}
(u-z)\,du =\cr
=\E\left [\mathcal L_\ell (z)^2\right ]\ ,
}$$
the last step following by continuity of the map $u\mapsto \E[\mathcal L_\ell (u)^2]$.
\end{proof}
To conclude the section, we observe that the previous result suggests that the random variable $\mathcal L_\ell(z)$  can be formally written as
\begin{equation}\label{formalS}
\mathcal L_\ell(z) = \int_{ \mathbb S^2} \delta_z(T_\ell(\theta))\| \nabla T_\ell \|
\,d\theta\ ,
\end{equation}
where $\delta_z$ denotes the Dirac mass in $z$.

\subsection{The chaotic expansion for $\mathcal L_\ell(z)$}

In view of the convention \eqref{e:normaS}, throughout the section we will rewrite \paref{napproxS} as
\begin{equation}\label{formal2S}
\mathcal L_\ell^\eps (z)= \sqrt{\frac{\ell(\ell+1)}{2}}\frac{1}{2\eps}\int_{\mathbb S^2}
1_{[z-\eps,z+ \eps]}(T_\ell(\theta)) \sqrt{\partial_1 \widetilde
T_\ell(\theta)^2+\partial_2 \widetilde T_\ell(\theta)^2}\,d\theta\ .
\end{equation}
We also need to introduce two collection of coefficients
$\{\alpha_{n,m} : n,m\geq 1\}$ and $\{\beta_{l}(z) : l\geq 0\}$, that are connected to the (formal) Hermite expansions of the norm $\| \cdot
\|$ in $\R^2$ and the Dirac mass $ \delta_z(\cdot)$ respectively.
These are given by
\begin{equation}\label{e:beta}
\beta_{l}(z):= \phi(z)H_{l}(z)\ ,
\end{equation}
where $\phi$ is the standard Gaussian pdf, $H_{l}$ denotes the $l$-th Hermite polynomial \paref{hermite} and $\alpha_{n,m}=0$ but for the case $n,m$ even
\begin{equation}\label{e:alpha}
\alpha_{2n,2m}=\sqrt{\frac{\pi}{2}}\frac{(2n)!(2m)!}{n!
m!}\frac{1}{2^{n+m}} p_{n+m}\left (\frac14 \right)\ ,
\end{equation}
where for $N=0, 1, 2, \dots $ and $x\in \R$
\begin{equation}\label{pN}
p_{N}(x) :=\sum_{j=0}^{N}(-1)^{j}\ \ (-1)^{N}{N
\choose j}\ \ \frac{(2j+1)!}{(j!)^2} x^j \ ,
\end{equation}
$\frac{(2j+1)!}{(j!)^2}$ being the so-called {\it swinging factorial}
restricted to odd indices.
{{}
\begin{prop}[\bf Chaotic expansion of $\mathcal L_\ell(z)$]\label{teoexpS} For every $\ell\in \N$ and $q\geq 2$,
\begin{eqnarray}\label{e:ppS}
&&{\rm proj}(\mathcal L_\ell(z)\, | \, C_{q}) \\
&&= \sqrt{\frac{\ell(\ell+1)}{2}}\sum_{u=0}^{q}\sum_{k=0}^{u}
\frac{\alpha _{k,u-k}\beta _{q-u}(z)
}{(k)!(u-k)!(q-u)!} \!\!\int_{\mathbb S^2}\!\! H_{q-u}(T_\ell (\theta))
H_{k}(\partial_1 \widetilde T_\ell (\theta))H_{u-k}(\partial_2
\widetilde T_\ell (\theta))\,d\theta.\notag
\end{eqnarray}
As a consequence, one has the representation
\begin{eqnarray}\label{chaosexpS}
\mathcal L_\ell(z) &=& \E\mathcal L_\ell(z) + \sqrt{\frac{\ell(\ell+1)}{2}}\sum_{q=2}^{+\infty}\sum_{u=0}^{q}\sum_{k=0}^{u}
\frac{\alpha _{k,u-k}\beta _{q-u}(z)
}{(k)!(u-k)!(q-u)!} \!\!\times\\
&&\hspace{4.5cm} \times \int_{\mathbb S^2}\!\! H_{q-u}(T_\ell (\theta))
H_{k}(\partial_1 \widetilde T_\ell (\theta))H_{u-k}(\partial_2
\widetilde T_\ell (\theta))\,d\theta,\notag
\end{eqnarray}
where the series converges in $L^2(\P)$.
\end{prop}
}
\noindent\begin{proof}[Proof of Proposition \ref{teoexpS}] The proof is divided into four steps.

\medskip

\noindent\underline{\it Step 1: dealing with indicators.} .  We start by expanding the function $\frac{1}{2\eps}{1}_{[z-\eps ,z+\eps ]}(\cdot)$ into Hermite polynomials, as defined in \S \ref{ss:berryintroS}:

\begin{equation*}
\frac{1}{2\eps}{1}_{[z-\eps ,z+\eps]}(\cdot
)=\sum_{l=0}^{+\infty }\frac{1}{l!}\beta_l^\eps (z)\,H_{l}(\cdot )\ .
\end{equation*}%
 {{}One has that $\beta_0^\eps(z) = \frac{1}{2\eps} \int_{z-\eps}^{z+\eps} \phi(x)\,dx$, and, for $l\geq 1$}
$$\displaylines{
\beta_l^\eps(z) = \frac{1}{2\eps} \int_{z-\eps}^{z+\eps} \phi(x) H_l(x)    \,dx=
\frac{1}{2\eps} \int_{z-\eps}^{z+\eps} \phi(x) (-1)^l \phi^{-1}(x) \frac{d^l}{dx^l} \phi(x)  \,dx =\cr
=(-1)^l\frac{1}{2\eps} \int_{z-\eps}^{z+\eps}  \frac{d^l}{dx^l} \phi(x)  \,dx\ .
}$$
Using the notation \eqref{e:beta}, we have that
$$
\lim_{\eps\to 0} \beta_0^\eps(z) = \phi(z) = \phi(z) H_0(z)=\beta_0(z)\ ,
$$
and for all $l\geq 1$,
\begin{equation}\label{e:satS}
\lim_{\eps\to 0} \beta_{l}^\eps(z) =(-1)^l \frac{d^l}{dx^l} \phi(x)_{|_{x=z}} =
\phi(z) H_l(z)= \beta_{l}(z) \\ .
\end{equation}

\noindent\underline{\it Step 2: dealing with the Euclidean norm.} {{} Fix $x\in \mathbb{S}^2$, and recall that, according to Proposition \ref{p:fieldS}, the vector $$
\nabla \widetilde T_\ell := (\partial_1 \widetilde T_\ell, \partial_2 \widetilde T_\ell)\ ,
$$ is composed of centered independent Gaussian random variables with variance one. Now, since the random variable $\| \nabla \widetilde T_\ell (\theta) \|$ is square-integrable, it can be expanded into the following  infinite series of Hermite polynomials:
\begin{equation*}
\| \nabla \widetilde T_\ell (\theta) \| = \sum_{u=0}^{+\infty}
\sum_{m=0}^{u} \frac{\alpha_{u,u-m}}{u! (u-m)!} H_u(\partial_1 \widetilde T_\ell (\theta)) H_{u-m}(\partial_2 \widetilde T_\ell (\theta)),
\end{equation*}
where
\begin{equation}
\alpha_{n,n-m}=\frac{1}{2\pi} \int_{\R^2} \sqrt{y^2 + z^2} H_{n}(y) H_{n-m}(z)
\mathrm{e}^{-\frac{y^2+z^2}{2}}\,dy dz\ .
\end{equation}
Our aim is to compute $\alpha_{n,n-m}$ as explicitly as possible. }First of all, wet observe that, if $n$ or $n-m$ is odd, then the above integral
vanishes {{} (since the two mappings $z\mapsto \sqrt{y^2 + z^2}$ and $y\mapsto \sqrt{y^2 + z^2}$ are even)}. It follows therefore that
\begin{equation*}
\| \nabla \widetilde T_\ell (\theta) \| = \sum_{n=0}^{+\infty}
\sum_{m=0}^{n} \frac{\alpha_{2n,2n-2m}}{(2n)! (2n-2m)!} H_{2n}(\partial_1 \widetilde T_\ell (\theta)) H_{2n-2m}(\partial_2 \widetilde T_\ell (\theta)).
\end{equation*}
We are therefore left with the task of showing that the integrals
\begin{equation}  \label{coeff}
\alpha_{2n,2n-2m}=\frac{1}{2\pi} \int_{\R^2} \sqrt{y^2 + z^2} H_{2n}(y)
H_{2n-2m}(z) \mathrm{e}^{-\frac{y^2+z^2}{2}}\,dy dz,
\end{equation}
where $n\ge 0$ and $m=0, \dots, n$, can be evaluated according to \eqref{e:alpha}. One elegant way for dealing with this task is to use the following Hermite polynomial expansion
\begin{equation}
\mathrm{e}^{\lambda y - \frac{\lambda^2}{2}} = \sum_{a=0}^{+\infty} H_a(y)
\frac{\lambda^a}{a!}, \quad {{} \lambda \in \R}.
\end{equation}
We start by considering the integral
\begin{equation*}
\frac{1}{2\pi} \int_{\mathbb R^2} \sqrt{y^2 + z^2} \e^{\lambda y - \frac{\lambda^2}{2}} \e^{\mu z - \frac{\mu^2}{2}}\e^{-\frac{y^2+z^2}{2}}\,dy dz
=\frac{1}{2\pi} \int_{\mathbb R^2} \sqrt{y^2 + z^2} \e^{-\frac{(y-\lambda)^2+(z-\mu)^2}{2}}\,dy dz\ .
\end{equation*}
This integral coincides with the expected value of the random variable $W:=\sqrt{%
Y^2 + Z^2}$ where $(Y,Z)$ is a vector of independent Gaussian random variables with variance one and mean $\lambda $ and $\mu$, respectively. Since $W^2=Y^2 + Z^2$ has a non central $%
\chi^2$ distribution (more precisely, $Y^2 + Z^2\sim \chi^2(2,\lambda^2 +
\mu^2)$) it is easily checked that the density of $W$ is given by
\begin{equation}
f_W(t) = \sum_{j=0}^{+\infty} \mathrm{e}^{-(\lambda^2 + \mu^2)/2}\frac{%
((\lambda^2 + \mu^2)/2)^j }{j!}f_{2+2j}(t^2)\, 2t \,{\mathbb{I}}_{\{t> 0\}}\
,
\end{equation}
where $f_{2+2j}$ is the density function of a $\chi^2_{2+2j}$ random variable. The expected
value of $W$ is therefore
\begin{equation}
2\sum_{j=0}^{+\infty} \mathrm{e}^{-(\lambda^2 + \mu^2)/2}\frac{((\lambda^2 +
\mu^2)/2)^j }{j!}\int_{0}^{+\infty} f_{2+2j}(t^2)\, t^2\,dt\ .
\end{equation}
From the definition of $f_{2+2j}$ we have
\begin{eqnarray*}
\int_{0}^{+\infty} f_{2+2j}(t^2)\, t^2\,dt &=&
\frac{1}{2^{1+j}\Gamma(1+j)} \int_{0}^{+\infty} t^{2j+2}\e^{-t^2/2}\,dt \\
&=& \frac{\prod_{i=1}^{1+j} (2i-1)\sqrt{\frac{\pi}{2}}}{2^{1+j}\Gamma(1+j)}\ .
\end{eqnarray*}
As a consequence,
\begin{eqnarray*}
&& \frac{1}{2\pi} \int_{\mathbb R^2} \sqrt{y^2 + z^2} \e^{-\frac{(y-\lambda)^2+(z-\mu)^2}{2}}\,dy dz
\\ &&=2\e^{-(\lambda^2 + \mu^2)/2}\sum_{j=0}^{+\infty} \frac{((\lambda^2
+ \mu^2)/2)^j }{j!}\frac{\prod_{i=1}^{1+j} (2i-1)\sqrt{\frac{\pi}{2}}}{2^{1+j}\Gamma(1+j)} =: F(\lambda, \mu)\ .
\end{eqnarray*}
We can develop the function $F$ as follows:
\begin{equation*}
\displaylines{
F(\lambda, \mu) =
2\sum_{a=0}^{+\infty} \frac{(-1)^a\lambda^{2a}}{2^a a!}
\sum_{b=0}^{+\infty} \frac{(-1)^b\mu^{2b}}{2^b b!}
\sum_{j=0}^{+\infty} \frac{1 }{j!} \sum_{l=0}^{j} {j \choose l} \lambda^{2l}
\mu^{2j-2l}
\frac{\prod_{i=1}^{1+j} (2i-1)\sqrt{\frac{\pi}{2}}}{2^{1+2j}\Gamma(1+j)} = \cr
= \sum_{a,b=0}^{+\infty} \frac{(-1)^a}{2^a a!}
 \frac{(-1)^b}{2^b b!}\sum_{j=0}^{+\infty}
 \frac{\prod_{i=1}^{1+j} (2i-1)\sqrt{\frac{\pi}{2}}}{j!2^{2j}\Gamma(1+j)}
 \sum_{l=0}^{j} {j \choose l} \lambda^{2l+2a}
\mu^{2j+2b-2l}\ .
}
\end{equation*}
On the other hand
\begin{eqnarray*}
{{} F(\lambda, \mu) }&=&\frac{1}{2\pi} \int_{\mathbb R^2} \sqrt{y^2 + z^2} \e^{\lambda y - \frac{\lambda^2}{2}} \e^{\mu z - \frac{\mu^2}{2}}\e^{-\frac{y^2+z^2}{2}}\,dy dz\\
&& = \frac{1}{2\pi} \int_{\mathbb R^2} \sqrt{y^2 + z^2} \sum_{a=0}^{+\infty} H_a(y) \frac{\lambda^a}{a!} \sum_{b=0}^{+\infty} H_b(z) \frac{\mu^b}{b!}\e^{-\frac{y^2+z^2}{2}}\,dy dz \\
&& =  \sum_{a,b=0}^{+\infty} \left( \frac{1}{a!b!2\pi} \int_{\mathbb R^2} \sqrt{y^2 + z^2}  H_a(y) H_b(z) \e^{-\frac{y^2+z^2}{2}}\,dy dz \right)\lambda^a \mu^b.
\end{eqnarray*}
By the same reasoning as above, if $a$ or $b$ is odd, then the
integral coefficient in the previous expression must be zero. Setting $n:=l+a$ and $%
m:=j+b-l$, we also have that
\begin{eqnarray*}
&&\sum_{a,b=0}^{+\infty} \frac{(-1)^a}{2^a a!}
 \frac{(-1)^b}{2^b b!}\sum_{j=0}^{+\infty} \frac{\prod_{i=1}^{1+j} (2i-1)\sqrt{\frac{\pi}{2}}}{j!2^{2j}\Gamma(1+j)}
 \sum_{l=0}^{j} {j \choose l} \lambda^{2l+2a}
\mu^{2j+2b-2l} \\
&&=\sum_{n,m}^{} \sum_{j}^{}\frac{\prod_{i=1}^{1+j} (2i-1)\sqrt{\frac{\pi}{2}}}{j!2^{2j}\Gamma(1+j)}  \sum_{l=0}^{j}\frac{(-1)^{(n-l)}}{2^{n-l} {(n-l)}!}
 \frac{(-1)^{m+l-j}}{2^{m+l-j} {(m+l-j)}!}
  {j \choose l} \lambda^{2n}
\mu^{2m}\ .
\end{eqnarray*}
Thus we obtain
\begin{eqnarray*}
\alpha_{2n,2m} &=& \frac{1}{2\pi} \int_{\mathbb R^2} \sqrt{y^2 + z^2}  H_{2n}(y) H_{2m}(z) \e^{-\frac{y^2+z^2}{2}}\,dy dz \\ &=& (2n)!(2m)!\frac{(-1)^{m+n}}{2^{n+m}}\sum_{j}^{}(-1)^{j}\frac{\prod_{i=1}^{1+j} (2i-1)
\sqrt{\frac{\pi}{2}}}{2^j j!\Gamma(1+j)}  \sum_{l=0}^{j}\frac{{j \choose l}}{ {(n-l)}! {(m+l-j)}!}
  \ .
\end{eqnarray*}
Representation \eqref{e:alpha} now follows from the computations:
\begin{eqnarray*}
\alpha_{2n,2m}&=&\frac{1}{2\pi} \int_{\mathbb R^2} \sqrt{y^2 + z^2}  H_{2n}(y) H_{2m}(z) \e^{-\frac{y^2+z^2}{2}}\,dy dz \\
&=& (2n)!(2m)!\frac{(-1)^{m+n}}{2^{n+m}}\sum_{j}^{}(-1)^{j}
\frac{\prod_{i=1}^{1+j} (2i-1)\sqrt{\frac{\pi}{2}}}{2^j j!\Gamma(1+j)}
\sum_{l=0}^{j}\frac{{j \choose l}}{ {(n-l)}! {(m+l-j)}!} \\
&=& (2n)!(2m)!\frac{(-1)^{m+n}}{2^{n+m}}
\sum_{j}^{}(-1)^{j}\frac{(2j+1)!!\sqrt{\frac{\pi}{2}}}{2^j(j!)^2}
 \sum_{l=0}^{j}\frac{{j \choose l}}{ {(n-l)}! {(m+l-j)}!}\\
 &=&\frac{(2n)!(2m)!}{n! m!}\frac{(-1)^{m+n}}{2^{n+m}}
\sum_{j=0}^{n+m}(-1)^{j}\frac{(2j+1)!!\sqrt{\frac{\pi}{2}}}{2^j j!}
 \sum_{l=0}^{j}{ n\choose l}{ m\choose j-l} \\
&=&\sqrt{\frac{\pi}{2}}\frac{(2n)!(2m)!}{n! m!}\frac{(-1)^{m+n}}{2^{n+m}}
\sum_{j=0}^{n+m}(-1)^{j}\frac{(2j+1)!!}{2^j j!}
 { n+m\choose j} \\
&=&\sqrt{\frac{\pi}{2}}\frac{(2n)!(2m)!}{n! m!}\frac{(-1)^{m+n}}{2^{n+m}}
\sum_{j=0}^{n+m}(-1)^{j}\frac{(2(j+1))!}{2^{j+1} 2^j j! (j+1)!}
 { n+m\choose j} \\
&=&\sqrt{\frac{\pi}{2}}\frac{(2n)!(2m)!}{n! m!}\frac{(-1)^{m+n}}{2^{n+m}}
\sum_{j=0}^{n+m}(-1)^{j}\frac{(2j+1)!}{2^{2j} (j!)^2}
 { n+m\choose j}\ .
\end{eqnarray*}

\medskip
{{}
\noindent\underline{\it Step 3: letting $\eps \to 0$.} In view of Definition \ref{d:chaosS}, the computations at Step 1 and Step 2 (together with the fact that the three random variables $T_\ell (\theta),\, \partial_1 \widetilde T_\ell (\theta)$ and $\partial_2 \widetilde T_\ell (\theta)$ are stochastically independent ) show that, for fixed $\theta\in \mathbb{S}^2$, the projection of the random variable
$$
\frac{1}{2\eps} 1_{[z-\eps, z+\eps]}(T_\ell (\theta)) \sqrt{\partial_1 \widetilde
T_\ell (\theta)^2+\partial_2 \widetilde T_\ell (\theta)^2}
$$
on the chaos $C_q$ equals
$$
\sum_{u=0}^{q}\sum_{m=0}^{u}
\frac{\alpha _{m,u-m}\beta^\eps _{q-u}(z)
}{(m)!(u-m)!(q-u)!}  H_{q-u}(T_\ell (\theta))
H_{m}(\partial_1 \widetilde T_\ell (\theta))H_{u-m}(\partial_2
\widetilde T_\ell (\theta))\ .
$$
Since $\int_{\mathbb{S}^2}\,dx<\infty$, standard arguments based on Jensen inequality and dominated convergence yield that, for every $q\geq 1$,
\begin{eqnarray*}
&&{\rm proj}(\mathcal L^\eps_\ell(z)\, | \, C_{q}) \\
&&= \sqrt{\frac{\ell(\ell+1)}{2}}\sum_{u=0}^{q}\sum_{m=0}^{u}
\frac{\alpha _{m,u-m}\beta^\eps _{q-u}(z)
}{(m)!(u-m)!(q-u)!}\times \\
&& \hspace{4cm} \times\!\!\int_{\mathbb S^2}\!\! H_{q-u}(T_\ell (\theta))
H_{m}(\partial_1 \widetilde T_\ell (\theta))H_{u-m}(\partial_2
\widetilde T_\ell (\theta))\,d\theta.\notag
\end{eqnarray*}
One has that, as $\eps\to 0$, ${\rm proj}(\mathcal L^\eps_\ell (z)\, | \, C_{q})$ converges necessarily to ${\rm proj}(\mathcal L_\ell(z)\, | \, C_{q})$ in probability. Using \eqref{e:satS}, we deduce from this fact that representation \eqref{e:ppS} is valid for every $q\geq 1$.
}
\end{proof}
\begin{remark}\label{Gilles Becker}\rm The coefficients $\alpha_{2n,2m}$ can be found also first using polar coordinates and then the explicit expression for Hermite polynomials \cite{szego}. Briefly,
\begin{equation*}
\displaylines{
\frac{1}{2\pi} \int_{\mathbb R^2} \sqrt{y^2 + z^2}  H_{2n}(y) H_{2m}(z) \e^{-\frac{y^2+z^2}{2}}\,dy dz = \cr
= \frac{1}{2\pi} \int_{0}^{2\pi} \int_{0}^{+\infty} \rho^2  H_{2n}(\rho \cos \vartheta) H_{2m}(\rho \sin \vartheta) \e^{-\frac{\rho^2}{2}}\,d \rho d \vartheta = \cr
=  \frac{(2n)!(2m)!}{2\pi} \sum_{a=0}^{n} \frac{(-1)^a}{2^aa!(2n - 2a)!}\sum_{b=0}^{m} \frac{(-1)^b}{2^bb!(2m - 2b)!}\times \cr
\times \int_{0}^{2\pi}\cos \vartheta^{2n - 2a}  \sin \vartheta^{2m - 2b} d \vartheta \int_{0}^{+\infty} \rho^{2+2n - 2a+2m - 2b} \e^{-\frac{\rho^2}{2}} \,d \rho\ .
}
\end{equation*}
It remains to solve the previous integrals (which are well-known).
\qed
\end{remark}
%
%

\section{Asymptotic study of $\text{proj}(\mathcal L_\ell(z) | C_2)$}
In this section we find an explicit expression for the second order chaotic projection of the length of level curves.
\begin{prop}\label{2sfera}
We have
$$
\displaylines{
\text{proj}(\mathcal L_\ell(z) | C_2)= \sqrt{\frac{\ell(\ell+1)}{2}}
 \sqrt{\frac{\pi}{8}} \phi(z) z^2 \int_{\mathbb S^2} H_2(T_\ell(x))\,dx=\cr
=\sqrt{\frac{\ell(\ell+1)}{2}}
 \sqrt{\frac{\pi}{8}} \phi(z) z^2 \sum_{m=1}^{2\ell+1} \left (a_{\ell,m}^2 - \frac{4\pi}{2\ell +1} \right )\ .}
$$
\end{prop}
\begin{proof}
The second chaotic projection is, omitting the factor $\sqrt{\frac{\ell(\ell+1)}{2}}$,
$$
\displaylines{
  \frac{\alpha_{0,0}\beta_{2}(z)}{2}
\int_{\mathbb S^2} H_{2}(T_{\ell}(x))\,dx
+ \frac{\alpha_{0,2}\beta_{0}(z)}{2}
\int_{\mathbb S^2}  H_{2}(\widetilde \partial_2 T_\ell(x))\,dx +\cr
+  \frac{\alpha_{2,0}\beta_{0}(z)}{2}
\int_{\mathbb S^2} H_{2}(\widetilde \partial_1 T_\ell(x))\,dx\
= \frac12 \Big (\alpha_{0,0}\beta_{2}(z)
\int_{\mathbb S^2} (T_{\ell}(x)^2-1)\,dx
+ \cr
+\alpha_{0,2}\beta_{0}(z)
\int_{\mathbb S^2}  ((\widetilde \partial_2 T_\ell(x))^2-1)\,dx
+  \alpha_{2,0}\beta_{0}(z)
\int_{\mathbb S^2} ((\widetilde \partial_1 T_\ell(x))^2-1)\,dx \Big )\ =\cr
= \frac12 \Big (\alpha_{0,0}\beta_{2}(z)
\int_{\mathbb S^2} T_{\ell}(x)^2\,dx
+\alpha_{0,2}\beta_{0}(z)
\frac{2}{\ell(\ell+1)}\int_{\mathbb S^2}  ( \partial_2 T_\ell(x))^2\,dx
+  \cr
+\alpha_{2,0}\beta_{0}(z)
\frac{2}{\ell(\ell+1)}\int_{\mathbb S^2} (\partial_1 T_\ell(x))^2\,dx - 4\pi(\alpha_{0,0}\beta_{2}(z) + \frac{4}{\ell(\ell+1)}\alpha_{0,2}\beta_{0}(z))\Big )\ .
}
$$
Now, by Green's formula, we have for $j=1,2$
$$\displaylines{
\int_{\mathbb S^2} (\partial_j T_\ell(x))^2\,dx = -
\int_{\mathbb S^2} T_\ell(x) \partial_{j}^2 T_\ell(x)\,dx
}$$
and putting things together
$$\displaylines{
\frac12 \Big (\alpha_{0,0}\beta_{2}(z)
\int_{\mathbb S^2} T_{\ell}(x)^2\,dx
+\alpha_{0,2}\beta_{0}(z)
\frac{2}{\ell(\ell+1)}\int_{\mathbb S^2}  ( \partial_2 T_\ell(x))^2\,dx
+  \cr
+\alpha_{2,0}\beta_{0}(z)
\frac{2}{\ell(\ell+1)}\int_{\mathbb S^2} (\partial_1 T_\ell(x))^2\,dx - 4\pi(\alpha_{0,0}\beta_{2}(z) + \frac{4}{\ell(\ell+1)}\alpha_{0,2}\beta_{0}(z))\Big )=}$$
$$\displaylines{
=\frac12 \Big (\alpha_{0,0}\beta_{2}(z)
\int_{\mathbb S^2} T_{\ell}(x)^2\,dx
-\alpha_{0,2}\beta_{0}(z)
\frac{2}{\ell(\ell+1)}
\int_{\mathbb S^2} T_\ell(x) \partial_{2}^2 T_\ell(x)\,dx
+  \cr
-\alpha_{2,0}\beta_{0}(z)
\frac{2}{\ell(\ell+1)}
\int_{\mathbb S^2} T_\ell(x) \partial_{1}^2 T_\ell(x)\,dx - 4\pi(\alpha_{0,0}\beta_{2}(z) + \frac{4}{\ell(\ell+1)}\alpha_{0,2}\beta_{0}(z))\Big )=\cr}$$
$$\displaylines{
=\frac12 \Big (\alpha_{0,0}\beta_{2}(z)
\int_{\mathbb S^2} T_{\ell}(x)^2\,dx
-\alpha_{0,2}\beta_{0}(z)
\frac{2}{\ell(\ell+1)}
\int_{\mathbb S^2} T_\ell(x) (\partial_{1}^2 T_\ell(x)+\partial_{2}^2 T_\ell(x))\,dx
+  \cr
- 4\pi(\alpha_{0,0}\beta_{2}(z) + \frac{4}{\ell(\ell+1)}\alpha_{0,2}\beta_{0}(z))\Big )=\cr}$$
$$\displaylines{
=\frac12 \Big (\alpha_{0,0}\beta_{2}(z)
\int_{\mathbb S^2} T_{\ell}(x)^2\,dx
-\alpha_{0,2}\beta_{0}(z)
\frac{2}{\ell(\ell+1)}
\int_{\mathbb S^2} T_\ell(x) \Delta T_\ell(x)\,dx
+  \cr
- 4\pi(\alpha_{0,0}\beta_{2}(z) + \frac{4}{\ell(\ell+1)}\alpha_{0,2}\beta_{0}(z))\Big )=\cr}$$
$$\displaylines{
=\frac12 \Big (\alpha_{0,0}\beta_{2}(z)
\int_{\mathbb S^2} T_{\ell}(x)^2\,dx
+\alpha_{0,2}\beta_{0}(z)
\frac{2}{\ell(\ell+1)} \ell(\ell+1)
\int_{\mathbb S^2} T_\ell(x)^2\,dx
+  \cr
- 4\pi(\alpha_{0,0}\beta_{2}(z) + \frac{4}{\ell(\ell+1)}\alpha_{0,2}\beta_{0}(z))\Big )=\cr
=\frac12 \Big ((\alpha_{0,0}\beta_{2}(z)+2\alpha_{0,2}\beta_{0}(z))
\int_{\mathbb S^2} T_{\ell}(x)^2\,dx
- 4\pi(\alpha_{0,0}\beta_{2}(z) + 2\alpha_{0,2}\beta_{0}(z))\Big )=\cr
=\frac12 (\alpha_{0,0}\beta_{2}(z)+2\alpha_{0,2}\beta_{0}(z))
\int_{\mathbb S^2} (T_{\ell}(x)^2-1)\,dx=\cr
=\frac12\sqrt{\frac{\pi}{2}}\phi(z) z^2 \int_{\mathbb S^2} H_2(T_{\ell}(x))\,dx\ .
}$$
Moreover
$$\displaylines{
\int_{\mathbb S^2} H_2(T_{\ell}(x))\,dx = \int_{\mathbb S^2} \sum_{m,m'}\left( a_{\ell,m} a_{\ell,m'} Y_{\ell,m}(x) Y_{\ell,m'}(x) - 1\right )\,dx =\cr
= \sum_{m=1}^{2\ell+1} \left(a_{\ell,m}^2 - \frac{4\pi}{2\ell+1}\right)\ ,
}$$
since $Y_{\ell,m}$ are an orthonormal family.
\end{proof}
Now it immediately follows that
\begin{cor}\label{corNodal}
The second chaotic projection of the length $\mathcal L_\ell(z)$ vanishes if and only if $z=0$.
\end{cor}
\begin{remark}\rm
Previous computations in the proof of Proposition \ref{2sfera} indeed holds on every two dimensional compact Riemannian manifold, actually we can always use Green's formula.
\end{remark}

\section{The CLT}

In this section we prove the main result of this chapter, that is a CLT for the length of $z$-level curve for $z\ne 0$.
Let us first show the following.

\begin{lemma}
For $z\ne 0$, we have
$$
\frac{\text{proj}(\mathcal L_\ell(z) | C_2)}{\sqrt{\Var(\text{proj}(\mathcal L_\ell(z) | C_2))}}\mathop{\goto}^{\mathcal L} Z\ ,
$$
where $Z\sim \mathcal N(0,1)$.
\end{lemma}
\begin{proof}
The variance of the second chaotic projection (Proposition \ref{2sfera}) is
$$\displaylines{
\Var(\text{proj}(\mathcal L_\ell(z) | C_2))= \ell(\ell+1) \frac{\pi}{16} \phi(z)^2 z^4 2 \cdot 4\pi\cdot 2\pi \int_0^{\pi} P_\ell(\cos\vartheta)^2\,d\vartheta=\cr
= \ell(\ell+1) \frac{\pi}{16} \phi(z)^2 z^4 2 \cdot 4\pi\cdot 2\pi \frac{2}{2\ell +1}=\ell(\ell+1) \frac{\pi}{16} \frac{1}{2\pi} \e^{-z^2} z^4 2 \cdot 4\pi\cdot 2\pi \frac{2}{2\ell +1}=\cr
=\ell(\ell+1)\frac{2}{2\ell +1}\cdot \frac{ \pi^2}{2}  \e^{-z^2} z^4  \sim
\ell \cdot \frac{ \pi^2}{2}  \e^{-z^2} z^4\ , \quad \ell\to +\infty\ ,
}$$
where we used the identity \paref{momento 2}
$$
\int_0^\pi P_\ell(\cos \vartheta)^2\,d\vartheta = \frac{2}{2\ell+1}\ .
$$
Moreover we can rewrite the second chaotic projection as
$$\displaylines{
\text{proj}(\mathcal L_\ell(z) | C_2) = \sqrt{\frac{\ell(\ell+1)}{2}}
 \sqrt{\frac{\pi}{8}} \phi(z) z^2 \sum_m \left (a_{\ell,m}^2 - \frac{4\pi}{2\ell+1} \right )=\cr
=\sqrt{\frac{\ell(\ell+1)}{2}}
 \sqrt{\frac{\pi}{4}} \phi(z) z^2\frac{4\pi}{\sqrt{2\ell+1}} \frac{1}{\sqrt{2(2\ell+1)}}\sum_m \left (\left (\sqrt{\frac{2\ell+1}{4\pi}}a_{\ell,m} \right )^2 - 1\right )\ .
}$$
Now we can apply the standard CLT to the sequence of normalized sums
$$
 \frac{1}{\sqrt{2(2\ell+1)}}\sum_m \left (\left (\sqrt{\frac{2\ell+1}{4\pi}}a_{\ell,m} \right )^2 - 1\right )\mathop{\goto}^{\mathcal L} Z\ ,
$$
where $Z\sim \mathcal N(0,1)$. Finally this implies the CLT for the second chaotic projection
$$
\frac{\text{proj}(\mathcal L_\ell(z) | C_2)}{\sqrt{\Var(\text{proj}(\mathcal L_\ell(z) | C_2))}} \mathop{\goto}^{\mathcal L} Z\ ,
$$
which conclude the proof.
\end{proof}
Now we can easily prove Thereom \ref{main th S}.

\begin{proof}[Proof of Theorem \ref{main th S}]

We have, for $z\ne 0$,
\begin{equation}\label{nice}
\lim_\ell \frac{\Var(\text{proj}(\mathcal L_\ell(z) | C_2))}{\Var(\mathcal L_\ell(z))} = 1\ .
\end{equation}
It follows  from the chaotic decomposition, that as $\ell\to \infty$
$$
\frac{\mathcal L_\ell(z)}{\sqrt{\Var(\mathcal L_\ell(z))}} = \frac{\text{proj}(\mathcal L_\ell(z) | C_2)}{\sqrt{\Var(\mathcal L_\ell(z))}} + o_\P (1)\ ,
$$
therefore $\frac{\mathcal L_\ell(z)}{\sqrt{\Var(\mathcal L_\ell(z))}}$ and
$\frac{\text{proj}(\mathcal L_\ell(z) | C_2)}{\sqrt{\Var(\mathcal L_\ell(z))}}$ have the same asymptotic distribution. Previous lemma allows to conclude the proof, recalling moreover that if the limit distribution is absolutely continuous, than the convergence in distribution is equivalent to the convergence   in Kolmogorov distance.
\end{proof}

\chapter{Nodal lengths  for arithmetic random waves}

\section{Introduction and main results}

In this chapter we investigate the asymptotic behavior of nodal lengths for arithmetic random waves.

\subsection{Arithmetic random waves}

Let $\T:=\R^2/\Z^2$ be the standard
$2$-torus and $\Delta$ the Laplace operator on $\Tb$. We are interested
in the (totally discrete) spectrum of $\Delta$ i.e. eigenvalues $E>0$
of the Schr\"{o}dinger equation
\begin{equation}
\label{eq:Schrodinger}
\Delta f + Ef=0.
\end{equation}
Let $$S=\{{{} n \in \Z : n} =  a^2+b^2 \,\,  \mbox{{} for some} \:a,\, b\in\Z\}$$ be the collection of all numbers
expressible as a sum of two squares. Then the eigenvalues of \eqref{eq:Schrodinger}
(also called ``energy levels" of the torus) are all numbers of the form $E_{n}=4\pi^{2}n$ with $n\in S$.

In order to describe a Laplace eigenspace corresponding to $E_{n}$ denote
$\Lambda_n$ to be the set of ``frequencies":
\begin{equation*}
\Lambda_n := \lbrace \lambda =(\lambda_1,\lambda_2)\in \Z^2 : \lambda_1^2 + \lambda_2^2 = n\rbrace\
\end{equation*}
of cardinality $| \Lambda_n |$. (Geometrically $\Lambda_{n}$ are all the standard lattice points
lying on the centered radius-$\sqrt{n}$ circle.)
For $\lambda\in \Lambda_{n}$ denote the complex exponential associated to the frequency $\lambda$
\begin{equation*}
e_{\lambda}(\theta) = \exp(2\pi i \langle \lambda, \theta \rangle)
\end{equation*}
with $\theta=(\theta_{1},\theta_{2})\in\Tb$.
The collection
\begin{equation*}
\{e_{\lambda}(\theta)\}_{\lambda\in \Lambda_n}
\end{equation*}
of complex exponentials corresponding to frequencies $\lambda \in \Lambda_{n}$
is an $L^{2}$-orthonormal basis of the eigenspace of $\Delta$ corresponding to
eigenvalue $E_{n}$. In particular, the dimension of $E_{n}$
equals the number of ways to express $n$ as a sum of two squares
\begin{equation*}
\mathcal N_n := \dim E_{n} = |\Lambda_n|
\end{equation*}
(also denoted in the number theoretic literature $r_{2}(n)=|\Lambda_{n}|$). The number
$\Nc_{n}$ is subject to large and erratic fluctuation; it grows ~\cite{La} {\em on average}
as $\sqrt{\log{n}}$, but could be as small as $8$ for (an infinite sequence of) prime numbers
$p\equiv 1\mod{4}$, or as large as a power of $\log{n}$.

Following ~\cite{RW} and ~\cite{AmP} we define
the ``arithmetic random waves" (random Gaussian toral Laplace eigenfunctions)
to be the random fields
\begin{equation}\label{defrf}
T_n(\theta)=\frac{1}{\sqrt{\mathcal N_n}}\sum_{ \lambda\in \Lambda_n}
a_{\lambda}e_\lambda(\theta),
\end{equation}
$\theta\in\Tb$, where the coefficients $a_{\lambda}$ are standard Gaussian i.i.d. save
to the relations $$a_{-\lambda}= \overline{a_{\lambda}}$$ (ensuring that $T_{n}$ are real-valued).
By the definition \eqref{defrf}, $T_n$ is a centered
Gaussian random field with covariance function
\begin{equation*}
r_n(\theta,\zeta) = r_{n}(\theta-\zeta) := \E[T_n(\theta) \overline{T_n(\zeta)}] = \frac{1}{\mathcal N_n}
\sum_{\lambda\in \Lambda_n}e_{\lambda}(\theta-\zeta)=\frac{1}{\mathcal N_n}\sum_{\lambda\in \Lambda_n}\cos\left(2\pi\langle \theta-\zeta,\lambda \rangle\right),
\end{equation*}
$\theta,\zeta\in\Tb$ (by the standard abuse of notation). Note that $r_{n}(0)=1$, i.e. $T_{n}$ is unit variance.

\subsection{Nodal length{{} : mean and variance}}

Consider the total {\em nodal length} of random
eigenfunctions, i.e. the sequence the random variables $\{\mathcal L_{n}\}_{n\in S}$ given by
\begin{equation}\label{e:length}
\mathcal L_n := \text{length}(T_n^{-1}(0)).
\end{equation}
The expected value of $\mathcal L_{n}$ was computed ~\cite{RW} to be
\begin{equation}
\E[\mathcal L_n]= \frac{1}{2\sqrt{2}}\sqrt{E_n},
\end{equation}
consistent to Yau's conjecture ~\cite{Yau,D-F}.
The more subtle question of asymptotic behaviour of the variance $\Var(\Lc_{n})$ of $\Lc_{n}$ was
addressed ~\cite{RW}, and fully resolved \cite{AmP} as follows.

Given $n\in S$ define a probability measure $\mu_{n}$ on the unit circle $\Sc^{1}\subseteq\R^{2}$
supported on angles corresponding to lattice points in $\Lambda_{n}$:
\begin{equation*}
\mu_{n} := \frac{1}{\mathcal N_n} \sum_{\lambda\in \Lambda_n} \delta_{\frac{\lambda}{\sqrt{n}}}.
\end{equation*}
It is known  that for a density $1$ sequence of numbers $\{n_{j}\}\subseteq S$ the angles
of lattice points in $\Lambda_{n}$ tend to be equidistributed in the sense that
\begin{equation*}
\mu_{n_{j}}\Rightarrow \frac{d\theta}{2\pi}
\end{equation*}
(where $\Rightarrow$ is weak-$*$ convergence of probability measures). However the sequence $\{\mu_{n}\}_{n\in S}$
has other weak-$*$ partial limits ~\cite{Ci,AmP} (``attainable measures"), partially
classified in \cite{KW}.

It was proved ~\cite{AmP} that one has
\begin{equation}
\label{eq:var leading KKW}
\var(\Lc_{n}) =c_n \frac{E_n}{\Nc_{n}^2}(1 + o_{\Nc_{n}\rightarrow\infty}(1)),
\end{equation}
where
\begin{equation}\label{cn}
c_n = \frac{1+\widehat{\mu_n}(4)^2}{512},
\end{equation}
and for a measure $\mu$ on $\mathbb{S}^{1}$,
\begin{equation}\label{e:smet}
\widehat \mu_n(k) = \int_{\mathbb S^1} z^{-k}\,d\mu_n(z)
\end{equation}
are the usual Fourier coefficients of $\mu$ on the unit circle.
As $$|\widehat{\mu_{n}}(4)|\le 1$$ by the triangle inequality, the result \paref{eq:var leading KKW}
shows that the order of magnitude of $\var(\Lc_{n})$ is
$ \frac{E_n}{\Nc_n^2} $, that is, of smaller order than what would be a natural guess
$ \frac{E_n}{\Nc_n} $; this situation (`arithmetic Berry's cancellation' --  see \cite{AmP}) is similar to the {\em cancellation phenomenon} observed by Berry
in a different setting ~\cite{Berry2002}.

In addition, \eqref{eq:var leading KKW} shows that for $\var(\Lc_{n})$ to exhibit an
asymptotic law (equivalent to $\{c_{n}\}$ in \eqref{cn} convergent along a subsequence)
we need to pass to a subsequence $\{ n_{j} \}\subset S$ such that the limit
$$\lim\limits_{j\rightarrow\infty}| \widehat{\mu_{n_{j}}}(4)|$$ exists. For example,
if $\{ n_{j}\} \subset S$ is a subsequence such that $\mu_{n_{j}}\Rightarrow\mu$
for some probability measure $\mu$ on $\Sc^{1}$,
then \eqref{eq:var leading KKW} reads (under the usual extra-assumption $\Nc_{n_{j}}\rightarrow\infty$)
\begin{equation}\label{e:varz}
\var(\Lc_{n_j})\sim c({{} \mu}) \frac{E_{n_j}}{\Nc_{n_j}^2}
\end{equation}
with $$c(\mu) = \frac{1+\widehat{\mu}(4)^2}{512},$$
where, here and for the rest of the chapter, we write $a_n\sim b_n$ to
indicate that the two positive sequences $\{a_n\}$ and $\{b_n\}$ are such
that $a_n/b_n \rightarrow 1$, as $n\to\infty$. Note that the set of the possible
values for the $4$th Fourier coefficient $\widehat{\mu}(4)$ covers the whole interval
$[-1,1]$ (see \cite{AmP, KW}). This implies in particular that the possible values of
the constant $c(\mu)$ cover the whole interval
$$\left[\frac{1}{512},\frac{1}{256}\right];$$ the above discussion provides
a complete classification of the asymptotic behaviour of $\var(\Lc_{n})$.

\subsection{Main results}

{{}

Let $\{n_j : j\geq 1\}\subset S$ be a sequence within $S$, and assume that
$\lim_{j\to\infty }{\mathcal N}_{n_j}  = \infty$. As it is customary, generic
subsequences of $\{n_j\}$ will be denoted by $\{n'_j\}$, $\{n''_j\}$, and so on.
Our principal aim in this chapter is to study the asymptotic behaviour, as $j\to \infty$,
of the distribution of the sequence of normalized random variables
\begin{equation}\label{e:culp}
\widetilde{\Lc}_{n_j} :=  \frac{\mathcal{L}_{n_j} - \E[\mathcal{ L}_{n_j}]}
{\sqrt{\Var[\mathcal{L}_{n_j} ]}}, \quad j\geq 1.
\end{equation}
Since, in this setting, the variance \eqref{eq:var leading KKW} {diverges to infinity}, it {seems}
reasonable to expect a central limit result, that is, that the sequence $\widetilde{\Lc}_{n_j}$, $ j\geq 1$,
{converges in distribution to a standard Gaussian random variable.}
Our main findings not only contradict this {somewhat naive prediction},
but also show the following non-trivial facts:
\begin{itemize}
\item[\bf (i)] the sequence $\big\{\widetilde{\Lc}_{n_j} \big\}$ does not
necessarily converge in distribution, and
\item[\bf (ii)] the adherent points of the sequence
$\big\{\widetilde{\Lc}_{n_j} : j\geq 1 \big\}$
(in the sense of the topology induced by the convergence in distribution of random variables)
coincide with the distributions spanned by a class of linear combinations of independent
squared Gaussian random variables; such linear combinations are moreover parameterized by
the adherent points of the numerical sequence
$$
j \mapsto  \left|\widehat{\mu_{n_j}}(4)\right|, \quad j\geq 1.
$$
\end{itemize}

One should note that the phenomenon described at Point {\bf (ii)} is consistent with the fact that the
variance ${\rm Var}(\mathcal{L}_n)$ explicitly depends on the constant $\widehat{\mu_n}(4)^2$ (see
\eqref{e:varz}).  In order to formally state our main findings, we introduce some further notation: for every
$\eta\in [0,1]$, we write $\mathcal{M}_\eta$ to indicate the random variable
\begin{equation}\label{e:r}
\mathcal{M}_\eta := \frac{1}{2\sqrt{1+\eta^2}} (2 - (1+\eta) X_1^2-(1-\eta) X_2^2),
\end{equation}
where $X=(X_{1},X_{2})$ is a two-dimensional centered Gaussian vector with identity covariance matrix (more
information on the distributions of the random variables $\Mc_{\eta}$ is provided in Proposition
\ref{p:meta}). For every $n\in S$, we write
\begin{equation}\label{e:k}
\Mc^n :=  \Mc_{ |  \widehat{\mu_{n}}(4) | },
\end{equation}
where the quantity $\widehat{\mu_{n}}(4)$ is defined according to formula \eqref{e:smet}.
\smallskip

The following statement is the main result of the chapter.

\begin{theorem}\label{thm:lim dist sep}
Let the above notation and assumptions prevail. Then, the sequence $\big\{\mathbf{D}\big(\widetilde{\Lc}_{n_j}\big) : j\geq 1\big\}$ is relatively compact with respect to the topology of weak convergence, and a subsequence $\big\{ \widetilde{\Lc}_{n'_j}\big\}$ admits a limit in distribution if and only if the corresponding numerical subsequence $\big\{\big|\widehat{\mu_{n'_j}}(4)\big| : j\geq 1 \big\}$ converges to some $\eta\in [0,1]$, and in this case $$\widetilde{\Lc}_{n'_j}  \stackrel{\rm d}{\longrightarrow} \mathcal{M}_\eta.$$
In particular, letting $d$ denote either the Kolmogorov distance \paref{prob distance}, or an arbitrary distance metrizing weak convergence on $\mathscr{P}$ (the space of all probability mesures on $\R$ - see Chapter 4) this implies that
\begin{equation}\label{e:b}
\lim_{j\to\infty} d\big(\widetilde{\Lc}_{n_j} , \Mc^{n_j}\big) = 0.
\end{equation}
\end{theorem}
The next result is a direct consequence of Theorem \ref{thm:lim dist sep}, of \cite[Theorem 11.7.1]{D} and of the fact that $\big\{\widetilde{\Lc}_{n_j} \big\}$ is a bounded sequence in $L^2$: it shows that one can actually couple the elements of the sequences $\{\widetilde{\Lc}_{n_j}\}$ and $\big\{\Mc^{n_j}\big\}$ on the same probability space, in such a way that their difference converges to zero almost surely and in $L^p$, for every $p<2$.

\begin{corollary}\label{c:coupling} There exists a probability space $(\Omega^*, \mathcal{F}^*, \P^*)$ as well as random variables $\{A_j, B_j : j\geq 1\}$ defined on it such that, for every $j\geq 1$, $A_j \stackrel{\rm d}{=} \widetilde{\Lc}_{n_j}$, $B_j  \stackrel{\rm d}{=}  \Mc^{n_j}$, and, as $j\to \infty$,
$$
A_j-B_j \to 0, \quad \mbox{a.s.}-\P^*.
$$
Also, for every $p\in (0,2)$, $\E^*[ | A_j-B_j| ^p ]\to 0$.
\end{corollary}

We conclude this section by stating some elementary properties of the random variables $\Mc_\eta$, $\eta\in [0,1]$, whose proof (left to the reader) can be easily deduced from the representation
\begin{equation}
\mathcal{M}_\eta = a(\eta) H_2(X_1)+b(\eta) H_2(X_2),
\end{equation}
where $H_2 (x) = x^2-1$ is the second Hermite polynomial,
$a(\eta) := -(1+\eta)/\sqrt{4(1+\eta^2)}$ and $b(\eta) := -(1-\eta)/\sqrt{4(1+\eta^2)}$,
as well as from the (classical) results presented in \cite[Section 2.7.4]{noupebook}.

In what follows, we will use the elementary fact that, if $\Mc_\eta$ is the random variable defined in \eqref{e:r} and if $\eta\to \eta_0\in [0,1]$, then $\Mc_\eta \stackrel{\rm d}{\longrightarrow} \Mc_{\eta_0}$.

\begin{proposition}[\bf About $\Mc_{\eta}$]\label{p:meta}

\begin{enumerate} Let the above notation prevail.

\item For every $\eta\in [0,1]$, the distribution of $\Mc_\eta$ is absolutely continuous with respect to the Lebesgue measure, with support equal to $\big(-\infty, (1+\eta^2)^{-1/2} \big)$.

\item For every $\eta\in [0,1]$, the characteristic function of $\Mc_\eta$ is given by
$$
\varphi_\eta(\mu) := \E[\exp(i\mu \Mc_\eta)] = \frac{e^{-i\mu(a(\eta)+b(\eta))}}{\sqrt{(1-2i\mu a(\eta))(1-2i\mu b(\eta))}}, \quad \mu\in \R.
$$
\item  For every $\eta\in [0,1]$, the distribution of $\Mc_\eta$ is determined
 by its moments (or, equivalently, by its cumulants). Moreover, the sequence of
  the cumulants of $\Mc_\eta$, denoted by $\{\kappa_p(\Mc_\eta) : p\geq 1\}$,
  admits the representation: $\kappa_p(\Mc_\eta) = 2^{p-1}(p-1)!
  (a(\eta)^p + b(\eta)^p)$, for every $p\geq 1$ (in particular, $\mathcal{M}_\eta$ has unit variance).
\item Let $\eta_0, \eta_1\in [0,1]$ be such that $\eta_0\neq \eta_1$. Then,
${\bf D}(\Mc_{\eta_0})  \neq {\bf D}(\Mc_{\eta_1})$.
 \end{enumerate}
\end{proposition}

We observe that Point 4 in the previous statement is an immediate consequence of Point 1 and of the fact that the mapping $\eta\mapsto (1+\eta^2)^{-1/2}$ is injective on $[0,1]$. In the next section, we will discuss the role of chaotic expansions in the proofs of our main findings.

}

{{}

\subsection{ Chaos and the Berry cancellation phenomenon}\label{ss:berryintro}

As in the previous chapter, the proofs of our results rely on a pervasive use of  Wiener-It\^o chaotic expansions for non-linear functionals of Gaussian fields (the reader is referred to the two monographs \cite{noupebook, P-T} for an exhaustive discussion).

\medskip

According to \eqref{defrf}, the arithmetic random waves considered in this work are built starting from a family of complex-valued Gaussian random variables $\{a_\lambda : \lambda\in \mathbb{Z}^2\}$, defined on some probability space $(\Omega, \mathscr{F}, \mathbb{P})$ and verifying the following properties: {\bf (a)} each $a_\lambda$ has the form $x_\lambda+iy_\lambda$, where $x_\lambda$ and $y_\lambda$ are two independent real-valued Gaussian random variables with mean zero and variance $1/2$; {\bf (b)} $a_\lambda$ and $a_\tau$ are stochastically independent whenever $\lambda \notin\{ \tau, -\tau\}$, and {\bf (c)} $a_\lambda = \overline{a_{-\lambda}}$. We define ${\bf A}$ to be the closure in $L^2(\mathbb{P})$ of all real finite linear combinations of random variables $\xi$ having the form $\xi = z \, a_\lambda + \overline{z} \, a_{-\lambda}$, where $\lambda\in \mathbb{Z}^2$ and $z\in \mathbb{C}$. It is easily verified that ${\bf A}$ is a real centered Gaussian space (that is, a linear space of jointly Gaussian centered real-valued random variables, that is stable under convergence in $L^2(\mathbb{P})$).

\begin{defn}\label{d:chaos}{\rm For every $q=0,1,2,...$ the $q$th {\it Wiener chaos} associated with ${\bf A}$, written $C_q$, is the closure in $L^2(\mathbb{P})$ of all real finite linear combinations of random variables with the form
$$
H_{p_1}(\xi_1)H_{p_2}(\xi_2)\cdots H_{p_k}(\xi_k),
$$
where the integers $p_1,...,p_k \geq 0$ verify $p_1+\cdot+p_k = q$, and $(\xi_1,...,\xi_k)$ is a real centered Gaussian vector with identity covariance matrix extracted from ${\bf A}$ (note that, in particular, $C_0 = \mathbb{R}$).}
\end{defn}
Again $C_q \,\bot\, C_m$ (where the orthogonality holds in the sense of $L^2(\mathbb{P})$) for every $q\neq m$, and moreover
\begin{equation}\label{e:chaos}
L^2(\Omega, \sigma({\bf A}), \mathbb{P}) = \bigoplus_{q=0}^\infty C_q,
\end{equation}
that is: each real-valued functional $F$ of ${\bf A}$ can be (uniquely) represented in the form
\begin{equation}\label{e:chaos2}
F = \sum_{q=0}^\infty {\rm proj}(F \, | \, C_q),
\end{equation}
where ${\rm proj}(\bullet \, | \, C_q)$ stands for the projection operator onto $C_q$, and the series converges in $L^2(\mathbb{P})$. Plainly, ${\rm proj}(F \, | \, C_0) = \E F$. Now recall the definition of $T_n$ given in \eqref{defrf}: the following elementary statement shows that the Gaussian field
$$
\left\{ T_n(\theta),\,  \frac{\partial}{\partial \theta_1} T_n(\theta),\,  \frac{\partial}{\partial \theta_2} T_n(\theta) : \theta =(\theta_1,\theta_2)\in \mathbb{T}\right\}
$$
is a subset of ${\bf A}$, for every $n\in S$.

\begin{proposition} \label{p:field} Fix $n\in S$, let the above notation and conventions prevail. Then, for every $j=1,2$ one has that
\begin{equation}\label{e:partial}
\partial_j T_n(\theta) := \frac{\partial}{\partial \theta_j} T_n(\theta) = \frac{2\pi i}{\sqrt{\mathcal{N}_n} }\sum_{(\lambda_1,\lambda_2)\in \Lambda_n} \lambda_j a_\lambda e_\lambda(\theta),
\end{equation}
and therefore $T_n(\theta), \, \partial_1 T_n(\theta), \, \partial_2 T_n(\theta) \in {\bf A}$, for every $\theta\in \mathbb{T}$. Moreover, for every fixed $\theta\in \mathbb{T}$, one has that $T_n(\theta), \, \partial_1 T_n(\theta), \, \partial_2 T_n(\theta)$ are stochastically independent.
\end{proposition}

We shall often use the fact that
$$\displaylines{
\Var[\partial_j T_n(\theta)] = \frac{4\pi^2}{\mathcal N_n}
\sum_{\lambda\in \Lambda_n} \lambda_j^2 = 4\pi^2 \frac{n}{2}\ ,
}$$
and, accordingly, for $\theta=(\theta_1, \theta_2)\in \mathbb T$ and $j=1,2$, we will denote by $\partial_j \widetilde T_n(\theta)$ the normalized derivative
\begin{equation}\label{e:norma}
\partial_j \widetilde T_n(\theta) := \frac{1}{2\pi} \sqrt{\frac{2}{n}} \frac{\partial}{\partial \theta_j}  T_n(\theta) = \sqrt{\frac{2}{n}}\frac{ i}{\sqrt{\mathcal N_n}}\sum_{ \lambda\in \Lambda_n}\lambda_j\,
a_{\lambda}e_\lambda(\theta)\ .
\end{equation}

The next statement gathers together some of the main technical achievements of the present chapter. It shows in particular that the already evoked `arithmetic Berry cancellation phenomenon' (see \cite{AmP}, as well as \cite{Berry2002}) -- according to which the variance of the nodal length $\mathcal{L}_n$  (as defined in \eqref{e:length}) has asymptotically the same order as $ \frac{E_n}{\Nc_n^2}$ (rather than the expected order $ \frac{E_n}{\Nc_n}$) -- is a consequence of the following two facts:
\begin{itemize}

\item[\bf (i)] The projection of $\mathcal{L}_n$ on the second Wiener chaos $C_2$ is {\it exactly equal to zero} for every $n\in S$ (and so is the projection of $\mathcal{L}_n$ onto any chaos of odd order $q\geq 3$).
\item[\bf (ii)] The variance of $ {\rm proj}(\mathcal{L}_{n} \, | C_4)$ has the order $ \frac{E_n}{\Nc_n^2}$, as $\Nc_n\to \infty$, and one has moreover that
$$
\var(\Lc_{n}) = \var\left( {\rm proj}(\mathcal{L}_{n} \, | C_4)\right)+ o\left(\frac{E_{n}}{\Nc_{n}^2}\right).
$$

\end{itemize}

Note that, in principle, if ${\rm proj}(\mathcal{L}_n \, | C_2)$ did not vanish, then the sequence $n\mapsto  \var\left( {\rm proj}(\mathcal{L}_{n} \, | C_2)\right)$ would have provided the leading term (of the order $ \frac{E_n}{\Nc_n}$) in the asymptotic development of $\var(\Lc_{n})$.

\begin{proposition}[\bf Berry cancellation phenomenon]\label{p:berry} For every fixed $n\in S$,\\ one has that
\begin{equation}\label{e:berry}
{\rm proj}(\mathcal{L}_n \, | C_2) ={\rm proj}(\mathcal{L}_n \, | C_{2k+1}) = 0, \quad k=0,1,...,
\end{equation}
Moreover, if $\{n_j : j\geq 1\}\subset S$ is a sequence contained in $S$ such that $\lim_{j\to\infty }{\mathcal N}_{n_j}  = \infty$, then (as $j \to \infty$)
$$
\var(\Lc_{n_j})\sim c({\mu_{n_j} }) \frac{E_{n_j}}{\Nc_{n_j}^2} \sim \var\left( {\rm proj}(\mathcal{L}_{n_j} \, | C_4)\right),
$$
and therefore,
$$
\sum_{k=3}^\infty \var\left( {\rm proj}\left(\mathcal{L}_{n_j} \, | C_{2k} \right)\right) = o\left( \frac{E_{n_j}}{\Nc_{n_j}^2}\right).
$$
\end{proposition}

\begin{remark}\label{remMau}\rm Nodal lengths of Gaussian Laplace eigenfunctions $T_\ell$, $\ell\in \N$,
on the two-dimensional sphere
have the same qualitative behavior as their toral counterpart. Indeed, in Proposition \ref{2sfera} it is shown that
the second chaotic term in the Wiener-It\^o
 expansion of the length of level curves $T_\ell^{-1}(u)$, $u\in \mathbb R$ disappears
 if and only if $u=0$.  These findings shed some light on the Berry's cancellation
 phenomenon, indeed they explain why
 the asymptotic variance of the length of level curves respects the natural scaling --
except for the nodal case \cite{Wig,wigsurvey}.
 \end{remark}

\begin{conjecture}\rm
Consider Gaussian eigenfunctions $T$ on some manifold
 $\mathbb M$ and define as usual the $u$-excursion set as
\[
A_u(T,\mathbb M):=\{x \in \mathbb M : T(x)>u\}, \qquad u \in \mathbb{R}.
\]
Toral (resp. spherical) nodal lengths can be viewed as the length of the boundary of $A_u$ for
$u=0$ for $\mathbb M =\mathbb T$ the $2$-torus (resp. $\mathbb M = \mathbb S^2$ the $2$-sphere).
In this sense, as stated in the Introduction of this thesis,
 they represent a special case of the second {\it Lipschitz-Killing curvature} of $A_u$, $u \in \mathbb{R}$
(see \cite{adlertaylor} for the definition and a comprehensive treatment of Lipschitz-Killing curvatures
on Gaussian excursion sets). We conjecture that for excursion sets $A_u$ of Gaussian
eigenfunctions on \emph{compact} manifolds $\mathbb M$ the projection of each Lipschitz-Killing
curvature on the second-order
Wiener chaos vanishes if and only if $u=0$; clearly the proof of this conjecture would represent a major step
towards a global understanding of the Berry's cancellation phenomenon. In the two-dimensional case, there
are three Lipschitz-Killing curvatures, which correspond to the area, half the boundary length and the
Euler-Poincar\'e characteristic of the excursion sets; for the $2$-sphere, we refer to
\cite{Nonlin,maudom,fluct}  for
results supporting our conjecture in the case of the area and the Euler-Poincar\'e characteristic, and to
Remark
\ref{remMau} and Chapter 7
for the boundary lengths.
\end{conjecture}

\subsection{Plan}

The rest of the chapter is organized as follows:  \S 8.2 contains a study of the chaotic
representation of nodal lengths, \S 8.3 focuses on the projection of nodal lengths on the fourth
Wiener chaos, whereas \S 8.4 contains a proof of our main result.
}

%

\section{Chaotic expansions}\label{expan}

The aim of this section is to derive an explicit expression for each projection of the type ${\rm proj}(\Lc_n \, | \, C_q)$, $q\geq 1$. In order to accomplish this task, as in Chapter 7, we first focus on a sequence of auxiliary random variables $\{\Lc_n^\eps : \eps>0\}$ that approximate $\Lc_n$ in the sense of the $L^2(\P)$ norm.

\subsection{Preliminary results}

Fix $n\geq 1$, and let $T_n$ be defined according to \eqref{defrf}. Define, for $\varepsilon >0$, the approximating random variables
\begin{equation}\label{napprox}
\mathcal L_n^\varepsilon :=\frac{1}{2\varepsilon}
\int_{\mathbb T} 1_{[-\varepsilon,\varepsilon]}(T_n(\theta))\|\nabla T_n(\theta) \|\,d\theta\ .
\end{equation}
\begin{lemma}\label{approx}
We have
$$
\lim_{\varepsilon\to 0} \E[|\mathcal L_n^\varepsilon - \mathcal L_n|^2]=0\ .
$$
\end{lemma}
\begin{proof}
We have that
$$
\lim_{\varepsilon\to 0} \mathcal L_n^\varepsilon = \mathcal L_n\ ,\qquad a.s.
$$
Moreover, for every $\varepsilon$
$$\displaylines{
|\mathcal L_n^\varepsilon - \mathcal L_n|^2\le 2((\mathcal L_n^\varepsilon)^2 +(\mathcal L_n)^2 )\le\cr
\le 2((12\sqrt{4\pi^2n})^2 +  (\mathcal L_n)^2)\ ,
}$$
where the last equality follows form Lemma 3.2 in \cite{RW}. We can hence apply dominated convergence theorem to conclude.
\end{proof}

We observe that the previous result suggests that the random variable $\mathcal L_n$  can be formally written as
\begin{equation}\label{formal}
\mathcal L_n = \int_{ \mathbb T} \delta_0(T_n(\theta))\| \nabla T_n(\theta) \|
\,d\theta\ ,
\end{equation}
where $\delta_0$ denotes the Dirac mass in $0$.

\subsection{Chaotic expansion of nodal length $\mathcal{L}_n$}


In view of the convention \eqref{e:norma},  we will rewrite \paref{napprox} as
\begin{equation}\label{formal2}
\mathcal L_n^\eps = \frac{1}{2\eps}\sqrt{4\pi^2}\sqrt{\frac{n}{2}}\int_{\mathbb T}
1_{[-\eps, \eps]}(T_n(\theta)) \sqrt{\partial_1 \widetilde
T_n(\theta)^2+\partial_2 \widetilde T_n(\theta)^2}\,d\theta\ .
\end{equation}
Recall from Chapter 7 the two collection of coefficients
$\{\alpha_{2n,2m} : n,m\geq 1\}$ and $\{\beta_{2l} : l\geq 0\}$, that are connected to the (formal) Hermite expansions of the norm $\| \cdot
\|$ in $\R^2$ and the Dirac mass $ \delta_0(\cdot)$ respectively.
These are given by \paref{e:beta} and \paref{e:alpha}
\begin{equation*}
\beta_{2l}:= \frac{1}{\sqrt{2\pi}}H_{2l}(0)\ ,
\end{equation*}
where $H_{2l}$ denotes the $2l$-th Hermite polynomial and
\begin{equation*}
\alpha_{2n,2m}=\sqrt{\frac{\pi}{2}}\frac{(2n)!(2m)!}{n!
m!}\frac{1}{2^{n+m}} p_{n+m}\left (\frac14 \right)\ ,
\end{equation*}
where for $N=0, 1, 2, \dots $ and $x\in \R$ (as in \paref{pN})
\begin{equation*}
\displaylines{ p_{N}(x) :=\sum_{j=0}^{N}(-1)^{j}\ \ (-1)^{N}{N
\choose j}\ \ \frac{(2j+1)!}{(j!)^2} x^j \ , }
\end{equation*}
$\frac{(2j+1)!}{(j!)^2}$ being the so-called {\it swinging factorial}
restricted to odd indices. The following result provides the key tool in order to prove Proposition \ref{p:berry}.

{{}

\begin{proposition}[\bf Chaotic expansion of $\Lc_n$]\label{teoexp} Relation \eqref{e:berry} holds for every $n\in S$ and also, for every $q\geq 2$,
\begin{eqnarray}\label{e:pp}
&&{\rm proj}(\Lc_n\, | \, C_{2q}) \\
&&= \sqrt{\frac{4\pi^2n}{2}}\sum_{u=0}^{q}\sum_{k=0}^{u}
\frac{\alpha _{2k,2u-2k}\beta _{2q-2u}
}{(2k)!(2u-2k)!(2q-2u)!} \!\!\int_{\mathbb T}\!\! H_{2q-2u}(T_n(\theta))
H_{2k}(\partial_1 \widetilde T_n(\theta))H_{2u-2k}(\partial_2
\widetilde T_n(\theta))\,d\theta.\notag
\end{eqnarray}
As a consequence, one has the representation
\begin{eqnarray}\label{chaosexp}
\Lc_n &=& \E \Lc_n + \sqrt{\frac{4\pi^2n}{2}}\sum_{q=2}^{+\infty}\sum_{u=0}^{q}\sum_{k=0}^{u}
\frac{\alpha _{2k,2u-2k}\beta _{2q-2u}
}{(2k)!(2u-2k)!(2q-2u)!}\times\\
&&\hspace{4.5cm} \times \int_{\mathbb T}H_{2q-2u}(T_n(\theta))
H_{2k}(\partial_1 \widetilde T_n(\theta))H_{2u-2k}(\partial_2
\widetilde T_n(\theta))\,d\theta,\notag
\end{eqnarray}
where the series converges in $L^2(\P)$.
\end{proposition}
}
\noindent\begin{proof}[Proof of Proposition \ref{teoexp}] The proof of the chaotic projection formula is based on the same arguments as the proof of Proposition \ref{teoexpS}, therefore we can skip details.

Let us show that the ${\rm proj}(\Lc_n\, | \, C_{2})$ vanishes. It equals the quantity
$$\displaylines{
\sqrt{4\pi^2}\sqrt{\frac{n}{2}}\Big(
\frac{\alpha_{0,0}\beta_{2}}{2} \int_{\mathbb T}
H_{2}(T_n(\theta))\,d\theta + \frac{\alpha_{0,2}\beta_{0}}{2}
\int_{\mathbb T}  H_{2}(\partial_2 \widetilde
T_n(\theta))\,d\theta +  \cr
+\frac{\alpha_{2,0}\beta_{0}}{2}
\int_{\mathbb T} H_{2}(\partial_1 \widetilde T_n(\theta))\,d\theta
\Big )\ . }$$ Using the explicit expression $H_2(x)=x^2-1$, we deduce that
\begin{eqnarray*}
\int_{\mathbb T} H_{2}(T_n(\theta))\,d\theta &=& \int_{\mathbb T}
\left ( T_n(\theta)^2 -1                           \right)\,d\theta=
\int_{\mathbb T}
\left ( \frac{1}{\mathcal N_n}
\sum_{\lambda, \lambda'\in \Lambda_n} a_\lambda \overline a_{\lambda'}
e_{\lambda - \lambda'}(\theta)
 -1                           \right)\,d\theta\\
&=&
 \frac{1}{\mathcal N_n}\sum_{\lambda, \lambda'\in \Lambda_n} a_\lambda \overline a_{\lambda'}
 \underbrace{\int_{\mathbb T}e_{\lambda - \lambda'}(\theta)\,d\theta}_{\delta_{\lambda}^{\lambda'}}
-1= \frac{1}{\mathcal N_n} \sum_{\lambda_\in \Lambda_n} (|a_\lambda|^2 -1),
\end{eqnarray*}
where $\delta_{\lambda}^{\lambda'}$ is the Kronecker symbol (observe that $\E[|a_\lambda|^2]=1$, hence the expected value of the integral
$\int_{\mathbb T} H_{2}(T_n(\theta))\,d\theta$ is $0$, as expected). Analogously, for $j=1,2$ we have that
\begin{eqnarray*}
\int_{\mathbb T} H_{2}(\partial_j \widetilde T_n(\theta))\,d\theta
&=& \int_{\mathbb T} \left ( \frac{2}{n}\frac{1}{\mathcal N_n}
\sum_{\lambda, \lambda'\in \Lambda_n} \lambda_j \lambda'_j
a_\lambda \overline a_{\lambda'} e_{\lambda - \lambda'}(\theta )
 -1                           \right)\,d\theta\\
 &=&\frac{2}{n}\frac{1}{\mathcal N_n}
\sum_{\lambda\in \Lambda_n} \lambda_j^2 |a_\lambda|^2 - 1
=\frac{1}{\mathcal N_n} \frac{2}{n} \sum_{\lambda \in \Lambda_n}
\lambda_j^2 (|a_\lambda|^2 - 1)\ ,
\end{eqnarray*}
{{} where the last equality follows from the elementary identity
$$
 \sum_{\lambda \in \Lambda_n}\lambda_j^2 = \frac{n \mathcal N_n} {2}.
$$
}Since $\alpha_{2n,2m}=\alpha_{2m,2n}$ we can rewrite ${\rm proj}(\Lc_n\, | \, C_{2})$ as
\begin{eqnarray*}
&&\sqrt{4\pi^2}\sqrt{\frac{n}{2}}\left( \frac{\alpha_{0,0}\beta_{2}}{2}
\frac{1}{\mathcal N_n} \sum_{\lambda_\in \Lambda_n} (|a_\lambda|^2 -1)
+ \frac{\alpha_{0,2}\beta_{0}}{2}
\frac{1}{\mathcal N_n} \frac{2}{n} \sum_{\lambda \in \Lambda_n}
\underbrace{(\lambda_1^2+\lambda_2^2)}_{=n} (|a_\lambda|^2 - 1)
  \right)\\
&&=\sqrt{4\pi^2}\sqrt{\frac{n}{2}}\frac{1}{2\mathcal N_n}
\left( \alpha_{0,0}\beta_{2}
 \sum_{\lambda_\in \Lambda_n} (|a_\lambda|^2 -1)
+ 2\alpha_{0,2}\beta_{0}
 \sum_{\lambda \in \Lambda_n}
 (|a_\lambda|^2 - 1)
  \right)\\
&&=\sqrt{4\pi^2}\sqrt{\frac{n}{2}}\frac{1}{2\mathcal N_n}
\left( \alpha_{0,0}\beta_{2}+ 2\alpha_{0,2}\beta_{0}\right)
 \sum_{\lambda_\in \Lambda_n} (|a_\lambda|^2 -1)\ .
\end{eqnarray*}
Easy computations show that
$$
\alpha_{0,0} = \sqrt{\frac{\pi}{2}}\ ,\quad \alpha_{0,2}=\alpha_{2,0} = \frac12
\sqrt{\frac{\pi}{2}}\ , \quad \beta_0 =\frac{1}{ \sqrt{2\pi}}\ , \quad \beta_2 = - \frac{1}{ \sqrt{2\pi}}\ ,
$$
and therefore
\begin{eqnarray*}
{\rm proj}(\Lc_n\, | \, C_{2}) &=&\sqrt{4\pi^2}\sqrt{\frac{n}{2}}\frac{1}{2\mathcal N_n}
\left(- \sqrt{\frac{\pi}{2}}\frac{1}{\sqrt{2\pi}}
+ 2\frac{1}{2}\sqrt{\frac{\pi}{2}}\frac{1}{\sqrt{2\pi}}
\right)
 \sum_{\lambda_\in \Lambda_n} (|a_\lambda|^2 -1)
\\
&=&\sqrt{4\pi^2}\sqrt{\frac{n}{2}}\frac{1}{2\mathcal N_n}
\left(- \frac12 + \frac12
\right)
 \sum_{\lambda_\in \Lambda_n} (|a_\lambda|^2 -1) = 0,
\end{eqnarray*}
thus concluding the proof.
\end{proof}

\section{Asymptotic study of ${\rm proj}(\Lc_n \, | \, C_4)$}

\subsection{Preliminary considerations}

As anticipated in the Introduction, one of the main findings of the present chapter is that, whenever $\Nc_{n_j}\to \infty$, the asymptotic behaviour of the normalized sequence $\{ \widetilde{\Lc}_{n_j}\}$ in \eqref{e:culp} is completely determined by that of the fourth order chaotic projections
\begin{equation}\label{e:ron}
{\rm proj}(\widetilde{\Lc}_{n_j} \, | \, C_4) =  \frac{ {\rm proj}(\mathcal{L}_{n_j} \, | \, C_4)}{\sqrt{\Var[\mathcal{L}_{n_j} ]}} , \quad j\geq 1.
\end{equation}
The aim of this section is to provide a detailed asymptotic study of the sequence appearing in \eqref{e:ron}, by using in particular the explicit formula \eqref{e:pp}. For the rest of the chapter, we use the notation
\begin{equation}\label{e:pzi}
\psi(\eta ) := \frac{3+\eta}{8}, \quad \eta\in [-1,1],
\end{equation}
and will exploit the following elementary relations, valid as $\Nc_{n_j}\to \infty$:
\begin{enumerate}
\item[(i)] if $\widehat{\mu}_{n_j}(4)\to \eta\in [-1,1]$, then, for $\ell=1,2$,
\begin{equation}\label{e:easy1}
\frac{2}{n_j^2 \Nc_{n_j}} \sum_{\substack{\lambda = (\lambda_1,\lambda_2)\in \Lambda_{n_j} \\ \lambda_2\geq 0 }}\!\!\!\lambda_\ell^4 \longrightarrow \psi(\eta);
\end{equation}

\item[(ii)] if $\widehat{\mu}_{n_j}(4)\to \eta\in [-1,1]$, then
\begin{equation}\label{e:easy2}
\frac{2}{n_j^2 \Nc_{n_j}} \sum_{\substack{\lambda = (\lambda_1,\lambda_2)\in \Lambda_{n_j} \\ \lambda_2\geq 0 }}\!\!\!\lambda_1^2\lambda_2^2 \longrightarrow \frac12 - \psi(\eta).
\end{equation}
\end{enumerate}
Note that  \eqref{e:easy1}--\eqref{e:easy2} follow immediately from the fact that, for every $n$,
$$
\widehat{\mu}_n(4) = \frac{1}{n^2 \Nc_n} \sum_{\lambda\in \Lambda_n} \left(\lambda_1^4 +\lambda_2^4-6\lambda_1^2\lambda_2^2\right),
$$
as well as from elementary symmetry considerations. We will also use the identity:
\begin{equation}\label{e:magic}
64\, \psi(\eta)^2-48\, \psi(\eta)+10 = \eta^2+1.
\end{equation}

One of our principal tools will be the following multivariate central limit theorem (CLT).

\begin{proposition}\label{p:clt} Assume that the subsequence $\{n_j\}\subset S$ is such that $\Nc_{n_j} \to \infty$ and $\widehat{\mu}_{n_j}(4) \to \eta\in [-1,1]$. Define
\begin{eqnarray*}
H(n_j)=\left(
\begin{array}{c}
H_{1}(n_j) \\
H_{2}(n_j) \\
H_{3}(n_j) \\
H_{4}(n_j)
\end{array}%
\right ) & :=& \frac{1}{n_j\sqrt{\Nc_{n_j}/2}}\sum_{\substack{\lambda = (\lambda_1,\lambda_2)\in \Lambda_{n_j} \\ \lambda_2\geq 0 }}\left(
|a_{\lambda }|^{2}-1\right) \left(
\begin{array}{c}
n_j \\
\lambda _{1}^{2} \\
\lambda _{2}^{2} \\
\lambda _{1}\lambda _{2}%
\end{array}%
\right) .
\end{eqnarray*}
Then, as $n_j\to \infty$, the following CLT holds:
\begin{eqnarray}\label{e:punz}
H(n_j) \stackrel{\rm d}{\longrightarrow} Z(\eta)= \left(
\begin{array}{c}
Z_{1} \\
Z_{2} \\
Z_{3} \\
Z_{4}%
\end{array}%
\right )\text{ ,}
\end{eqnarray}
where $Z(\eta)$ is a centered Gaussian vector with covariance
\begin{equation}\label{e:sig}
\Sigma =\left(
\begin{array}{cccc}
1 & \frac{1}{2} & \frac{1}{2} & 0 \\
\frac{1}{2} & \psi  & \frac{1}{2}-\psi  & 0 \\
\frac{1}{2} & \frac{1}{2}-\psi  & \psi  & 0 \\
0 & 0 & 0 & \frac{1}{2}-\psi
\end{array}%
\right) \text{.}
\end{equation}%
with $\psi = \psi(\eta)$, as defined in \eqref{e:pzi}.
\end{proposition}
\noindent\begin{proof}
Each component of the vector $H(n_j)$ is an element of the second Wiener chaos associated with ${\bf A}$ (see Section \ref{ss:berryintro}). As a consequence, according e.g. to \cite[Theorem 6.2.3]{noupebook}, in order to prove the desired result it is enough to establish the following relations (as $n_j\to \infty$): (a) the covariance matrix of $H(n_j)$ converges to $\Sigma$, and (b) for every $k=1,2,3,4$, $H_k(n_j)$ converges in distribution to a one-dimensional centered Gaussian random variable. Point (a) follows by a direct computation based on formulae \eqref{e:easy1}--\eqref{e:easy2}, as well as on the fact that the random variables in the set $$\left\{
|a_{\lambda }|^{2}-1  : \lambda \in \Lambda_{n_j}, \, \lambda_2\geq 0 \right\}$$ are centered, independent, identically distributed and with unit variance. To prove Point (b), write $\Lambda_{n_j}^+ := \{ \lambda \in \Lambda_{n_j}, \, \lambda_2\geq 0\}$ and observe that, for every $k$ and every $n_j$, the random variable $H_k(n_j)$ has the form
$$
H_k(n_j) = \sum_{\lambda \in \Lambda_{n_j}^+} c_k(n_j, \lambda)\times  (|a_{\lambda }|^{2}-1)
$$
where $\{c_k(n_j, \lambda)\}$ is a collection of positive deterministic coefficients such that $$\max_{\lambda \in  \Lambda_{n_j}^+} c_k(n_j, \lambda)\to 0,$$ as $n_j\to\infty$. An application of the Lindeberg criterion, e.g. in the quantitative form stated in \cite[Proposition 11.1.3]{noupebook}, consequently yields that $H_k(n_j)$ converges in distribution to a Gaussian random variable, thus concluding the proof.
\end{proof}

\begin{remark}{\rm
The eigenvalues of the matrix $\Sigma $ are given by: $\left\{ 0,\frac{3}{2},%
\frac{1}{2}-\psi ,2\psi -\frac{1}{2}\right\} .$
}
\end{remark}

Following \cite{AmP}, we will abundantly use the structure of
the {\it length-$4$ correlation set of  frequencies}:
$$
S_n(4) := \lbrace  (\lambda, \lambda', \lambda'', \lambda''')\in (\Lambda_n)^4 :
\lambda + \dots + \lambda''' = 0          \rbrace\ .
$$
It is easily seen that an element of $S_n(4)$ necessarily verifies one of the following properties (A)--(C):
\begin{eqnarray*}
{\rm (A)} &:&\left\{
\begin{array}{l}
\lambda = - \lambda ^{\prime } \\
\lambda ^{\prime \prime }= - \lambda ^{\prime \prime \prime }%
\end{array}%
\right. , \\
{\rm (B)} &:&\left\{
\begin{array}{l}
\lambda =-\lambda ^{\prime \prime } \\
\lambda ^{\prime }=-\lambda ^{\prime \prime \prime }%
\end{array}%
\right. , \\
{\rm (C)} &:&\left\{
\begin{array}{l}
\lambda =- \lambda ^{\prime \prime \prime } \\
\lambda ^{\prime }= - \lambda ^{\prime \prime }%
\end{array}%
\right. .
\end{eqnarray*}%
We also have the following identities between sets:
\[
\{(\lambda, \lambda', \lambda'', \lambda''') : \mbox{(A) and (B) are verified} \} =  \left\{ (\lambda, \lambda', \lambda'', \lambda''') : \lambda = - \lambda ^{\prime }=-\lambda ^{\prime \prime
}=\lambda ^{\prime \prime \prime }\right\} ,
\]%
\[
\{(\lambda, \lambda', \lambda'', \lambda''') : \mbox{(A) and (C) are verified} \}=  \left\{(\lambda, \lambda', \lambda'', \lambda''') : \lambda =- \lambda ^{\prime }=\lambda ^{\prime \prime
}=- \lambda ^{\prime \prime \prime }\right\} ,
\]%
\[
\{(\lambda, \lambda', \lambda'', \lambda''') : \mbox{(B) and (C) are verified} \} =  \left\{(\lambda, \lambda', \lambda'', \lambda''') :  \lambda =\lambda ^{\prime }=-\lambda ^{\prime \prime
}=- \lambda ^{\prime \prime \prime }\right\} ,
\]
whereas, for $n\ne 0$, there is no element of $S_n(4)$ simultaneously verifying (A), (B) and (C).  In view of these remarks, we can apply the inclusion-exclusion principle to deduce that, for $n\neq 0$,
\[
| S_{n}(4)| =3\Nc_{n}(\Nc_{n}-1)\text{ .}
\]
In the next subsections, we will establish a non-central limit theorem for the sequence defined in \eqref{e:ron}.

\subsection{Non-central convergence of the fourth chaotic projection: statement}

One of the main achievements of the present chapter is the following statement.

\begin{proposition}\label{p:4nclt} Let $\{n_j\} \subset S$ be such that $\Nc_{n_j}\to \infty$ and $\widehat{\mu}_{n_j}\to \eta\in [-1,1]$; set
\begin{equation}\label{e:varz}
v(n_j) :=\sqrt{\frac{4\pi^2n_j}{512}}\frac{1}{\Nc_{n_j}}, \quad j\geq 1.
\end{equation}
Then, as $n_j\to \infty$,
\begin{equation}\label{e:4nclt}
Q(n_j) := \frac{ {\rm proj}(\mathcal{L}_{n_j} \, | \, C_4)}{v(n_j) } \stackrel{\rm d}{\longrightarrow} 1+Z_1^2-2Z_2^2-2Z_3^2-4Z_4^2,
\end{equation}
where the four-dimensional vector $Z^{\top} = Z^{\top}(\eta) = (Z_1,Z_2,Z_3,Z_4)$ is defined in \eqref{e:punz}. Moreover, one has that
\begin{equation}\label{e:magic2}
{\rm Var} ( 1+Z_1^2-2Z_2^2-2Z_3^2-4Z_4^2) = 1+ \eta^2.
\end{equation}
\end{proposition}
Since the multidimensional CLT stated in \eqref{e:punz} implies that
$$
(H_1(n_j)^2, H_2(n_j)^2, H_3(n_j)^2, H_4(n_j)^2) \stackrel{\rm d}{\longrightarrow} (Z_1^2, Z_2^2, Z_3^2, Z_4^2),
$$
in order to prove Proposition \ref{p:4nclt} it is sufficient to show that
\begin{eqnarray}\label{e:zut}
&&Q(n_j) = H_1(n_j) ^2-2H_2(n_j)-2H_3(n_j)^2 - 4H_4(n_j)^2+ R(n_j),
\end{eqnarray}
where, as $n_j\to \infty$, the sequence of random variables $\{R_{n_j}\}$ converges in probability to some numerical constant $\alpha\in \R$. To see this, observe that, in view of \eqref{eq:var leading KKW} and of the orthogonal chaotic decomposition \eqref{e:chaos}, one has that $\{Q(n_j)\}$ is a centered sequence of random variables such that $\sup_j \E Q(n_j)^2 <\infty$.  By uniform integrability, this fact implies that, as $n_j\to \infty$,
$$
\E[Z_1^2-2Z_2^2-2Z_3^2-4Z_4^2 ]+\alpha = \lim_{n_j\to \infty} \E[Q(n_j)] = \lim_{n_j\to \infty} 0 = 0,
$$
and therefore, since $\E[ Z_1^2-2Z_2^2-2Z_3^2-4Z_4^2 ] = 1-2\psi-2\psi-4(2^{-1}-\psi)=-1$, one has necessarily that $\alpha=1$. (Note that our results yield indeed a rather explicit representation of the term $R(n_j)$, so that the fact that $\alpha=1$ can alternatively be verified by a careful bookkeeping of the constants appearing in the computations to follow).

\medskip

Our proof of \eqref{e:zut} is based on a number of technical results that are gathered together in the next subsection. These results will be combined with the following representation of ${\rm proj}(\mathcal{L}_{n_j} \, | \, C_4)$, that is a direct consequence of \eqref{e:pp} in the case $q=2$:
\begin{eqnarray}\label{e:nus}
{\rm proj}(\mathcal{L}_{n_j} \, | \, C_4)&=&\sqrt{4\pi^2}\sqrt{\frac{n}{2}}\Big(\frac{\alpha_{0,0}\beta_4}{4!}\int_{\mathbb
T} H_4(T_n(\theta))\,d\theta\\ \notag
&&+\frac{\alpha_{0,4}\beta_0}{4!}\int_{\mathbb T} H_4(\partial_2
\widetilde T_n(\theta))\,d\theta
+\frac{\alpha_{4,0}\beta_0}{4!}\int_{\mathbb T} H_4(\partial_1
\widetilde T_n(\theta))\,d\theta+\\ \notag
&&+\frac{\alpha_{0,2}\beta_2}{2!2!}\int_{\mathbb T}
H_2(T_n(\theta))H_2(\partial_2 \widetilde
T_n(\theta))\,d\theta\\
\notag
&&+\frac{\alpha_{2,0}\beta_2}{2!2!}\int_{\mathbb T}
H_2(T_n(\theta))H_2(\partial_1 \widetilde
T_n(\theta))\,d\theta+\\ \notag
&& +\frac{\alpha_{2,2}\beta_0}{2!2!}\int_{\mathbb T}
H_2(\partial_1 \widetilde T_n(\theta))H_2(\partial_2 \widetilde
T_n(\theta))\,d\theta \Big),
\end{eqnarray}
where the coefficients $\alpha_{\cdot, \cdot}$ and $\beta_{\cdot}$ are defined according to equation \eqref{e:alpha} and equation \eqref{e:beta}, respectively.

\subsection{Some ancillary lemmas}

The next four lemmas provide some useful representations for the six summands appearing on the right-hand side of \eqref{e:nus}. In what follows, $n$ always indicates a positive integer different from zero and, moreover, in order to simplify the discussion we sometimes use the shorthand notation:
\begin{eqnarray*}
&& \sum_{\lambda} = \sum_{\substack{\lambda = (\lambda_1,\lambda_2)\in \Lambda_{n} }}, \quad\quad \sum_{\lambda, \lambda'} = \sum_{\substack{\lambda, \lambda' \in \Lambda_{n} }}  \quad \quad\mbox{and}\quad \quad \sum_{\lambda : \lambda_2\geq 0} = \sum_{\substack{\lambda = (\lambda_1,\lambda_2)\in \Lambda_{n} \\ \lambda_2\geq 0 }},
\end{eqnarray*}
in such a way that the exact value of the integer $n$ will always be clear from the context. Also, the symbol $\{n_j\}$ will systematically indicate a subsequence of integers contained in $S$ such that $\Nc_{n_j}\to \infty$ and $\widehat{\mu}_{n_j} (4) \to \eta\in [-1,1]$, as $n_j\to \infty$. As it is customary, we write `$\stackrel{\mathbb{P}}\longrightarrow$' to denote convergence in probability, and we use the symbol $o_{\P}(1)$ to indicate a sequence of random variables converging to zero in probability, as $\mathcal{N}_n\to\infty$.

\begin{lemma}\label{lem1} One has the following representation:
$$\displaylines{
\int_{\mathbb T} H_{4}(T_n(\theta))\,d\theta
=\frac{6}{\Nc_{n}}\left (\frac{1}{\sqrt{\Nc_{n}/2}}\sum_{\lambda :\lambda
_{2}\geq 0}(|a_{\lambda }|^{2}-1) + o_{\P}(1)  \right )^{2}-\frac{3}{\Nc_{n}^{2}}\sum_{\lambda
}|a_{\lambda }|^{4}.
}$$
Also, as $n_j\to\infty$,
$$
\frac{3}{\Nc_{n_j}}\sum_{\lambda
}|a_{\lambda }|^{4}\stackrel{\mathbb{P}}\longrightarrow 6.
$$

\end{lemma}
\noindent\begin{proof}
Using the explicit expression $H_4(x)=x^4 - 6x^2 +3$, we deduce that
\begin{eqnarray*}
\int_{\mathbb T} H_{4}(T_n(\theta))\,d\theta &=& \int_{\mathbb T}
\left ( T_n(\theta)^4 -6T_n(\theta)^2  +3 \right)\,d\theta\\
&=& \frac{1}{\mathcal N_n^2}
\sum_{\lambda, \dots, \lambda''' \in \Lambda_n} a_{\lambda} \overline a_{\lambda'} a_{\lambda''}
\overline a_{\lambda'''}
\int_{\mathbb T}
\exp(2\pi i\langle \lambda -\lambda' + \lambda'' - \lambda''', \theta\rangle )\,d\theta+\\
 && -\, 6\frac{1}{\mathcal N_n}
\sum_{\lambda, \lambda' \in \Lambda_n} a_{\lambda}
\overline a_{\lambda'}
\int_{\mathbb T}\exp(2\pi i\langle \lambda - \lambda', \theta\rangle )\,d\theta   +3 \\
&=& \frac{1}{\mathcal N_n^2}
\sum_{ \lambda - \lambda' + \lambda'' - \lambda'''=0} a_{\lambda} \overline a_{\lambda'} a_{\lambda''}
\overline a_{\lambda'''}
  -6\frac{1}{\mathcal N_n}
\sum_{\lambda \in \Lambda_n} |a_{\lambda}|^2
 +3,
 \end{eqnarray*}
where the subscript $ \lambda - \lambda' + \lambda'' - \lambda'''=0$ indicates that $(\lambda, -\lambda', \lambda'', -\lambda''')\in S_n(4)$. Owing to the structure of $S_n(4)$ discussed above,
the previous expression simplifies to
\begin{eqnarray*}
&&3\frac{1}{\mathcal N_n^2}\Big( \sum_{\lambda, \lambda'\in \Lambda_n} |a_\lambda|^2 |a_{\lambda'}|^2
   -\sum_\lambda |a_{\lambda}|^4   \Big)
  -6\frac{1}{\mathcal N_n}
\sum_{\lambda \in \Lambda_n} |a_{\lambda}|^2 +3 \\
&& =3\frac{1}{\mathcal N_n} \Big( \frac{1}{\sqrt{\mathcal N_n}}
   \sum_{\lambda \in \Lambda_n} (|a_\lambda|^2-1 ) \Big)^2 -
   3\frac{1}{\mathcal N_n^2}
   \sum_{\lambda \in \Lambda_n} |a_\lambda|^4\\
   && =  \frac{6}{\Nc_{n}}\left (\frac{1}{\sqrt{\Nc_{n}/2}}\sum_{\lambda
:\lambda _{2}\geq 0}(|a_{\lambda }|^{2}-1) +o_{\P}(1) \right
)^{2}-\frac{3}{\Nc_{n}^{2}}\sum_{\lambda }|a_{\lambda }|^{4},
\end{eqnarray*}
where $o_{\P}(1) =  - (2\Nc_{n_j})^{-1/2} (|a_{(n^{-1/2}, 0)}|^2+|a_{(-n^{-1/2}, 0)}|^2-2)$, thus immediately yielding the first part of the statement (after developing the square). The second part of the statement follows from a standard application of the law of large numbers to the sum,
$$
\frac{3}{\Nc_{n_j}}\sum_{\lambda }|a_{\lambda }|^{4} =\frac{3}{\Nc_{n_j}/2}\sum_{\lambda : \lambda_2\geq 0}|a_{\lambda }|^{4} + o_{\P}(1),
$$
as well as from the observation that the set $\{|a_\lambda|^4 : \lambda\in \Lambda_{n_j}, \, \lambda_2\geq 0\}$ is composed of i.i.d. random variables such that $\E |a_\lambda|^4 =2$.

\end{proof}
\begin{lemma}\label{lem2} For $\ell =1,2$,
$$\displaylines{
\int_{\mathbb T} H_{4}(\partial_\ell \widetilde T_n(\theta))\,d\theta
= \frac{24}{\Nc_{n}}\left[ \frac{1}{\sqrt{\Nc_{n}/2}}%
\sum_{\lambda ,\lambda _{2}\geq 0}\left (\frac{\lambda _{\ell}^{2}}{n}\left( |a_{\lambda
}|^{2}-1\right) \right) + o_{\P}(1) \right] ^{2} +\cr
-\left( \frac{2}{n}\right) ^{2}\frac{3}{\Nc_{n}^{2}}\sum_{\lambda }\lambda
_{\ell}^{4}|a_{\lambda }|^{4}.}$$
Moreover, as $n_j\to \infty$,
\[
\left( \frac{2}{n_j}\right) ^{2}\frac{3}{\Nc_{n_j}}\sum_{\lambda }\lambda
_{\ell}^{4}|a_{\lambda }|^{4} \stackrel{\P}{\longrightarrow} 24\, \psi(\eta).
\]
\end{lemma}
\noindent\begin{proof}
The proof is similar to that of Lemma \ref{lem1}. We have that
\begin{eqnarray*}
&&\int_{\mathbb T} H_{4}(\partial_\ell \widetilde T_n(\theta))\,d\theta
=\int_{\mathbb T}
(\partial_\ell \widetilde T_n(\theta)^4 -6 \partial_\ell \widetilde T_n(\theta)^2+3)\,d\theta\\
&& = \frac{1}{\mathcal N_n^2}\frac{4}{n^2}
\sum_{\lambda, \dots, \lambda''' \in \Lambda_n} \lambda_\ell\lambda'_\ell\lambda''_\ell\lambda'''_\ell a_{\lambda} \overline a_{\lambda'} a_{\lambda''}
\overline a_{\lambda'''}
\int_{\mathbb T}
\exp(2\pi i\langle \lambda -\lambda' + \lambda'' - \lambda''', \theta\rangle )\,d\theta+\\
&&  -6\frac{1}{\mathcal N_n}\frac{2}{n}
\sum_{\lambda, \lambda' } \lambda_\ell\lambda'_\ell a_{\lambda}
\overline a_{\lambda'}
\int_{\mathbb T}\exp(2\pi i\langle \lambda - \lambda', \theta\rangle )\,d\theta   +3 \\
&&=\frac{1}{\mathcal N_n^2}\frac{4}{n^2}
\sum_{\lambda-\lambda'+\lambda''- \lambda''' =0}\lambda_\ell\lambda'_\ell\lambda''_\ell\lambda'''_\ell a_{\lambda} \overline a_{\lambda'} a_{\lambda''}
\overline a_{\lambda'''}
  -6\frac{1}{\mathcal N_n}\frac{2}{n}
\sum_{\lambda \in \Lambda_n} \lambda_\ell^2 |a_{\lambda}|^2
   +3\\
&&=\frac{3}{\mathcal N_n^2}\frac{4}{n^2}\left (
\sum_{\lambda,\lambda'}\lambda_\ell^2(\lambda'_\ell)^2 |a_{\lambda}|^2 |a_{\lambda'}|^2 - \sum_{\lambda} \lambda_\ell^4 |a_{\lambda}|^4\right )
  -6\frac{1}{\mathcal N_n}\frac{2}{n}
\sum_{\lambda \in \Lambda_n}\lambda_\ell^2 |a_{\lambda}|^2
   +3\\
&& =\frac{24}{\Nc_{n}}\left[ \frac{1}{\sqrt{\Nc_{n}/2}}%
\sum_{\lambda ,\lambda _{2}\geq 0}\left (\frac{\lambda _{\ell}^{2}}{n}\left( |a_{\lambda
}|^{2}-1\right) \right) + o_{\P}(1) \right] ^{2} -\left( \frac{2}{n}\right) ^{2}\frac{3}{\Nc_{n}^{2}}\sum_{\lambda }\lambda
_{\ell}^{4}|a_{\lambda }|^{4}\ .
\end{eqnarray*}
To conclude the proof, we observe that
\begin{eqnarray*}
&&\left( \frac{2}{n_j}\right) ^{2}\frac{3}{\Nc_{n_j}}\sum_{\lambda }\lambda
_{\ell}^{4}|a_{\lambda }|^{4} \\
&&= o_{\P}(1)+ \left( \frac{2}{n_j}\right) ^{2}\frac{3}{\Nc_{n_j}/2}\sum_{\lambda:\lambda_2\geq 0 }\lambda
_{\ell}^{4}(|a_{\lambda }|^{4} -2) +\frac{2\times 24}{n^2_j\Nc_{n_j} }\sum_{\lambda:\lambda_2\geq 0 }\lambda
_{\ell}^{4} := K_1(n_j) + K_2(n_j),
\end{eqnarray*}
so that the conclusion follows from \eqref{e:easy1}, as well as from the fact that, since the random variables $\{|a_\lambda|^4 -2: \lambda\in \Lambda_{n_j}, \, \lambda_2\geq 0\}$ are i.i.d., square-integrable and centered and $\lambda_\ell^4/n^2\leq 1$, $\E K_1(n_j)^2 = O(\Nc_{n_j}^{-1})\to 0$.
\end{proof}

\medskip

\begin{lemma}\label{lem3} One has that
\begin{eqnarray}\label{e:barber}
&& \int_{\mathbb T} H_2(T_n(\theta))\Big( H_2(\partial_1 \widetilde
T_n(\theta)) + H_2(\partial_2 \widetilde
T_n(\theta))\Big)\,d\theta\\
&&\hspace{3cm}= \frac{4}{\Nc_{n}}\left\{ \frac{1}{\sqrt{\Nc_{n}/2}}\sum_{\lambda
,\lambda _{2}\geq 0}\left( |a_{\lambda }|^{2}-1\right) +o_{\P}(1) \right\} ^{2}-\frac{2%
}{\Nc_{n}^{2}}\sum_{\lambda }|a_{\lambda ^{\prime }}|^{4}.\notag
\end{eqnarray}
\end{lemma}
\noindent \begin{proof} For $\ell=1,2$,
\begin{eqnarray*}
&& \int_{\mathbb T} H_2(T_n(\theta))H_2(\partial_\ell \widetilde
T_n(\theta))\,d\theta=\int_{\mathbb T} (T_n(\theta)^2-1)(\partial_\ell \widetilde
T_n(\theta)^2-1)\,d\theta\\
&&=\int_{\mathbb T} \left (\frac{1}{\mathcal N_n} \sum_{\lambda, \lambda'} a_\lambda \overline{a}_{\lambda'} e_\lambda(\theta) e_{-\lambda'}(\theta)-1\right )\left (\frac{2}{n}\frac{1}{\mathcal N_n} \sum_{\lambda'', \lambda'''} \lambda_\ell'' \lambda_\ell''' a_{\lambda''} \overline{a}_{\lambda'''} e_{\lambda''}(\theta) e_{-\lambda'''}(\theta)-1 \right )\,d\theta\\
&& =\frac{2}{n}\frac{1}{\mathcal N_n^2} \sum_{\lambda-\lambda'+\lambda''-\lambda'''=0} \lambda_\ell'' \lambda_\ell''' a_\lambda \overline{a}_{\lambda'} a_{\lambda''} \overline{a}_{\lambda'''}-
\frac{1}{\mathcal N_n} \sum_{\lambda} |a_\lambda|^2
-
\frac{2}{n}\frac{1}{\mathcal N_n} \sum_{\lambda}\lambda_\ell^2 |a_\lambda|^2 +1\ .
\end{eqnarray*}
An application of the inclusion-exclusion principle yields that the first summand in the previous expression equals
$$\displaylines{
\frac{2}{n}\frac{1}{\mathcal N_n^2} \left ( \sum_{\lambda,\lambda'} \lambda_j^2 |a_\lambda|^2 |a_{\lambda'}|^2 +2\sum_{\lambda,\lambda'} \lambda_j \lambda'_j |a_\lambda|^2 |a_{\lambda'}|^2 - 2\sum_\lambda \lambda_j^2|a_\lambda|^4 + \sum_\lambda \lambda_j^2|a_\lambda|^4 \right )\
}$$
Using the relation $a_{-\lambda} = \overline a_\lambda$, we also infer that
$$
\sum_{\lambda,\lambda'} \lambda_j \lambda'_j |a_\lambda|^2 |a_{\lambda'}|^2 = \left (\sum_\lambda \lambda_j |a_\lambda|^2 \right)^2 =0\ .
$$
Summing the terms corresponding to $\partial _{1}$ and $\partial _{2}$ we deduce that the left-hand side of \eqref{e:barber} equals
obtain%
\begin{eqnarray*}
&=&\frac{2}{n}\frac{1}{\Nc_{n}^{2}}\left\{ \sum_{\lambda ,\lambda ^{\prime
}}\left( \lambda _{1}^{2}+\lambda _{2}^{2}\right) |a_{\lambda
}|^{2}|a_{\lambda ^{\prime }}|^{2}-\sum_{\lambda }\left( \lambda
_{1}^{2}+\lambda _{2}^{2}\right) |a_{\lambda ^{\prime }}|^{4}\right\} \\
&&-\frac{2}{\Nc_{n}}\sum_{\lambda }|a_{\lambda }|^{2}-\frac{2}{n}\frac{1}{\Nc_{n}%
}\sum_{\lambda }\left( \lambda _{1}^{2}+\lambda _{2}^{2}\right) |a_{\lambda
}|^{2}+2= \\
&=&\frac{2}{n}\frac{1}{\Nc_{n}^{2}}\left\{ \sum_{\lambda ,\lambda ^{\prime
}}n|a_{\lambda }|^{2}|a_{\lambda ^{\prime }}|^{2}-\sum_{\lambda
}n|a_{\lambda ^{\prime }}|^{4}\right\} -\frac{2}{\Nc_{n}}\sum_{\lambda }|a_{\lambda }|^{2}-\frac{2}{n}\frac{1}{\Nc_{n}%
}\sum_{\lambda }n|a_{\lambda }|^{2}+2 \\
&=&2\frac{1}{\Nc_{n}^{2}}\left\{ \sum_{\lambda ,\lambda ^{\prime }}|a_{\lambda
}|^{2}|a_{\lambda ^{\prime }}|^{2}-\sum_{\lambda }|a_{\lambda ^{\prime
}}|^{4}\right\} -\frac{2}{\Nc_{n}}\sum_{\lambda }|a_{\lambda }|^{2}-\frac{2}{\Nc_{n}}%
\sum_{\lambda }|a_{\lambda }|^{2}+2\\
&=&\frac{2}{\Nc_{n}}\left\{ \frac{\sqrt{2}}{\sqrt{\Nc_{n}/2}}\sum_{\lambda
,\lambda _{2}\geq 0}\left( |a_{\lambda }|^{2}-1\right) +o_{\P}(1)  \right\} ^{2}-\frac{2%
}{\Nc_{n}^{2}}\sum_{\lambda }|a_{\lambda ^{\prime }}|^{4},
\end{eqnarray*}
which corresponds to the desired conclusion.
\end{proof}

\medskip

Our last lemma allows one to deal with the most challenging term appearing in formula \eqref{e:nus}.

\begin{lemma}\label{lem4}
We have that
\begin{eqnarray*}
 &&\int H_{2}(\partial _{1}\widetilde{T}_{n})H_{2}(\partial _{2}\widetilde{T}%
_{n})\,d\theta  \\
&&=-4\left[ \frac{1}{\sqrt{\Nc_{n}/2}}\frac{1}{n}\sum_{\lambda ,\lambda _{2}\geq
0}\lambda _{2}^{2}(|a_{\lambda }|^{2}-1) \right] ^{2}
-4\left[ \frac{1}{\sqrt{\Nc_{n}/2}}\frac{1}{n}\sum_{\lambda ,\lambda _{2}\geq
0}\lambda _{1}^{2}(|a_{\lambda }|^{2}-1)+o_{\P}(1)\right] ^{2}\\
&&\quad +4\left[  \frac{1}{\sqrt{\Nc_{n}/2}}\sum_{\lambda ,\lambda _{2}\geq
0}(|a_{\lambda }|^{2}-1) +o_{\P}(1)
\right] ^{2} \\
&& \quad+16\left[
\frac{1}{\sqrt{\Nc_{n}/2}}\frac{1}{n}\sum_{\lambda ,\lambda _{2}\geq
0}\lambda _{1}\lambda _{2}\left( |a_{\lambda }|^{2}-1\right)+o_{\P}(1)
\right] ^{2}
-\frac{12}{n^{2}}\frac{1}{\Nc_{n}^{2}}\sum_{\lambda
}\lambda _{1}^{2}\lambda _{2}^{2}|a_{\lambda }|^{4}\ .
\end{eqnarray*}
And the following convergence takes place as $n_j\to\infty$:
$$
\frac{12}{n_j^{2}}\frac{1}{\Nc_{n_j}^{2}}\sum_{\lambda
}\lambda _{1}^{2}\lambda _{2}^{2}|a_{\lambda }|^{4}\stackrel{\P}{\longrightarrow} 12 - 24\, \psi(\eta).
$$

\end{lemma}

\noindent\begin{proof} One has that
\begin{eqnarray}
&&\int H_{2}(\partial _{1}\widetilde{T}_{n})H_{2}(\partial _{2}\widetilde{T}%
_{n})d\theta \notag \\
&&=\frac{4}{n^{2}}\frac{1}{\Nc_{n}^{2}}\sum_{\lambda -\lambda ^{\prime
}+\lambda ^{\prime \prime }-\lambda ^{\prime \prime \prime }=0}\lambda
_{1}\lambda _{1}^{\prime }\lambda _{2}^{\prime \prime }\lambda _{2}^{\prime
\prime \prime }a_{\lambda }\overline{a}_{\lambda ^{\prime }}a_{\lambda
^{\prime \prime }}\overline{a}_{\lambda ^{\prime \prime \prime }} \label{hope1}\\
&&-\frac{2}{n}\frac{1}{\Nc_{n}}\sum_{\lambda }\lambda _{1}^{2}|a_{\lambda
}|^{2}-\frac{2}{n}\frac{1}{\Nc_{n}}\sum_{\lambda }\lambda _{2}^{2}|a_{\lambda
}|^{2}+1\text{ .}\label{hope2}
\end{eqnarray}%
First of all, we note that%
\[
\mathbb{E}\left[ \frac{2}{n}\frac{1}{\Nc_{n}}\sum_{\lambda }(\lambda
_{1}^{2}+\lambda _{2}^{2})|a_{\lambda }|^{2}\right] =\mathbb{E}\left[ \frac{2%
}{\Nc_{n}}\sum_{\lambda }|a_{\lambda }|^{2}\right] =2\text{ .}
\]%
Let us now focus on (\ref{hope1}). Using the structure of $S_4(n)$ recalled above, we obtain
\[
\frac{4}{n^{2}}\frac{1}{\Nc_{n}^{2}}\sum_{\lambda -\lambda ^{\prime }+\lambda
^{\prime \prime }-\lambda ^{\prime \prime \prime }=0}\lambda _{1}\lambda
_{1}^{\prime }\lambda _{2}^{\prime \prime }\lambda _{2}^{\prime \prime
\prime }a_{\lambda }\overline{a}_{\lambda ^{\prime }}a_{\lambda ^{\prime
\prime }}\overline{a}_{\lambda ^{\prime \prime \prime }}
\]%
\[
=\frac{4}{n^{2}}\frac{1}{\Nc_{n}^{2}}\left\{ \sum_{\lambda ,\lambda
^{\prime }}\lambda _{1}^{2}(\lambda _{2}^{\prime })^{2}|a_{\lambda
}|^{2}|a_{\lambda ^{\prime }}|^{2}+2\sum_{\lambda ,\lambda
^{\prime }}\lambda _{1}\lambda _{2}\lambda _{1}^{\prime }\lambda
_{2}^{\prime }|a_{\lambda }|^{2}|a_{\lambda ^{\prime
}}|^{2}-3\sum_{\lambda }\lambda _{1}^{2}\lambda
_{2}^{2}|a_{\lambda }|^{4}\right\}.
\]%
Let us now write%
\begin{eqnarray*}
\frac{4}{n^{2}}\frac{1}{\Nc_{n}^{2}}\sum_{\lambda ,\lambda ^{\prime }}\lambda
_{1}^{2}(\lambda _{2}^{\prime })^{2}|a_{\lambda }|^{2}|a_{\lambda ^{\prime
}}|^{2} &:=&A\text{ ,} \\
\frac{4}{n^{2}}\frac{1}{\Nc_{n}^{2}}2\sum_{\lambda ,\lambda ^{\prime }}\lambda
_{1}\lambda _{2}\lambda _{1}^{\prime }\lambda _{2}^{\prime }|a_{\lambda
}|^{2}|a_{\lambda ^{\prime }}|^{2} &:=&B\text{ ,} \\
-3\frac{4}{n^{2}}\frac{1}{\Nc_{n}^{2}}\sum_{\lambda }\lambda _{1}^{2}\lambda
_{2}^{2}|a_{\lambda }|^{4} &:=&C\text{ ,} \\
\frac{4}{n^{2}}\frac{1}{\Nc_{n}^{2}}\left\{ -N_{n}\frac{n}{2}\sum_{\lambda
}|a_{\lambda }|^{2}+\frac{\Nc_{n}^{2}n^{2}}{4}\right\}  &:=&D\text{ .}
\end{eqnarray*}%
We have that $A$ equals%
\begin{eqnarray*}
&&\frac{4}{n^{2}}\frac{1}{\Nc_{n}^{2}}\sum_{\lambda ,\lambda ^{\prime
}}\lambda _{1}^{2}(\lambda _{2}^{\prime })^{2}|a_{\lambda }|^{2}|a_{\lambda
^{\prime }}|^{2} \\
&&=\frac{4}{n^{2}}\frac{1}{\Nc_{n}^{2}}\frac{1}{2}\left\{ \sum_{\lambda
,\lambda ^{\prime }}\lambda _{1}^{2}(\lambda _{2}^{\prime })^{2}|a_{\lambda
}|^{2}|a_{\lambda ^{\prime }}|^{2}+\sum_{\lambda ,\lambda ^{\prime }}\lambda
_{1}^{2}(\lambda _{2}^{\prime })^{2}|a_{\lambda }|^{2}|a_{\lambda ^{\prime
}}|^{2}\right\}
\\
&&=\frac{4}{n^{2}}\frac{1}{\Nc_{n}^{2}}\frac{1}{2}\left\{ \sum_{\lambda ,\lambda
^{\prime }}(n-\lambda _{2}^{2})(\lambda _{2}^{\prime })^{2}|a_{\lambda
}|^{2}|a_{\lambda ^{\prime }}|^{2}+\sum_{\lambda ,\lambda ^{\prime }}\lambda
_{1}^{2}(n-(\lambda _{1}^{\prime })^{2})|a_{\lambda }|^{2}|a_{\lambda
^{\prime }}|^{2}\right\},
\end{eqnarray*}
an expression that can be rewritten as
\begin{eqnarray*}
&&\frac{4}{n^{2}}\frac{1}{\Nc_{n}^{2}}\frac{1}{2}\left\{ -\sum_{\lambda
,\lambda ^{\prime }}\lambda _{2}^{2}(\lambda _{2}^{\prime })^{2}|a_{\lambda
}|^{2}|a_{\lambda ^{\prime }}|^{2}-\sum_{\lambda ,\lambda ^{\prime }}\lambda
_{1}^{2}(\lambda _{1}^{\prime })^{2}|a_{\lambda }|^{2}|a_{\lambda ^{\prime
}}|^{2}\right\}  \\
&&+\frac{4}{n^{2}}\frac{1}{\Nc_{n}^{2}}\frac{1}{2}\left\{ n\sum_{\lambda
,\lambda ^{\prime }}(\lambda _{2}^{\prime })^{2}|a_{\lambda
}|^{2}|a_{\lambda ^{\prime }}|^{2}+n\sum_{\lambda ,\lambda ^{\prime
}}\lambda _{1}^{2}|a_{\lambda }|^{2}|a_{\lambda ^{\prime }}|^{2}\right\}
\\
&&=\frac{4}{n^{2}}\frac{1}{\Nc_{n}^{2}}\frac{1}{2}\left\{ -\sum_{\lambda
,\lambda ^{\prime }}\lambda _{2}^{2}(\lambda _{2}^{\prime })^{2}|a_{\lambda
}|^{2}|a_{\lambda ^{\prime }}|^{2}-\sum_{\lambda ,\lambda ^{\prime }}\lambda
_{1}^{2}(\lambda _{1}^{\prime })^{2}|a_{\lambda }|^{2}|a_{\lambda ^{\prime
}}|^{2}\right\}  \\
&&+\frac{4}{n^{2}}\frac{1}{\Nc_{n}^{2}}\frac{1}{2}\left\{ n\sum_{\lambda
,\lambda ^{\prime }}(\lambda _{2}^{\prime })^{2}|a_{\lambda
}|^{2}|a_{\lambda ^{\prime }}|^{2}+n\sum_{\lambda ,\lambda ^{\prime
}}\lambda _{1}^{2}|a_{\lambda }|^{2}|a_{\lambda ^{\prime }}|^{2}\right\}.
\end{eqnarray*}%
As a consequence,
\begin{eqnarray*}
A+D &=&-\frac{4}{n^{2}}\frac{1}{\Nc_{n}^{2}}\frac{1}{2}\left[ \sum_{\lambda
}\lambda _{2}^{2}(|a_{\lambda }|^{2}-1)\right] ^{2} \\
&&-\frac{4}{n^{2}}\frac{1}{\Nc_{n}^{2}}\frac{1}{2}\left[ \sum_{\lambda
}\lambda _{1}^{2}(|a_{\lambda }|^{2}-1)\right] ^{2} \\
&&+\frac{4}{\Nc_{n}^{2}}\frac{1}{2}\left[ \sum_{\lambda }(|a_{\lambda }|^{2}-1)%
\right] ^{2}.
\end{eqnarray*}%
On the other hand,
\begin{eqnarray*}
B &=&\frac{4}{n^{2}}\frac{1}{\Nc_{n}^{2}}2\sum_{\lambda ,\lambda ^{\prime
}}\lambda _{1}\lambda _{2}\lambda _{1}^{\prime }\lambda _{2}^{\prime
}|a_{\lambda }|^{2}|a_{\lambda ^{\prime }}|^{2} \\
&=&\frac{4}{n^{2}}\frac{1}{\Nc_{n}^{2}}2\left[ \sum_{\lambda }\lambda
_{1}\lambda _{2}\left( |a_{\lambda }|^{2}-1\right) \right] ^{2}.
\end{eqnarray*}%
The last assertion in the statement, which concerns the term $C$ defined above, is a direct consequence of \eqref{e:easy2} and of an argument similar to the one that concluded the proof of Lemma \ref{lem2}.

\end{proof}

\subsection{End of the proof of Proposition \ref{p:4nclt}}
Plugging the explicit expressions appearing in Lemma \ref{lem1}, Lemma \ref{lem2}, Lemma \ref{lem3} and Lemma \ref{lem4} into \eqref{e:nus} (and exploiting the fact that $p_2(1/4) =-1/8$), one deduces after some standard simplification that representation \eqref{e:zut} is indeed valid, so that the conclusion of Proposition \ref{p:4nclt} follows immediately. In order to prove relation \eqref{e:magic2}, introduce the centered Gaussian vector $\widetilde{Z}^{\top} := (\widetilde{Z}_{1}, \widetilde{Z}_{2}, \widetilde{Z}_{3}, \widetilde{Z}_{4})$, with covariance matrix given by
$$
\widetilde{\Sigma }%
:=\left(
\begin{array}{cccc}
1 & \frac{1}{2\sqrt{\psi }} & \frac{1}{2\sqrt{\psi }} & 0 \\
\frac{1}{2\sqrt{\psi }} & 1 & \frac{1}{2\psi }-1 & 0 \\
\frac{1}{2\sqrt{\psi }} & \frac{1}{2\psi }-1 & 1 & 0 \\
0 & 0 & 0 & 1%
\end{array}%
\right) \text{ .}
$$
Then, %
\begin{eqnarray*}
&& \Var\left [ {Z}_{1}^{2}-2{Z}_{2}^{2}-2{Z}_{3}^{2}-4{Z}_{4}^{2}\right ]
\\
&& =\Var\left [ \widetilde{Z}_{1}^{2}-2\psi \widetilde{Z}_{2}^{2}-2\psi
\widetilde{Z}_{3}^{2}-4(\frac{1}{2}-\psi )\widetilde{Z}_{4}^{2}\right ]  \\
&&=\Var\left [ H_{2}(\widetilde{Z}_{1})-2\psi H_{2}(\widetilde{Z}_{2})-2\psi
H_{2}(\widetilde{Z}_{3})-4\left (\frac{1}{2}-\psi \right )H_{2}(\widetilde{Z}%
_{4})\right ] \\
&&=2+8\psi ^{2}+8\psi ^{2}+32\left (\frac{1}{2}-\psi \right )^{2}-4\psi \Cov\left[ H_{2}(%
\widetilde{Z}_{1}),H_{2}(\widetilde{Z}_{2})\right]  \\
&&\quad-4\psi \Cov\left[ H_{2}(\widetilde{Z}_{1}),H_{2}(\widetilde{Z}_{3})\right]
+8\psi ^{2} \Cov\left[ H_{2}(\widetilde{Z}_{2}),H_{2}(\widetilde{Z}_{3})\right]
\\
&&=2+8\psi ^{2}+8\psi ^{2}+32(\frac{1}{2}-\psi )^{2} \\
&&\quad-8\psi (\frac{1}{2\sqrt{\psi }})^{2}-8\psi (\frac{1}{2\sqrt{\psi }}%
)^{2}+16\psi ^{2}(\frac{1}{2\psi }-1)^{2}
\\
&&=2+8\psi ^{2}+8\psi ^{2}+32(\frac{1}{4}+\psi ^{2}-\psi )-2-2+16\psi ^{2}(%
\frac{1}{4\psi ^{2}}+1-\frac{1}{\psi })
\\
&&=64\psi ^{2}-48\psi +10\text{,}
\end{eqnarray*}
and the conclusion follows from \eqref{e:magic}.

\section{End of the proof of Theorem \ref{thm:lim dist sep}}\label{s:mainproof}

A direct computation (obtained e.g. by diagonalising the covariance matrix $\Sigma$  appearing in \eqref{e:sig}) reveals that, for every $\eta\in [-1,1]$, the random variable
$$
\frac{1}{\sqrt{1+\eta^2}}\Big( 1+Z_1^2-2Z_2^2-2Z_3^2-4Z_4^2 \Big)
$$
has the same law as $\mathcal{M}_{|\eta|}$, as defined in $\eqref{e:r}$. This implies, in particular, that such a random variable has unit variance, and has a distribution that does not depend on the sign of $\eta$. Now let the assumptions and notations of Theorem \ref{thm:lim dist sep} prevail
(in particular, $\Nc_{n_j} \to \infty$). Since the sequence $\{|\widehat{\mu}_{n_j}(4)| : j\geq 1\}$ is nonnegative and bounded by 1, there exists a subsequence $\{n'_j\}$ such that $|\widehat{\mu}_{n'_j}(4)|$ converges to some $\eta\in [0,1]$. It follows that $\{n'_j\}$ necessarily contains a subsequence $\{n''_j\}\subset \{n'_j\}$ such that one of the following properties holds: either (i) $\widehat{\mu}_{n''_j}(4)$ converges to $\eta$, or (ii) $\widehat{\mu}_{n''_j}(4)$ converges to $-\eta$. Now, if $\{n''_j\}$ is of type (i), then our initial remarks together with \eqref{eq:var leading KKW} and \eqref{e:4nclt} imply that
$$
\lim_{n''_j} \frac{\Var[\mathcal L_{n''_j}]}{\Var[{\rm proj}(\Lc_{n''_j} \, | \, C_4) ]} = 1.
$$
In view of the chaotic decomposition \eqref{e:chaos2}, this result implies that, as $n''_j\to \infty$,
$$
\widetilde{\Lc}_{n''_j}  = \frac{{\rm proj}(\Lc_{n''_j} \, | \, C_4)}{\sqrt{\Var[\mathcal L_{n''_j}]}} + o_{\P}(1),
$$
and consequently that $\widetilde{\Lc}_{n''_j}$ converges in distribution to $\Mc_{\eta}$, by virtue of Proposition \ref{p:4nclt}. An analogous argument shows that, if $\{n''_j\}$ is of type (ii), then necessarily $\widetilde{\Lc}_{n''_j}$ converges in distribution to $\Mc_{| -\eta|} = \Mc_{\eta}$. The results described above readily imply the following three facts: (a) if the subsequence $\{n'_j\}\subset \{n_j\}$ is such that $|\widehat{\mu}_{n'_j}(4)|\to \eta\in [0,1]$, then $\widetilde{\Lc}_{n'_j}\stackrel{\rm d}{\longrightarrow}\Mc_{\eta}$, (b) any subsequence $\{n'_j\}\subset \{n_j\} $ contains a further subsequence $\{n''_j\}\subset \{n'_j\}$ such that $|\widehat{\mu}_{n''_j}(4)|$ converges to some $\eta\in [0,1]$, and therefore $\widetilde{\Lc}_{n''_j}\stackrel{\rm d}{\longrightarrow}\Mc_{\eta}$, and (c) if the subsequence $\{n'_j\}$ is such that $|\widehat{\mu}_{n'_j}(4)|$ is not converging, then $\widetilde{\Lc}_{n'_j}$ is not converging in distribution, since in this case the set $\{{\bf D}(\widetilde{\Lc}_{n'_j})\}$ has necessarily two distinct adherent points (thanks to Point 4 in Proposition \ref{p:meta}). The first part of the statement is therefore proved. To prove \eqref{e:b}, use fact (b) above to deduce that, for every subsequence $\{n'_j\}$, there exists a further subsequence $\{n''_j\}$ such that $\widetilde{\Lc}_{n''_j}\stackrel{\rm d}{\longrightarrow}\Mc_{\eta}$ and $\Mc^{n''_j} \stackrel{\rm d}{\longrightarrow}\Mc_{\eta}$ (where we have used the notation \eqref{e:k}), and consequently
$$
d\big(\widetilde{\Lc}_{n''_j} \,,\,  \Mc^{n''_j}\big)\leq  d\big(\widetilde{\Lc}_{n''_j}\, ,\, \Mc_\eta \big)+d\big(\Mc_\eta\,,\, \Mc^{n''_j}\big)\longrightarrow 0, \quad n_j''\to \infty.
$$
The previous asymptotic relation is obvious whenever $d$ is a distance
metrizing weak convergence on $\mathscr{P}$. To deal with the case where
$d=d_K$ equals the Kolmogorov distance, one has to use the standard fact that,
since $\Mc_\eta$ has an absolutely continuous distribution,
then $X_n \stackrel{\rm d}{\longrightarrow} \Mc_\eta$ if and only if
$d_K(X_n,\Mc_\eta)\longrightarrow 0$. The proof of Theorem  \ref{thm:lim dist sep} is complete.

\chapter*{Part 3\\ Spin random fields}
\addcontentsline{toc}{chapter}{Part 3: Spin random fields}

\chapter{Representation of Gaussian isotropic spin random fields}

This chapter is based on the second part of \cite{mauspin}: we investigate spin random fields on the sphere, extending the representation formula for Gaussian isotropic random fields on homogeneous spaces of a compact group in Chapter 2 to the spin case. Moreover we introduce a powerful tool for studying spin random fields and more generally random sections of homogeneous vector bundles, that is, the ``pullback'' random field.

\section{Random sections of vector bundles}
We now investigate the case of Gaussian isotropic spin random fields
on $\cS^2$, with the aim of extending the representation result
of Theorem \ref{real-general}. As stated in the Introduction of this thesis, these models have
received recently much attention (see \cite{bib:LS}, \cite{malyarenko} or \cite{dogiocam}),
being motivated by the modeling of CMB data. Actually our point of view begins from \cite{malyarenko}.

We consider first the case of a general vector bundle. Let $\xi= (E,
p, B)$ be a finite-dimensional \emph{complex vector bundle} on the
topological space $B$, which is called the \emph{base space}. The
surjective map
\begin{equation}
p: E\goto B
\end{equation}
is the
\emph{bundle projection}, $p^{-1}(x)$, $x\in B$ is the {\it fiber}
above $x$.
Let us denote $\B(B)$ the Borel $\sigma$-field of $B$.
A section of $\xi$ is a map $u: B \to E$ such that $p \circ u=id_B$. As $E$ is itself a topological space, we can speak of continuous sections.

We suppose from now on that every fibre $p^{-1}(x)$ carries an inner
product and a measure $\mu$ is given on the base space. Hence we can
consider square integrable sections, as those such that
$$
\int_B\langle u(x),u(x)\rangle_{p^{-1}(x)}\,d\mu(x)<+\infty
$$
and define the corresponding $L^2$ space accordingly.

Let $(\Omega, \F, \P)$ be a probability space.
\begin{definition}\label{definizione di sezione aleatoria}
A \emph{random section $T$
of the vector bundle $\xi$} is a collection of $E$-valued random variables
$(T_x)_{x\in B}$ indexed by the elements of the base space $B$ such that
the map $\Omega \times B \ni(\omega, x)  \mapsto T_x(\omega)$
is $\F \otimes \B(B)$-measurable and, for every $\omega$, the path
$$
B\ni x\to T_x(\omega)\in E
$$
is a section of $\xi$, i.e. $p\circ T_\cdot(\omega) =id_B$.
\end{definition}
Continuity of a random section $T$ is easily defined by requiring that
for every $\omega\in \Omega$ the map $x \mapsto T_x$
is a continuous section of $\xi$. Similarly a.s. continuity is defined.
A random section $T$ of $\xi$ is a.s. square integrable if  the map
$x \mapsto \| T_x (\omega)\|^2 _{p^{-1}(x)}$ is a.s. integrable, it is second order if $\E[ \| T_x \|^2_{p^{-1}(x)}] < +\infty$ for every $x\in B$ and
mean square integrable
if
$$
\E\Bigl[\int_B\| T_x \|^2 _{p^{-1}(x)} \, d\mu(x)\Bigr]< +\infty\ .
$$
As already remarked in $\cite{malyarenko}$, in defining the notion of mean square continuity for
a random section, the naive approach
$$
\lim_{y\to x} \E [\| T_x - T_y \|^2] =0
$$
is not immediately meaningful as $T_x$ and $T_y$ belong to different
(even if possibly isomorphic) spaces (i.e. the fibers).
A similar difficulty arises for the definition of strict sense invariance w.r.t. the action of a topological group on the bundle.
We shall investigate these points below.

A case of particular interest to us are the homogeneous (or twisted) vector bundles. Let $G$ be a compact group, $K$ a closed subgroup and $\X=G/K$.
Given an irreducible unitary representation $\tau$ of $K$ on the complex (finite-dimensional) Hilbert space $H$,
$K$ acts on the Cartesian product $G\times H$ by the action
$$
k(g,z):= (gk, \tau(k^{-1})z)\ .
$$
Let   $G\times_\tau H=\lbrace \theta(g,z) : (g,z) \in G\times H
\rbrace$ denote the quotient space of the orbits $\theta(g,z) =
\lbrace (gk, \tau(k^{-1})z) : k\in K \rbrace$ under the above
action. $G$ acts on $G\times_\tau H$ by
\begin{equation}\label{action}
h \theta(g,z) := \theta(hg,z)\ .
\end{equation}
The map $G\times H \to \X: (g,z)\to gK$ is constant on the
orbits $\theta(g,z)$ and   induces the projection
$$
G\times_\tau H\ni\theta(g,z)\enspace\mathop{\to}^{\pi_\tau\,}\enspace gK\in \X
$$
which is a continuous $G$-equivariant map. $\xi_\tau=
(G\times_\tau H, \pi_\tau, \X)$ is a $G$-vector bundle: it is the \emph{homogeneous vector bundle associated to
the representation $\tau$}. The  fiber
$\pi^{-1}_\tau(x)$ is isomorphic to $H$ for every $x\in \X$ (see
\cite{B-D}). More precisely, for $x\in\X$ the fiber $\pi_\tau^{-1}(x)$
is the set of
elements of the form $\th(g,z)$ such that $gK=x$. We define the scalar
product of two such elements as
\begin{equation}\label{prod scalare}
\langle \th(g,z),\th(g,w)\rangle_{\pi_\tau^{-1}(x)}=\langle z,w\rangle_{H}
\end{equation}
for some fixed $g\in G$ such that $gK=x$, as it is immediate that this
definition does not depend on the choice of such a $g$.
Given a function $f:G \to H$ satisfying
\begin{equation}\label{funzioni di tipo W}
f(gk)=\tau(k^{-1})f(g)\ ,
\end{equation}
then to it we can associate the section of $\xi_\tau$
\begin{equation}\label{proiezione}
u(x)=u(gK)=\theta(g,f(g))
\end{equation}
as again this is a good definition, not depending
of the choice of $g$ in the coset $gK$. The functions $f$ satisfying to
(\ref{funzioni di tipo W}) are called right $K$-covariant functions
of type $\tau$ (\emph{functions of type $\tau$} from now on).

More interestingly, also the converse is true.
\begin{prop}\label{pullback-s-deterministic}
Given a section $u$ of $\xi_\tau$, there exists a unique function
$f$ of type $\tau$ on $G$ such that $u(x)=\theta(g,f(g))$ where
$gK=x$. Moreover $u$ is continuous if and only if
$f:G\to H$ is continuous.
\end{prop}
\begin{proof} Let $(g_x)_{x\in\X}$ be
a measurable selection such that
$g_xK=x$ for every $x\in\X$. If $u(x)=\theta(g_x, z)$, then define
$f(g_x):=z$ and extend the definition to the elements of the coset
$g_xK$ by $f(g_xk):= \tau(k^{-1})z$; it is easy to check that such a
$f$ is of type $\tau$, satisfies (\ref{proiezione}) and is the unique
function of type $\tau$ with this property.

Note that the continuity of $f$ is equivalent to the continuity of
the map
\begin{equation}\label{mappa1}
F: g\in G \to (g,f(g))\in G\times H\ .
\end{equation}
Denote $pr_1: G\to \X$ the canonical projection onto the quotient space $\X$
and $pr_2: G\times H \to G\times_\tau H$ the canonical projection
onto the quotient space $G\times_\tau H$. It is immediate that
\begin{equation*}\label{mappa2}
pr_2 \circ F = u \circ pr_1\ .
\end{equation*}
Therefore $F$ is continuous if and only if $u$ is continuous,
the projections $pr_1$ and $pr_2$ being continuous open mappings.

\end{proof}
We shall again call $f$ the \emph {pullback} of $u$.
Remark that, given two sections $u_1, u_2$ of $\xi_\tau$ and their respective pullbacks $f_1,f_2$, we have
\begin{equation}\label{prod_scalare}
\langle u_1, u_2 \rangle := \int_\X \langle u_1(x),
u_2(x)\rangle_{\pi_\tau^{-1}(x)}\,dx=
\int_G \langle f_1(g),f_2(g)\rangle_H\,dg
\end{equation}
so that $u\longleftrightarrow f$ is an isometry between the space $L^2(\xi_\tau)$ of
the square integrable sections of $\xi_\tau$ and the space $L^2_\tau(G)$ of the square
integrable functions of type $\tau$.

The left regular action of $G$ on $L^2_\tau(G)$ (also called the
\emph{representation of $G$ induced by $\tau$})
$$
L_h f(g) := f(h^{-1}g)
$$
can be equivalently realized on $L^2(\xi_\tau)$ by
\begin{equation}\label{indotta}
 U_h u(x) = h u(h^{-1}x)\ .
\end{equation}
We have
$$
U_h u(gK) = h u(h^{-1}gK) = h \theta( h^{-1} g, f(h^{-1}g)) =
\theta(g, f(h^{-1}g)) = \theta(g, L_h f(g))
$$
so that, thanks to the uniqueness of the pullback function:
\begin{prop}\label{action-pullback}
If $f$ is the pullback function of the section $u$ then $L_hf$ is
the pullback of the section $U_hu$.
\end{prop}
Let  $T=(T_x)_{x\in \X}$ be
a random section of the homogeneous vector bundle $\xi_\tau$. As, for
fixed $\omega$, $x\mapsto T_x(\omega)$ is a section of $\xi_\tau$, by Proposition
\ref{pullback-s-deterministic} there
exists a unique function $g\mapsto X_g(\omega)$ of type $\tau$ such that
$T_{gK}(\omega)=\theta(g, X_g(\omega))$. We refer to the
random field $X=(X_g)_{g\in G}$ as the \emph{pullback random field
of $T$}. It is a random field on $G$ of type $\tau$, i.e.  $X_{gk}(\omega)=\tau(k^{-1})X_g(\omega)$ for each $\omega$.
Conversely every random field $X$ on $G$ of type $\tau$ uniquely defines a
random section of $\xi_\tau$ whose pullback random field is $X$. It is immediate that
\begin{prop}\label{prop-pull1}
Let $T$ be a random section of $\xi_\tau$.

a) $T$  is  a.s. square integrable if and only if its pullback random field $X$ is a.s. square
integrable.

b) $T$ is second order if and only if its pullback random field $X$ is second order.

c) $T$ is a.s. continuous if and only if its pullback random field $X$ is a.s. continuous.
\end{prop}
Proposition \ref{prop-pull1} introduces the fact that many properties of random sections of
the homogeneous bundle can be stated or investigated through corresponding properties of
their pullbacks, which are just ordinary random fields to whom the results of previous
sections can be applied. A first instance is the following definition.
\begin{definition}\label{definizione di continuita in media quadratica}
The random section $T$ of the homogeneous vector bundle $\xi_\tau$ is
said to be \emph{mean square continuous} if its  pullback  random
field $X$ is mean square continuous, i.e.,
\begin{equation}
\lim_{h\to g} \E [ \| X_h - X_g \|_H^2 ]=0\ .
\end{equation}
\end{definition}
Recalling Definition \ref{invarian}, we state now the definition of
strict-sense invariance.
Let $T$ be an a.s. square integrable random section of
$\xi_\tau$.
For every $g\in G$, the ``rotated'' random section
$T^g$ is defined as
\begin{equation}
T^g_x(\cdot):= g^{-1} T_{gx}(\cdot)
\end{equation}
which is still an a.s. square integrable random section of $\xi_\tau$. For any square integrable
section $u$ of $\xi_\tau$, let
\begin{equation}
T(u):= \int_{\X} \langle T_x, u(x)\rangle_{\pi^{-1}(x)}\,dx\ .
\end{equation}
\begin{definition}\label{isotropia per sezioni aleatorie}
Let $T$ be an a.s. square integrable random section of the homogeneous vector bundle
$\xi_\tau$. It is said to be \emph{(strict-sense) $G$-invariant}
or \emph{isotropic} if and only if
for every choice of square integrable sections
$u_1, u_2, \dots, u_m$ of $\xi_\tau$, the random vectors
\begin{equation}\label{= in legge}
\bigl( T(u_1), \dots, T(u_m) \bigr)\quad
\mbox{and} \quad\bigl( T^g(u_1), \dots, T^g(u_m) \bigr)=\bigl( T( U_g u_1),  \dots, T( U_g u_m) \bigr)
\end{equation}
have the same law for every $g\in G$.
\end{definition}
\begin{prop}\label{pullback-invariant}
Let $T$ be an a.s. square integrable random section of $\xi_\tau$ and let
$X$ be its pullback random field on $G$. Then $X$ is isotropic
if and only if $T$ is an isotropic random section.
\end{prop}
\begin{proof}
Let us denote $X(f) := \int_G \langle X_g, f(g) \rangle_H\,dg$.
Thanks to Proposition \ref{action-pullback} the equality in law (\ref{= in legge}) is equivalent to the requirement that
for every choice of square integrable functions $f_1, f_2, \dots, f_m$ of type $\tau$ (i.e. the pullbacks of corresponding sections of $\xi_\tau$), the random vectors
\begin{equation}\label{pullback invariante}
( X(f_1), \dots, X(f_m) )\quad \mbox{and}\quad( X(L_g f_1), \dots, X(L_g  f_m) )
\end{equation}
have the same law for every $g\in G$.
As $L^2_\tau(G)$ is a closed subspace of $L^2(G)$ and is invariant under the left regular
representation of $G$, every square integrable function $f:G\to H$
can be written as the sum $f^{(1)}+ f^{(2)}$
with $f^{(1)}\in L^2_\tau(G)$, $ f^{(2)}\in L^2_\tau(G)^{\bot}$. As the paths of the random field $X$ are of type $\tau$ we have $X(f^{(2)})=X(L_h f^{(2)})=0$ for every $h\in G$ so that
\begin{equation}
X(f)=X(f^{(1)})\quad  \mbox{and}\quad  X(L_h f) = X(L_h f^{(1)})\ .
\end{equation}
Therefore (\ref{pullback invariante}) implies that, for every choice $f_1, f_2, \dots, f_m$
of square integrable $H$-valued functions on $G$, the random vectors
\begin{equation}
( X(L_g f_1),  \dots, X(L_g f_m) )\quad  \mbox{and} \quad  ( X(f_1),  \dots, X(f_m) )
\end{equation}
have the same law for every $h\in G$ so that the pullback random field $X$ is a strict-sense isotropic random field on $G$.

\end{proof}
As a consequence of Proposition \ref{Mean square continuity of invariant}
 we have
\begin{cor} Every a.s.
square integrable, second order and isotropic random section $T$ of
the homogeneous vector bundle $\xi_\tau$ is mean square
continuous.
\end{cor}

In order to make a comparison with the pullback approach developed above, we briefly recall
the approach to the theory of random fields in vector bundles introduced by Malyarenko in \cite{malyarenko}.
The main tool is the
scalar random field associated to the random section $T$ of $\xi=(E,p,B)$.
More precisely, it is the complex-valued random field $T^{sc}$ indexed by the elements $\eta\in E$ given by
\begin{equation}\label{scalar random field}
T^{sc}_{\eta} := \langle \eta, T_{b} \rangle_{p^{-1}(b)}, \; b\in B, \eta\in p^{-1}(b)\ .
\end{equation}
$T^{sc}$ is a scalar random field on $E$ and we can give the definition that $T$ is mean square continuous
if and only if $T^{sc}$  is mean square continuous, i.e., if the map
\begin{equation}
E \ni \eta \mapsto T^{sc}_{\eta}\in L_\C^2(\P)
\end{equation}
is continuous.
Given a  topological group $G$ acting with a continuous action
$(g,x)\mapsto gx, g\in G$ on the base space
$B$, an action of $G$ on $E$ is called associated if its restriction to any fiber $p^{-1}(x)$ is a linear isometry between
$p^{-1}(x)$ and $p^{-1}(gx)$.  In our case of interest, i.e. the homogeneous vector
bundles $\xi_\tau=(G\times_\tau H, \pi_\tau, \X)$, we can consider the action defined in \paref{action} which is obviously associated. We can now define that $T$ is strict sense $G$-invariant w.r.t. the action of $G$ on $B$ if the finite-dimensional distributions of $T^{sc}$ are invariant under the associated action \paref{action}. In the next statement we prove the equivalence of the two approaches.
\begin{prop}\label{noimal} The square integrable random section $T$ of the homogeneous bundle $\xi_\tau$ is mean square continuous (i.e. its pullback random field on $G$ is mean square continuous)  if and only if the associated scalar random field $T^{sc}$ is mean square continuous. Moreover if $T$ is a.s. continuous then it is isotropic if and only if the associated scalar random field $T^{sc}$ is $G$-invariant.
\end{prop}
\begin{proof}
Let $X$ be the pullback random field of $T$. Consider the scalar random field on $G\times H$ defined as $X^{sc}_{(g,z)} := \langle z, X_g \rangle_H$. Let us denote $pr$ the projection $G\times H\to G\times_\tau H$: keeping in mind (\ref{prod scalare}) we have
\begin{equation}\label{2=}
T^{sc} \circ pr =X^{sc}\ ,
\end{equation}
i.e.
$$
T^{sc}_{\th(g,z)} (\omega) = X^{sc}_{(g,z)}(\omega)
$$
for every $(g,z)\in G\times H$, $\omega\in \Omega$.
Therefore the map $G\times_\tau H \ni \theta(g,z) \mapsto T^{sc}_{\theta(g,z)}\in L^2_\C(P)$ is continuous if and only if the map $G\times H \ni (g,z) \mapsto X^{sc}_{(g,z)}\in L^2_\C(P)$ is continuous, the projection $pr$ being open and continuous.
Let us  show that the continuity of the latter map is equivalent to the mean square continuity of the pullback random field $X$, which will allow to conclude.
The proof of this equivalence is inspired by the one of a similar statement in \cite{malyarenko}, $\S 2.2$.

Actually consider an orthonormal basis $\lbrace v_1, v_2, \dots, v_{\dim\tau} \rbrace$ of $H$, and
denote $X^i=\langle X,v_i\rangle$ the $i$-th component of $X$ w.r.t. the above basis. Assume that the map $G\times H \ni (g,z) \mapsto X^{sc}_{(g,z)}\in L^2_\C(P)$ is continuous, then the random field on $G$
$$
g\mapsto   X^{sc}_{(g,v_i)}=\overline{ X_g ^i}
$$
is continuous for every $i=1, \dots, \dim\tau$.
As $\E[| \overline{X_g^i}-\overline{X_h^i}|^2]=\E [| X_g^i - X_h^i|^2]$,
$$
\lim_{h\to g} \E [\| X_g - X_h \|_H^2] =\lim_{h\to g} \sum_{i=1}^{\dim\tau}
\E [| X_g^i - X_h^i|^2] = 0\ .
$$
Suppose that the pullback random field $X$ is mean square continuous.
Then for each $i=1, \dots, \dim\tau$
$$
\dlines{
0 \le \limsup_{h\to g} \E[ | X_g^i - X_h^i|^2] \le \lim_{h\to g} \E[ \| X_g - X_h \|^2_H ]= 0
}$$
so that the maps $G\ni g\mapsto    X_g ^i \in L^2_\C(\P)$ are continuous.
Therefore
$$\dlines{
\lim_{(h,w) \to (g,z)} \E [| X^{sc}_{(h,w)} - X^{sc}_{(g,z)} |^2] \le 2 \sum_{i=1}^{\dim\tau} \lim_{(h,w) \to (g,z)} \E[|w_i X_h^i - z_i X_g^i |^2] = 0\ ,
}$$
$a_i$ denoting the $i$-th component of $a\in H$.

Assume that $T$ is a.s. continuous and let us show that it is  isotropic if and only if the associated scalar random field $T^{sc}$ is $G$-invariant.
Note first that, by \paref{2=} and $(T^{sc})^h=(X^{sc})^h\circ pr$ for any $h\in G$,  $T^{sc}$ is $G$-invariant if and only if  $X^{sc}$ is $G$-invariant.
Actually if the random fields $X^{sc}$ and its rotated $(X^{sc})^h$
have the same law, then $T^{sc} \mathop{=}^{law} X^{sc}$ and
vice versa.
Now recalling the definition of $X^{sc}$, it is obvious that  $X^{sc}$ is $G$-invariant
if and only if $X$ is isotropic.

\end{proof}
\section{Random sections of the homogeneous line bundles on $\cS^2$}
We now concentrate on the case of the homogeneous line bundles on
$\X=\cS^2$ with $G=SO(3)$ and $K\cong SO(2)$.  For every character $\chi_s$ of $K$, $s\in\Z$, let $\xi_s$ be the corresponding homogeneous vector bundle on $\cS^2$, as
explained in the previous section.
Given the action of $K$ on $SO(3)\times \mathbb{C}$:
$k(g,z)=(gk, \chi_s(k^{-1})z)$, $k\in K$, let $\mE_s:=SO(3)\times_s\C$ be the space of the orbits
$\mE_s=\lbrace \theta(g,z), (g,z)\in G\times \mathbb{C}\rbrace$
where $\theta(g,z) = \lbrace (gk, \chi_s(k^{-1})z); k\in K \rbrace$.
If  $\pi_s:  \mE_s \ni\theta(g,z)\to gK\in \cS^2$,
$\xi_s=(\mE_s, \pi_s, \cS^2)$ is an homogeneous line bundle (each fiber $\pi_s^{-1}(x)$ is isomorphic  to $\C$ as a vector space).

The space $L^2(\xi_s)$ of the square integrable sections of
$\xi_s$ is therefore isomorphic to the space $L^2_s(SO(3))$ of the
square integrable \emph{functions of type $s$}, i.e. such that, for every $g\in G$ and $k \in K$,
\begin{equation}
f(gk)=\chi_s(k^{-1})f(g)=\overline{\chi_s(k)}f(g)\ .
\end{equation}
Let us investigate the Fourier expansion of a function of type $s$.

\begin{prop}\label{infinite-linear}
Every function of type $s$ is an infinite linear combination of the
functions appearing in the $(-s)$-columns of Wigner's $D$ matrices
$D^\ell$, $\ell \ge |s|$. In particular functions of type $s$ and type
$s'$ are orthogonal if $s\not=s'$.

\end{prop}
\begin{proof}
For every $\ell \ge |s|$, let $\widehat f(\ell)$ be as in (\ref{coefficiente ellesimo}). We have,
for every $k\in K$,
\begin{equation}\label{leading}
\begin{array}{c}
\displaystyle\widehat f(\ell)
 = \sqrt{2\ell+1}\,\int_{SO(3)}f(g)D^\ell(g^{-1})\,dg =\\
 \sqrt{2\ell+1}\,\chi_s(k)\int_{SO(3)} f(gk) D^\ell(g^{-1})\,dg=\\
\displaystyle=\sqrt{2\ell+1}\,\chi_s(k)\int_{SO(3)} f(g) D^\ell(kg^{-1})\,dg
=\\=\sqrt{2\ell+1}\,\chi_s(k)D^\ell(k)\int_{SO(3)} f(g)
D^\ell(g^{-1})\,dg
 =\chi_s(k)D^\ell(k)\widehat f(\ell) \ ,\\
\end{array}
\end{equation}
i.e. the image of $\widehat f(\ell)$ is contained in the subspace
$H_{\ell}^{(-s)} \subset H_\ell$ of the vectors such that $D^\ell(k)v=
\chi_{-s}(k)v$ for every $k\in K$. In particular $\widehat f(\ell) \ne 0$ only if
$\ell \ge |s|$, as for every $\ell$ the restriction to $K$ of the representation
$D^\ell$
is unitarily equivalent to the direct sum of the representations $\chi_m$,
$m=-\ell, \dots, \ell$ as recalled at the end of \S2.

Let $\ell \ge |s|$ and $v_{-\ell}, v_{-\ell + 1}, \dots, v_{\ell}$ be the orthonormal basis of $H_\ell$
as in (\ref{restrizione}), in other words $v_m$ spans $H_{\ell}^{m}$, the
one-dimensional subspace of $H_\ell$
formed by the vectors that transform under the action of $K$ according to the
representation $\chi_{m}$.
It is immediate that
\begin{align}
\widehat f(\ell)_{i,j} = \langle \widehat f(\ell)  v_j, v_i \rangle=0\ ,
\end{align}
unless $i=-s$. Thus the Fourier coefficients of $f$ vanish but those corresponding to
the column $(-s)$ of the matrix representation $D^\ell$ and the Peter-Weyl expansion  (\ref{PW SO(3)}) of $f$ becomes, in  $L^2(SO(3))$,
\begin{equation}\label{sviluppo per una funzione di tipo s}
f=\sum_{\ell \ge |s|} \sqrt{2\ell + 1} \sum_{m=-\ell}^{\ell}
\widehat f(\ell)_{-s,m} D^\ell_{m,-s}\ .
\end{equation}
\end{proof}
We introduced the spherical harmonics in (\ref{armoniche sferiche1}) from the entries $D^{\ell}_{m,0}$ of the central columns of Wigner's $D$ matrices.
Analogously to the case of $s=0$, for any $s\in \Z$ we define for $\ell \ge |s|, m=-\ell, \dots, \ell$
\begin{equation}\label{armoniche di spin}
_{-s} Y_{\ell,m} (x) := \theta \Bigl( g, \sqrt{\frac{2\ell +1}{4\pi}}\, \overline{D^\ell_{m,s}(g)} \Bigr)\ , \qquad x=gK\in \cS^2\ .
\end{equation}
$ _{-s} Y_{\ell,m}$
is a section of $\xi_s$ whose pullback function (up to a factor)
is $g\mapsto D^{\ell}_{-m,-s}(g)$ (recall the relation
$\overline{D^\ell_{m,s}(g)} = (-1)^{m-s} D^{\ell}_{-m,-s}(g)$, see \cite{dogiocam} p.~55 e.g.).
Therefore thanks to Proposition \ref{infinite-linear} the sections
$_{-s} Y_{\ell,m},\, \ell \ge |s|, m=-\ell, \dots, \ell$, form an
\emph{orthonormal} basis of $L^2(\xi_s)$. Actually recalling
(\ref{prod scalare}) and (considering the total mass equal to $4\pi$ on the sphere and to $1$ on $SO(3)$)
$$
\int_{\cS^2}\, _{-s} Y_{\ell,m}\, \overline{_{-s}Y_{\ell',m'}}\,dx =
4\pi \int_{SO(3)} \sqrt{\frac{2\ell +1}{4\pi}}\, \overline{D^\ell_{m,s}(g)}\,
\sqrt{\frac{2\ell' +1}{4\pi}}\,D^{\ell'}_{m',s}(g)\,dg = \delta_{\ell'}^{\ell}\delta_{m'}^{m}\ .
$$
The sections  $_{-s} Y_{\ell,m},\, \ell \ge |s|, m=-\ell, \dots, \ell$ in
(\ref{armoniche di spin}) are called  \emph{spin $-s$ spherical harmonics}.
Recall that the spaces $L^2_s(SO(3))$ and $L^2(\xi_s)$ are isometric through
the identification $u \longleftrightarrow f$ between a section $u$ and its
pullback $f$ and the definition of the scalar product on $L^2(\xi_s)$ in
(\ref{prod_scalare}).
Proposition (\ref{infinite-linear}) can be otherwise stated as

\noindent \emph{Every square integrable section $u$ of the homogeneous line
bundle $\xi_s=(\mE_s, \pi_s, \cS^2)$ admits a Fourier expansion in terms of spin
$-s$ spherical harmonics of the form
\begin{equation}
u(x) = \sum_{\ell \ge |s|} \sum_{m=-\ell}^{\ell} u_{\ell,m}\, _{-s} Y_{\ell,m}(x)\ ,
\end{equation}
where $u_{\ell,m} := \langle u,\, _{-s} Y_{\ell,m} \rangle_2$, the above series converging in $L^2(\xi_s)$.}

In particular we have the relation
$$
\dlines{
 u_{\ell,m} = \int_{\cS^2} u(x)\,   _{-s} Y_{\ell,m}(x) \,dx =
4\pi \int_{SO(3)} f(g) \sqrt{ \frac{2\ell +1}{4\pi}} D^\ell_{m,s}(g)\,dg =\cr
(-1)^{s-m} \sqrt{4\pi (2\ell+1)} \int_{SO(3)} f(g)  \overline{  D^\ell_{-m,-s}(g)}\,dg = (-1)^{s-m} \sqrt{4\pi}\, \widehat f(\ell)_{-s,-m}\ . }
$$
\begin{definition}
Let $s\in \Z$. A square integrable function $f$ on $SO(3)$ is said
to be \emph{bi-$s$-associated} if for every $g\in SO(3), k_1, k_2 \in K$,
\begin{equation}
f(k_1 g k_2) = \chi_s (k_1) f(g) \chi_s(k_2)\ .
\end{equation}
\end{definition}
Of course for $s=0$ {bi-$0$-associated} is equivalent to bi-$K$-invariant.
We are particularly interested in bi-$s$-associated functions
as explained in the remark below.
\begin{remark}\label{associate-bi} \rm Let $X$ be an isotropic random field of type $s$ on
$SO(3)$. Then its associate positive definite function $\phi$ is
bi-$(-s)$-associated. Actually, assuming for simplicity that $X$ is
centered, as $\phi(g)=\E[X_g\overline{X_e}]$, we have, using
invariance on $k_1$ and type $s$ property on $k_2$,
$$\displaylines{
\phi(k_1gk_2)=\E[X_{k_1gk_2}\overline{X_e}]=
\E[X_{gk_2}\overline{X_{k_1^{-1}}}]=\cr
=
\chi_s(k_1^{-1})\E[X_g\overline{X_e}]\chi_s(k_2^{-1})=
\chi_{-s}(k_1)\phi(g)\chi_{-s}(k_2)\ .
}$$\qed
\end{remark}

Let us investigate the Fourier expansion of a bi-$s$-associated
function $f$: note first that a bi-$s$-associated function is also a
function of type $(-s)$, so that $\widehat f(\ell) \ne 0$
only if $\ell \ge |s|$ as above and all its rows vanish but for
the $s$-th. A repetition of the computation
leading to \paref{leading} gives easily that
$$
\widehat f(\ell)=\chi_{-s}(k)\widehat f(\ell)D^\ell(k)
$$
so that the only non-vanishing entry of the matrix $\widehat f(\ell)$ is
the $(s,s)$-th.

Therefore the Fourier expansion of a bi-$s$-associated function $\phi$ is
\begin{equation}\label{sviluppo per una funzione bi-s-associata}
f= \sum_{\ell \ge |s|} \sqrt{2\ell + 1}\, \alpha_\ell
D^\ell_{s,s}\ ,
\end{equation}
where we have set $\alpha_\ell=\widehat f(\ell)_{s,s}$.

Now let $T$ be an a.s. square integrable random section of the line
bundle $\xi_s$ and $X$ its pullback random field. Recalling that $X$
is a random function of type $s$ and its sample paths are a.s.
square integrable, we easily obtain the stochastic Fourier
expansion of $X$ applying (\ref{sviluppo per una funzione di tipo s})
to the functions $g\mapsto X_g(\omega)$. Actually define, for every $\ell \ge |s|$, the random operator
\begin{equation}
\widehat X(\ell)= \sqrt{2\ell + 1}\int_{SO(3)} X_g D^\ell(g^{-1})\,dg\ .
\end{equation}
The basis of $H_\ell$ being fixed as above and recalling (\ref{sviluppo per una funzione di tipo s}), we obtain, a.s. in $L^2(SO(3))$,
\begin{equation}\label{sviluppo di Fourier per X}
X_g =\sum_{\ell \ge |s|} \sqrt{2\ell + 1}
\sum_{m=-\ell}^{\ell} \widehat X(\ell)_{-s,m} D^\ell_{m,-s}(g)\ .
\end{equation}

If $T$ is isotropic, then by Definition
\ref{isotropia per sezioni aleatorie} its pullback
random field $X$ is also isotropic in the sense of Definition
\ref{invarian}. The following is a consequence of well known general properties of the random
coefficients of invariant random fields (see \cite{balditrapani} Theorem 3.2 or \cite{malyarenko} Theorem 2).
\begin{prop}\label{teorema di struttura}
Let $s\in \Z$ and $\xi_s=(\mE_s, \pi_s, \cS^2)$ be the homogeneous
line bundle on $\cS^2$ induced by the $s$-th linear character $\chi_s$ of $SO(2)$.
Let $T$ be a random section
of $\xi_s$ and $X$ its pullback random field. If $T$ is second order and strict-sense isotropic, then the Fourier coefficients $X(\ell)_{-s,m}$ of $X$
in its stochastic expansion \paref{sviluppo di Fourier per X}
are pairwise orthogonal and the variance, $c_\ell$, of $\widehat X(\ell)_{-s,m}$ does not depend on $m$.
Moreover $\E [ \widehat X(\ell)_{-s,m} ]=0$ unless $\ell=0, s=0$.
\end{prop}

For the random field $X$ of Proposition \ref{teorema di struttura} we have
immediately
\begin{equation}\label{conv}
\E[|X_g|^2]=\sum_{\ell \ge |s|} (2\ell + 1) c_\ell < +\infty\ .
\end{equation}
The convergence of the series above is also a consequence of
Theorem \ref{gangolli-true}, as the positive definite function $\phi$ associated to
$X$ is given by
$$
\dlines{
\phi(g)=\E[X_g\overline{X_e}]=\sum_{\ell \ge |s|} (2\ell + 1)c_\ell
\sum_{m=-\ell}^{\ell} D^\ell_{m,-s}(g)\overline{D^\ell_{m,-s}(e)}=\cr
=\sum_{\ell \ge |s|} (2\ell + 1)c_\ell
\sum_{m=-\ell}^{\ell} D^\ell_{m,-s}(g)D^\ell_{-s,m}(e)=
\sum_{\ell \ge |s|} (2\ell + 1)c_\ell D^\ell_{-s,-s}(g)\ . \cr
}
$$
\begin{remark} \rm
Let $X$ be a type $s$ random field on $SO(3)$ with $s\not=0$. Then the
relation $X_{gk}=\chi_s(k^{-1})X_{g}$ implies that $X$ cannot be real
(unless it is vanishing). If in addition it was Gaussian,
then, the identity in law between $X_g$ and $X_{gk}=\chi_s(k^{-1})X_{g}$
would imply that, for every $g\in G$, $X_g$ is a complex Gaussian r.v.
\end{remark}

\section{Construction of Gaussian isotropic spin random fields}

We now give an extension of the construction of \S\ref{sec4}
and prove that every complex Gaussian random section of a homogeneous
line bundle on $\cS^2$ can be obtained in this way, a result much
similar to Theorem \ref{real-general}.
Let $s\in \Z$ and let $\xi_s$ be the homogeneous line bundle associated to
the representation $\chi_s$.

Let $(X_n)_n$ be a sequence of i.i.d. standard Gaussian r.v.'s on some probability space $(\Omega, \F, \mathbb{P})$, and
$\mathscr{H}\subset L^2(\Omega,\F,\P)$ the \emph{complex} Hilbert space
generated by $(X_n)_n$.
Let $(e_n)_n$ be an orthonormal basis of $L^2(SO(3))$ and
define an isometry $S$ between $L^2(SO(3))$ and
$\mathscr{H}$ by
$$
L^2(SO(3))\ni \sum_k \alpha_k e_k\enspace \to\enspace \sum_k \alpha_k X_k \in \mathscr{H}\ .
$$
Let $f\in L^2(SO(3))$, we define a random field $X^f$ on $SO(3)$ by
\begin{align}\label{spin-def}
X^f_g=S(L_gf)\ ,
\end{align}
$L$ denoting as usual the left regular representation.
\begin{prop}\label{propspin=}
If $f$ is a square integrable bi-$s$-associated function
on $SO(3)$, then $X^f$ defined in \paref{spin-def} is a second order,
square integrable Gaussian isotropic random field of type $s$. Moreover it is
complex Gaussian.
\end{prop}
\begin{proof}
It is immediate that $X^f$ is second order as
$\E[|X^f_g|^2]=\Vert L_gf\Vert^2_2=\Vert f\Vert^2_2$. It
 is of type $s$ as for every $g\in SO(3)$ and $k\in K$,
$$
X^f_{gk}=S(L_{gk}f)=\chi_s(k^{-1})S(L_gf)=\chi_s(k^{-1})X^f_g\ .
$$
Let us prove strict-sense invariance. Actually, $S$ being an isometry,
for every $h\in SO(3)$
$$
\dlines{
\E[ X^f_{hg}\overline{X^f_{hg'}}] = \E[ S(L_{hg}f)\overline{S(L_{hg'}f)}]
= \langle L_{hg}f, L_{hg'}f \rangle_2 =
\langle L_{g}f, L_{g'}f \rangle_2
= \E[X^f_{g}\overline{X^f_{g'}}]\ .
}$$
Therefore the random fields $X^f$ and its rotated $(X^f)^h$
have the same covariance
kernel. Let us prove that they also have the same relation function.
Actually we have, for every $g,g'\in SO(3)$,
\begin{equation}\label{complex-spin1}
\E[X^f_{g}X^f_{g'}]=\E[ S(L_{hg}f){S(L_{hg'}f)}]
= \langle L_{hg}f, \overline {L_{hg'}f} \rangle_2 =\langle L_{g}f, \overline{L_{g'}f} \rangle_2=0
\end{equation}
as the function $\overline{L_{g'}f}$ is bi-$(-s)$-associated and therefore of
type $s$ and orthogonal to $L_{g}f$ which is of type $-s$ (orthogonality of functions
of type $s$ and $-s$ is a consequence of Proposition \ref{infinite-linear}).

In order to prove that $X^f$ is complex Gaussian we must show that
for every $\psi\in L^2(SO(3))$, the r.v.
$$
Z=\int_{SO(3)}X^f_g\psi(g)\, dg
$$
is complex Gaussian. As $Z$ is Gaussian by construction we must just prove that
$\E[Z^2]=0$. But as, thanks to \paref{complex-spin1}, $\E[X^f_{g}X^f_{g'}]=0$
$$\displaylines{
\E[Z^2]=\E\Bigl[\int_{SO(3)}\int_{SO(3)}X^f_gX^f_h\psi(g)\psi(h)\,dg dh\Bigr]=\cr
=
\int_{SO(3)}\int_{SO(3)}\E[X^f_gX^f_h]\psi(g)\psi(h)\,dg dh=0\ .
}$$
\end{proof}

Let us investigate the stochastic Fourier expansion of
$X^f$.
Let us consider first the random field $X^\ell$ associated to
$f=D^\ell_{s,s}$. Recall first that the r.v. $Z=S(D^\ell_{s,s})$
has variance $\E[|Z|^2]=\Vert D^\ell_{s,s}\Vert_2^2=(2\ell+1)^{-1}$ and that
$\overline{D^\ell_{m,s}}=(-1)^{m-s}D^\ell_{-m,-s}$. Therefore
$$
\dlines{
X^\ell_g =S(L_g D^\ell_{s,s})=
\sum_{m=-\ell}^{\ell} S(D^\ell_{m,s}) D^\ell_{s,m}(g^{-1})=\cr
=\sum_{m=-\ell}^{\ell} S(D^\ell_{m,s})
\overline{D^\ell_{m,s}(g)}=
\sum_{m=-\ell}^{\ell} S(D^\ell_{m,s}) (-1)^{m-s}D^\ell_{-m,-s}(g)\ .\cr
}
$$
Therefore the r.v.'s
$$
a_{\ell,m}=\sqrt{2\ell+1}\, S(D^\ell_{m,s})(-1)^{m-s}
$$
are complex Gaussian, independent and with variance $\E[|a_{\ell,m}|^2]=1$ and
we have the expansion
\begin{equation}
X^\ell_g=
\frac{1}{\sqrt{2\ell+1}}\sum_{m=-\ell}^{\ell} a_{\ell, m} D^\ell_{m,-s}(g)\ .
\end{equation}
Note that the coefficients
$a_{\ell m}$ are independent complex Gaussian  r.v.'s.
This is a difference with respect to the case $s=0$,
where in the case of a real random field, the coefficients $a_{\ell,m}$ and $a_{\ell,-m}$ were not independent. Recall that random fields of type $s\ne 0$ on $SO(3)$ cannot be real.

In general, for a square integrable bi-$s$-associated function $f$
\begin{equation}\label{fs}
f=\sum_{\ell \ge |s|} \sqrt{2\ell + 1}\, \alpha_\ell D^\ell_{s,s}
\end{equation}
with
$$
\Vert f\Vert_2^2=\sum_{\ell \ge |s|} |\alpha_\ell|^2 < +\infty\ ,
$$
the Gaussian random field $X^f$ has the expansion
\begin{align}\label{sviluppo per X}
X^f_g&=\sum_{\ell \ge |s|} \alpha_\ell
\sum_{m=\ell}^{\ell} a_{\ell, m} D^\ell_{m,-s}(g)\ ,
\end{align}
where $(a_{\ell,m} )_{\ell,m}$ are independent  complex Gaussian
r.v.'s with zero mean and unit variance.

The associated positive definite function of $X^f$,
$\phi^f(g):=\E[X^f_g \overline{X^f_e}]$ is bi-$(-s)$-associated (Remark
\ref{associate-bi}) and
continuous (Theorem \ref{gangolli-true}) and, by \paref{convolution for phi},
is related to $f$ by
$$
\phi^f=f \ast \breve f (g^{-1})\ .
$$
This allows to derive its Fourier expansion:
$$
\dlines{
\phi^f(g)=f \ast \breve f (g^{-1})=
\int_{SO(3)} f(h) \overline{f(g h)}\,dh
=\cr
=\sum_{\ell, \ell' \ge |s|} \sqrt{2\ell + 1}\,\sqrt{2\ell'+1}\, \alpha_\ell \overline{\alpha_{\ell'}}
\int_{SO(3)} D^\ell_{s,s}(h) \overline{D^{\ell'}_{s,s}(g h)}\,dh=\cr
=\sum_{\ell, \ell' \ge |s|} \sqrt{2\ell + 1}\,\sqrt{2\ell'+1}\, \alpha_\ell \overline{\alpha_{\ell'}}
\sum_{j=-\ell}^{\ell} \underbrace{\Bigl (\int_{SO(3)} D^\ell_{s,s}(h)
\overline{D^{\ell'}_{j,s}(h)}\,dh\, \Bigr)}_{= \frac 1 {2\ell +1}
\delta_{\ell,\ell'} \delta_{s,j}} \overline{D^{\ell}_{s,j}(g)}=\cr
=\sum_{\ell \ge |s|} |\alpha_\ell |^2 D^\ell_{-s,-s}(g)\ .
}
$$
Note that in accordance with Theorem \ref{gangolli-true}, as
$|D^\ell_{-s,-s}(g)|\le D^\ell_{-s,-s}(e)=1$, the above series converges uniformly.

Conversely, it is immediate that, given a continuous positive definite
bi -$(-s)$- associated function $\phi$, whose expansion is
$$
\phi^f(g)=\sum_{\ell \ge |s|} |\alpha_\ell |^2 D^\ell_{-s,-s}(g)\ ,
$$
by choosing
$$
f(g)=\sum_{\ell \ge |s|} \sqrt{2\ell+1}\,\beta_\ell D^\ell_{-s,-s}(g)
$$
with $|\beta_\ell|=\sqrt{\alpha_\ell}$,
there exist a square integrable bi-$s$-associated function $f$ as in \paref{fs}
such that $\phi(g)=f*\breve f(g^{-1})$. Therefore, for every random field
$X$ of type
$s$ on $SO(3)$ there exists a square integrable bi-$s$-associated
function $f$ such that $X$ and $X^f$ coincide in law. Such a function $f$ is not
unique.

From $X^f$ we can define a random section $T^f$ of the homogeneous
line bundle $\xi_s$ by
\begin{equation}
T^f_{x} := \theta(g, X^f_g)\ ,
\end{equation}
where $x=gK\in \cS^2$.
Now, as for the case $s=0$ that was treated in \S\ref{sec4}, it is natural to ask whether every Gaussian isotropic section of $\xi_s$ can be obtained in this way.
\begin{theorem}\label{teospin=}
Let $s\in \Z\setminus \{0\}$.
For every square integrable, isotropic, (complex) Gaussian
random section  $T$ of the homogeneous  $s$-spin line bundle $\xi_s$, there exists a square integrable and bi-$s$-associated function $f$ on $SO(3)$
such that
\begin{equation}\label{=}
T^f\enspace \mathop{=}^{law}\enspace T\ .
\end{equation}
Such a function $f$ is not unique.
\end{theorem}
\begin{proof}
Let $X$ be the pullback random field (of type $s$) of $T$. $X$ is of course
mean square continuous. Let $R$ be its covariance kernel. The function
$\phi(g):=R(g,e)$  is  continuous,
positive definite and bi-$(-s)$-associated, therefore has the expansion
\begin{equation}\label{sviluppo per fi 2}
\phi=\sum_{\ell \ge |s|} \sqrt{2\ell + 1}\,\beta_\ell D^\ell_{-s,-s}\ ,
\end{equation}
where
$\beta_\ell=\sqrt{2\ell + 1}\,\int_{SO(3)} \phi(g)
\overline{D^\ell_{-s,-s}(g)}\,dg \ge 0$. Furthermore,
by Theorem \ref{gangolli-true}, the series
in (\ref{sviluppo per fi 2}) converges uniformly, i.e.
$$
\sum_{\ell \ge |s|} \sqrt{2\ell + 1}\,\beta_\ell < +\infty\ .
$$
Now set $f:=\sum_{\ell \ge |s|} (2\ell +1)\sqrt{\beta_\ell} D^\ell_{s,s}$.
Actually, $f\in L^2_s(SO(3))$ as $\| f\|^2_{L^2(SO(3))} = \sum_{\ell \ge |s|} (2\ell +1)\beta_\ell < +\infty$ so that it is bi-$s$-associated.

Note that every function $f$ of the form
$f=\sum_{\ell \ge |s|} (2\ell +1)\alpha_\ell D^\ell_{s,s}$ where
$\alpha_\ell$ is such that $\alpha_\ell^2=\beta_\ell$ satisfies
(\ref{=}) (and clearly every function $f$ such that $\phi(g)=f*\breve f(g^{-1})$ is of this form).

\end{proof}
\section{The connection with classical spin theory}
There are different approaches to the theory of random sections of homogeneous line bundles
 on $\cS^2$ (see \cite{marinuccigeller}, \cite{bib:LS}, \cite{malyarenko}, \cite{bib:NP}  e.g.). In this section
we compare them, taking into account, besides the one outlined
 in \S 6,  the classical Newman and Penrose  spin theory (\cite{bib:NP})
later formulated in a more mathematical framework by Geller and Marinucci
(\cite{marinuccigeller}).

Let us first recall some basic notions about vector bundles. From now on $s\in \Z$. We shall state them concerning
the complex line bundle $\xi_s=(\mE_s, \pi_s, \cS^2)$ even if they can be immediately extended
to more general situations. An atlas of $\xi_s$ (see \cite{Husemoller} e.g.) can be defined as follows.
Let $U\subset \cS^2$ be an open set and $\Psi$ a diffeomorphism between $U$ and an open set of $\R^2$. A chart $\Phi$ of $\xi_s$ over $U$ is an isomorphism
\begin{equation}\label{def chart}
\Phi: \pi^{-1}_s(U)\goto \Psi(U)\times \C\ ,
\end{equation}
whose restriction to every fiber $\pi_s^{-1}(x)$ is a linear isomorphism $\leftrightarrow\C$. An atlas of $\xi_s$ is a family $( U_j, \Phi_j)_{j\in J}$ such that $\Phi_j$ is a chart of $\xi_s$ over $U_j$
and the family $(U_j)_{j\in J}$ covers $\cS^2$.

Given an atlas $( U_j, \Phi_j)_{j\in J}$, For each pair $i,j\in J$ there exists a unique map
(see \cite{Husemoller} Prop. 2.2) $\lambda_{i,j}: U_i\cap U_j \goto \C\setminus 0$
such that for $x\in U_i\cap U_j, z\in \C$,
\begin{equation}
\Phi_i^{-1}(\Psi_i(x),z)=\Phi_j^{-1}(\Psi_j(x),\lambda_{i,j}(x)z)\ .
\end{equation}
The map $\lambda_{i,j}$ is called the \emph{transition function} from the chart
$(U_j,\Phi_j)$ to the chart $(U_i,\Phi_i)$.
Transition functions satisfy the cocycle conditions, i.e.
for every $i,j,l\in J$
\begin{equation}
\begin{array}{l}\label{cociclo}
\nonumber
\lambda_{j,j} = 1\ \ \ \  \qquad  \text{on}\ \ \ U_j\ ,\cr
\lambda_{j,i} = \lambda_{i,j}^{-1}\ \ \ \   \quad \text{on}\ \  \ U_i\cap U_j\ ,\cr
\nonumber
\lambda_{l,i}\lambda_{i,j}=\lambda_{l,j}\ \ \  \text{on}\ \  \ U_i\cap U_j\cap U_l\ .
\end{array}
\end{equation}
Recall that we denote $K\cong SO(2)$ the isotropy group of the north pole as in \S 6, \S7,  so that $\cS^2\cong SO(3)/K$.
We show now that an atlas of the line bundle $\xi_s$ is given as soon as we specify

a) an atlas $(U_j,\Psi_j)_{j\in J}$ of the manifold $\cS^2$,

b) for every $j\in J$ a family $(g_x^j)_{x\in U_j}$ of representative
elements $g_x^j\in G$ with $g_x^jK=x$.

\noindent More precisely,  let $(g_x^j)_{x\in U_j}$ be as in b) such that $x\mapsto g^j_x$ is smooth for each $j\in J$.
Let $\eta \in  \pi^{-1}_s(U_j)\subset \mE_s$ and $x:=\pi_s(\eta)\in U_j$,
therefore $\eta=\theta(g^j_x,z)$, for a unique $z\in \C$.
Define the chart $\Phi_j$ of $\xi_s$ over $U_j$ as
\begin{equation}\label{triv}
\Phi_j(\eta)=  (\Psi_j(x), z)\ .
\end{equation}
Transition functions of this atlas are easily determined.
If $\eta \in \xi_s$ is such that $x=\pi_s(\eta)\in U_i\cap U_j$,
then $\Phi_j(\eta)=(\Psi_j(x), z_j)$, $\Phi_i(\eta)=(\Psi_i(x), z_i)$.
As $g_x^iK=g^j_xK$, there exists a unique $k=k_{i,j}(x)\in K$ such that
 $g^j_x=g^i_xk$, so that
$\eta=\theta(g^i_x,z_i)=\theta(g^j_x, z_j)=\theta(g^i_xk, z_j)=\theta(g^i_x, \chi_s(k)z_j)$ which implies $z_i=\chi_s(k)z_j$.
Therefore
\begin{equation}\label{transizione}
\lambda_{i,j}(x)=\chi_s(k)\ .
\end{equation}
\bigskip

\noindent The spin $s$ concept was introduced by Newman and Penrose in \cite{bib:NP}:
\emph{a quantity $u$ defined on $\cS^2$ has spin weight $s$  if, whenever a tangent vector $\rho$
 at any point $x$ on the sphere transforms under coordinate change  by
$\rho'=e^{i \psi} \rho$, then the quantity at this point $x$ transforms
by  $u'=e^{is\psi} u$}. Recently, Geller and Marinucci in \cite{marinuccigeller}
have put this notion in a more mathematical framework modeling such a $u$
as a section of a complex line bundle on $\cS^2$ and they describe this line bundle by giving charts and fixing transition functions to express the transformation laws under
changes of coordinates.

More precisely,
they  define an atlas of $\cS^2$ as follows. They consider the open
covering $(U_R)_{R\in SO(3)}$ of $\cS^2$  given by
\begin{equation}\label{charts}
U_e := \cS^2 \setminus \lbrace x_0, x_1 \rbrace \qquad\text{and}\qquad U_R:=R U_e\ ,
\end{equation}
where $x_0=$the north pole (as usual), $x_1=$the south pole. On $U_e$ they consider the usual spherical coordinates $(\vartheta, \varphi)$, $\vartheta=$colatitude, $\varphi=$longitude
 and on any  $U_R$ the ``rotated'' coordinates $(\vartheta_R, \varphi_R)$  in such a way that $x$ in $U_e$ and $Rx$ in $U_R$ have the same coordinates.

The transition functions are defined as follows.
For each $x\in U_R$, let $\rho_R(x)$ denote the unit  tangent vector at $x$, tangent to the circle $\vartheta_R = const$ and
pointing to the direction of increasing $\varphi_R$. If
$x\in U_{R_1}\cap U_{R_2}$,  let $\psi_{R_2,R_1}(x)$ denote the (oriented) angle from
$\rho_{R_1}(x)$ to $\rho_{R_2}(x)$.
They prove that the quantity
\begin{equation}\label{triv-mg}
e^{is\psi_{R_2,R_1}(x)}
\end{equation}
satisfies the cocycle relations \paref{cociclo} so that this defines a unique (up to isomorphism)
structure of complex line bundle on $\cS^2$ having \paref{triv-mg} as transition
functions at $x$  (see \cite{Husemoller} Th. 3.2).

We shall prove that
this spin $s$ line bundle is the same as the
homogeneous line bundle $\xi_{-s}=(\mE_{-s}, \pi_{-s}, \cS^2)$.
To this aim we have just to check that, for a suitable choice of the atlas
$(U_R, \Phi_R)_{R\in SO(3)}$ of $\xi_{-s}$ of the type described in a), b) above,
the transition functions \paref{transizione}  and \paref{triv-mg} are the same.
Essentially we have to determine the family  $(g^R_x)_{R\in SO(3), x\in U_R}$ as in b).

Recall first that every rotation $R\in SO(3)$ can be realized as a composition of three rotations:
(i) a rotation by an angle $\gamma_R$ around the z axis, (ii) a rotation by an angle
$\beta_R$ around the y axis and (iii) a rotation by an angle $\alpha_R$
around the z axis (the so called z-y-z convention), ($\alpha_R$, $\beta_R$, $\gamma_R$) are the {\it Euler angles} of $R$.
Therefore the rotation $R$ acts on the north pole $x_0$
of $\cS^2$ as mapping $x_0$ to the new location on $\cS^2$
whose spherical coordinates are $(\beta_R, \alpha_R)$ after rotating  the tangent plane at $x_0$
by an angle $\gamma_R$. In each coset $\cS^2\ni x=gK$ let us choose the element $g_x\in SO(3)$ as the rotation such that $g_xx_0=x$ and having its third Euler angle $\gamma_{g_x}$ equal to $0$. Of course
if $x\ne x_0,x_1$, such $g_x$ is unique.

Consider the atlas $(U_R, \Psi_R)_{R\in SO(3)}$ of $\cS^2$ defined as follows.
Set the charts as
\begin{align}
&\Psi_e(x) := (\beta_{g_x}, \alpha_{g_x})\ , \qquad x\in U_e\ ,\\
&\Psi_R(x) := \Psi_e (R^{-1}x)\ , \qquad x\in U_R\ .
\end{align}
Note that for each $R$, $\Psi_R(x)$ coincides with the ``rotated'' coordinates $(\vartheta_R, \varphi_R)$ of $x$.
Let us choose now the family $(g^R_x)_{x\in U_R, R\in SO(3)}$.
For $x\in U_e$ choose $g^e_x:=g_x$
 and for $x\in U_R$
\begin{equation}
g^R_x:=Rg_{R^{-1}x}\ .
\end{equation}
Therefore the corresponding atlas
 $( U_R, \Phi_R)_{R\in SO(3)}$ of $\xi_s$ is given, for $\eta\in \pi^{-1}_s(U_R)$, by
 \begin{equation}
 \Phi_R(\eta)=( \Psi_R(x), z)\ ,
 \end{equation}
 where $x:=\pi_s(\eta)\in U_R$ and $z$ is such that $\eta=\theta(g_x^R,z)$.
Moreover for $R_1, R_2\in SO(3)$, $x\in U_{R_1}\cap U_{R_2}$ we have
\begin{equation}\label{eq k}
k_{R_2,R_1}(x)= (g_{R_2^{-1}x})^{-1} R_2^{-1}R_1 g_{R_1^{-1}x}
\end{equation}
and the transition function
from the chart $(U_{R_1}, \Phi_{R_1})$ to the chart $(U_{R_2}, \Phi_{R_2})$ at $x$ is given by \paref{transizione}
\begin{equation}\label{fnz transizione}
\lambda^{(-s)}_{R_2, R_1}(x):=\chi_s(k)\ .
\end{equation}
From now on let us denote $\omega_{R_2,R_1}(x)$ the rotation angle of $k_{R_2,R_1}(x)$.
Note that, with this choice of the family $(g^R_x)_{x\in U_R, R\in SO(3)}$,
 $\omega_{R_2,R_1}(x)$ is the third Euler angle of the rotation $R_2^{-1}R_1g_{R_1^{-1}x}$.

\begin{remark}\label{particular} \rm Note that we have
$$
R^{-1}g_x=g_{R^{-1}x}\ ,
$$
i.e. $g^R_x=g_x$, in any of the following two situations

a) $R$ is a rotation around the north-south axis (i.e. not changing the latitude of the points of $\cS^2$).

b) The rotation axis of $R$ is orthogonal to the plane $[x_0,x]$ (i.e. changes the colatitude of $x$ leaving its longitude unchanged).

Note that if each of the rotations $R_1,R_2$ are of type a) or of type b), then
$$
k_{R_2,R_1}(x)=g^{-1}_{R_2^{-1}x}R_2^{-1}R_1g_{R_1^{-1}x}=(R_2g_{R_2^{-1}x})^{-1}R_1g_{R_1^{-1}x}=
g_x^{-1}g_x= \mbox{the identity}
$$
and in this case the rotation angle of $k_{R_2,R_1}(x)$ coincides with the angle $-\psi_{R_2,R_1}(x)$, as neither $R_1$ nor $R_2$ change the orientation of the tangent plane at $x$.

Another situation in which the rotation $k$ can be easily computed appears when $R_1$ is the
identity and $R_2$ is a rotation of an angle $\gamma$ around an axis passing through $x$.
Actually
\begin{equation}\label{rot}
k_{R_2,e}(x)=g_x^{-1}R_2^{-1}g_x
\end{equation}
which, by conjugation, turns out to be a rotation of the angle $-\gamma$ around the north-south axis. In this case also it is immediate that the rotation angle $\omega_{R_2,R_1}(x)$ coincides with $-\psi_{R_2,R_1}(x)$.

\qed
\end{remark}
The following relations will be useful in the sequel, setting $y_1=R_1^{-1}x$, $y_2=R_2^{-1}x$,
\begin{align}
k_{R_2,R_1}(x)&=g^{-1}_{R_2^{-1}x}R_2^{-1}R_1g_{R_1^{-1}x}=g^{-1}_{R_2^{-1}R_1y_1}R_2^{-1}R_1g_{y_1}=
k_{R_1^{-1}R_2,e}(R_1^{-1}x)\label{al1}\ ,\cr
k_{R_2,R_1}(x)&=g^{-1}_{R_2^{-1}x}R_2^{-1}R_1g_{R_1^{-1}x}=g^{-1}_{y_2}R_2^{-1}R_1g_{R_1^{-1}R_2y_2}=
k_{e,R_2^{-1}R_1}(R_2^{-1}x)\ .\
\end{align}
We have already shown in Remark \ref{particular} that  $\omega_{R_2,R_1}(x)=-\psi_{R_2,R_1}(x)$ in two particular situations:
rotations that move $y_1=R_1^{-1}x$ to $y_2=R_2^{-1}x$
without turning the tangent plane and rotations that turn the tangent plane without moving the point.
In the next statement, by combining these two particular cases,
we prove that actually they coincide always.
\begin{lemma}\label{lemma angolo}
Let $x\in U_{R_1}\cap U_{R_2}$, then
$\omega_{R_2,R_1}(x)=-\psi_{R_2,R_1}(x)$\ .
\end{lemma}
\begin{proof}
The matrix $R_2^{-1}R_1$ can be decomposed as $R_2^{-1}R_1=EW$ where $W$ is the product of a rotation around an axis that is orthogonal to the plane $[x_0,y_1]$ bringing $y_1$ to a point having the same colatitude as $y_2$ and of a rotation around the north-south axis taking this point to $y_2$. By Remark \ref{particular} we have
$Wg_{y_1}=g_{Wy_1}=g_{y_2}$. $E$ instead is a rotation around an axis passing by $y_2$ itself.

We have then, thanks to \paref{rot} and \paref{al1}
$$
\dlines{
k_{R_2,R_1}(x)=k_{R_1^{-1}R_2,e}(R_1^{-1}x)=k_{W^{-1}E^{-1},e}(y_1)=
g_{EWy_1}^{-1}EWg_{y_1}
=g_{y_2}^{-1}Eg_{y_2}=k_{E^{-1},e}(y_2)\ .\cr
}
$$
By the previous discussion, $\omega_{E^{-1},e}(y_2)=-\psi_{E^{-1},e}(y_2)$.
To finish the proof it is enough to show that
\begin{equation}\label{ss}
\psi_{R_2,R_1}(x)=\psi_{E^{-1},e}(y_2)\ .
\end{equation}
Let us denote $\rho(x)=\rho_e(x)$ the tangent vector at $x$ which is parallel to the curve $\vartheta=const$ and pointing in the direction of increasing $\varphi$. Then in coordinates
$$
\rho(x)=\frac 1{\sqrt{x_1^2+x_2^2}}\ \bigl(-x_2,x_1,0\bigr)
$$
and the action of $R$ is given by (\cite{marinuccigeller},\S3) $\rho_R(x)=R\rho(R^{-1}x)$. As $W\rho(y_1)=\rho(y_2)$ ($W$ does not change the orientation of the tangent plane),
$$
\dlines{
\langle \rho_{R_2}(x), \rho_{R_1}(x) \rangle =
\langle R_2 \rho(R_2^{-1}x), R_1 \rho(R_1^{-1}x) \rangle
= \langle R_1^{-1}R_2 \rho(R_2^{-1}x),\rho(R_1^{-1}x) \rangle=\cr
=\langle W^{-1}E^{-1} \rho(EWR_1^{-1}x), \rho(W^{-1}E^{-1}R_2^{-1}x) \rangle=
\langle E^{-1} \rho(Ey_2), W\rho(W^{-1}y_2) \rangle =\cr
=\langle E^{-1} \rho(y_2)),  W\rho(y_1) \rangle=\langle E^{-1} \rho(y_2)), \rho(y_2) \rangle\ ,
}
$$
so that the oriented angle $\psi_{R_2,R_1}(x)$ between $\rho_{R_2}(x)$ and $\rho_{R_1}(x)$
is actually the rotation angle of $E^{-1}$.

\end{proof}

\clearpage
\fancyhf{} \fancyfoot[CE,CO]{\thepage}
\fancyhead[CO]{\textit{Acknowledgments}}
\fancyhead[CE]{\textit{Acknowledgments}}
\renewcommand{\headrulewidth}{0.5pt}
\renewcommand{\footrulewidth}{0.0pt}
\addtolength{\headheight}{0.5pt}
\fancypagestyle{plain}{\fancyhead{}\renewcommand{\headrulewidth}{0pt}}

\chapter*{Acknowledgments}\addcontentsline{toc}{chapter}{Acknowledgments}

This is the part of my  thesis that maybe I  like most. Indeed, here I can write whatever I want, with no definition, label, rule...  For the same reasons  it is the most difficult for me,
as if I feel something,
then I deeply  feel  it. With no rule, logic step and reason.  And to write it down with no guideline and of course no usual word, I should work hard somehow.

If you think that what I will write for you is not enough, then yes, you are right.
But I am quite sure that you know how much indebted I am with you for your words, help, contribution... whatever.

\section*{Part 1}

\begin{center}

\emph{To Prof. Paolo Baldi:\\ you taught me fundamental things for life.  To not use too much the emph-style in \LaTeX,  how to deal with  the Grushin operator, where to see the painting ``Battaglia di San Romano'', why the novel ``Persuasion'' is great, how to fast destroy some Introduction section. {\rm To think}. \\
I thank you for each of these things, for the great opportunity to work with you, for
your necessary help and solving ideas during these research years (master and PhD), all the fruitful discussions, computations, proofs, mistakes  we did together. \\ But  most of all I sincerely thank you for your friendship.  }

\smallskip

\newpage

To Prof. Domenico Marinucci:\\ you made me improve several aspects of my life. Now I like both, the National Gallery and Clebsch-Gordan coefficients,  I appreciate very much the movie ``Million dollar baby'' as well as the Cosmic Microwave Background radiation,  I can show the best of Rome to some tourist in at most three hours, walking of course. \\
I thank you for all these improvements,
and for many other things. During these last years, your solving ideas, suggestions and help were valuable: I mean not only  the  discussions, long computations and proofs we did together, but also  the special opportunity you gave me to join your research project ERC Grant 277742 \emph{Pascal} and to meet great mathematicians all around the world... Thanks, a lot.

\smallskip

To Prof. Lucia Caramellino: \\
 I heartily thank you for your help, suggestions and ideas, for the time we spent together proving some theorem and most of all talking as good friends do.

\smallskip

To Prof. Giovanni Peccati:\\  I sincerely thank you for your help and  ideas, the time we spent together in Luxembourg
discussing abouth maths and the opportunity you gave me to visit your nice department, to meet your friendly research team and especially to work with you.

\smallskip

To Prof. Igor Wigman: \\
 I thank you for fruitful discussions we had  in London.   Your suggestions and  ideas were very important and I heartily thank you for your kindness, for the great opportunity you gave me to visit your amazing department and most of all to work with you on topics which are collected in this thesis.

\smallskip

 I wish to thank also Proff. Andrea Iannuzzi and Stefano Trapani \\for fruitful discussions on sereval topics and
valuable assistance in \cite{mauSO(3)}, which is completely collected in Chapter 3,
and  Gilles Becker  for the idea in Remark \ref{Gilles Becker}.

Finally, I sincerely thank all ERC \emph{Pascal} project members: Valentina Cammarota, Simon Campese and Claudio Durastanti for valuable comments on an earlier version of this thesis, Alessandro Renzi and Yabebal Fantaye for useful suggestions.

\end{center}

\section*{Part 2}

Many other people helped me to write this thesis, somehow.  First of all my friends, and of course my family.
It is not easy for me to write this part, as you know, so I will give just one thought for each of you. Only one, but full.

When reading your part, remember that the main reason for which your name is here is simple: I thank you for making me smile.

\smallskip
\begin{center}

To Giada: \\you are always close to  me. Everywhere.

\smallskip

To Camilla, Stefano, Chiara, Stefania, Eloisa, Serena and Erika: \\thanks for all the time we spent together - laughing, basically.

\smallskip

To Andrea, Vincenzo, Stefano, Claudio, Valentina, Gianluca, Simon, Alessandro and Yabebal:\\
I am so lucky to have colleagues and friends as nice as you are. My time at the Department would had been worse if any of you had not been there - but our favourite \emph{Cinsenke} (see e.g. The Red Bar) would had been useful, as well.

\smallskip

To my Cardio Combat Dance team:\\ sometimes, at night,  kicking something with you (the air e.g.) helped me to forget sad situations.

\smallskip

To everyone with whom I shared the office somewhere in the world, even if for few days: thanks.

\smallskip

\emph{Finally my greatest thanks. \\To my family: \\
every reason I think of, is too much deep  to
write it here.}
\end{center}

\vspace{17pt}

\begin{center}
No page of this thesis
would have had any meaning, if any of you had not been close to me.
\end{center}

\clearpage
\fancyhf{} \fancyfoot[CE,CO]{\thepage}
\fancyhead[CO]{\textit{Bibliography}}
\fancyhead[CE]{\textit{Biblyography}}
\renewcommand{\headrulewidth}{0.5pt}
\renewcommand{\footrulewidth}{0.0pt}
\addtolength{\headheight}{0.5pt}
\fancypagestyle{plain}{\fancyhead{}\renewcommand{\headrulewidth}{0pt}}

\addcontentsline{toc}{chapter}{Bibliography}

\providecommand{\bysame}{\leavevmode\hbox to3em{\hrulefill}\thinspace}
\providecommand{\MR}{\relax\ifhmode\unskip\space\fi MR }
\providecommand{\MRhref}[2]{%
  \href{http://www.ams.org/mathscinet-getitem?mr=#1}{#2}
}
\providecommand{\href}[2]{#2}

%

\begin{thebibliography}{10}

\bibitem{adlertaylor}
Robert~J. Adler and Jonathan~E. Taylor, \emph{Random fields and geometry},
  Springer Monographs in Mathematics, Springer, New York, 2007. \MR{2319516
  (2008m:60090)}

\bibitem{andrews}
George~E. Andrews, Richard Askey, and Ranjan Roy, \emph{Special functions},
  Encyclopedia of Mathematics and its Applications, vol.~71, Cambridge
  University Press, Cambridge, 1999. \MR{1688958 (2000g:33001)}

\bibitem{simon}
Ehsan Azmoodeh, Simon Campese, and Guillaume Poly, \emph{Fourth {M}oment
  {T}heorems for {M}arkov diffusion generators}, J. Funct. Anal. \textbf{266}
  (2014), no.~4, 2341--2359. \MR{3150163}

\bibitem{ld}
Paolo Baldi, Lucia Caramellino, and Maurizia Rossi, \emph{Large deviation
  asymptotics for the exit from a domain of the bridge of a general diffusion},
  Preprint (ArXiv 1406.4649) (2014).

\bibitem{sld}
\bysame, \emph{On sharp large deviations for the bridge of a general
  diffusion}, Preprint (ArXiv 1410.0863) (2014).

\bibitem{BMV}
Paolo Baldi, Domenico Marinucci, and Veeravalli~S. Varadarajan, \emph{On the
  characterization of isotropic {G}aussian fields on homogeneous spaces of
  compact groups}, Electron. Comm. Probab. \textbf{12} (2007), 291--302.
  \MR{2342708 (2008h:60019)}

\bibitem{mauSO(3)}
Paolo Baldi and Maurizia Rossi, \emph{On {L}\'evy's {B}rownian motion indexed
  by elements of compact groups}, Colloq. Math. \textbf{133} (2013), no.~2,
  227--236. \MR{3145515}

\bibitem{mauspin}
\bysame, \emph{Representation of {G}aussian isotropic spin random fields},
  Stochastic Process. Appl. \textbf{124} (2014), no.~5, 1910--1941.
  \MR{3170229}

\bibitem{balditrapani}
Paolo Baldi and Stefano Trapani, \emph{Fourier coefficients of invariant random
  fields on homogeneous spaces of compact {L}ie groups}, Ann. Inst. Henri
  Poincar\'e Probab. Stat. \textbf{51} (2015), no.~2, 648--671. \MR{3335020}

\bibitem{berry}
Michael~V. Berry, \emph{Regular and irregular semiclassical wavefunctions}, J.
  Phys. A 10 \textbf{12} (1977), 2083�--2091.

\bibitem{Berry2002}
\bysame, \emph{Statistics of nodal lines and points in chaotic quantum
  billiards: perimeter corrections, fluctuations, curvature}, J. Phys. A
  \textbf{35} (2002), no.~13, 3025--3038. \MR{1913853 (2003c:81060)}

\bibitem{bogomolnyschmit}
Eugene Bogomolny and Charles Schmit, \emph{Random wavefunctions and
  percolation}, J. Phys. A \textbf{40} (2007), no.~47, 14033--14043.
  \MR{2438110 (2010a:81072)}

\bibitem{cannarsa}
Haim Brezis, \emph{Functional analysis, {S}obolev spaces and partial
  differential equations}, Universitext, Springer, New York, 2011. \MR{2759829
  (2012a:35002)}

\bibitem{vale3}
Valentina Cammarota and Domenico Marinucci, \emph{On the limiting behaviour of
  needlets polyspectra}, Ann. Inst. Henri Poincar\'e Probab. Stat. (In press).

\bibitem{vale2}
\bysame, \emph{The stochastic properties of {$\ell^1$}-regularized spherical
  {G}aussian fields}, Appl. Comput. Harmon. Anal. \textbf{38} (2015), no.~2,
  262--283. \MR{3303675}

\bibitem{fluct}
Valentina Cammarota, Domenico Marinucci, and Igor Wigman, \emph{Fluctuations of
  the euler-poincar\'e characteristic for random spherical harmonics}, Preprint
  (ArXiv 1504.01868) (2011).

\bibitem{simonS}
Simon Campese, \emph{Optimal convergence rates and one-term edge-worth
  expansions for multidimensional functionals of {G}aussian fields}, ALEA Lat.
  Am. J. Probab. Math. Stat. \textbf{10} (2013), no.~2, 881--919. \MR{3149458}

\bibitem{simonmaudom}
Simon Campese, Domenico Marinucci, and Maurizia Rossi, \emph{Gaussian behavior
  for hyperspherical eigenfunctions}, Preprint (2015).

\bibitem{cheng2}
Dan Cheng, \emph{Excursion probabilities of isotropic and locally isotropic
  gaussian random fields on manifolds}, Preprint (ArXiv 1504.08047) (2015).

\bibitem{cheng}
\bysame, \emph{Excursion probability of certain non-centered smooth gaussian
  random fields}, Preprint (ArXiv 1502.04414) (2015).

\bibitem{chenschwar}
Dan Cheng and Armin Schwartzman, \emph{On the explicit height distribution and
  expected number of local maxima of isotropic gaussian random fields},
  Preprint (ArXiv 1503.01328) (2015).

\bibitem{yimin-cheng}
Dan Cheng and Yimin Xiao, \emph{The mean euler characteristic and excursion
  probability of gaussian random fields with stationary increments}, Preprint
  (ArXiv 1211.6693) (2012).

\bibitem{yimin}
\bysame, \emph{Excursion probability of gaussian random fields on sphere},
  Preprint (ArXiv 1401.5498) (2014).

\bibitem{Ci}
Javier Cilleruelo, \emph{The distribution of the lattice points on circles}, J.
  Number Theory \textbf{43} (1993), no.~2, 198--202. \MR{1207499 (94c:11097)}

\bibitem{dehlingtaqqu}
Herold Dehling and Murad~S. Taqqu, \emph{The empirical process of some
  long-range dependent sequences with an application to {$U$}-statistics}, Ann.
  Statist. \textbf{17} (1989), no.~4, 1767--1783. \MR{1026312 (91c:60025)}

\bibitem{D-F}
Harold Donnelly and Charles Fefferman, \emph{Nodal sets of eigenfunctions on
  {R}iemannian manifolds}, Invent. Math. \textbf{93} (1988), no.~1, 161--183.
  \MR{943927 (89m:58207)}

\bibitem{D}
Richard~M. Dudley, \emph{Real analysis and probability}, Cambridge Studies in
  Advanced Mathematics, vol.~74, Cambridge University Press, Cambridge, 2002,
  Revised reprint of the 1989 original. \MR{1932358 (2003h:60001)}

\bibitem{claudiospin}
Claudio Durastanti, Daryl Geller, and Domenico Marinucci, \emph{Adaptive
  nonparametric regression on spin fiber bundles}, J. Multivariate Anal.
  \textbf{104} (2012), 16--38. \MR{2832184 (2012j:62141)}

\bibitem{faraut}
Jacques Faraut, \emph{Analysis on {L}ie groups}, Cambridge Studies in Advanced
  Mathematics, vol. 110, Cambridge University Press, Cambridge, 2008, An
  introduction. \MR{2426516 (2009f:22009)}

\bibitem{MR2088027}
Sylvestre Gallot, Dominique Hulin, and Jacques Lafontaine, \emph{Riemannian
  geometry}, third ed., Universitext, Springer-Verlag, Berlin, 2004.
  \MR{2088027 (2005e:53001)}

\bibitem{GAN:60}
Ramesh Gangolli, \emph{Positive definite kernels on homogeneous spaces and
  certain stochastic processes related to {L}\'evy's {B}rownian motion of
  several parameters}, Ann. Inst. H. Poincar\'e Sect. B (N.S.) \textbf{3}
  (1967), 121--226. \MR{0215331 (35 \#6172)}

\bibitem{marinuccigeller}
Daryl Geller and Domenico Marinucci, \emph{Spin wavelets on the sphere}, J.
  Fourier Anal. Appl. \textbf{16} (2010), no.~6, 840--884. \MR{2737761
  (2011j:42064)}

\bibitem{Husemoller}
Dale Husemoller, \emph{Fibre bundles}, third ed., Graduate Texts in
  Mathematics, vol.~20, Springer-Verlag, New York, 1994. \MR{1249482
  (94k:55001)}

\bibitem{AmP}
Manjunath Krishnapur, P{\"a}r Kurlberg, and Igor Wigman, \emph{Nodal length
  fluctuations for arithmetic random waves}, Ann. of Math. (2) \textbf{177}
  (2013), no.~2, 699--737. \MR{3010810}

\bibitem{KW}
Par Kurlberg and Igor Wigman, \emph{On probability measures arising from
  lattice points on circles}, Preprint (ArXiv 1501.01995) (2015).

\bibitem{bib:LS}
Nikolai Leonenko and Ludmila Sakhno, \emph{On spectral representations of
  tensor random fields on the sphere}, Stoch. Anal. Appl. \textbf{30} (2012),
  no.~1, 44--66. \MR{2870527}

\bibitem{levy}
Paul L`'evy, \emph{Le mouvement brownien fonction d'un point de la sph\'ere de
  riemann}, Rendiconti del Circolo Matematico di Palermo, vol.~8, 1959,
  pp.~297--310.

\bibitem{malyarenko}
Anatoliy Malyarenko, \emph{Invariant random fields in vector bundles and
  application to cosmology}, Ann. Inst. Henri Poincar\'e Probab. Stat.
  \textbf{47} (2011), no.~4, 1068--1095. \MR{2884225 (2012k:60151)}

\bibitem{malyarenkobook}
\bysame, \emph{Invariant random fields on spaces with a group action},
  Probability and its Applications (New York), Springer, Heidelberg, 2013, With
  a foreword by Nikolai Leonenko. \MR{2977490}

\bibitem{dogiocam}
Domenico Marinucci and Giovanni Peccati, \emph{Random fields on the sphere},
  London Mathematical Society Lecture Note Series, vol. 389, Cambridge
  University Press, Cambridge, 2011, Representation, limit theorems and
  cosmological applications. \MR{2840154 (2012j:60136)}

\bibitem{maupec}
\bysame, \emph{Mean-square continuity on homogeneous spaces of compact groups},
  Electron. Commun. Probab. \textbf{18} (2013), no. 37, 10. \MR{3064996}

\bibitem{mistosfera}
Domenico Marinucci, Giovanni Peccati, Maurizia Rossi, and Igor Wigman,
  \emph{High-energy limit distribution of nodal lengths for random spherical
  harmonics}, In preparation (2015).

\bibitem{misto}
\bysame, \emph{Non-{G}aussian fluctuations for arithmetic random waves}, In
  preparation (2015).

\bibitem{maudom}
Domenico Marinucci and Maurizia Rossi, \emph{Stein-{M}alliavin approximations
  for nonlinear functionals of random eigenfunctions on {$\Bbb{S}^d$}}, J.
  Funct. Anal. \textbf{268} (2015), no.~8, 2379--2420. \MR{3318653}

\bibitem{def}
Domenico Marinucci and Igor Wigman, \emph{The defect variance of random
  spherical harmonics}, J. Phys. A: Math. Theor. 44 355206 (2011).

\bibitem{wigexc}
\bysame, \emph{On the area of excursion sets of spherical {G}aussian
  eigenfunctions}, J. Math. Phys. \textbf{52} (2011), no.~9, 093301, 21.
  \MR{2867816 (2012k:60102)}

\bibitem{Nonlin}
\bysame, \emph{On nonlinear functionals of random spherical eigenfunctions},
  Comm. Math. Phys. \textbf{327} (2014), no.~3, 849--872. \MR{3192051}

\bibitem{meckes}
Elizabeth Meckes, \emph{On the approximate normality of eigenfunctions of the
  {L}aplacian}, Trans. Amer. Math. Soc. \textbf{361} (2009), no.~10,
  5377--5399. \MR{2515815 (2010f:58046)}

\bibitem{MR0163331}
John Milnor, \emph{Morse theory}, Based on lecture notes by M. Spivak and R.
  Wells. Annals of Mathematics Studies, No. 51, Princeton University Press,
  Princeton, N.J., 1963. \MR{0163331 (29 \#634)}

\bibitem{nazarovsodin}
Fedor Nazarov and Mikhail Sodin, \emph{On the number of nodal domains of random
  spherical harmonics}, Amer. J. Math. \textbf{131} (2009), no.~5, 1337--1357.
  \MR{2555843 (2011b:60098)}

\bibitem{bib:NP}
Ezra~T. Newman and Roger Penrose, \emph{Note on the {B}ondi-{M}etzner-{S}achs
  group}, J. Mathematical Phys. \textbf{7} (1966), 863--870. \MR{0194172 (33
  \#2385)}

\bibitem{nou-pe2}
Ivan Nourdin and Giovanni Peccati, \emph{Stein's method on {W}iener chaos},
  Probab. Theory Related Fields \textbf{145} (2009), no.~1-2, 75--118.
  \MR{2520122 (2010i:60087)}

\bibitem{noupebook}
\bysame, \emph{Normal approximations with {M}alliavin calculus}, Cambridge
  Tracts in Mathematics, vol. 192, Cambridge University Press, Cambridge, 2012,
  From Stein's method to universality. \MR{2962301}

\bibitem{nou-pe}
\bysame, \emph{Poisson approximations on the free {W}igner chaos}, Ann. Probab.
  \textbf{41} (2013), no.~4, 2709--2723. \MR{3112929}

\bibitem{taqqu}
Giovanni Peccati and Murad~S. Taqqu, \emph{Limit theorems for multiple
  stochastic integrals}, ALEA Lat. Am. J. Probab. Math. Stat. \textbf{4}
  (2008), 393--413. \MR{2461790 (2010e:60057)}

\bibitem{P-T}
\bysame, \emph{Wiener chaos: moments, cumulants and diagrams}, Bocconi \&
  Springer Series, vol.~1, Springer, Milan; Bocconi University Press, Milan,
  2011, A survey with computer implementation, Supplementary material available
  online. \MR{2791919 (2012d:60157)}

\bibitem{peccatitudor}
Giovanni Peccati and Ciprian~A. Tudor, \emph{Gaussian limits for vector-valued
  multiple stochastic integrals}, S\'eminaire de {P}robabilit\'es {XXXVIII},
  Lecture Notes in Math., vol. 1857, Springer, Berlin, 2005, pp.~247--262.
  \MR{2126978 (2006i:60071)}

\bibitem{pham}
Viet-Hung Pham, \emph{On the rate of convergence for central limit theorems of
  sojourn times of {G}aussian fields}, Stochastic Process. Appl. \textbf{123}
  (2013), no.~6, 2158--2174. \MR{3038501}

\bibitem{mauproc}
Maurizia Rossi, \emph{On the high-energy behavior of nonlinear functionals of
  random eienfunctions on $\mathbb s^d$}, Proc. EYSM in Prague 2015 (ArXiv
  1506.01841) (2015).

\bibitem{mau}
\bysame, \emph{On the high-energy defect distribution for random hyperspherical
  harmonics}, In preparation (2015).

\bibitem{RW}
Ze{\'e}v Rudnick and Igor Wigman, \emph{On the volume of nodal sets for
  eigenfunctions of the {L}aplacian on the torus}, Ann. Henri Poincar\'e
  \textbf{9} (2008), no.~1, 109--130. \MR{2389892 (2009b:58072)}

\bibitem{schoenberg}
I.~J. Schoenberg, \emph{Metric spaces and positive definite functions}, Trans.
  Amer. Math. Soc. \textbf{44} (1938), no.~3, 522--536. \MR{1501980}

\bibitem{sugiura}
Mitsuo Sugiura, \emph{Unitary representations and harmonic analysis}, second
  ed., North-Holland Mathematical Library, vol.~44, North-Holland Publishing
  Co., Amsterdam; Kodansha, Ltd., Tokyo, 1990, An introduction. \MR{1049151
  (91c:22028)}

\bibitem{szego}
G{\'a}bor Szeg{\H{o}}, \emph{Orthogonal polynomials}, fourth ed., American
  Mathematical Society, Providence, R.I., 1975, American Mathematical Society,
  Colloquium Publications, Vol. XXIII. \MR{0372517 (51 \#8724)}

\bibitem{TKU}
Shigeo Takenaka, Izumi Kubo, and Hajime Urakawa, \emph{Brownian motion
  parametrized with metric space of constant curvature}, Nagoya Math. J.
  \textbf{82} (1981), 131--140. \MR{618812 (83m:60071)}

\bibitem{B-D}
E.~T. Whittaker and G.~N. Watson, \emph{Representations of compact lie groups},
  Bull. Amer. Math. Soc. (N.S.), Bull. Amer. Math. Soc. (N.S.), 1987.

\bibitem{WW}
\bysame, \emph{A course of modern analysis}, Cambridge Mathematical Library,
  Cambridge University Press, Cambridge, 1996, An introduction to the general
  theory of infinite processes and of analytic functions; with an account of
  the principal transcendental functions, Reprint of the fourth (1927) edition.
  \MR{1424469 (97k:01072)}

\bibitem{Wig}
Igor Wigman, \emph{Fluctuations of the nodal length of random spherical
  harmonics}, Comm. Math. Phys. \textbf{298} (2010), no.~3, 787--831.
  \MR{2670928 (2012f:58073)}

\bibitem{wigsurvey}
\bysame, \emph{On the nodal lines of random and deterministic {L}aplace
  eigenfunctions}, Spectral geometry, Proc. Sympos. Pure Math., vol.~84, Amer.
  Math. Soc., Providence, RI, 2012, pp.~285--297. \MR{2985322}

\bibitem{Yau2}
S.-T. Yau, \emph{Open problems in geometry}, J. Ramanujan Math. Soc.
  \textbf{15} (2000), no.~2, 125--134. \MR{1754714 (2001c:53001)}

\bibitem{Yau}
Shing~Tung Yau, \emph{Survey on partial differential equations in differential
  geometry}, Seminar on {D}ifferential {G}eometry, Ann. of Math. Stud., vol.
  102, Princeton Univ. Press, Princeton, N.J., 1982, pp.~3--71. \MR{645729
  (83i:53003)}

\end{thebibliography}

\end{document}